\documentclass[11pt]{amsart}

\usepackage{amsmath, amsthm, amssymb}
\pdfoutput=1

\usepackage[T1]{fontenc}                    
\usepackage{newpxtext}                      
\usepackage[bigdelims,vvarbb]{newpxmath}    
\usepackage[scale=0.95]{sourcecodepro}      
\makeatletter
\def\section{\@startsection{section}{1}%
  \z@{2.5ex plus 1ex minus .2ex}%
  {1.6ex plus .2ex}%
  {\centering\normalfont\bfseries\large}}
\makeatother



\usepackage[abs]{overpic}		  

\usepackage{xcolor}
\definecolor{indigo}{rgb}{0.29, 0.0, 0.51}  
\definecolor{dgreen}{RGB}{0, 161, 75}
\usepackage[colorlinks, urlcolor=indigo, linkcolor=indigo, citecolor=indigo]{hyperref}

\usepackage{enumitem}

\usepackage[hcentering, vcentering, total={5.9in, 8.3in}]{geometry}  

\theoremstyle{plain}
\newtheorem{theorem}{Theorem}
\newtheorem{corollary}[theorem]{Corollary}
\newtheorem{proposition}[theorem]{Proposition}
\newtheorem{lemma}[theorem]{Lemma}
\newtheorem{question}[theorem]{Question}
\newtheorem{conjecture}[theorem]{Conjecture}

\theoremstyle{definition}
\newtheorem{definition}[theorem]{Definition}

\theoremstyle{remark}
\newtheorem{remark}[theorem]{Remark}
\newtheorem{example}[theorem]{Example}

\numberwithin{theorem}{section}

\newcommand{\dfn}[1]{{\em #1}}        
\newcommand{\R}{\mathbb{R}}           
\newcommand{\Q}{\mathbb{Q}}           
\newcommand{\Z}{\mathbb{Z}}           
\newcommand{\C}{\mathbb{C}}           
\newcommand{\N}{\mathbb{N}}           

\makeatletter
\newcommand*\bigcdot{\mathpalette\bigcdot@{0.6}}
\newcommand*\bigcdot@[2]{\mathbin{\vcenter{\hbox{\scalebox{#2}{$\m@th#1\bullet$}}}}}
\makeatother

\newcommand{\vects}[2]{\left(\begin{smallmatrix} #1 \\ #2 \end{smallmatrix}\right)}  

\DeclareMathOperator{\Tight}{Tight}
\DeclareMathOperator\tb{tb}                   
\DeclareMathOperator\rot{rot}                 
\DeclareMathOperator\self{sl}                 
\DeclareMathOperator\tw{tw}                   
\DeclareMathOperator{\tor}{tor}               
\DeclareMathOperator{\dep}{d}
\DeclareMathOperator{\ten}{t}    

\DeclareMathOperator{\spinc}{\mathrm{Spin}^{\textit{c}}}  

\DeclareFontFamily{U} {cmr}{}
\DeclareFontShape{U}{cmr}{m}{n}{
  <-6> cmr5
  <6-7> cmr6
  <7-8> cmr7
  <8-9> cmr8
  <9-10> cmr9
  <10-12> cmr8
  <12-> cmr9}{}
\DeclareSymbolFont{Xcmr} {U} {cmr}{m}{n}
\DeclareMathSymbol{\Phi}{\mathord}{Xcmr}{8}

\begin{document}

\title{Non-loose torus knots}

\author{John B. Etnyre}

\author{Hyunki Min}

\author{Anubhav Mukherjee}

\address{School of Mathematics \\ Georgia Institute of Technology \\  Atlanta, GA}
\email{etnyre@math.gatech.edu}

\address{Department of Mathematical Sciences \\ KAIST \\ Daejeon, Korea}
\email{hmin27@kaist.ac.kr}

\address{School of Mathematics \\ Georgia Institute of Technology \\ Atlanta, GA}
\email{etnyre@math.gatech.edu}

\begin{abstract}
We give a complete coarse classification of Legendrian and transverse torus knots in any contact structure on $S^3$. We also give the first examples of virtually overtwisted contact structures that have (arbitrarily large) Giroux torsion. 
\end{abstract}

\maketitle
\tableofcontents

\section{Introduction}
The study of Legendrian and transverse knots in contact $3$-manifolds has gone hand in hand with the development and application of contact geometry, with many key features (like tightness) and constructions (like Legendrian and transverse surgery) relying on them. Moreover, over the last 20 or so years, a rich and beautiful theory of Legendrian and transverse knots has developed. However, there has been surprisingly little work on Legendrian knots in overtwisted contact manifolds. This might partially be due to the fact that overtwisted contact structures are classified and are determined by their algebraic topology \cite{Eliashberg89}. However, non-loose Legendrian and transverse knots in overtwisted contact structures, which are those with tight complements, are of great interest. For example, Legendrian surgery on a non-loose Legendrian knot might produce tight contact structures and hence be the key to the classification of tight contact structures on certain manifolds. Indeed, in a forthcoming paper by the first two authors, Tosun and Varvarezos~\cite{EtnyreMinTosunpre}, the results in this paper will be used to classify tight contact structures on some small Seifert fibered spaces where such a classification has remained elusive. The result in this paper will also be used in joint work of the first two authors and Piccirillo and Roy \cite{EtnyreMinPiccirilloRoyPre} to construct explicit symplectic embeddings of rational homology balls into $\C P^2$ and its blowups. 
In addition, the results in this paper will illuminate many new features of non-loose knots, showing that there is as rich a structure to them as for the much studied Legendrian and transverse knots in tight contact manifolds. 

In this paper, we give a complete coarse classification of non-loose Legendrian and transverse torus knots in any overtwisted contact structure on $S^3$. Combining with the results of~\cite{Etnyre13} and ~\cite{EtnyreHonda01}, this completes the coarse classification of Legendrian and transverse torus knots in any contact structure on $S^3$. In particular, this gives the first classification of Legendrian knots that involves Giroux torsion in their complement. The existence of such knots was shown in \cite{Etnyre13}, but all previous classification results only considered the case without Giroux torsion. 

Previously, non-loose knots were only completely classified for Legendrian and transverse representatives of the unknot \cite{EliashbergFraser98}. There have been several partial classifications for other knots \cite{Etnyre13} and in particular torus knots \cite{GeigesOnaran20a, Matkovic20Pre}. In Section~\ref{thealgorithm}, we will give a simple algorithm to classify non-loose Legendrian and transverse torus knots. In Section~\ref{sec:2n+1} and~\ref{sec:(5,8)}, we will also give closed form classifications for non-loose Legendrian and transverse $(2, \pm(2n+1))$-torus knots and $(5,\pm 8)$-torus knots. 

\subsection{New techniques}
The proofs of our main results relay on two main ingredients, convex surfaces and the geometry of the Farey graph. While convex surface theory is now a well-known part of contact geometry, we have to develop several new techniques. The most interesting might be the ability to add Giroux torsion to some virtually overtwisted contact structures, see Lemma~\ref{lem:staytight} and its associated lemmas, which relies on what appears to be a novel application of Honda's work on tricky non-rotative layers \cite{Honda01}. To the authors' knowledge, all previous work involving Giroux torsion --- its existence or adding it to an existing contact structure --- has been restricted to universally tight contact structures. In fact, Conjecture~1.4 of Giroux \cite{Giroux00}, in part, asks if on a closed oriented $3$-manifold there are only finitely many isotopy classes of tight but virtually overtwisted contact structures, which would imply that one could not add an arbitrary amount of Giroux torsion to such contact structures. (In past examples, there was no Giroux torsion on tight but virtually overtwisted contact structures, so it would seem that one might not be able to add any torsion to a virtually overtwisted contact structure while preserving tightness.) We give the first examples of the existence of Giroux torsion for tight but virtually overtwisted contact structures and, in fact, produce infinite families of tight but virtually overtwisted contact structures on a fixed manifold. These examples are on the complement of some torus knots, but in a future work, the first two authors will give similar examples on closed $3$-manifolds \cite{EtnyreMinPre}, thus disproving the above mentioned conjecture. 

Another interesting technique is the ability to detect non-loose knots by carefully applying the state transition technique in overtwisted contact structures. This allows us to determine when all Legendrian $(p,q)$-knots are non-loose without relying on contact surgery diagrams or invariants from Heegaard Floer homology. This is done for $(p,q)$-torus knots with $\tb<pq$ in Propositions~\ref{wings} and~\ref{diamonds} and the same arguments work for $\tb=pq$ as well; moreover, Propositions~\ref{propxwing} and~\ref{inftyV} show that all non-loose knots with $\tb>pq$ are destabilizations of ones with $\tb=pq$, so the non-looseness can also be seen without contact surgery diagrams or invariants from Heegaard Floer homology. 

We also develop new techniques 
to analyze pairs of paths in the Farey graph that approach a given fraction from different directions, see Section~\ref{subsec:pathsinFG}. From this, we can, among other things, determine when two non-loose knots stabilize to become equivalent, see Propositions~\ref{merge} and~\ref{alldiamonds}. We can also use this to calculate the classical invariants of non-loose torus knots without relying on contact surgery diagrams, see Lemma~\ref{computer}; and this, in particular, allows us to distinguish non-loose Legendrian knots after adding Giroux torsion to their complements, see Lemma~\ref{postorsion}  (there does not seem to be a way to do this using the more classical surgery diagram approach to computing rotation numbers). 

\subsection{Prior classification results}\label{sec:prior}
So far, with the exception of \cite{Etnyre13, GeigesOnaran15}, non-loose Legendrian knots have only been studied in $S^3$. To discuss these results, we recall that Eliashberg \cite{Eliashberg89} classified overtwisted contact structures on all $3$-manifolds, and on the $3$-sphere they are in one-to-one correspondence with $\Z$. We will denote by $\xi_n$, for $n\in\Z$, the overtwisted contact structure on $S^3$ with $d_3(\xi_n)=n$. See Section~\ref{htpclasses} for the definition of the $d_3$-invariant.\footnote{We adopt the convention that the $d_3$-invariant of contact structures on $S^3$ are integers and the standard tight contact structure has $d_3$-invariant $0$. This differs from the original definition in \cite{Gompf98} by $1/2$. We also note that some papers enumerate overtwisted contact structures with their Hopf invariant which is the negation of the $d_3$-invariant.}

\subsubsection{Basic notation for Legendrian and transversal knots.}
We denote the Thurston-Bennequin invariant of a Legendrian knot $L$ by $\tb(L)$ and its rotation number by $\rot(L)$. For a transverse knot $T$ we denote its self-linking number by $\self(T)$. We use the notation $S_+(L)$ to denote the positive stabilization of $L$ and $S_-(L)$ its negative stabilization. We also denote the amount of convex torsion in the complement of a standard neighborhood of a Legendrian knot $L$ by $\tor(L)$. See Section~\ref{oldclassification} for the definition of convex torsion. Note that all classification results so far, except those in \cite{Etnyre13}, only considered Legendrian knots with $\tor = 0$. We say knots are \dfn{coarsely classified} if they are classified up to co-orientation preserving contactomorphism, smoothly isotopic to the identity. It is well-known that loose knots are coarsely classified by the classical invariants \cite[Theorem~1.4]{Etnyre13}. We say two Legendrian knots $L_1$ and $L_2$ are \dfn{equivalent} if there is a co-orientation preserving contactomorphism, smoothly isotopic to the identity, sending $L_1$ to $L_2$, and similarly for transverse knots. 

\subsubsection{Non-loose unknots}
Previously, the only knot type for which there was a complete coarse classification of non-loose knots was the unknot. In \cite{EliashbergFraser98}, Eliashberg and Fraser showed that only the contact structure $\xi_1$ supports non-loose Legendrian unknots and the non-loose representatives are: $L_\pm^i$ for $i\geq 2$ and $L^1$ such that 
\begin{align*}
  \tb(L_\pm^i)=i &\text{ and } \rot(L_\pm^i)=\pm(i-1),\\
  \tb(L^1) = 1 &\text{ and } \rot(L^1) = 0,
\end{align*}
and satisfy
\begin{align*}
  &S_\pm(L_\pm^{i+1})=L_\pm^i \text{ and } S_\pm(L_\pm^2)=L^1,\\ 
  &S_\mp(L_\pm^i) \text{ and } S_\pm(L^1) \text{ are loose.}
\end{align*}
From this, one can also see that there are no non-loose transverse unknots. 

To visualize the classical invariants of Legendrian knots, we consider the \dfn{mountain range} of Legendrian knots for a given smooth knot type. \begin{definition}
Given a smooth knot type $K$ and a fixed contact structure $\xi$, we denote by $\mathcal{L}(K)$ the set of Legendrian knots in $\xi$ up to coarse equivalence\footnote{Usually, $\mathcal{L}(K)$ denotes the set of Legendrian knots up to Legendrian isotopy. In this paper however, since we only consider coarse classification, we adopt this definition.} and consider a map $\Phi:\mathcal{L}(K)\to \Z^2$ that sends $L\in \mathcal{L}(K)$ to $(\rot(L), \tb(L))$. The image of $\Phi$ is called the \dfn{mountain range} of $K$. We can also restrict $\Phi$ to the subset $\mathcal{L}_{nl}(K)$ of non-loose Legendrian knots realizing $K$; and since we completely understand loose Legendrian knots isotopic to $K$, we will refer to the mountain range of $K$ in some overtwisted contact structure, as the image of $\Phi$ restricted to $\mathcal{L}_{nl}(K)$. We also note that a non-loose knot can have convex torsion in its complement. So when discussing mountain ranges, we will sometimes emphasize if the non-loose knots have no convex torsion in their complement (we will say $\tor=0$) or specify what the torsion is for given points in the mountain range.
\end{definition} 

\subsubsection{Partial results for non-loose torus knots}
In \cite{GeigesOnaran20a}, Geiges and Onaran gave the next coarse classification results for some torus knots with specific classical invariants. They considered only ``strongly exceptional" knots. The term exceptional is what we are calling non-loose, and strongly means there is no Giroux torsion in the complement. In this paper we will say such knots are non-loose without convex torsion, or non-loose knots with $\tor=0$, see Definition~\ref{defnoftor}. Their results are as follows.

\begin{itemize}
  \item[] {\bf Left-handed trefoil:} There are exactly two non-loose Legendrian representatives without convex torsion with $\tb=-5$ or $\tb<-6$, and there is at least one with $\tb=1$ and at least two for all other values of $\tb$. All these examples are in $\xi_2$.\\
  \item[] {\bf Right-handed trefoil:} There are exactly four non-loose Legendrian knots having $\tor=0$ and $\tb=7$. Two have $\rot=\pm 4$ and live in $\xi_1$ and the other two have $\rot=\pm 8$ and live in $\xi_{-1}$. They also constructed non-loose Legendrian knots in $\xi_{-1}$ with $\tb$ realizing any integer less than or equal to $5$ and such Legendrian knots in $\xi_{1}$ with $\tb$ realizing any integer greater than or equal to $6$. \\
  \item[] {\bf Other torus knots:} For $p \geq 2$ and $n \geq 1$, there are exactly $2p$ non-loose Legendrian $(p,np+1)$-knots having $\tor=0$ and $\tb=np^2+p+1$. If $n \geq 2$, then there are exactly $2(p-1)(n-1)$ such non-loose Legendrian $(p, -(np-1))$-torus knots having $\tor=0$ and $\tb=-np^2+p-1$. They also worked out the rotation numbers of these knots and which overtwisted contact structures in which they live.  
\end{itemize}

In \cite{Matkovic20Pre}, Matkovi\v c coarsely classified non-loose negative $(p,q)$-torus knots with $\tb < pq$ and $\tor =0$. The classification is in terms of specific contact surgery descriptions in Figure~\ref{fig:torus-knots}, and to determine if a given surgery description is non-loose one must determine if Legendrian surgery on the knot produces a tight contact manifold. This is translated into information about the rotation numbers in the contact surgery diagram; which, in turn, is equivalent to our ``pairs of decorated paths'' description used in our classification algorithm given in Section~\ref{thealgorithm}.

\subsubsection{Other results on non-loose Legendrian knots}
While we do not consider links in this paper, we do mention that Geiges and Onaran have coarsely classified all non-loose Legendrian Hopf links (including ones with $\tor > 0$) in \cite{GeigesOnaran20b}. This and Eliashberg and Fraser's result above are the only complete coarse classification of non-loose representatives of a link type. 

The only results involving the classification up to Legendrian isotopy (not just the coarse classification) of non-loose Legendrian knots is the work of Vogel, \cite{Vogel18}. He showed that for each $\tb$ and $\rot$ realized by a non-loose unknot above, there are exactly two non-loose unknots up to Legendrian isotopy. We believe that some of our results below can also be upgraded to classifications up to Legendrian isotopy, but that will be the subject of future work. 

We also note that in addition to the above works, there have been many constructions of non-loose knots, see for example \cite{Ghiggini06b, Ghiggini06, LiscaOzsvathStipsiczSzab09}. 

\subsection{General non-loose torus knots}\label{introgeneral}
From now on, we assume that $|q| > p > 0$. In Section~\ref{thealgorithm}, we will give a simple algorithm to classify non-loose Legendrian and transverse representatives of any $(p,q)$-torus knot. This algorithm is justified in Section~\ref{nonloosetorusknots} after a careful analysis of the tight contact structure on the complement of torus knots is carried out in Section~\ref{classificationoncomplement}. In Section~\ref{overview} will will outline the strategy for classifying non-loose torus knots to motivate the many technical results that are proven in the rest of Section~\ref{classificationoncomplement}. (To understand Section~\ref{overview}, one will need to read Section~\ref{pairsodecorated} and~\ref{subsec:pathsinFG} for the notation used, and possibly some earlier sections if one is not familiar with paths in the Farey graph.) In this section, we discuss the properties of the classifications that have a simple closed form. 
The first observation from our classification concerns destabilizing non-loose Legendrian knots. 

\begin{theorem}\label{gen1}
  Any non-loose $(p,q)$-torus knot with $\tb\not=pq$ destabilizes if $pq>0$. Similarly, any non-loose $(p,q)$-torus knot with $\tb\not=pq$ or $|pq|-|p|-|q|$ destabilizes if $pq<0$. 
   Non-loose Legendrian knots with $\tb=pq$ sometimes destabilize and sometimes do not. 
\end{theorem}

We can also restrict the potential $d_3$-invariants of overtwisted contact structures that support non-loose torus knots. 

\begin{theorem}\label{parity}
Suppose $\xi$ is an overtwisted contact structure on $S^3$ supporting a non-loose Legendrian $(p,q)$-torus knot $L$ with $\tor(L)=n$ and if $pq>0$ then assume that $\xi$ is not $\xi_0$ or $\xi_{-pq+p+q}$, then 
  \[
    d_3(\xi)=
    \begin{cases}
    \text{odd} & \text{ if $pq>0$ and $n \in \Z$, or $pq<0$ and $n$ is a half-integer,}\\
    \text{even} & \text{ if $pq<0$ and $n \in \Z$, or $pq>0$ and $n$ is a half-integer.}
    \end{cases}
  \]
When $pq>0$ and $\xi=\xi_0$, respectively $\xi_{-pq+p+q}$, then knots with $\tb> pq-p-q$ have half-integer, respectively integer, torsion, while those with $\tb\leq pq-p-q$ have integer, respectively half-integer, torsion. 
\end{theorem}

To state our classification, we first let $q/p$ be a rational number with $|q/p|>1$. If $pq<0$, we have the continued fraction
\[
  \frac qp = [a_1,\ldots, a_m] = a_1-\frac{1}{a_2-\frac{1}{\cdots - \frac{1}{a_m}}}
\]
where $a_i \leq -2$ for $1 \leq i \leq m$. If $pq>0$, we consider the continued fraction
\[
  \left( \frac pq -1 \right)^{-1}=[a_1,\ldots, a_m].
\]
In either case, we consider the continued fraction
\[
  \left(\frac qp - \left\lceil \frac qp\right\rceil\right)^{-1}=[b_1,\ldots, b_n].
\]
We now set
\[
  m(p,q)=|(a_1+1)\cdots(a_{m-1}+1)a_m|\cdot|(b_1+1)\cdots (b_{n-1}+1)b_n|
\]
and
\[
  n(p,q)=|(a_1+1)\cdots(a_m+1)|\cdot|(b_1+1)\cdots (b_n+1)|.
\]
According to Giroux and Honda's classification of tight contact structures on solid tori, which we recall in Section~\ref{oldclassification}, $m(p,q)$ is the number of tight contact structures on a solid torus with convex boundary having two dividing curves of slope $q/p$ times the number of such contact structures with dividing slope $p/q$, and $n(p,q)$ is the number of tight contact structures on $L(p,-q)\#L(q,-p)$, where we use the convention that $L(p,q)$ is ${-p/q}$-surgery on the unknot.

We can now enumerate all the non-loose Legendrian $(p,q)$-torus knots with $\tb>pq$ and $\tor = 0$. 

\begin{theorem}\label{thm:>pq} 
  Suppose $|q| > p > 1$. There are exactly $2n(p,q)$ non-loose $(p,q)$-torus knots with $\tb = i>pq$ and $\tor = 0$ which we denote by 
  \[
    L_{\pm,k}^{i}\, \text{ for }1\leq k\leq n(p,q)
  \]
  except when $pq<0$ and $i=|pq|-|p|-|q|$. We also know that $\rot(L_{+,k}^{i})=-\rot(L_{-,k}^{i})$  and
  \begin{align*}
    &S_\pm(L_{\pm,k}^{i})=L_{\pm,k}^{i-1}\, \text{ for $i>pq+1$,}\\&S_\mp(L_{\pm,k}^{i}) \text{ is loose for $i > pq$.}
  \end{align*}
Moreover, $S^j_\pm(L_{\pm,k}^{pq+1})$ is non-loose for any $j > 0$. Any $L_{\pm,k}^{i}$ can be realized as a Legendrian knot shown in Figure~\ref{fig:tb=pq+m}. 
  
 When $pq<0$ and $\tb=|pq|-|p|-|q|$, there are $2n(p,q)+1$ non-loose representatives. Of these, $2n(p,q)$ are denoted as above and have the same properties under stabilization. The extra Legendrian knot is denoted by $L_e$. We have $\rot(L_e)=0$ and
  \[
    S_\pm(L_e)=L_{\pm, 1}^{|pq|-|p|-|q|-1}.
  \]
Moreover, $L_e$ is contained in the contact structure $\xi_{|pq|-|p|-|q|+1}$.
\end{theorem}

\begin{figure}[htbp]{\tiny
  \vspace{0.1cm}
  \begin{overpic}[scale=1,tics=10]{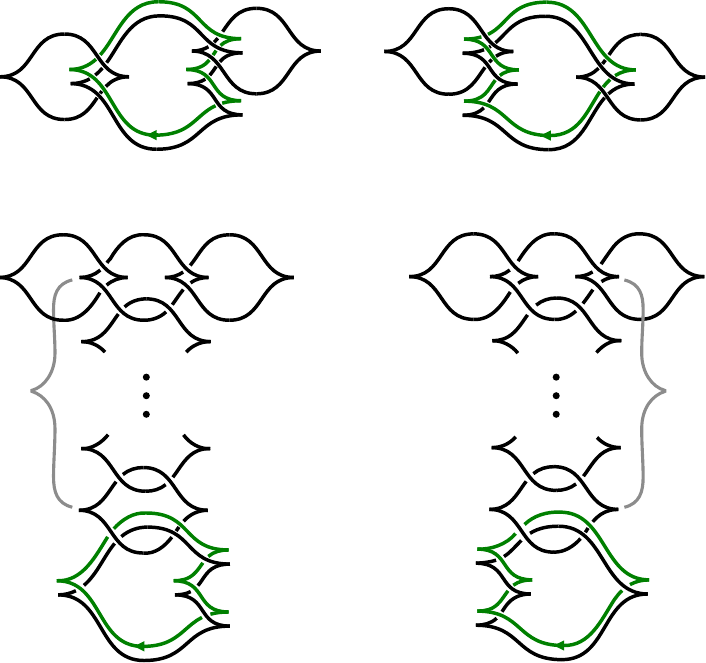}
    \put(-5,305){$(\frac{p'-p}{p'})$}
    \put(145,280){$(\frac{q'}{q'-q})$}
    \put(100,250){$(+1)$}
    \put(45,310){$L_+$}
    
    \put(179,312){$(\frac{p'-p}{p'})$}
    \put(330,265){$(\frac{q'}{q'-q})$}
    \put(225,250){$(+1)$}
    \put(285,310){$L_-$}
    
    \put(-5,205){$(\frac{p'-p}{p'})$}
    \put(130,205){$(\frac{q'}{q'-q})$}
    \put(75,209){$(-1)$}
    \put(105,153){$(-1)$}
    \put(105,102){$(-1)$}
    \put(105,72){$(-1)$}
    \put(115,55){$L_+$}
    \put(113,12){$(+1)$}
    \put(-8,130){$m-1$}

    \put(192,205){$(\frac{p'-p}{p'})$}
    \put(327,205){$(\frac{q'}{q'-q})$}
    \put(247,209){$(-1)$}
    \put(217,153){$(-1)$}
    \put(217,102){$(-1)$}
    \put(217,72){$(-1)$}
    \put(217,55){$L_-$}
    \put(211,12){$(+1)$}
    \put(325,130){$m-1$}
  \end{overpic}}
  \caption{Half of the $2n(p,q)$ of non-loose Legendrian $(p,q)$-torus knots with $\tb = pq + 1$ and $\tor=0$ are shown on the top left and the other half on the top right. Similarly, half of the $2n(p,q)$ of non-loose Legendrian $(p,q)$-torus knots with $\tb = pq + m$ for $m>1$ and $\tor=0$ are shown on the bottom left and the other half on the bottom right. Here, ${q'/p'}$ is the largest rational number such that $pq'-p'q=1$.}
  \label{fig:tb=pq+m}
\end{figure}
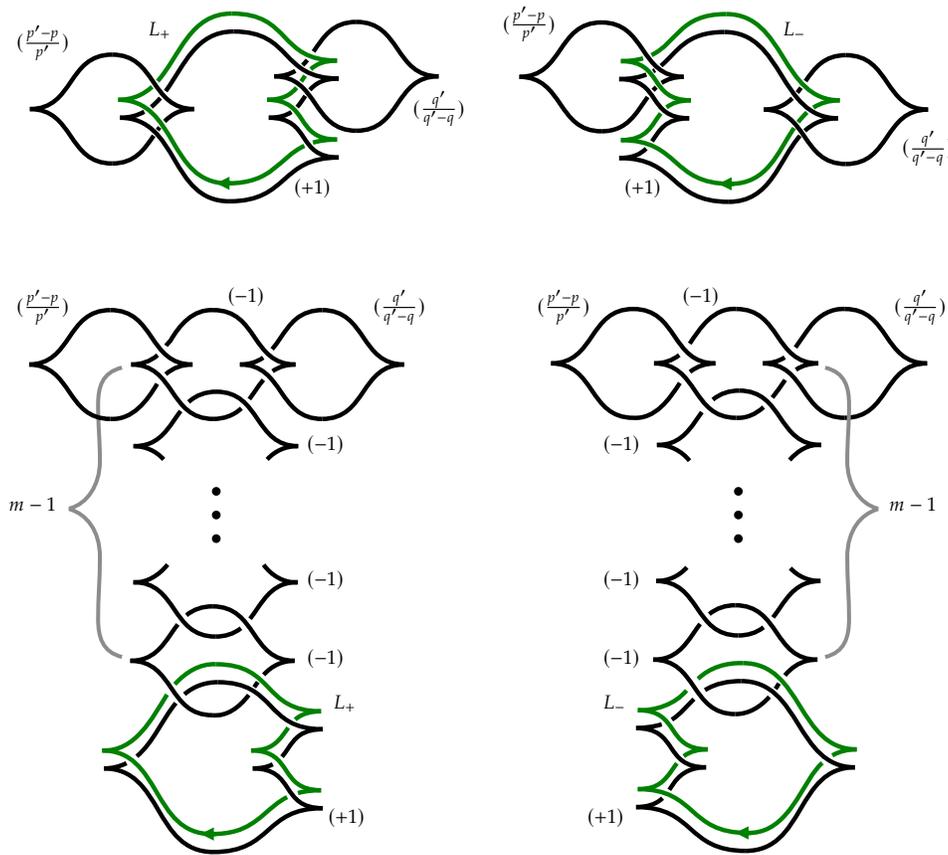

In Section~\ref{htpclasses}, we show how to compute the rotation number of $L_{\pm,k}^{i}$ and the $d_3$-invariant of the contact structure on which it lives, but we note that $L_{+,k}^{i}$ and $L_{-,k}^{i}$ live in the same contact structure. We note that surgery diagrams as in Figure~\ref{fig:tb=pq+m} first appeared in work of Geiges and Onaran \cite{GeigesOnaran20a} for specific torus knots, but it is clear from their work that one can construct examples of non-loose $(p,q)$-torus knots with $\tb>pq$ for any $p$ and $q$. Our work shows that all such knots, except one when $pq<0$, come from these diagrams. 

We can similarly enumerate non-loose Legendrian knots with $\tb=pq$ and $\tor=0$. 

\begin{theorem}\label{thm:<=pq}
  Suppose $|q| > p > 1$. The number of non-loose Legendrian $(p,q)$-torus knots with $\tb=pq$ and $\tor =0$ is exactly
  \[
    \begin{cases}
    m(p,q) & \text{ if } pq>0,\\
    m(p,q)-2\left|\left\lceil \frac qp\right\rceil\right| & \text{ if } pq<0
    \end{cases}
  \]
  and any such Legendrian knot can be realized as a Legendrian knot shown in Figure~\ref{fig:torus-knots}. 
\end{theorem}

In Section~\ref{htpclasses}, we show how to compute the rotation numbers of these knots and the $d_3$-invariants of the contact structures in which it lives. We note that the surgery diagram in Figure~\ref{fig:torus-knots} first appeared in work of Lisca and Stipsicz  \cite{LiscaStipsicz07} in the context of small Seifert fibered spaces, and then in work of Lisca, Ozsv\'ath, Stipsicz, and Szabo \cite{LiscaOzsvathStipsiczSzab09} to construct some non-loose Legendrian torus knots using Heegaard-Floer theory.  

\begin{figure}[htbp]
  \vspace{0.1cm}
  \begin{overpic}[scale=1,tics=10]{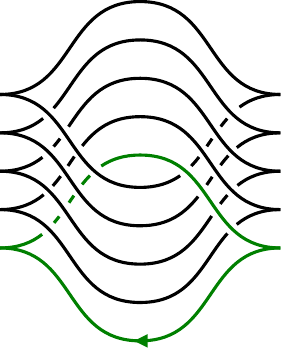}
    \put(140,120){$(-\frac{p}{p'})$}
    \put(140,101){$(-\frac{q}{q-q'})$}
    \put(140,82){$(+1)$}
    \put(140,63){$(+1)$}
    \put(120,30){$L$}
  \end{overpic}
  \caption{$L$ is a non-loose torus knot with $\tb(L)=pq$ and $\tor(L)=0$. Here, ${q'/p'}$ is the largest rational number such that $pq'-p'q=1$.}
  \label{fig:torus-knots}
\end{figure}
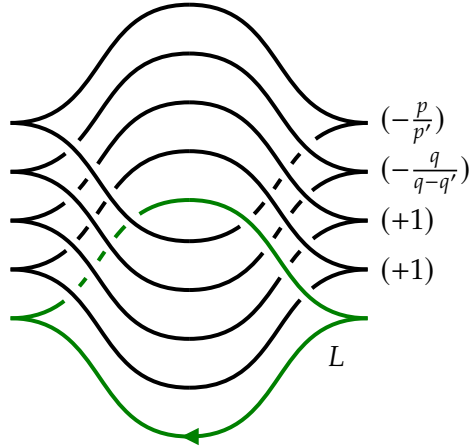

\begin{remark}
  We will see there are always $m(p,q)$ Legendrian $(p,q)$-knots with $\tb=pq$ and $\tor=0$, but when $pq<0$, it turns out that $2\left|\left\lceil q/p\right\rceil\right|$ of those are in $(S^3,\xi_{std})$.
\end{remark}

The algorithm in Section~\ref{thealgorithm} will give a complete classification of non-loose Legendrian and torus knots, but we can easily describe the qualitative features of the classification through mountain ranges. We begin this discussion with some terminology. 
\begin{definition}\label{quantfeatures}
We say a mountain range contains an \dfn{infinite $\textsf{X}$ based at $r$} if there are non-loose Legendrian representatives $L^\pm_n$ for all $n\in \Z$ such that 
\begin{align*}
  \tb(L^\pm_n)=n &\text{ and } \rot(L^\pm_n)=\mp(r-n)
\end{align*}
and satisfy
\[
S^\pm(L_n^\pm)=L_{n-1}^\pm \text{ and }  S^\mp(L_n^\pm) \text{ is loose.}
\]
See Figure~\ref{fig-genericXwing} (but ignore the grey shaded region). 

We say a mountain range for a knot type \dfn{contains an infinite $\textsf{V}$ with vertex at $(a,b)$} if there are non-loose Legendrian representatives 
 $L^\pm_n$ for $n\geq 2$ and $L_1$ such that 
\begin{align*}
  \tb(L^\pm_n)=b+n&\text{ and } \rot(L^\pm_n)=a \mp(n-1),\\
  \tb(L_1) = b &\text{ and } \rot(L_1) = a,
\end{align*}
and satisfy
\begin{align*}
  &S_\pm(L^\pm_{n})=L^\pm_{n-1} \text{ for } n\geq3, \text{ and } S_\pm(L^\pm_2)=L_1,\\ 
  &S_\mp(L^\pm_n) \text{ and } S_\pm(L_1) \text{ are loose.}
\end{align*}
See the right drawing of Figure~\ref{fig-exceptionMR} (but ignore the grey shaded region). 

We say a mountain range for a knot type that contains an infinite $\textsf{X}$ also has \dfn{wings} if there are other Legendrian representatives indicated by the grey region in Figure~\ref{fig-genericXwing}. See Theorem~\ref{leg58} for a description of a wing in a specific example and Definition~\ref{wingdef} and Theorem~\ref{merge} for the precise definition.

We say a mountain range for a knot type that contains an infinite $\textsf{V}$ also has \dfn{diamonds} if there are other Legendrian representatives indicated in the grey region in Figure~\ref{fig-exceptionMR}. See Theorem~\ref{leg58} for a description of a wing in a specific example and Definition~\ref{rigorousdiamond} for the precise definition. 
\end{definition}

We can now summarize the qualitative behavior of mountain ranges for non-loose torus knots. 
\begin{theorem}
For all but one of the overtwisted contact structures supporting non-loose Legendrian $(p,q)$-torus knots, the mountain range for the non-loose representatives with $\tor=0$ will contain an infinite $\textsf{X}$, possibly with wings. If there are wings, their peaks have $\tb=pq$. There is an algorithm to determine which overtwisted contact structures support such non-loose knots, see Section~\ref{thealgorithm}. We call these mountain ranges \dfn{generic}. 

There exists one overtwisted contact structure for each $(p,q)$-torus knot where the classification of non-loose representatives is different, see Figure~\ref{fig-exceptionMR}. If $pq>0$, the contact structure $\xi_1$ will contain an infinite $\textsf{V}$, possibly with diamonds. If $pq<0$, then the contact structure $\xi_{|pq|-|p|-|q|}$ will contain an infinite $\textsf{X}$ with an extra Legendrian at the crossing point. We call these mountain ranges \dfn{exceptional}.
\end{theorem}
We note that the proof of this theorem follows directly from the algorithm in Section~\ref{thealgorithm} that classifies all non-loose Legendrian torus knots, and that algorithm is justified in Section~\ref{nonloosetorusknots}.


\begin{figure}[htbp]
  \begin{overpic}{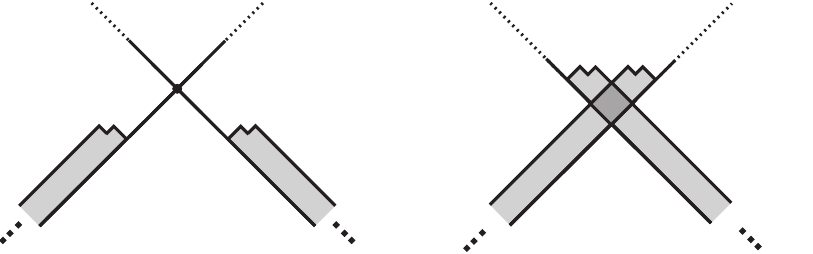}
  \end{overpic}
  \caption{Generic mountain ranges for non-loose Legendrian $(p,q)$-torus knots with $\tor = 0$. On the left is the case where $pq<0$ and on the right is where $pq>0$. The peaks occur at $\tb=pq$. Each integral point in the lightly shaded region, whose coordinates sum to be odd, is realized by a unique non-loose Legendrian knot, while in the darker shaded region on the right and crossing point on the left there are exactly two representatives with those invariants.}
  \label{fig-genericXwing}
\end{figure}

The number of peaks in the wings will depend on $(p,q)$ and decorated paths in the Farey graph from $\infty$ to $q/p$ and then to $0$, see Section~\ref{thealgorithm}.
However, we observe the following about the wings. 

\begin{theorem}\label{bigwings}
  Given any positive integers $n$ and $m_1, \ldots, m_{n-1}$ there is some $(p,q)$-torus knot whose mountain range of non-loose knots with $\tor=0$ in some overtwisted contact structure has $2n$ peaks, and the distance between the $i^{th}$ and $(i+1)^{st}$ peak is at least $m_i$ (we label the peaks according to their distance from the infinite $\textsf{X}$). 
\end{theorem}

\begin{figure}[htbp]
  \begin{overpic}{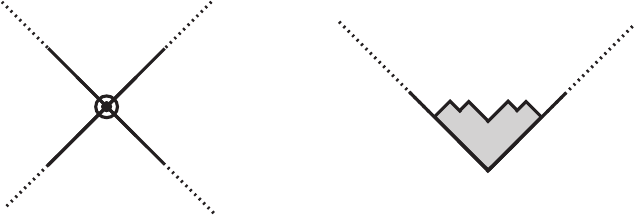}
  \end{overpic}
  \caption{The mountain range of non-loose Legendrian $(p,q)$-torus knots with $\tor=0$ for the exceptional contact structures. On the left is the mountain range for a negative $(p,q)$-torus knot in $\xi_{|pq|-|p|-|q|+1}$. The crossing is at $(\rot,\tb)=(0,|pq|-|p|-|q|)$ and there are three distinct non-loose Legendrian knots. On the right is the mountain range for a positive $(p,q)$ torus knot in $\xi_1$. The lower vertex is at $(\rot,\tb)=(0,pq-p-q+2)$. } 
  \label{fig-exceptionMR}
\end{figure}

We note that it might be possible that the mountain range for some $(p,q)$-torus knot and some overtwisted contact structure consists of the union of more than one diagram shown in Figures~\ref{fig-genericXwing} and~\ref{fig-exceptionMR}. For example, a mountain range could contain two infinite $\textsf{X}$s. If this happens, then the Legendrian knots depicted in each figure are never equivalent to those in another. In all our computed examples, we see that this never occurs and conjecture that it never does. 

\begin{conjecture} \label{conjecture}
  In each overtwisted contact structure that supports non-loose $(p,q)$-torus knots, the mountain range of such knots is given by only one of the diagrams indicated in Figures~\ref{fig-genericXwing} and~\ref{fig-exceptionMR}.
\end{conjecture}
        
In \cite{Matkovic20Pre} Matkovi\v{c} classified non-loose Legendrian $(p,q)$-torus knots with $pq<0$, $\tb<pq$ and $\tor = 0$, from this she could classify all non-loose transverse knots with $\tor = 0$, as well. She could then show that if two non-loose transverse knots were not related by stabilization, then they were in distinct overtwisted contact structures. This verifies our conjecture for negative torus knots.

We now consider the number of contact structures supporting non-loose Legendrian $(p,q)$-torus knots. 

\begin{theorem}\label{numbersupportingnonloose}
  There are at most  $n(p,q)$ overtwisted contact structures supporting non-loose Legendrian $(p,q)$-torus knots with $\tor = 0$. And at most 
  \[
    n(p,q) + |(a_1+1)\cdots(a_{m-1}+1)|\cdot|(b_1+1)\cdots (b_{n-1}+1)|
  \]
  overtwisted contact structures supporting any non-loose Legendrian $(p,q)$-torus knots. 
\end{theorem}

We note that if Conjecture~\ref{conjecture} is true, then the upper bound in Theorem~\ref{numbersupportingnonloose} gives the exact number of such contact structures. 

We also give some qualitative properties of non-loose torus knots with $\tor >0$. One notable fact about these knots is that $\tor$ is always finite. This is not true for general Legendrian knots. For example, any loose Legendrian knot has $\tor = \infty$, meaning that its complement can contain arbitrarily large amounts of convex torsion. It is an open problem whether there exist non-loose knots with $\tor = \infty$. See \cite[Problem~3.42]{K3:problems}.

\begin{theorem}\label{thm:torsion}
  Let $L$ be a non-loose Legendrian $(p,q)$-torus knot with $\tor(L) > 0$. Then
  \begin{enumerate}
    \item $\tor(L)$ is finite, 
    \item by performing an $n$-fold convex Lutz twist (see Section~\ref{lutzsection}) along the boundary of the standard neighborhood of $L$ with an appropriate sign, we obtain a new non-loose Legendrian $L'$ with the same classical invariants as $L$ and $\tor(L')=\tor(L)+n$, and 
    \item there exists a unique Legendrian knot $L_0$ with $\tb(L_0) = pq$ and $\tor(L_0)=0$ such that the complement of $L$ is obtained by attaching a convex torsion layer in the complement of $L_0$.
  \end{enumerate}
\end{theorem}

\subsection{Non-loose transverse torus knots} 
We now turn to transverse knots. As is well-known \cite[Theorem~2.10]{EtnyreHonda01}, the classification of transverse knots is equivalent to the classification of Legendrian knots up to negative stabilization. Thus, our algorithm for classifying non-loose Legendrian $(p,q)$-torus knots will also classify non-loose transverse $(p,q)$-torus knots. 

\begin{theorem}\label{gentransverse}
  Suppose $\xi$ is an overtwisted contact structure supporting non-loose transverse $(p,q)$-torus knots. If we suppose Conjecture~\ref{conjecture} is true, then in $\xi$, either 
  \begin{enumerate}
    \item\label{1} there are a finite number of non-loose transverse knots $T_1,\ldots T_n$, the stabilization of $T_i$ is $T_{i+1}$ for $i<n$, and the stabilization of $T_n$ is loose; or 
    \item\label{2} there are an infinite number of non-loose transverse knots with the same self-linking number and they are distinguished by the Giroux torsion in their complement. 
  \end{enumerate}
  In the former case, all the $T_i$ have zero Giroux torsion in their complement. If Conjecture~\ref{conjecture} is not true, then the set of non-loose transverse knots in $\xi$ could be a union of several copies of non-loose knots of type~\eqref{1} and~\eqref{2} above. 
\end{theorem}

We note that in \cite{Matkovic20Pre}, Matkovi\v{c} proved that negative torus knots (with no Giroux torsion in their complements) are transversely simple and gave an algorithm that could be used to obtain the above results for these knots. 

\subsection{Non-loose \texorpdfstring{$(2,\pm (2n+1))$}{(2,±(2n+1))}-torus knots} \label{sec:2n+1}
Here we give an explicit classification of non-loose Legendrian and transverse $(2,\pm (2n+1))$-torus knots for $n \in \mathbb{N}$. All the results in this section will be established in Section~\ref{fullclass22np1}.

\subsubsection{Non-loose Legendrian \texorpdfstring{$(2,\pm (2n+1))$}{(2,±(2n+1))}-torus knots} We begin with Legendrian $(2,2n+1)$-torus knots. 

\begin{theorem}\label{thm:(2,2n+1)}
  The $(2,2n+1)$-torus knot has non-loose Legendrian representatives only in $\xi_1,\xi_0,$ and $\xi_{1-2n}$. The classification in each of these contact structures is as follows.
  \begin{enumerate}
    \item In $(S^3,\xi_{0})$, there are non-loose Legendrian $(2,2n+1)$-torus knots $L_\pm^{i,k+\scriptscriptstyle\frac 12}$ for $i \in \mathbb{Z}$ and $k \in \mathbb{N} \cup \{0\}$ such that \label{item:RHT-torsions}
    \[
      \tb(L_\pm^{i,k+\scriptscriptstyle\frac 12})=i, \text{ and } \rot(L_\pm^{i,k+\scriptscriptstyle\frac 12})=\mp(i-2n+1),
    \]
    \[
      \tor(L_\pm^{i,k+\scriptscriptstyle\frac 12}) = \begin{cases} k+ \frac12 \,&\text{ if }\,i>2n-1,\\ k+1 &\text{ if }\,i\leq 2n-1, \end{cases}
    \]  
    \[
      S_\pm(L_\pm^{i,k+\scriptscriptstyle\frac 12})=L_\pm^{i-1,k+\scriptscriptstyle\frac 12} \text{ and } S_\mp(L_\pm^{i,k+\scriptscriptstyle\frac 12}) \text{ is loose}.
    \]
    \item  In $(S^3,\xi_{1-2n})$, there are non-loose Legendrian $(2,2n+1)$-torus knots $L_\pm^{i,k}$ for $i \in \Z$ and $k \in \mathbb{N}\cup \{0\}$ such that 
    \[
      \tb(L_\pm^{i,k})=i, \text{ and } \rot(L_\pm^{i,k})=\mp(i+2n-1),
    \]
    \[
      \tor(L_\pm^{i,k}) = \begin{cases} k \,&\text{ if } i>2n-1,\\ k + \frac12 &\text{ if } i\leq 2n-1, \end{cases}
    \] 
    \[
      S_\pm(L_\pm^{i,k})=L_\pm^{i-1,k} \text{ and } S_\mp(L_\pm^{i,k}) \text{ is loose.}
    \]
    \item  In $(S^3,\xi_{1})$, there are non-loose Legendrian knots 
    \begin{align*}
      &L_\pm^i, \,\,\,\text{ for }\,\,i>2n+1,\\
      &L_{2,\pm}^i, \,\,\,\text{ for }\,\, 2n+4 \leq i \leq 4n+2, \\
      & L^{2n+1}, \text{ and } L_2^{2n+3}
    \end{align*}
    with
    \[
      \tb(L_\pm^i)=i, \text{ and } \rot(L_\pm^i)=\mp(i-2n-1),
    \]
    \[
      \tb(L^{2n+1})=2n+1, \text{ and } \rot(L^{2n+1})=0,
    \]
    \[
      \tb(L_{2,\pm}^i)=i, \text{ and } \rot(L_{2,\pm}^i)= \mp (i - 2n-3),
    \]
    \[
      \tb(L_2^{2n+3})=2n+3, \text{ and } \rot(L_2^{2n+3})=0
    \]
    such that 
    \[
      S_\pm(L_\pm^i)=L_\pm^{i-1}, \text{ for $i \geq 2n+3$, and } S_\pm(L_\pm^{2n+2})=L^{2n+1},
    \]
    \[
      S_\pm(L_{2,\pm}^i)=L_{2,\pm}^{i-1}, \text{ for $i\geq 2n+5$, and } S_\pm(L_{2,\pm}^{2n+4})= L_2^{2n+3},
    \]
    \[
      S_\mp(L_{2,\pm}^i)=L_\pm^{i-1}, \text{ for } i\geq 2n+4, S_\mp(L_2^{2n+3})=L_\pm^{2n+2}
    \]
    and $S_\mp(L_\pm^i)$ and $S_\pm(L^{2n+1})$ are loose. All these Legendrian knots have $\tor = 0$. 
  \end{enumerate}
\end{theorem}

See Figure~\ref{fig:RHT-mountain} for the mountain ranges of non-loose Legendrian right-handed trefoils. 

\begin{figure}[htbp]{\footnotesize
  \vspace{0.2cm}
  \begin{overpic}{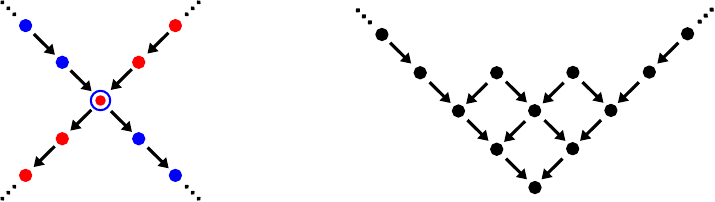}
    \put(7, 98){$-2$}
    \put(23, 98){$-1$}
    \put(45, 98){$0$}
    \put(65, 98){$1$}
    \put(83, 98){$2$}

    \put(165, 77){$7$}
    \put(165, 59){$6$}
    \put(165, 39.5){$5$}
    \put(165, 22){$4$}
    \put(165, 4){$3$}

    \put(195, 98){$-3$}
    \put(212, 98){$-2$}
    \put(231, 98){$-1$}
    \put(255, 98){$0$}
    \put(273, 98){$1$}
    \put(292, 98){$2$}
    \put(310, 98){$3$}
  \end{overpic}}
  \vspace{0.1cm}
  \caption{The left is the mountain range for the non-loose Legendrian right-handed trefoils in $\xi_0$ and $\xi_{-1}$. In $\xi_0$, the crossing point of the $\textsf{X}$ is at $\tb=1$ while in $\xi_{-1}$, it is at $\tb=-1$. Each dot represents an infinite family of Legendrian representatives distinguished by convex torsion, and at the cross we have two infinite families. The right is the mountain range in $\xi_1$. Each dot represents a unique non-loose Legendrian representative.}
  \label{fig:RHT-mountain}
\end{figure}
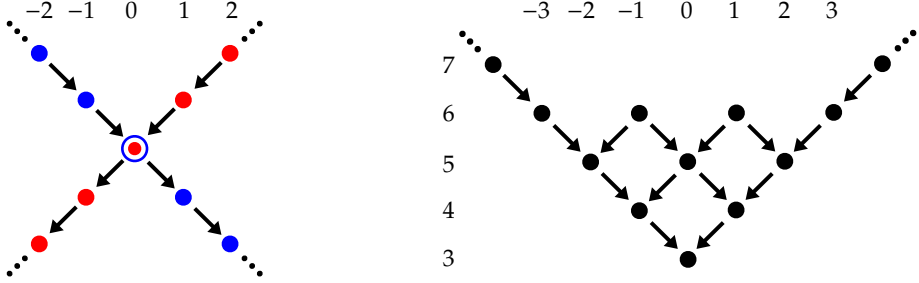

We now turn to Legendrian $(2,-(2n+1))$-torus knots. 

\begin{theorem}\label{thm:(2,-(2n+1))}
  The $(2,-(2n+1))$-torus knot has non-loose Legendrian representatives only in $\xi_{n+l+1}$ and $\xi_{n-l}$ for $l\in \{-n+1, -n+3, \ldots, n-3, n-1\}$. The classification in each of these contact structures is as follows.
  \begin{enumerate}
    \item  In $(S^3,\xi_{2n})$, there are non-loose Legendrian $(2,2n+1)$-torus knots $L_{n-1,\pm}^{i, k}$ for $i\in \Z, k\in\N\cup\{0\}$, and $L_e$ with
    \[
      \tb(L_{n-1,\pm}^{i,k})=i,\,\, \rot(L_{n-1,\pm}^{i,k})=\mp(i -2n + 1), \text { and } \tor(L_{n-1,\pm}^{i,k})=k,
    \]
    \[
      \tb(L_e)=2n-1,\,\, \rot(L_e)=0, \text{ and } \tor(L_e) = 0,
    \]
    \[
      L_{n-1,\pm}^{i,k}=S_\pm(L_{n-1,\pm}^{i-1,k}) \text{ and } S_\pm(L_e)=L_{n-1,\pm}^{2n-2,0},
    \]
    \[
      S_\mp(L_{n-1,\pm}^{i,k}) \text{ is loose.}
    \]
    \item  In $(S^3,\xi_{n+l+1})$ for $l\in \{-n+1,-n+3, \ldots, n-3\}$, there are non-loose Legendrian $(2,2n+1)$-torus knots $L_{l,\pm}^{i,k}$ having 
    \[
      \tb(L_{l,\pm}^{i,k})=i,\,\, \rot(L_{l,\pm}^{i,k})=\mp(i - 2l - 1), \text{ and } \tor(L_{l,\pm}^{i,k}) = k,
    \]
    \[
      S_\pm(L_{l,\pm}^{i,k})=L_{l,\pm}^{i-1,k} \text{ and } S_\mp(L_{l,\pm}^{i,k}) \text{ loose.}
    \]
    \item In $(S^3,\xi_{n-l})$ for $l\in \{-n+1,-n+3, \ldots, n-3, n-1\}$, there are non-loose Legendrian $(2,2n+1)$-torus knots $L_{l,\pm}^{i,k+\scriptscriptstyle\frac 12}$ having 
    \[
      \tb(L_{l,\pm}^{i,k+\scriptscriptstyle\frac 12})=i,\,\, \rot(L_{l,\pm}^{i,k+\scriptscriptstyle\frac 12})=\mp(i+2l+1), \text{ and } \tor(L_{l,\pm}^{i,k+\scriptscriptstyle\frac 12}) = k + \frac12,
    \]
    \[
      S_\pm(L_{l,\pm}^{i,k+\scriptscriptstyle\frac 12})=L_{l,\pm}^{i-1,k+\scriptscriptstyle\frac 12} \text{ and } S_\mp(L_{l,\pm}^{i,k+\scriptscriptstyle\frac 12}) \text{ loose}.
    \]
  \end{enumerate}
\end{theorem}

See Figure~\ref{fig:LHT-mountain} for the mountain ranges of non-loose Legendrian left-handed trefoils.

\begin{figure}[htbp]{\footnotesize
  \vspace{0.4cm}
  \begin{overpic}{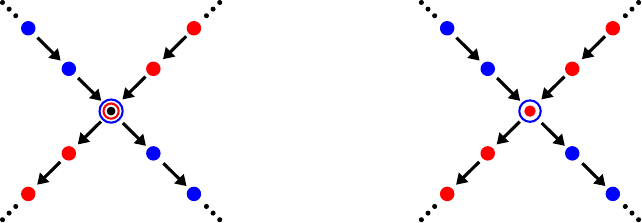}
    \put(5, 114){$-2$}
    \put(25, 114){$-1$}
    \put(50, 114){$0$}
    \put(70, 114){$1$}
    \put(90, 114){$2$}

    \put(205, 114){$-2$}
    \put(225, 114){$-1$}
    \put(250, 114){$0$}
    \put(270, 114){$1$}
    \put(290, 114){$2$}

    \put(-15, 89){$3$}
    \put(-15, 69){$2$}
    \put(-15, 52){$1$}
    \put(-15, 33){$0$}
    \put(-22, 12){$-1$}

    \put(180, 89){$1$}
    \put(180, 69){$0$}
    \put(173, 52){$-1$}
    \put(173, 32){$-2$}
    \put(173, 12){$-3$}
  \end{overpic}}
  \vspace{0.1cm}
  \caption{The left is the mountain range for the non-loose Legendrian left-handed trefoils in $\xi_2$. The right is in $\xi_1$. Each red and blue dot or circle represents an infinite family of Legendrian representatives distinguished by convex torsion. The black dot represents the extra Legendrian $L_e$ }
  \label{fig:LHT-mountain}
\end{figure}

\subsubsection{Non-loose transverse \texorpdfstring{$(2,\pm (2n+1))$}{(2,±(2n+1))}-torus knots}
We now consider the classification of non-loose transverse $(2,2n+1)$-torus knots.

\begin{theorem}\label{thm:(2,2n+1)-transverse}
  The $(2,2n+1)$-torus knot has non-loose transverse representatives only in $\xi_0,$ and $\xi_{1-2n}$. The classification in each of these structures is as follows.
  \begin{enumerate}
    \item In $(S^3,\xi_0)$, there is a family of non-loose transverse $(2,2n+1)$-torus knots $T^{k}$ for $k \geq 1$ with 
    \[
      \self(T^k)=2n-1 \text{ and } \tor(T^k) = k
    \]
    and when stabilized $T^k$ becomes loose. Moreover, $T^{k+1}$ is obtained from a Lutz twist on $T^k$ for $i>i$ and $T^1$ is obtained from a Lutz twist on the maximal self-linking transverse representative of the $(2,2n+1)$-torus knot in $(S^3,\xi_{std})$. 
    \item In $(S^{3},\xi_{1-2n})$, there is a family of non-loose transverse $(2,2n+1)$-torus knots $T^{k+\scriptscriptstyle\frac{1}{2}}$ for $i \geq 0$ with 
    \[
      \self(T^{k+\scriptscriptstyle\frac{1}{2}}) = -2n+1 \text{ and } \tor(T^{k+\scriptscriptstyle\frac{1}{2}}) = k + \frac12
    \]
    any when stabilized $T^{k+\scriptscriptstyle\frac 12}$ becomes loose. Moreover, $T^{k+\scriptscriptstyle\frac{1}{2}}$ is obtained by the half Lutz twist of $T^{k}$.
  \end{enumerate}
\end{theorem}

Finally we discuss non-loose transverse $(2,-(2n+1))$-torus knots. 

\begin{theorem}\label{thm:(2,-(2n+1))-transverse}
  The $(2,-(2n+1))$-torus knot has non-loose transverse representatives only in $\xi_{n+l+1}$ and $\xi_{n-l}$ for $l\in \{-n+1, -n+3, \ldots, n-3, n-1\}$. The classification in each of these structures is as follows.
  \begin{enumerate}
    \item In $(S^3,\xi_{n+l+1})$, for $l\in \{-n+1, -n+3, \ldots, n-3, n-1\}$ there are non-loose transverse $(2,-(2n+1))$-torus knots $T^k_l$ for $k\geq 0$ with 
    \[
      \self(T^k_l)=2l+1 \text{ and } \tor(T^k_l) = k 
    \]
    and when stabilized $T^k_l$ becomes loose. Moreover, $T_l^{k+1}$ is obtained from a Lutz twist on $T^k_l$ for $k\geq0$.
    \item In $(S^3, \xi_{n-l})$ for $l\in \{-n+1, -n+3, \ldots, n-3, n-1\}$, there are non-loose transverse $(2,-(2n+1))$-torus knots $T^{k+\scriptscriptstyle\frac 12}_l$ for $k\geq 0$ with 
    \[
      \self(T^{k+\scriptscriptstyle\frac 12}_l)=-2l-1 \text{ and } \tor(T^{k+\scriptscriptstyle\frac 12}_l) = k+\frac12  
    \]
    and when stabilized $T^{k+\scriptscriptstyle\frac 12}_l$ becomes loose. Moreover, $T_l^{k+\scriptscriptstyle\frac 12}$ is obtained from the half Lutz twist on $T^{k}_l$.  
  \end{enumerate}
\end{theorem}

\subsection{Non-loose \texorpdfstring{$(5,\pm 8)$}{(5,8)}-torus knots} \label{sec:(5,8)}
Here we give a classification of non-loose Legendrian and transverse $(5,\pm 8)$-torus knots as their classification shows some features not seen in the non-loose $(2,\pm(2n+1))$-torus knots. These knots are part of the family of $(5, \pm(5n+3))$-torus knots and their classification is quite similar and is left as an exercise for the reader. All the results in this section will be established in Section~\ref{fullclass58}.

\begin{theorem}\label{leg58}
  The $(5,8)$-torus knot has non-loose Legendrian representatives only in $\xi_1$, $\xi_0$, $\xi_{-1}$, $\xi_{-2}$, $\xi_{-3}$, $\xi_{-4}$, $\xi_{-7}$, $\xi_{-8}$, $\xi_{-9}$ $\xi_{-15}$, $\xi_{-19}$, and $\xi_{-27}$. The classification in each of these contact structures is as follows. See Figures~\ref{fig:r85-1} and~\ref{fig:r85-2}.
  \begin{enumerate}
    \item in $(S^3,\xi_1)$ we have non-loose Legendrian knots $L_\pm^i$ for $i>29$ and $L^{29}$ such that $\tb(L_\pm^i)=i$, $\tb(L^{29})=29$ and 
    \[
    \rot(L_\pm^i)=\mp(i-29) \text{ and } \rot(L^{29})=0,
    \]
    \[
    S_\pm(L_\pm^i)=L_\pm^{i-1}, \text{ for $i>30$ }, S_\pm(L_\pm^{20})=L^{29}, \text{ and } S\mp(L_\pm^i) \text{ and } S\pm(L^{29}) \text{ are loose.}
    \]
    In addition, there are Legendrian knots $L_{k,\pm}^{40}, k=2,3,4,$ with Thurston-Bennequin invariant $40$ and 
    \[
    \rot(L_{2,\pm}^{40})=\mp 9, \rot(L_{3,\pm}^{40})=\mp 7, \text{ and } \rot(L_{4,\pm}^{40})=\mp 3.
    \]
    When these knots are stabilized to have the same invariants (or the invariants of the $L_\pm^i$ or $L^{29}$) they become equivalent and they are non-loose until stabilized outside the $\textsf{V}$ defined by the $L_\pm^i$ and $L^{29}$. None of these Legendrian knots have convex torsion in their complement. 
    \item In $(S^3,\xi_{-1})$ there are Legendrian knots $L_\pm^i$ for $i\in \Z$ and $L_{2,\pm}^i$ for $i\leq 40$ such that 
    \[
      \tb(L_\pm^i)=\tb(L_{2,\pm}^i)=i, \rot(L_\pm^i)=\mp (i-21), \text{ and } \rot(L_{2,\pm}^i)=\mp(i-19).
    \]
    Moreover, $S_\pm(L_\pm^i)=L_\pm^{i-1}, S_\pm(L_{2,\pm}^i)=L_{2,\pm}^{i-1}, S_\mp(L_{2,\pm}^i)=L_\pm^{i-1}$, and $S_\mp(L_\pm^i)$ is loose. No stabilization of $L_+^i$ or $L_{2,+}^i$ is equivalent to a stabilization of $L_i^i$ or $L_{2,-}^i$. All these Legendrian knots have no convex torsion in their complement. 
    \begin{figure}[htbp]{\footnotesize
    \vspace{0.1cm}
    \begin{overpic}{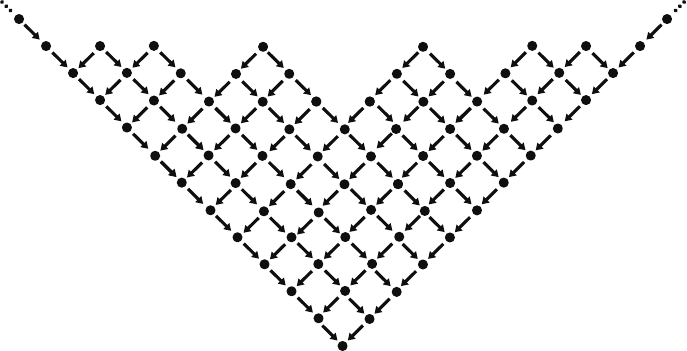}
      \put(-1, 172){$-12$}
      \put(41, 172){$-9$}
      \put(67, 172){$-7$}
      \put(120, 172){$-3$}
      \put(315, 172){$12$}
      \put(280, 172){$9$}
      \put(254, 172){$7$}
      \put(200, 172){$3$}
      \put(330, 143){$40$}
      \put(330, -1){$29$}
    \end{overpic}}
    \vspace{0.1cm}
    \caption{The mountain range for the non-loose Legendrian $(5,8)$-torus knots in $\xi_{1}$. Each dot or circle represents a unique non-loose Legendrian knot.}
    \label{fig:r85-1}
  \end{figure}
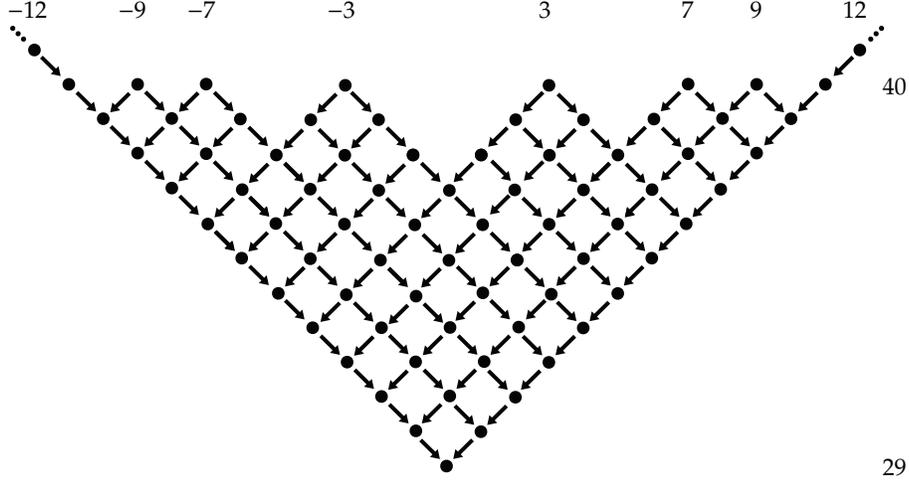
  \item In $(S^3, \xi_{-3})$ and $(S^3,\xi_{-7})$ we have Legendrian knots $L_\pm^i$ with $\tb(L_\pm^i)=i$ and $\tor(L_\pm^i)=0$ such that
  \[
    \rot(L_\pm^i)=\begin{cases}
    \mp(i-13) & \text{ in } \xi_{-3},\\
    \mp(i-3)& \text{ in } \xi_{-7},
    \end{cases}
  \]
  \[
    S_\pm(L_\pm^i)=L_\pm^{i-1} \text{ and } S_\mp(L_\pm^i) \text{ is loose.}
  \]
  \item In $S^3$ with the contact structures $\xi_{-9}. \xi_{-15}, \xi_{-19}, \xi_{-27}$ we have the Legendrian knots $L_\pm^{i,k}$ where $\tb(L_\pm^{i,k})=i$ and 
  \[
    \rot(L_\pm^{i,k})=\begin{cases}
    \mp(i+1) & \text{ in } \xi_{-9},\\
    \mp(i+11)& \text{ in } \xi_{-15},\\
    \mp(i+17)& \text{ in } \xi_{-19},\\
    \mp(i+27)& \text{ in } \xi_{-27},
    \end{cases}
  \]
  \[
    S_\pm(L_\pm^{i,k})=L_\pm^{i-1,k} \text{ and } S\mp(L_\pm^{i,k}) \text{ is loose.}
  \]
  Moreover, $\tor(L_\pm^{i,k})=k + 1/2$ if it is in $\xi_{-27}$ and $i \leq 27$. Otherwise, $\tor(L_\pm^{i,k})=k$. 
  \item In $S^3$ with the contact structures $\xi_{-8},\xi_{-4},\xi_{-2},$ and $\xi_0$ we have the Legendrian knots $L_{\pm}^{i,k+\scriptscriptstyle\frac 12}$ where $\tb(L_\pm^{i,k+\scriptscriptstyle\frac 12})=i$ and 
  \[
    \rot(L_\pm^{i,k+\scriptscriptstyle\frac 12})=\begin{cases}
    \mp(-1+i) & \text{ in } \xi_{-8},\\
    \mp(-11+i)& \text{ in } \xi_{-4},\\
    \mp(-17+i)& \text{ in } \xi_{-2},\\
    \mp(-27+i)& \text{ in } \xi_{-0},
    \end{cases}
  \]
  \[
    S_\pm(L_\pm^{i,k+\scriptscriptstyle\frac 12})=L_\pm^{i-1,k+\scriptscriptstyle\frac 12} \text{ and } S_\mp(L_\pm^{i,k+\scriptscriptstyle\frac 12}) \text{ is loose.}
  \]
  Moreover, $\tor(L_\pm^{i,k+\scriptscriptstyle\frac 12}) = k + 1$ if it is in $\xi_{0}$ and $i\leq 27$. Otherwise, $\tor(L_\pm^{i,k+\scriptscriptstyle\frac 12}) = k + 1/2$.
  \end{enumerate}
\end{theorem}

\begin{figure}[htbp]{\footnotesize
  \vspace{0.2cm}
  \begin{overpic}{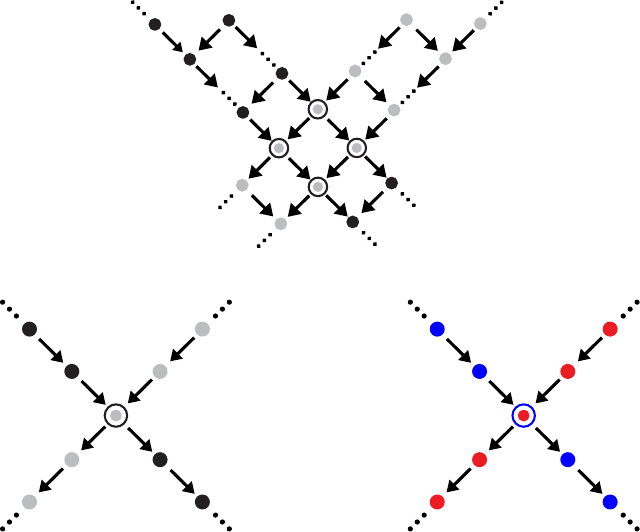}
    \put(65, 255){$-21$}
    \put(100, 255){$-19$}
    \put(191, 255){$19$}
    \put(226, 255){$21$}
    
    \put(40, 240){$40$}
    \put(40, 182){$20$}

    \put(5, 114){$-2$}
    \put(25, 114){$-1$}
    \put(50, 114){$0$}
    \put(70, 114){$1$}
    \put(90, 114){$2$}

    \put(205, 114){$-2$}
    \put(225, 114){$-1$}
    \put(250, 114){$0$}
    \put(270, 114){$1$}
    \put(290, 114){$2$}
        
    \put(-15, 53){$t$}
    \put(173, 53){$t$}
  \end{overpic}}
  \vspace{0.1cm}
  \caption{The top is the mountain range for the non-loose Legendrian $(5,8)$-torus knots in $\xi_{-1}$. The bottom left is in $\xi_{-3}$ and $\xi_{-7}$. In the first case, $t=13$ and in the second case, $t=3$. The bottom right is in $\xi_{-9}$, $\xi_{-15}$, $\xi_{-19}$, $\xi_{-27}$, $\xi_{-8}$, $\xi_{-4}$, $\xi_{-2}$, and $\xi_0$. The values of $t$ in those cases are $-1,-11,-17,-27, 1, 11, 17,$ and $27$, respectively. Each black and gray dot or circle represents a unique non-loose representative, while the colored ones represent infinitely many distinct Legendrian representatives.}
  \label{fig:r85-2}
\end{figure}

We now turn to the Legendrian $(5,-8)$-torus knots.

\begin{theorem}\label{leg5-8}
The $(5,-8)$-torus knots has non-loose Legendrian representatives only in $\xi_1$, $\xi_2$, $\xi_7$, $\xi_8$, $\xi_{14}$, and $\xi_{28}$. The classification in each of these contact structures is as follows. See Figure~\ref{fig:5-8}. 
\begin{figure}[htbp]{\footnotesize
  \vspace{0.1cm}
  \begin{overpic}{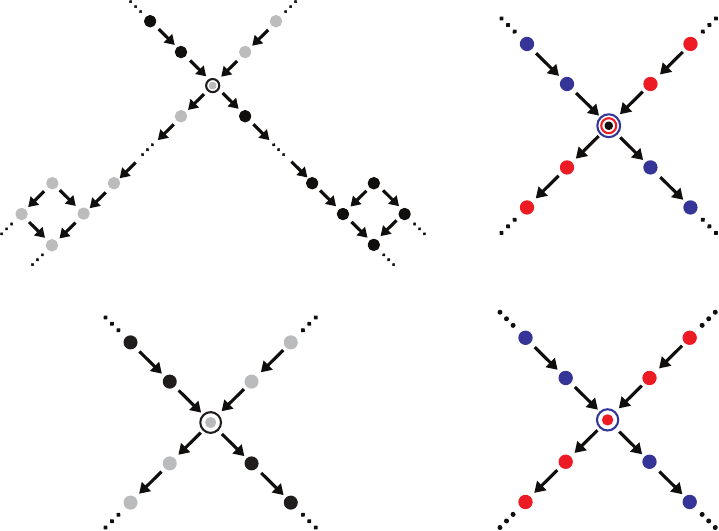}
    \put(330, 192){$27$}
    \put(-10, 211){$-25$}
    \put(-10, 165){$-40$}
      
    \put(15, 255){$-17$}
    \put(178, 255){$17$}

    \put(55, 109){$-2$}
    \put(75, 109){$-1$}
    \put(100, 109){$0$}
    \put(118, 109){$1$}
    \put(138, 109){$2$}

    \put(246, 109){$-2$}
    \put(266, 109){$-1$}
    \put(291, 109){$0$}
    \put(309, 109){$1$}
    \put(329, 109){$2$}
      
    \put(246, 248){$-2$}
    \put(266, 248){$-1$}
    \put(291, 248){$0$}
    \put(310, 248){$1$}
    \put(330, 248){$2$}

    \put(81, 255){$-1$}
    \put(100, 255){$0$}
    \put(115, 255){$1$}

    \put(155, 49){$-5$}
    \put(330, 50){$t$}
  \end{overpic}}
  \vspace{0.1cm}
  \caption{The upper left is the mountain range for non-loose $(5,-8)$-torus knots in $\xi_2$ and the upper right is in $\xi_{28}$. The bottom left is in $\xi_{8}$ and the bottom right is in $\xi_1$, $\xi_7$ and $\xi_{14}$ where $t$ is $27$, $7$, and $-7$, respectively. Each black and gray dot or circle represents a unique non-loose Legendrian knot while the colored ones represent infinitely many distinct Legendrian knots.}
  \label{fig:5-8}
\end{figure}

\begin{enumerate}
  \item In $(S^3,\xi_{28})$, there are non-loose Legendrian $(5,-8)$-torus knots $L_\pm^{i, k}, i\in \Z, k\in\N\cup\{0\}$ and $L_e$ with
  \[
    \tb(L_\pm^{i,k})=i,\,\, \rot(L_\pm^{i,k})=\mp(i-27), \text{ and } \tor(L_\pm^{i,k}) = k,
  \]
  \[
    \tb(L_e)=27,\,\, \rot(L_e)=0, \text{ and } \tor(L_e)=0,
  \]
  \[ 
    S_\pm(L_\pm^{i,k}) = L_\pm^{i-1,k},\,\,  S_\pm(L_e)=L_\pm^{2n-2,0} \text{ and } S_\mp(L_\pm^{i,k}) \text{ is loose.}
  \]
  \item In $(S^3,\xi_{2})$, there are non-loose Legendrian $(5,-8)$-torus knots $L_\pm^{i}$ for $i\in \Z$ and $L_{2,\pm}^i$ for $i\leq -40$ having 
  \[
    \tb(L_\pm^{i})=\tb(L_{2,\pm}^i)=i, \text{ and } \tor(L_\pm^{i})=\tor(L_{2,\pm}^i)=0, 
  \]
  \[
    \rot(L_\pm^{i})=\mp(i + 25), \text{ and } \rot(L_{2,\pm}^{i})=\mp(i + 23),
  \]
  \[
    S_\pm(L_\pm^{i})=L_\pm^{i-1}, \text{ and } S_\pm(L_{2,\pm}^i)=L_{2,\pm}^{i-1}, 
  \]
  \[
    S_\mp(L_{2,\pm}^i)=L_\pm^{i-1}, \text{ and } S_\mp(L_\pm^{i}) \text{ is loose.}
  \]
  \item In $(S^3,\xi_{8})$ there are non-loose Legendrian $(5,-8)$-torus knots $L_\pm^i$ for $i\in \Z$ with 
  \[
    \tb(L_\pm^i)=i,\,\, \rot(L_\pm^i)=\mp(i+5), \text{ and } \tor(L_\pm^i) = 0,
  \]
  \[
    S_\pm(L_\pm^i)=L_\pm^{i-1} \text{ and } S_\mp(L_\pm^i) \text{ is loose.}
  \]
  \item In $(S^3,\xi_{14})$ there are non-loose Legendrian $(5,-8)$-torus knots  $L_\pm^{i,k}$ for $i\in \Z$ and $k\in \N\cup \{0\}$ satisfying 
  \[
    \tb(L_\pm^{i,k})=i,\,\, \rot(L_\pm^{i,k})=\mp(i - 7), \text{ and } \tor(L_\pm^{i,k}) = k, 
  \]
  \[
    S_\pm(L_\pm^{i,k})=L_\pm^{i-1,k} \text{ and } S_\mp(L_\pm^{i,k}) \text{ is loose.}
  \]
  \item In $S^3$ with the contact structures $\xi_1$ and $\xi_7$ there are non-loose Legendrian $(5,-8)$-torus knots  $L_\pm^{i,k+\scriptscriptstyle\frac 12}$ satisfying $\tb(L_\pm^{i,k+\scriptscriptstyle\frac 12})=i$, $\tor(L_\pm^{i,k+\scriptscriptstyle\frac 12})=k+1/2$, 
  \[
    \rot(L_\pm^{i,k+\scriptscriptstyle\frac 12})=\begin{cases} 
    \mp(i + 27) & \text{ in } \xi_1,\\
    \mp(i + 7)& \text{ and } \xi_7,
    \end{cases}
  \]
  \[
    S_\pm(L_\pm^{i,k+\scriptscriptstyle\frac 12})=L_\pm^{i-1,k+\scriptscriptstyle\frac 12} \text{ and } S_\mp(L_\pm^{i,k+\scriptscriptstyle\frac 12}) \text{ is loose.}
  \]
\end{enumerate}
\end{theorem}

With the classification of non-loose Legendrian $(5,\pm 8)$-torus knots, one may easily classify non-loose transverse knots. 

\begin{theorem}\label{trans58}
  The $(5,8)$-torus knot has non-loose transverse representatives only in $\xi_{-1}$, $\xi_{-3}$, $\xi_{-7}$, $\xi_{-9}$, $\xi_{-15}$, $\xi_{-19}$, $\xi_{-27}$, $\xi_{-8}$, $\xi_{-4}$, $\xi_{-2}$, and $\xi_0$. The classification in each of these contact structures is as follows.
  \begin{enumerate}
  \item In $(S^3,\xi_{-1})$ there are two non-loose transverse knots $T$ and $T'$ with $sl(T)=-19$, $sl(T')=-21$, and the stabilization of $T$ is equivalent to $T'$. Neither knot has Giroux torsion in its complement. 
  \item In $S^3$ with the contact structure $\xi_{-3}$ or $\xi_{-7}$ there is exactly one non-loose transverse knot $T$ and it has $sl(T)=13$ in $\xi_{-3}$ and $-3$ in $\xi_{-7}$. Neither knot has Giroux torsion in its complement. 
  \item In $S^3$ with the contact structure $\xi_{-7}, \xi_{-9}, \xi_{-15}, \xi_{-19}$ or $\xi_{-27}$ there is a family of non-loose transverse knots $T^k$ for $k\geq 1$ with $\tor(T^k) = k$ and
  \[
    sl(T^k)=\begin{cases}
    -1& \text{ in } \xi_{-9}\\
    -11& \text{ in } \xi_{-15}\\
    -17& \text{ in } \xi_{-19}\\
    -27& \text{ in } \xi_{-27}.
    \end{cases}
  \]
  \item In $S^3$ with the contact structures $\xi_{-8},\xi_{-4},\xi_{-2},$ or $\xi_0$ there is a family of non-loose transverse knots $T^{k+\scriptscriptstyle\frac 12}$ for $k\geq 0$ with $\tor(T^{k+\scriptscriptstyle\frac 12}) = k+1/2$ and
  \[
    sl(T^{k+\scriptscriptstyle\frac 12})=\begin{cases}
    1& \text{ in } \xi_{-8}\\
    11& \text{ in } \xi_{-4}\\
    17& \text{ in } \xi_{-2}\\
    27& \text{ in } \xi_{0}.
    \end{cases}
  \]
  \end{enumerate}
\end{theorem}

Finally we can consider the transverse $(5,-8)$-torus knots.

\begin{theorem}\label{trans5-8}
  The $(5,-8)$-torus knot has non-loose transverse representatives only in $\xi_1, \xi_2, \xi_7,\allowbreak \xi_8, \xi_{14},$ and $\xi_{28}$. The classification in each of these contact structures is as follows.
  \begin{enumerate}
  \item In $(S^3,\xi_{2})$ there are two non-loose transverse knots $T$ and $T'$ with $sl(T)=27$, $sl(T')=25$, and the stabilization of $T$ is contactomorphic to $T'$. Neither knot has Giroux torsion in its complement.  
  \item In $(S^3, \xi_8)$ there is exactly one non-loose transverse knot $T$ and it has $sl(T)=-7$ and $\tor(T) = 0$. 
  \item In $S^3$ with the contact structure $\xi_{14}$ and $\xi_{28}$ there is a family of non-loose transverse knots $T^k$ for $k\geq 0$ with $\tor(T^k) = k$ and  
  \[
    sl(T^k)=\begin{cases}
    -7& \text{ in } \xi_{14}\\
    27& \text{ in } \xi_{28}
  \end{cases}
  \]
  \item In $S^3$ with the contact structures $\xi_{1}$ or $\xi_7$ there is a family of non-loose transverse knots $T^{k+\scriptscriptstyle\frac 12}$ for $k\geq 0$ with $\tor(T^{k+\scriptscriptstyle\frac12})=k+1/2$ and
  \[
    sl(T^{k+\scriptscriptstyle\frac 12})=\begin{cases}
    27& \text{ in } \xi_{1}\\
    7& \text{ in } \xi_{7}
    \end{cases}
  \]
  \end{enumerate}
\end{theorem}

\subsection{Qualitative features of non-loose Legendrian knots: known and new results}
Very little was known about the qualitative behavior of non-loose Legendrian knots, but we greatly illuminate their nature in this paper. The first most basic result about non-loose Legendrian and transverse knots is a Bennequin type inequality concerning their classical invariants. 

\begin{theorem}[\'Swi\c{a}tkowski, see \cite{Dymara01, Etnyre13}]\label{looseLegbound}
Let $(M,\xi)$ be an overtwisted contact $3$-manifold and $L$ a non-loose Legendrian knot in $\xi$. Then 
\[
  -|\tb(L)|+|\rot(L)|\leq -\chi(\Sigma)
\]
for any Seifert surface $\Sigma$ for $L$. For non-loose transverse knots we have
\[
  sl(T)\leq -\chi(\Sigma). 
\]
\end{theorem}

The theorem above gives restrictions on the mountain range for non-loose representatives of $K$, see Figure~\ref{geogfig}. 

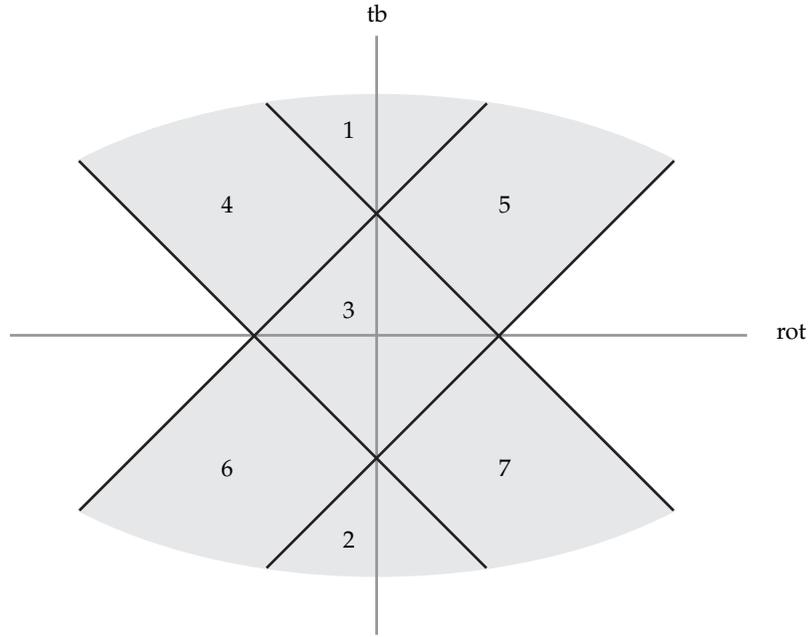
\begin{figure}[htbp]{\footnotesize
  \vspace{0.1cm}
  \begin{overpic}{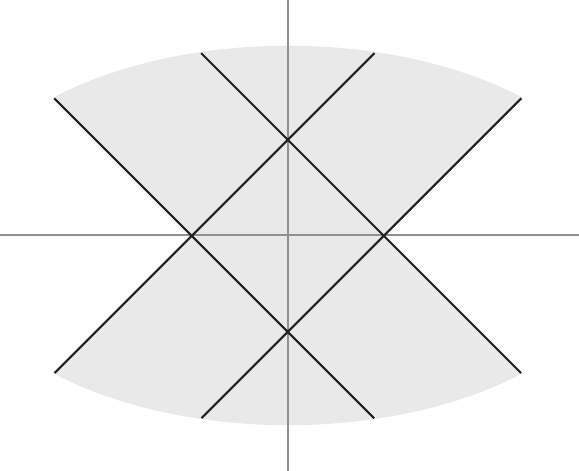}
    \put(290, 112){$\rot$}
    \put(135, 232){$\tb$}
    \put(126, 188){$1$}
    \put(126, 33){$2$}
    \put(126, 120){$3$}
    \put(80, 160){$4$}
    \put(185, 160){$5$}
    \put(80, 60){$6$}
    \put(185, 60){$7$}
  \end{overpic}}
  \vspace{0.1cm}
  \caption{The mountain range of non-loose Legendrian representatives of a knot type $K$ is contained in the shaded region defined by the four lines $l_1, l_2, l_3,$ and $l_4$. The lines divide the shaded region into the 7 parts shown.}
  \label{geogfig}
\end{figure}

Recall that $\Phi:\mathcal{L}(K)\to \Z^2$ is a map that sends $L\in \mathcal{L}(K)$ to $(\rot(L), \tb(L))$. If one considers the non-loose unknots discussed in Section~\ref{sec:prior}, we see that their image under $\Phi$ is an infinite $\textsf{V}$ with vertex at $(0,1)$.  Our classification of non-loose torus knots and \cite[Theorem~1.12 and 1.13]{Etnyre13} shows that the non-loose mountain range frequently \dfn{contains an $\textsf{X}$}, that is Legendrian knots whose invariants fill the integral points on a line of slope $1$ and a line of slope $-1$. Given our current knowledge it is reasonable to ask if all mountain ranges have such a feature.

\begin{question}
If $\xi$ admits non-loose representatives of a knot type $K$, then does the mountain range of non-loose Legendrian knots of the knot type $K$ in $\xi$ always contain a $\textsf{V}$ or $\textsf{X}$? Does it always contain an $\textsf{X}$ if $\xi$ is not $\xi_1$?
\end{question}

Looking back at \cite[Theorem~1.13]{Etnyre13}, a reasonable place to start looking for a knot where the answer was NO, would be to consider a knot type $K$ for which $-K$ is not smoothly isotopic to $K$. 

In \cite{Etnyre13}, there were many questions asked as to what the mountain ranges of non-loose Legendrian knots could look like. At the time the only known examples showed you could have $\textsf{X}$s and $\textsf{V}$s and that was all. In particular it was asked if there could ever be any non-loose Legendrian knots in Regions~1 or~2 in Figure~\ref{geogfig}, and if so could there be multiple (maybe infinitely many) representatives mapping to a fixed point in one of those regions. If one considers, for example, the $(n-1,n)$-torus knot in $\xi_1$ we can see that there are indeed non-loose Legendrian knots with invariants in Region~1, but at most one Legendrian representative can map to any point there. We do not know if Region~2 can be populated. So we as the following question. 

\begin{question}
Are there any non-loose Legendrian knots with invariants in $R_2$? Is the cardinality of the pre-image of any point in $R_1$ under the geography map $\Phi$ on non-loose Legendrian knots bounded by 1?
\end{question}

It was also asked in \cite{Etnyre13} if for any knot $K$, there are at most finitely many contact structures in which it could have non-loose representatives and if there were finitely many contact structures in which there could be infinitely many Legendrian representatives mapping to a fixed point $(\rot,\tb)$. (It was suggested that there might just be two such contact structures.) All our examples point to the answer being YES to both these questions, but our examples of the $(2,-(2n+1))$-torus knots show that this finite number of contact structures can be arbitrarily large. 

In \cite{BakerOnaran15}, Baker and Onaran defined three invariants that quantify ``how tight" the complements of non-loose Legendrian knots are.  Our results give quite a bit of new information about two of them, so we discuss those. Given a non-loose knot $L$, they were the \dfn{depth of $L$}, $\dep(L)$, defined to be the minimal number of times an overtwisted disk intersects $L$ and the \dfn{tension of $L$}, $\ten(L)$, defined to be the minimal number of stabilizations needed to make $L$ loose. Previously examples only show that $\ten$ can be $1$ and $\dep$ can be $2$. So, in \cite[Problem~6.1]{BakerOnaran15}, Baker and Onaran asked for constructions of non-loose Legendrian knots with arbitrarily large tension and depth. We can show the following.
\begin{corollary}
Given any integer $n$ there is a non-loose Legendrian torus knot $L$ with $\ten(L)>n$ and $\dep(L)>n$. 
\end{corollary}
\begin{proof}
Theorem~\ref{bigwings} shows there are Legendrian torus knots with arbitrarily large wings. The peaks in these wings that are farthest from the infinite $\textsf{X}$ give the desired Legendrian knots. Similarly, using the algorithm in Section~\ref{thealgorithm} we see that there are positive torus knots with arbitrarily large diamonds, given such examples in the fixed contact structure $\xi_1$. 

It was shown in \cite{BakerOnaran15} that $t(L)\leq d(L)$, so the examples above show that $\dep(L)$ can be arbitrarily large. 
\end{proof}


Baker and Onaran also noted that tension can be refined to consider only positive or negative stabilizations, so we set $\ten_\pm(L)$ to be the minimal number of $\pm$ stabilizations that are required to make $L$ loose. From prior work, {\em e.g.}~\cite{Etnyre13}, it is clear that $t_\pm(L)$ can be infinite. In \cite[Question 6.7]{BakerOnaran15} it was asked if there was an $L$ such that both $\ten_+(L)$ and $\ten_-(L)$ can be infinite. Similarly, Baker and Onaran \cite[Question 6.7]{BakerOnaran15} asked if one could have both $t_\pm(L)$ larger than $t(L)$. 

\begin{corollary}
Each negative torus knot contains a non-loose Legendrian representative $L$ such that both $\ten_+(L)$ and $\ten_-(L)$ are both infinite, while $\ten(L)=2$. 
\end{corollary}
\begin{proof}
The ``extra" Legendrian $L_e$ for negative torus knots given in Theorem~\ref{thm:>pq} is an example of such a non-loose Legendrian knot. 
\end{proof}

In \cite[Problem~6.6]{BakerOnaran15} it was asked if $\dep-\ten$ can be arbitrarily large. While we cannot answer this question, it seems likely that the knot $L_e$ from the previous proof will give a non-loose Legendrian knot where this difference is large. In particular, it would be interesting to see how $L_e$ intersects overtwisted disks. 


\subsection*{Acknowledgements}
The authors thank Rima Chatterjee for reminding us of interesting questions on non-loose knots, and Irena Matkovi\v c for a helpful discussion, in particular the idea of Theorem~\ref{parity}. We thank Shunyu Wan for pointing out an error in the first version of Theorem~\ref{parity}. We also thank B\"{u}lent Tosun for a useful conversation. The authors are grateful to the referee for many insightful comments on improving the manuscript. The first and third authors were partially supported by NSF grants DMS-1906414 and 2203312.

\section{Background and preliminary observations}

We assume the reader is familiar with basic Legendrian and transverse knot theory as discussed in \cite{EtnyreHonda01} and the convex surface theory in \cite{Honda00a}. 
In Section~\ref{oldclassification}, we review the classification of contact structures on simple manifolds that will be needed in the remainder of the paper. 
Before that, we recall the definition of the Farey graph and some important properties about paths in the Farey graph that are needed in those classifications and their applications. In Section~\ref{lutzsection} we recall the definition of Lutz twists and, for future use, relate them to Legendrian knots. 
In the following section, we discuss when Legendrian knots on the boundary of thickened tori are related by stabilization. 
Then in Section~\ref{subsec:pathsinFG} we discuss the construction of contact structures on $S^3$ using pairs of decorated paths in the Farey graph. 
We then show how to translate this description of the contact structure into a contact surgery diagram in $(S^3,\xi_{std})$. Then in Section~\ref{htpclasses}, we see how to compute the $d_3$-invariant of these contact structures as well as the rotation numbers of some Legendrian knots in them. We end this section by classifying contact structures, with certain boundary conditions, on $P\times S^1$ where $P$ is a pair of pants (that is, a disk with two disjoint sub-disks removed). 

\subsection{Continued fractions and paths in the Farey graph}\label{fareygraph}
We will keep track of curves on a torus using the Farey graph. The Farey graph is constructed as follows. Consider the unit disk in the $xy$-plane. Label the point $(0,1)$ with $0=0/1$ and $(0,-1)$ with $\infty=1/0$. Connect $0$ and $\infty$ with a straight line. Now if a point on the boundary of this disk with positive $x$-coordinate lies half way between two points labeled $a/b$ and $c/d$ that are connected by an edge, then label it ${(a+c)/(b+d)}$ and connect this point to each of the other two by a hyperbolic geodesic (put the hyperbolic metric on the interior of the unit disk). We call this the \dfn{Farey sum} of $a/b$ and $c/d$, and denote it by $\frac ab \oplus \frac cd$. (We will also use $\frac ab \ominus \frac cd$ to represent ${(a-c)/(b-d)}$.) If we iterate this construction, then all the positive rational numbers will appear. We can repeat this construction for points on the boundary with negative $x$-coordinate, but when we do we let $\infty={-1/0}$, so we will get all the negative rational numbers. 

\subsubsection{Curves on tori}
Recall that if one fixes a basis $\lambda, \mu$ for $H_1(T^2)$, then embedded curves on $T^2$ are in one-to-one correspondence with $\Q\cup \{\infty\}$, where $a/b$ is associated to the embedded curve on $T^2$ in the homology class $a\mu+b\lambda$. One may easily check that two curves associated to the numbers $r$ and $s$ form a basis for $H_1(T^2)$ if and only if there is an edge between $r$ and $s$ in the Farey graph. 

We also introduce the dot product of two rational numbers: $\frac ab \bigcdot \frac cd = ad - bc$ and note that $\left|\frac ab \bigcdot \frac cd\right|$ is the minimal number of times curves associated to $\frac ab$ and $\frac cd$ can intersect.

We end this section by setting up some notation that will be used frequently in the rest of the paper. 
Given two numbers $r$ and $s$ in $\Q\cup\{\infty\}$ we let $[r,s]$ denote the elements in  $\Q\cup\{\infty\}$ that are clockwise of $r$ in the Farey graph and anti-clockwise of $s$. 

\subsubsection{Paths in the Farey graph}\label{generalpathsinfg}
We have the following well-known lemma about continued fractions, see for example \cite[Lemma~2.1]{EtnyreLaFountainTosun12}.

\begin{lemma}\label{adjacentvertices}
  Suppose $q/p<-1$. Given $q/p=[a_1,\ldots, a_n]$, let $\left(q/p\right)^c= [a_1,\ldots, a_n+1]$ and $\left(q/p\right)^a=[a_1,\ldots, a_{n-1}]$. There will be an edge in the Farey graph between each pair of numbers $q/p$, $\left(q/p\right)^c$, and $\left(q/p\right)^a$. Moreover, $\left(q/p\right)^c$ will be the farthest clockwise point from $q/p$ that is larger than $q/p$ with an edge to $q/p$, while $\left(q/p\right)^a$ will be the farthest anti-clockwise point from $q/p$ that is less than $q/p$ with an edge to $q/p$.  
\end{lemma}

We note that when $a_n+1=-1$ the rational number $[a_1,\ldots, a_n+1]$ is the same as $[a_1,\ldots, a_{n-1}+1]$. Thus, by convention, if $a_n+1=-1$, then $[a_1,\ldots, a_n+1]$ is considered to be $[a_1,\ldots, a_{n-1}+1]$. 


A path in the Farey graph is a sequence of elements $p_1,\ldots, p_k$ in $\Q\cup\{\infty\}$ moving monotonically clockwise or anti-clockwise, such that each $p_i$ is connected to $p_{i+1}$ by an edge in the Farey graph, for $i<k$.

An immediate corollary of the above lemma is the following useful result. 
\begin{lemma}\label{lem:pathsinFG}
The minimal monotone path $\{q_1,\ldots q_l\}$ in the Farey graph going from $q_1=q/p<-1$ clockwise to $q_l=-1$ can be computed from the continued fraction expansion of $q/p$ as follows: if $q_i=[c_1,\ldots, c_j]$ then $q_{i+1}=[c_1,\ldots, c_j+1]$. 

Similarly a minimal monotone path $\{p_1,\ldots, p_k\}$ from $p_1=q/p$ anti-clockwise to $p_k=\lfloor q/p \rfloor$ can be computed as follows: if $p_i=[b_1,\ldots, b_j]$ then $p_{i+1}=[b_1,\ldots, b_{j-1}]$.

In particular, we can inductively compute all the $p_i$ and $q_j$. Also, $k=n$ and $l=|a_n|-n-1+\sum_{i=1}^{n-1} |a_i+1|$. 
\end{lemma}


We will now explore paths in the Farey graph further. The basic building blocks for such paths are {continued fraction blocks}. 

\begin{definition}
A path $\{p_0,\ldots, p_k\}$ in the Farey graph is a \dfn{continued fraction block} if there is some change of basis for $\Z^2$ such that $\{p_0,\ldots, p_k\}$ becomes the path $\{0,\pm1, \ldots, \pm k\}$. In other words, if we express a fraction $a/b$ as a vector $\vects{b}{a}$, then there is an element of $SL(2,\Z)$ that takes the vectors corresponding to $\{p_0,\ldots, p_k\}$ to $\{0,\pm1,\ldots, \pm k\}$. The sign is determined by whether the path is clockwise or anti-clockwise. 
\end{definition}

We will give another useful way to describe continued fraction blocks. Given points $p_0>p_1$ in the Farey graph that are connected by an edge, there will be a unique point $t$ that is not in $[p_1,p_0]$ and has an edge to $p_1$ and $p_0$. We define $p_l=p_0\oplus l t$ (this means we have ``Farey summed" $t$ to $p_0$, $l$ times). See Figure~\ref{CFB}.
\begin{figure}[htb]{\tiny
\begin{overpic}
{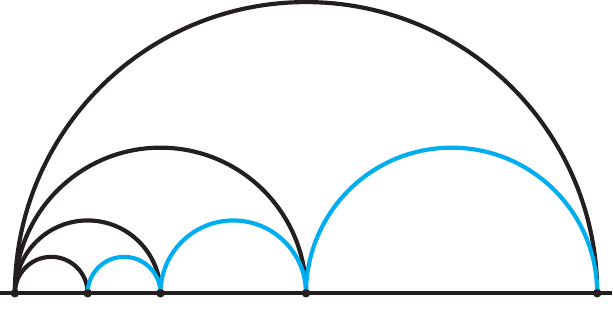}
\put(286, 4){$p_0$}
\put(146, 4){$p_1$}
\put(75, 4){$p_2$}
\put(39, 4){$p_3$}
\put(5, 4){$t$}
\end{overpic}}
\caption{A continued fraction block}
\label{CFB}
\end{figure}
Note that $p_0$ with this definition agrees with the original $p_0$. Fixing any $k$, the path $\{p_0,p_1,\ldots, p_k\}$ is an anti-clockwise continued fraction block. One may also check that the element of $SL(2,\Z)$ that takes this path to $0,-1,\ldots, -k$ has positive determinant. Similarly, if $p_0< p_1$ and $t$ is the unique point outside of $[p_1,p_0]$ with an edge to $p_0$ and $p_1$, then we can consider the points $p_l=p_0\oplus l t$ and obtain a clockwise continued fraction block $\{p_0,p_1,\ldots, p_k\}$.

Using Lemma~\ref{adjacentvertices} we can relate continued fraction blocks to continued fractions. Suppose $q/p=[a_1,\ldots, a_n]<-1$ and $a_n\leq -2$ for all $i$, then we can see that 
\[
\left\{ [a_1,\ldots, a_n], [a_1,\ldots, a_n+1],\ldots, [a_1,\ldots, a_{n-1}, -1] \right\}
\]
is a continued fraction block. Indeed, setting $p_0=[a_1,\ldots, a_{n-1}-1]$ and $t=[a_1,\ldots, a_{n-1}]$ one may check that the points above are of the form $p_0\oplus l t$. 

We now continue our discussion of minimal paths by seeing that any minimal path is a concatenation of continued fraction blocks. To this end, we examine how to continue a path after finishing a continued fraction block. We will consider anti-clockwise paths, but as above, a similar discussion holds for clockwise paths. Let $p_0$ and $t$ be vertices in the Farey graph that share an edge with $t<p_0$. Suppose $p_l= p_0\oplus l t$. We consider the continued fraction block $p_0,\ldots, p_k$ for some $k$. If we wish to continue this path (but not as part of a continued fraction block) then the next vertex must be in $(p_{k+1},p_k)$  since the edge from $p_{k+1}$ to $t$ prevents any vertex in $[t,p_{k+1})$ from having an edge to $p_k$ and the edge from $p_k$ to $t$ prevents any vertex in $[p_k,t]$ from having an edge to $p_k$. See Figure~\ref{NewCFB}. 
\begin{figure}[htb]{\tiny
\begin{overpic}
{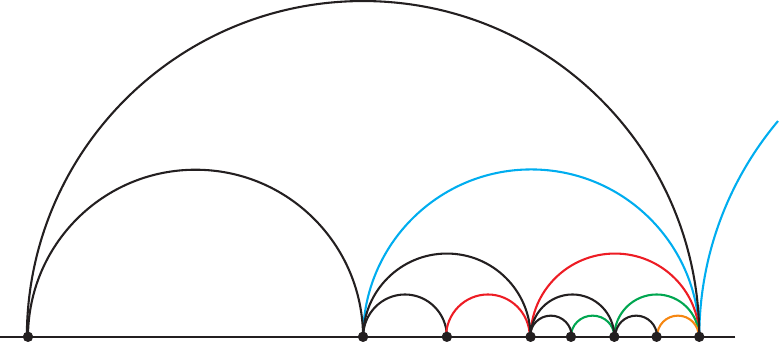}
\put(335, -2){$p_k$}
\put(171, -2){$p_{k+1}$}
\put(12, -2){$t$}
\put(253, -2){$w^1$}
\put(293, -2){$w^2$}
\put(312, -2){$w^3$}
\end{overpic}}
\caption{Transitioning between continued fraction blocks. The blue is the original continued fraction block. The red path is the continued fraction block $1$ down from the blue. The green path is $2$ down from the blue, and the orange is $3$ down from the blue. For the red and green paths we show two edges in the continued fraction block.}
\label{NewCFB}
\end{figure}
Now the vertices with an edge to $p_k$ in $[p_{k+1},p_k]$ are $w^i=p_{k+1}\oplus i p_{k}$. If the next jump in the path is to $w^n$, then we say we are starting a new continued fraction block that is \dfn{$n$-down} from the previous continued fraction block. Then the vertices in the next part of the path will be $w^n_l= w^n \oplus l w^{n-1}$ for $l=0,\ldots, m$ for some $m\geq 0$. (If $m=0$, then there is only one edge in this continued fraction block.) We can now repeat this process of adding ``continued fraction blocks down from the previous continued fraction block" until we have completed our path.

\subsubsection{Shortening paths in the Farey graph:}\label{shorteningpaths} We now discuss shortening of a non-minimal path in the Farey graph. Suppose $P'$ is a minimal path in the Farey graph from $r$ anti-clockwise to $s$, and $P''$ is a shortest path in the Farey graph from $s$ anti-clockwise to $t$. We will consider the path $P$ that is the concatenation of $P'$ and $P''$. If $P$ is not a shortest path in the Farey graph, then we can eliminate some vertices of $P$, we call this process \dfn{shortening} the path $P$. 

We now systematically consider this shortening process. Since $P'$ and $P''$ are shortest paths, we cannot eliminate any vertices in those paths, but we will be able to eliminate $s$. There are only two ways this can happen. Either the path $P'$ ends in a continued fraction block with length greater than $1$, or it ends in one with length $1$. In the first case, suppose $p_0,\ldots, p_k$ is the last continued fraction block in $P'$ and $k>1$. With the notation above, we see that the first edge in $P''$ must go from $p_k$ to $t$, since if the edge went to a vertex in $(t, p_k)$ then the path could not be shortened at $p_k$ (and the edge between $p_0$ and $t$ prevents an edge from $p_k$ to anything outside of $[t,p_k]$). See Figure~\ref{CFB}. Thus, we may shorten the path by removing the last edge of $P'$ (from $p_{k-1}$ to $p_k$) and the first edge of $P''$ (from $p_k$ to $t$) and replace them with an edge from $p_{k-1}$ to $t$. We may continue shortening the path until we have removed all the edges in the continued fraction block and replaced them with an edge from $p_0$ to $t$. At this stage, we have two new paths $P'_1$ and $P''_1$, where $P'_1$ is $P'$ with all the edges in the last continued fraction block removed and the edge from $p_0$ to $t$ added, and $P''_1$ is $P''$ with the edge from $t$ to $p_k$ removed. It is clear that both $P'_1$ and $P''_1$ are minimal paths ($P''_1$ since it is a subset of a minimal path, and $P'_1$ because it is $P'$ with the last continued fraction block removed, which is minimal, extended by either an edge extending the second to last continued fraction block, or edge in a continued fraction block that is ``down" from it, see Figure~\ref{NewCFB}). Now $P'_1$ and $P''_1$ are two minimal paths that can be concatenated. If they are minimal, then we have shortened $P'\cup P''$ to a minimal path; if not, we repeat the procedure above. 

In the case that $P'$ ends in a continued fraction block of length $1$, then either $P''$ starts with a continued fraction block of length greater than $1$ or it also has length $1$. In the former case, we can proceed as above but with the roles of $P'$ and $P''$ interchanged. 

In the latter case, we can remove the last edge from $P'$, say from $p$ to $s$, and the first edge from $P''$, say from $s$ to $q$, and then add the edge from $p$ to $q$ to either the first or last path to obtain $P'_1$ and $P''_2$ as above.


\subsubsection{Pairs of paths that can be shortened to a single edge}\label{sortentoone} 
We will see below that a key part of our study will understand when you can concatenate two shortest paths in the Farey graph and obtain a path that will shorten to a single edge in the Farey graph. In this section we study such pairs of paths. 

Since there is an element of $SL(2,\Z)$ that takes any edge in the Farey graph to another, we will consider the edge between two integers, say $n$ and $n-1$. Let $q/p$ be a number in $(n-1,n)$. We will consider the minimal anti-clockwise path $P_1$ from $q/p$ to $n-1$ and the minimal clockwise path $P_2$ from $q/p$ to $n$. Let $A$ be the continued fraction block in $P_1$ starting at $q/p$ and let $B$ be the continued fraction block in $P_2$ starting at $q/p$. The \dfn{length} of a continued fraction block $C$ is the number of points in $C$ minus 1, which is the number of edges in $C$. Denote the length of a continued fraction block $C$ by $|C|$. 

\begin{lemma}\label{morethanone-mustbeone}
  Let $q/p$ be a rational number in $(n-1,n)$, and let $P_1$ be the minimal anti-clockwise path from $q/p$ to $n-1$ and  $P_2$ the minimal clockwise path from $q/p$ to $n$. Let $A$ and $B$ be the continued fraction blocks in $P_1$ and $P_2$ starting at $q/p$, respectively. Then either $A$ or $B$ has length $1$. Moreover, if both have length $1$, then $q/p=(2n-1)/2$.
\end{lemma}

\begin{proof}
Since we know the union of the paths shortens, the discussion of shortening paths at the end of Section~\ref{shorteningpaths} shows that either $A$ or $B$ has length $1$. 

 One may readily check that for ${-(2n+1)/2}$ both $P_1=A$ and $P_2=B$ have length $1$. 
 
 Now if $q/p\not={-(2n+1)/2}$ then 
  the first edge in $P_1$ is from $q/p$ to $(q/p)^a$ and the first edge in $P_2$ is from $q/p$ to $(q/p)^c$. Recall there is an edge in the Farey graph from $(q/p)^a$ to $(q/p)^c$. Now, as any two vertices in the Farey graph that share an edge, also both share an edge to exactly two other vertices, we know that $(q/p)^a$ and $(q/p)^c$ share an edge to $q/p$ and another vertex $v$. Since $q/p \neq {-2(n+1)/2}$, we can assume $(q/p)^a > n-1$ or $(q/p)^c < n$. Since $n$ and $n+1$ have an edge, $v$ must be in $[n,n+1]$ and outside $[(q/p)^a,(q/p)^c]$. If $v > (q/p)^c$, then we see that $v$ is a vertex in $P_2$ and since $\{q/p, (q/p)^c, v\}$ is a continued fraction block, we see that $B$ has length greater than $1$. Similarly if $v$ is less than $(q/p)^a$ then $A$ has length greater than $1$. 
\end{proof}

If $|A|=1$ then we denote by $(A_1,A_3,\ldots, A_{n})$ the subdivision of $P_1$ such that each $A_i$ is a continued fraction block and $A_1 = A$, and denote by $(B_2,B_4 ..., B_{m})$ the subdivision of $P_2$ such that each $B_i$ is a continued fraction block and $B_2 = B$. If $|B|=1$, then we denote the continued faction blocks by $(A_2,A_4, ..., A_{n})$ and $(B_1,B_3,...,B_{m})$. (If $|A|=1=|B|$ then one may choose either numbering convention and we know $q/p={-(2n+1)/2}$ by Lemma~\ref{morethanone-mustbeone}.) 

\begin{example}
  Consider the two paths for $-21/8=[-3,-3,-3]$. In this case we have $P_1=\{-21/8, -8/3,-3\}$ and $P_2=\{-21/8, -13/5, -5/2,-2\}$, and the subdivisions $A_1=\{-21/8, -8/3\}$, $A_3=\{-8/3, -3\}$, $B_2=\{-21/8, -13/5,-5/2\}$, $B_4=\{-5/2,-2\}$. 
\end{example}

\begin{lemma}\label{lengthsofpaths}\label{obs}
There is a unique way to shorten the path $P_1\cup\overline{P}_2$, and the number of continued fraction blocks in $P_1$ and in $P_2$ differ by at most one: $|m-n|\leq 1$.
\end{lemma}

\begin{proof}
  The first statement of the lemma follows from our discussion of path shortening in Section~\ref{shorteningpaths}; however, we revisit it in our specific setting, as we will need the notation and ideas below.

  We recommend considering Figure~\ref{AB2example} while reading the proof. Let $(p_1,...,p_k)$ be the points in $P_1$ where $p_1=q/p$ and $(q_1,...,q_l)$ be the points in $P_2$ where $q_1 = q/p$. Clearly, the first point that can be removed is $p_1 = q_1 = q/p$. To continue, we suppose we are in the case where the continued fraction blocks are labeled 
  \[
    (A_2,\ldots, A_{n}) \text{ and } (B_1,B_3, \ldots, B_{m}),
  \] 
  so $B_1$ has length one. The proof of Lemma~\ref{morethanone-mustbeone} shows that there is an edge in the Farey graph between $q_2 = (q/p)^c$ and all the vertices in the continued fraction block $A_2=\{p_1,\ldots, p_i\}$. Thus we may first remove $p_1=q/p$ from the path, then $p_2$ and continue until we have removed $p_{i-1}$. Notice that there is an edge from $p_i$ to $q_3$, so the next vertex we can shorten is $q_2$, but before we do that, we will analyze the new pair of paths $P_1'$ from $q_2$ to $n-1$ and $P_2'$ from $q_2$ to $n$. We now have two cases to consider: if $B_3$ has length $1$ or greater.
  See Figure~\ref{AB2example}. 

  \begin{figure}[htbp]
    \begin{overpic}{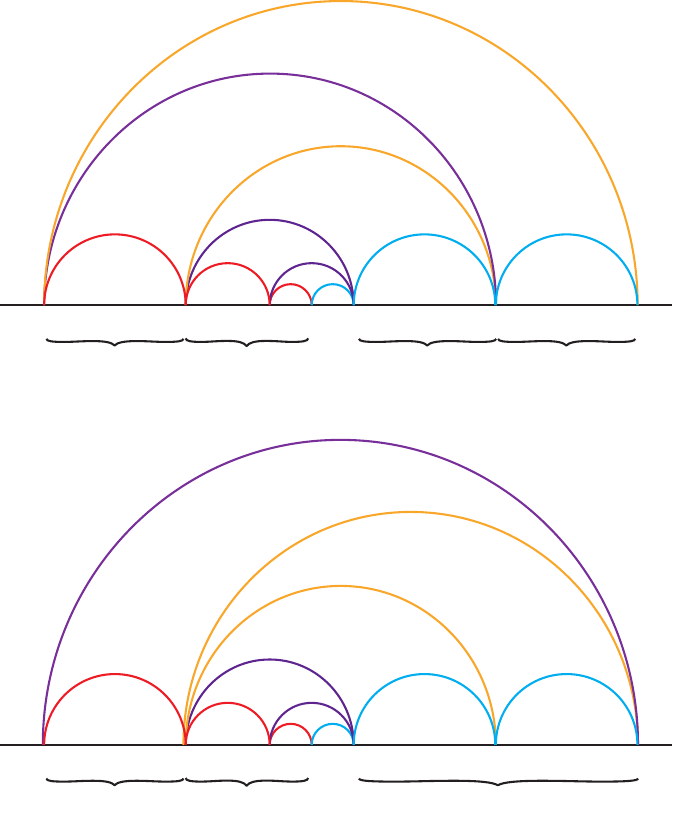}
      \put(15, 26){$p_4$}
      \put(87, 26){$p_3$}
      \put(127, 26){$p_2$}
      \put(144, 26){$q/p$}
      \put(115, 1){$A_2$}
      \put(50, 1){$A_4$}
      \put(158, 1){$B_1$}
      \put(236, 1){$B_3$}
      \put(170, 26){$q_2$}
      \put(237, 26){$q_3$}
      \put(305, 26){$q_4$}
      
      \put(15, 238){$p_4$}
      \put(87, 238){$p_3$}
      \put(127, 238){$p_2$}
      \put(144, 238){$q/p$}
      \put(115, 213){$A_2$}
      \put(50, 213){$A_4$}
      \put(158, 213){$B_1$}
      \put(203, 213){$B_3$}
      \put(270, 213){$B_5$}
      \put(170, 238){$q_2$}
      \put(237, 238){$q_3$}
      \put(305, 238){$q_4$}
    \end{overpic}
    \caption{Two types of paths that behave differently when shortening. The difference is whether or not $B_3$ has length $1$ or not. (The edges are not to scale to fit into the picture.)}
    \label{AB2example}
  \end{figure}

  If $B_3$ has length $1$, then there is an edge from $p_{i+1}$ to $q_3$ and thus the edge from $p_i$ to $q_2$ extends the continued fraction block $A_4$, i.e. $A'_4 = \{q_2, p_i\} \cup A_4$ is a continued fraction block. Thus we have
  \[
    P_1' = (A_4',A_6,\ldots,A_{n}) \text{ and } P_2' = (B_3,\ldots,B_{m})
  \]
  and $P'_1$ is a minimal path from $q_2$ anti-clockwise to $\lfloor q/p \rfloor$, and $P'_2$ is is a minimal path from $q_2$ clockwise to $n-1$. Thus, $(P_1',P_2')$ are a pair of minimal paths whose concatenation can be shortened at $q_2$ and each has one less continued fraction block than $P_1$ and $P_2$, respectively. We can now inductively continue to shorten the path until we have the path of length one from $n-1$ to $n$. 

  In the case when $|B_3|>1$, $A_2' = \{q_2,p_i\}$ replaces $A_2$ and we have 
  \[
    P_1'=(A'_2, A_4 \ldots, A_{2n}) \text{ and } P_2' = (B_3,\ldots,B_{2m-1})
  \]
and $P'_1$ is still a minimal path from $q_2$ anti-clockwise to $n-1$, but now $A'_2$ is its own continued fraction block with length $1$. So the number of continued fraction blocks in $P_1'$ is the same as for $P_1$, while the number in $P_2'$ is one less than in $P_2$. Moreover, numbering the continued fraction block by our convention above will give $P_1'$ the odd indices and $P_2'$ the even. Once again, we can inductively continue to shorten the paths until we have the path of length one from $n-1$ to $n$. 
  
We notice that (1) when the second odd-numbered continued fraction block has length $1$, then, after shortening, the new pairs of paths both have $1$ less continued fraction block and the parity of the numbering of the continued fraction blocks stays the same, and (2) when the second odd-numbered continued fraction block has length greater than one, then the new path with odd numberings has one less continued fraction block while the other has the same number and the parity of the numberings on the paths switches. 
 
This last observation implies that $|n-m|\leq 1$. 
\end{proof}

\subsection{Contact structures on \texorpdfstring{$T^2\times [0,1]$}{T2 x [0,1]}, solid tori, and lens spaces}\label{oldclassification}
Both Giroux \cite{Giroux00} and Honda \cite{Honda00a} classified tight contact structures on $T^2\times[0,1]$, solid tori, and lens spaces. Below, we discuss the classification along the lines of Honda. We note that these manifolds come with a given orientation, and the classification only considers positive contact structures. (Though we do not need this here, we note that the case for negative contact structures can easily be deduced from the statements below.)

\subsubsection{Contact structures on \texorpdfstring{$T^2\times [0,1]$}{T2 x [0,1]}}
Suppose $\xi$ is a tight contact structure on $T^2\times [0,1]$ with convex boundary and the dividing slope on $T^2\times\{i\}$ is $s_i$ for $i=0,1$. We say that $\xi$ is \dfn{minimally twisting} if any convex torus in $(T^2\times [0,1], \xi)$ that is parallel to the boundary has dividing slope in $[s_0,s_1]$.  

We denote by $\Tight^{min}(T^2\times [0,1]; s_0, s_1)$ the minimally twisting tight contact structures, up to isotopy, on $T^2\times[0,1]$ with convex boundary having two dividing curves of slope $s_0$ and $s_1$ on $T^2\times \{0\}$ and $T^2\times \{1\}$, respectively. Given any path $P$ in the Farey graph that starts at $s_0$ and goes clockwise to $s_1$, we say $P$ is a \dfn{decorated path} if its edges have each been labeled with a $+$ or a $-$. When $P$ is a minimal path, we say two decorated paths are \dfn{equivalent} if the number of $+$ signs in each continued fraction block is the same. 

With the discussion of paths in the Farey graph in Section~\ref{generalpathsinfg}, the following theorem is equivalent to \cite[Theorem~2.2]{Honda00a}.

\begin{theorem}\label{T2XIclass}
  The contact structures in $\Tight^{min}(T^2\times [0,1]; s_0, s_1)$ are in one-to-one correspondence with equivalence classes of decorations on the minimal path in the Farey graph from $s_0$ to $s_1$. Let $P$ be a decorated path from $s_0$ to $s_1$. We denote by $\xi_P$ the contact structure associated to $P$. 
\end{theorem}

Notice that if $s_0$ and $s_1$ share an edge in the Farey graph, then there are exactly two tight contact structures in $\Tight^{min}(T^2\times [0,1]; s_0, s_1)$. These are called \dfn{basic slices} and the correspondence in the theorem can be understood in terms of stacking basic slices according to the decorated path. 

Consider $\Tight^{min}(T^2\times [0,1]; q/p, -1)$ with $q/p<-1$ (note that given any $s_0$ and $s_1$ there is a diffeomorphism of the torus taking $s_1$ to $-1$ and $s_0$ to such a $q/p$). Let $q/p=[a_1,\ldots, a_n]$, and recall $a_i\leq -2$. According to the discussion in the previous subsection, we know the number of edges in the continued fraction blocks in the minimal path from $q/p$ to $-1$ is $|a_n+1|$, $|a_{n-1}+2|$, \ldots, $|a_1+2|$. So, according to the theorem above, the number of contact structures in $\Tight^{min}(T^2\times [0,1]; q/p, -1)$ is 
\[
  |(a_1+1) \cdots (a_{n-1}+1)a_n|.
\]

Suppose $P$ is a non-minimal path in the Farey graph, then there will be a vertex $v$ in $P$ such that there is an edge between its neighboring vertices. Let $P'$ be the path obtained by removing $v$ and the edges coming into $v$ and adding the edge between $v$'s neighbors. We say $P'$ is obtained from $P$ by \dfn{shortening} at $v$. We discussed this process thoroughly in Section~\ref{shorteningpaths}. If $P$ is a decorated path that can be shortened, then we say the shortening is \dfn{consistent} if the labels on the edges that were removed are the same, and in this case, the shortened path will be labeled by labeling the new edge with the same sign as the removed edges. 

Given any decorated path in the Farey graph, even a non-minimal one, one can construct a contact structure on $T^2\times [0,1]$ by stacking basic slices. It will be important to know when this contact structure is tight. To this end, we have the following result, which is essentially Part~(2) of \cite[Theorem~2.2]{Honda00a}.

\begin{theorem}\label{shortening}
  Let $\xi$ be a contact structure on $T^2\times [0,1]$ described by a non-minimal clockwise decorated path $P$ in the Farey graph. Then $\xi$ is tight if and only if one may construct a shortest path from $P$ by consistent shortening. When $\xi$ is tight, it will be minimally twisting and is described by the decorated shortest path between the endpoints of $P$ obtained by labeling the added edges in the shortening process with the label of the two replaced edges. 
\end{theorem}

A simple corollary of this is the following result we will find useful.

\begin{corollary}\label{lengthen}
  Let $\xi \in \Tight^{min}(T^2\times [0,1]; s_0, s_1)$ be a basic slice represented by a single edge $P = \{s_0, s_1\}$ in the Farey graph labeled by $\pm$. Then any non-minimal path $P'$ from $s_0$ to $s_1$ whose edges are all labeled with the same sign also represents $\xi$.
\end{corollary}
The corollary clearly follows from the pervious theorem as $P'$ can be consistently shortened to obtain $P$. 

We now introduce convex torsion. Consider the contact structure $\xi=\ker(\sin 2\pi z\, dx + \cos 2\pi z\, dy)$ on $T^2\times \R$, where $(x,y)$ are the coordinates on $T^2$ and $z$ is the coordinate on $\R$. Consider the region $T^2\times [0,k]$ for $k\in \frac 12 \N$ and notice that the contact planes twist $k$ times as $z$ goes from $0$ to $k$. We can now perturb $T^2\times\{0\}$ and $T^2\times \{k\}$ so that they become convex with two dividing curves of slope $0$. Let $\xi^k$ be the resulting contact structure on $T^2\times [0,1]$ (after $T^2\times [0,k]$ has been identified with $T^2\times [0,1]$ in the obvious way). Notice that inside $(T^2\times[0,1], \xi^k)$ there is a basic slice with boundary slopes $0$ and $\infty$, one of whose boundary components agrees with $T^2\times\{0\}$. By considering $\xi^k$ and $-\xi^k$, this basic slice can be taken to be either positive or negative; when we wish to distinguish between the two, we write $\pm \xi^k$ accordingly.

We call $(T^2\times[0,1],\xi^k)$ a \dfn{convex $k$-torsion layer} and if we have a contact structure $(M,\xi)$ into which $(T^2\times[0,1],\xi^k)$ embeds, we say $(M,\xi)$ has \dfn{convex $k$-torsion}. We will use the phrase $(M,\xi)$ has \dfn{exactly convex $k$-torsion} to the situation where one can embed  $(T^2\times[0,1],\xi^k)$ into $(M,\xi)$ but not $(T^2\times[0,1],\xi^{k+\scriptscriptstyle\frac 12})$. We say $(M,\xi)$ has no convex torsion, or convex $0$-torsion, if  $(T^2\times[0,1],\xi^k)$ does not embed in $(M,\xi)$ for any $k \in \frac12\N$. The following result is proved in \cite[Lemma~5.2]{Honda00a}.

\begin{theorem}\label{twolayers}
  For $k \in \frac12\N$, there are exactly two contact structures $\pm \xi^k$, up to isotopy, on $T^2\times [0,1]$ with convex boundary having two dividing curves, both of slope $0$, and exactly convex $k$-torsion. The two contact structures are co-orientation reversing contactomorphic. 
\end{theorem}

Lastly, we define a new invariant $\tor$ for Legendrian and transverse knots. 
\begin{definition}\label{defnoftor}
For a Legendrian knot $L$, we define $\tor(L)$ by
\begin{align*}
  \tor(L) = \sup\left\{k \in \tfrac 12\N \mid \begin{array}{l} \xi^k \text{ embeds into the complement of a standard} \\ 
  \text{neighborhood of $L$ along a boundary-parallel torus}
\end{array}
\right\}.
\end{align*}
That is, $\tor(L)=k$ if there is an embedding of a convex $k$-torsion layer $(T^2\times[0,1],\xi^k)$ into the complement of the standard neighborhood of $L$ such that $T^2\times \{0\}$ is smoothly isotopic to the boundary of the neighborhood, but there is no such embedding of a convex $(k+1/2)$-torsion layer.  
For a transverse knot $T$, we denote $\tor(T) = k$ if $(T^2 \times [0,1], \xi^k)$ embeds in the complement of $T$ in a neighborhood of the boundary, but $(T^2 \times [0,1],\xi^{k+\frac 12})$ does not.   
\end{definition}

Note that if knots $L$ and $T$ are loose, we have $\tor(L)=\tor(T)=\infty$.

\subsubsection{Contact structures on solid tori}
Notice we can construct a solid torus from $T^2\times[0,1]$ in two ways. If we choose a rational slope $s$ on $T^2\times\{0\}$ and collapse the linear curves of slope $s$ on this torus, we will get a solid torus $S_s$. We call this the solid torus with lower meridional slope $s$. Similarly, we can collapse the linear curves of slope $s$ on $T^2\times\{1\}$ to get a solid torus $S^s$, and we say it has upper meridian $s$. 

We denote by $\Tight(S_s;r)$ the isotopy classes of tight contact structures on the solid torus $S_s$ with lower meridian $s$ and convex boundary having two dividing curves of slope $r$. Similarly, $\Tight(S^s;r)$ is the isotopy class of tight contact structures on the solid torus $S^s$ with upper meridian $s$ and convex boundary having two dividing curves of slope $r$. Up to normalizing coordinates on the solid torus, the following result can be found in \cite[Theorem~2.3]{Honda00a}.

\begin{theorem}\label{torusclass}
  Let $P$ be a minimal path in the Farey graph from $r$ clockwise to $s$. Then, the elements of $\Tight(S_r;s)$ are in one-to-one correspondence with equivalence classes of decorations on the path $P$ where the first edge is left undecorated. Similarly, the elements of $\Tight(S^s;r)$ are in one-to-one correspondence with equivalence classes of decorations on the path $P$ where the last edge is left undecorated.
\end{theorem}

We now consider formulas for the number of tight contact structures on some solid tori. If $r<-1$ and $r=[a_1,\ldots, a_n]$, then we see that 
\begin{equation}\label{numbersolidtorizero}
  |\Tight(S^0;r)|=|(a_1+1)\cdots (a_{n-1}+1)a_n|,
\end{equation}
this is because the minimal path from $r$ to $0$ is the same as the minimal path from $r$ to $-1$ followed by the last edge to $0$. Decorations on this first path from $r$ to $-1$ also characterize elements of $\Tight^{min}(T^2\times [0,1]; q/p, -1)$, as discussed above. 

Notice that there is an orientation preserving diffeomorphism from $T^2\times [0,1]$ to itself that exchanges the two $S^1$ factors of $T^2$ and inverts $[0,1]$. This diffeomorphism identifies $\Tight(S_\infty;r)$ with $\Tight(S^0;r^{-1})$. So if $r \notin \mathbb{Z}$, then $r-\lceil r\rceil \in (-1,0)$ and $\Tight(S_\infty;r)=\Tight(S_\infty;r-\lceil r\rceil)$ via the diffeomorphism that cuts the solid torus along the meridian disk and adds $-\lceil r\rceil$ twists before re-glueing. Thus if $\left(r - \left\lceil r\right\rceil\right)^{-1}=[b_1,\ldots, b_n]$, then
\begin{equation}\label{numbersolidtoriinf}
  |\Tight(S_\infty;r)|=|(b_1+1)\cdots (b_{n-1}+1)b_n|.
\end{equation}
Now if $r>1$, then as above we have 
\[
  \Tight(S^0;r)=\Tight(S_\infty,r^{-1})=\Tight(S_\infty;r^{-1}-1)=\Tight(S^0;(r^{-1}-1)^{-1}),
\]
and $(r^{-1}-1)^{-1}<-1$. So if $(r^{-1}-1)^{-1}=[a_1,\ldots, a_n]$, then the number of tight contact structures in $\Tight(S^0;r)$ is also given by the formula on the right-hand side of Equation~\eqref{numbersolidtorizero}. Lastly, we note that when $r \in \Z$, there is a unique tight contact structure on $(S_\infty;r)$.

We end our discussion of contact structures on solid tori with a simple observation.

\begin{lemma}\label{gluingtoriandthickened}
  Let $\xi$ be the unique tight contact structure in $\Tight(S_\infty;m)$. Given any contact structure $\xi'\in\Tight^{min}(T^2\times [0,1];n,m)$, for $m>n$ integers, there is an embedding of the unique tight contact structure $\xi''\in \Tight(S_\infty;n)$ into $(S_\infty,\xi)$ whose complement is $(T^2\times [0,1],\xi')$. In particular, gluing $(S_\infty, \xi'')$ to $(T^2\times[0,1], \xi')$ along $T^2\times\{0\}$ is tight. 
\end{lemma}

\begin{proof}
Notice that $(S_\infty, \xi)$ is a standard neighborhood of a Legendrian knot $L$. Now inside $S_\infty$ we can stabilize $L$. Let $N_\pm$ be the standard neighborhood of $S_\pm(L)$ in $S_\infty$. Notice that the dividing curves on $N_\pm$ have slope $m-1$. 
Thus, $S_\infty\setminus N_\pm$ is $T^2\times [0,1]$ with dividing curves of slope $m-1$ on one boundary component and dividing curves of slope $m$ on the other. In other words, $S_\infty\setminus N_\pm$ is a basic slice. Recall, there are exactly $2$ basic slices, and since there were two possible stabilizations of $L$, one of those stabilizations gives one of the basic slices and the other gives the other basic slice. 
  This establishes the lemma for $m-n=1$; in general, the lemma follows by further stabilizing $L$. 
\end{proof}

\subsubsection{Contact structures on lens spaces}\label{sec:lens}
We can construct a lens space from $T^2\times [0,1]$ as follows: choose a slope $r$ on $T^2\times \{0\}$ and a slope $s$ on $T^2\times\{1\}$ and let $L_r^s$ be the result of collapsing the linear curves of the given slope on the boundary components. We say $L_r^s$ is the lens space with upper meridian $s$ and low meridian $r$. Note that the lens space $L(p,q)$, which is $-p/q$ surgery on the unknot, can also be described at $L_{-p/q}^0$ (this is essentially the definition of $-p/q$ surgery on the unknot) and similarly as $L_\infty^{-q/p}$. This latter expression is because there is an orientation preserving diffeomorphism of $T^2\times [0,1]$ that exchanges the $S^1$ factors of $T^2$ and inverts the interval. 

Let $\Tight(L_r^s)$ be the isotopy classes of tight contact structures on the lens space $L_r^s$. 

After a change of coordinates on the boundary of $L_r^s$, the following theorem is equivalent to \cite[Theorem~2.1]{Honda00a}.

\begin{theorem}\label{lensclass}
  Let $P$ be a minimal path in the Farey graph from $r$ clockwise to $s$. Then the elements of $\Tight(L_r^s)$ are in one-to-one correspondence with equivalence classes of decorations on the path $P$ where the first and last edges are left undecorated. 
\end{theorem}

Arguing as above to count the number of minimally twisting contact structures on $T^2\times [0,1]$, we can easily compute the well-known formula that 
\[
  \Tight(L_r^0)=|(a_1+1)\cdots(a_m+1)|
\]
if $r<-1$ and $r=[a_1,\ldots, a_m]$. 

\begin{lemma}\label{lensandtorus}
Given $r$ and $s$ rational numbers, let $r'$ be the rational number such that $r'$ is clockwise of $r$ in the Farey graph and as close to $s$ as possible with an edge back to $r$. Similarly, $s'$ is the rational number such that $s'$ is anti-clockwise of $s$ in the Farey graph and the closest point to $r$ with an edge to $s$. Then 
\[
  |\Tight(L_r^s)|=|\Tight(S_r;s')|=|\Tight(S^s; r')|.
\]
\end{lemma}
\begin{proof}
From Theorem~\ref{lensclass} the contact structures in $\Tight(L_r^s)$ are in one-to-one correspondence with decorations (up to equivalence) on all but the first and last edges of the minimal path from $r$ clockwise to $s$. Notice that from the definition of $r'$ and $s'$ the minimal path starts with $r, r'$ and ends with $s',s$. Thus, the contact structures in $\Tight(L_r^s)$ are in one-to-one correspondence with decorations (up to equivalence) on the edges of the minimal path from $r'$ to $s'$. The same reasoning, but using Theorem~\ref{torusclass}, shows that these decorated paths are in one-to-one correspondence with contact structures in $\Tight(S_r;s')$ and $\Tight(S^s; r')$.
\end{proof}

\subsection{Transverse knots and Lutz twists}\label{lutzsection}
In this section, we recall the definition of a Lutz twist (and half Lutz twist) on a transverse knot, how it affects the homotopy class of plane field, and relate it to Legendrian approximations of the transverse knot. 

Let $T$ be a transverse knot in a contact manifold $(M,\xi)$. It is well-known that $T$ has a standard neighborhood $N$ that is contactomorphic to $N_s$, where $N_s=\{(\phi, r,\theta)\in S^1\times D^2: r\leq s\}$ for some small $s$, with the contact structure $\xi_{std}=\ker(\cos r\, d\phi + r\sin r\, d\theta)$. Notice on the torus $\partial N_a$, the characteristic foliation has slope $-\frac 1a \cot a$. Thus $\partial N_s$ has a linear characteristic foliation of slope $-\frac 1s \cot s$. Notice that $-\frac 1a \cot a$ goes to $-\infty$ as $a$ goes to $0$. Thus, inside of $N_s$, one can find other solid tori neighborhoods of $T$ with any characteristic foliation having any slope less than $-\frac 1s \cot s$. 
\begin{definition}
A \dfn{half Lutz twist} on $T$ is the contact structure on $M$ obtained by removing $N_s$ from $M$ and replacing it with $N_t$, where is the smallest number larger than $s$ such that $-\frac 1t \cot t=-\frac 1s \cot s$. Notice that after the half Lutz twist, the core of the solid torus that was glued in is also a transverse knot, thus we may iterate this and perform another half Lutz twist, the result will be a \dfn{(full) Lutz twist}. 
\end{definition}

In \cite{DingGeigesStipsicz05}, Ding, Geiges and Stipsicz showed that if we perform the half Lutz twist on a transverse knot $T$ in $(S^3, \xi)$ and obtain a new contact structure $(S^3, \xi')$, then the relative $d_3$-invariant is
\[
  d_3(\xi',\xi) = d_3(\xi) - d_3(\xi') =\self(T).
\]
In \cite{DingGeigesStipsicz05}, this was only verified when $\xi$ was the standard tight contact structure on $S^3$, but it is true in general, see \cite[Proof of Theorem~4.3.1]{Geiges08}.

For later use, we will now relate Lutz twists to Legendrian knots. Given a transverse knot $T$ and a standard neighborhood $N$ that is contactomorphic to $N_s$, as above, we note that there is some $n_0$ such that for all $n\leq n_0$ there is a standard neighborhood of $T$ in $N_s$ with boundary having linear characteristic foliation of slope $n$. We can perturb this neighborhood to have convex boundary with two dividing curves of slope $n$. Denote the result $N'_n$. This is a standard neighborhood of a Legendrian knot $L_n$ with $\tb(L_n)=n$. We call $L_n$ a \dfn{Legendrian approximation of $T$} and note that the transverse push-off of $L_n$ is $T$. For all these facts, and those below, see \cite{EtnyreHonda01}. We also note that $N'_m\subset N'_n$ for any $m\leq n$ and $L_m$ is the result of negatively stabilizing $L_n$, $m-n$ times. Moreover, $\overline{N'_n\setminus N'_m}$ is a thickened torus, and the contact structure on it is a continued fraction block with all negative basic slices. This discussion is the basis for the well-known fact that the classification of transverse knots is equivalent to the classification of Legendrian knots up to negative stabilization, see \cite[Theorem~2.10]{EtnyreHonda01}.

Turning now to half Lutz twists, we notice that replacing $N_s$ with $N_t$, as in the definition of half Lutz twist, is the same as removing $N'_n$, gluing in a negative basic slice with dividing slopes $\infty$ and $n$, then gluing in a negative basic slice with dividing slopes $0$ and $\infty$, and finally gluing in a solid torus with dividing slope $0$ and meridional slope $\infty$. 

Iterating the above construction, we can perform an \emph{$n$-fold half convex Lutz twist} on any Legendrian knot $L$, (which corresponds to a $n$-fold half Lutz twist on a positive transverse push-off of $L$) by 
\begin{itemize}
  \item removing a standard neighborhood of $L$, 
  \item gluing in a negative basic slice with dividing slopes $\infty$ and $\tb(L)$, 
  \item gluing the negative convex torsion layer $-\xi^{k-1}$,
  \item gluing in a negative basic slice with dividing slopes $\tb(L)$ and $\infty$, and
  \item gluing in a solid torus with dividing slope $\tb(L)$ and meridional slope $\infty$.
\end{itemize}
The last solid torus glued in above is a neighborhood of a Legendrian knot in the manifold obtained from the iterated half convex Lutz twist. By changing the sign of basic slices from negative to positive, we can also define an \emph{$n$-fold positive half convex Lutz twist} on $L$, which corresponds to a $n$-fold half Lutz twist on a negative transverse push-off of $L$. Lastly, we call a $2$-fold half convex Lutz twist simply a \dfn{convex Lutz twist}.  

\subsection{Destabilizing Legendrian knots}\label{seestabsection}
We will need to understand when two Legendrian knots are related by stabilization. More specifically, when given a contact structure on a thickened torus, if the boundary is convex and $L_0$ and $L_1$ are Legendrian knots on the boundary, we would like to know if they are related by stabilization. The following lemma seems well-known, but we could not find a complete proof in the literature, so we provide a proof here. 

\begin{lemma}\label{seestab}
  Let $(T^2\times[0,1], \xi)$ be a $\pm$ basic slice with dividing slopes $s_i$ on $T^2\times\{i\}$ for $i=0,1$. Let $L_0$ be a Legendrian ruling curve of slope $s_1$ on $T^2\times\{0\}$ and $L_1$ a Legendrian divide on $T^2\times \{1\}$. Then $L_0$ is $S_\pm(L_1)$. Moreover, if $L'_0$ is a Legendrian divide on $T^2\times \{0\}$ and $L'_1$ is a ruling curve of slope $s_0$ on $T^2\times\{1\}$, then $L'_1$ is $S_\mp(L'_0)$.

  Let $s$ be a vertex in the Farey graph outside the interval $[s_0,s_1]$ for which there are vertices in the Farey graph in $[s,s_0)$ with an edge to $s_1$.  If $L''_i$ is a ruling curve of slope $s$ on $T^2\times\{i\}$, then $L''_0$ is $S_\pm^k(L''_1)$ where $k=|(s_1 \ominus s_0) \bigcdot s|$. 
  
  Similarly, let $s$ be a vertex outside of $[s_0,s_1]$ for which there are vertices in the Farey graph in $(s_1,s]$ with an edge to $s_0$. If $L''_i$ is a ruling curve of slope $s$ on $T^2\times\{i\}$, then $L''_1$ is $S_\pm^k(L''_0)$ where $k=|(s_1 \ominus s_0) \bigcdot s|$. 
\end{lemma}

\begin{proof}
  We will show how to build a solid torus in $T^2\times [0,1]$ that is a regular neighborhood of $L_1$ and isotope $L_0$ into this neighborhood so that it has a standard neighborhood with dividing slope $\tb(L_1)-1$. This will establish that $L_0$ is a stabilization of $L_1$ and the sign of the stabilization is determined by the relative Euler class. That is, if $A$ is an annulus in $T^2\times [0,1]$ with boundary $L_0\cup L_1$, then $\rot(L_0)-\rot(L_1)$ is the relative Euler class of the basic slice evaluated on $A$, which in turn is $\chi(A_+)-\chi(A_-)$, where $A_\pm$ are the positive and negative regions of $A$ once it is made convex. 

  To construct the claimed solid torus, we take parallel copies $T_i$ of $T^2\times \{i\}$ inside $T^2\times [0,1]$ that are contact isotopic to the respective boundary components. The contact isotopy moves $L_i$ to a curve lying on $T_i$, and we relabel it again by $L_i$. Now we can take $A$ to be the annulus with boundary $L_0\cup L_1$. As the twisting of each component of $\partial A$ is non-positive, we can make $A$ convex. We claim that $A$ can be chosen so that it has a single dividing curve, which is a boundary-parallel arc with both boundary components on $L_0$. See the left-hand side of Figure~\ref{pictureproof}. We moreover claim that we can assume that a neighborhood of $L_1$ in $A$ looks like a half neighborhood of a Legendrian divide (see below for details on this normalization).
    \begin{figure}[htbp]
  \begin{overpic}
  {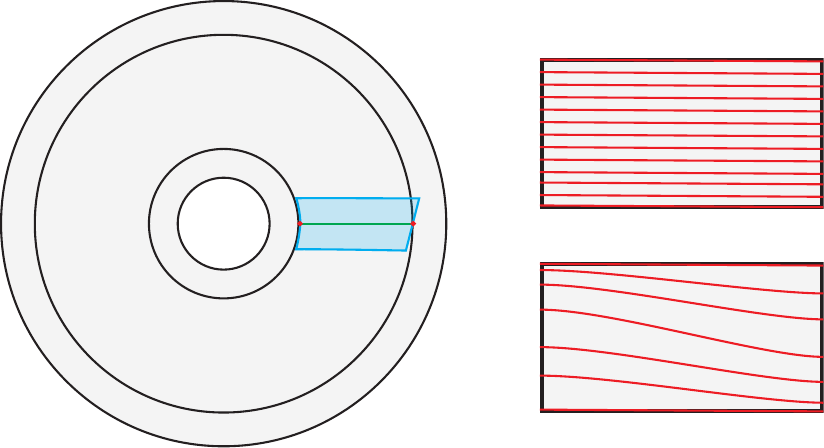}
    \put(130, 140){$T_0$}
    \put(155, 165){$T_1$}
    \put(188, 188){$T^2\times\{1\}$}
    \put(87, 103){$T^2\times \{0\}$}
    \put(165, 110){\color{dgreen}$A$}
    \put(201, 103){\color{red}$L_1$}
     \put(132, 103){\color{red}$L_0$}
     \put(300, 195){$(L_1\times [0,1])_\xi$}
     \put(315, 97){$A''_\xi$}
  \end{overpic}
  \caption{On the left is a cross-section of $T^2\times [0,1]$ with the tori $T_i$ labeled. We also see the Legendrian knots $L_i$ obtained by contact isotopy from the boundary. The green arc represents the annulus $A$ from $L_0$ to $L_1$. The blue shaded region is $N$, an $[-1,1]$-invariant neighborhood of $A$ discussed later in the proof. On the top right, we see $L_1\times [-1,1]\subset N$ and its characteristic foliation, and on the bottom right, we see the perturbation of that annulus making it convex.}
  \label{pictureproof}
\end{figure}
  
We now establish our claim. Since the $\tw(L_0) = -1$ with respect to $T_0$ and $\tw(L_1)=0$ with respect to $T_1$, we know the dividing curves of $A$ intersect $L_0$ twice and $L_1$ zero times. Thus, we know the dividing set is as claimed, except that there might also be some closed dividing curves isotopic to the core of $A$. We must show that $A$ can be chosen so that this is not the case (one must be careful, as there are choices for $A$ where there are such closed curves). To this end, notice that given the slopes $s_0$ and $s_1$ with an edge in the Farey graph connecting them, there will be exactly two slopes that both have an edge to $s_0$ and $s_1$. One will be in $[s_0, s_1]$ while the other $s$ will be outside this interval. We now perturb $T^2\times\{0\}$ and $T^2\times\{1\}$ so that their ruling curves have slope $s$. Let $B$ be an annulus of slope $s$ in $T^2\times [0,1]$ whose boundary consists of two ruling curves, one on $T^2\times \{0\}$ and the other on $T^2\times \{1\}$. We can make $B$ convex, and it must have exactly two dividing curves running from one boundary component to the other. This is because, if not, we could Legendrian realize the core curve in $B$ with contact twisting $0$ with respect to $B$. We could then find a torus $T$ parallel to the boundary of $T^2\times[0,1]$ containing this curve so that it also had twisting $0$ with respect to $T$. This implies that $T$ can be made convex with dividing slope $s$, contradicting the fact that, as a basic slice, $T^2\times[0,1]$ is minimally twisting. Thus, the dividing curves are as claimed. Now we can Legendrian realize a curve $\gamma$ on $B$ that runs from one boundary component to the other and has contact twisting $0$, we can moreover assume that one boundary component of $\gamma$ is on $L_1$ and the other is not on the Legendrian divides on $T^2\times \{0\}$. Now we can isotope $T^2\times \{0\}$, keeping it fixed near $\partial \gamma$ so that its ruling curves have slope $s_1$. This allows us to take an annulus $A$ from $L_0$ to $L_1$ that contains $\gamma$. The twisting of $\gamma$ with respect to this annulus is still $0$, so we can make $A$ convex relative to $\gamma$. Because the twisting of $\gamma$ is $0$ we see that it cannot intersect the dividing curves of $A$. This implies that there can be no closed curves in the dividing set of $A$, and hence the dividing curves of $A$ must be as claimed. The claim about the characteristic foliation on $A$ follows from Giroux flexibility.
  
Now let $N=A\times [-1,1]$ be an $[-1,1]$-invariant neighborhood of $A=A\times \{0\}$. This neighborhood is generated by a contact vector field. Since $L_0$ is a ruling curve in $T_0$, it has an invariant neighborhood in $T_0$, and using this we can assume the contact vector field is tangent to $T_0$. However, $L_1$ does not have an invariant neighborhood in $T_1$, so the contact vector field must be transverse to $T_1$. Thus, as indicated on the left of Figure~\ref{pictureproof}, we see that $L_0\times [-1,1]$ can be taken to be an annulus in $T_0$ while $L_1\times [-1,1]$ cannot be taken to be a subset of $T_1$.

Notice that $\partial N$ consists of four parts, $A_{-1}=A\times \{-1\}, A_1=A\times \{1\}, L_0\times [-1,1],$ and $L_1\times [-1,1]$. The first three are convex surfaces, with the first two having dividing set the same as $A$ while the third having dividing set consisting of two arcs, each running from one boundary component to the other. We can round the two corners between the first three surfaces to get a convex annulus $A'$ with one dividing curve isotopic to its center curve and intersecting $L_0$ twice. We notice that $L_1\times [-1,1]$ is foliated by Legendrian curves parallel to $L_1$ (that is, $L_1\times [-1,1]$ is formed as the image of $L_1$ under the flow of a contact vector field). See the upper right-hand diagram in Figure~\ref{pictureproof}. Thus, by our choice of the characteristic foliation on $A$, we see the characteristic foliation of $A'$ has Legendrian boundary consisting of two copies of $L_1$, and each boundary component looks like a Legendrian divide, by which we mean the boundary components are circles of singularities in the foliation and nearby the foliation is non-singular and transverse to the boundary, moreover one boundary will be an attracting circle of singularities and the other will be repelling (since neighborhoods of the boundary of $A'$ consists of two copies of $A$, but with opposite orientations).

  We now have $A'\cup (L_1\times [-1,1])$ is a torus bounding a solid torus. We now consider $L_1\times [-1,1]$. This has characteristic foliation given by $L_1\times \{t\}$ for $t\in[-1,1]$. See Figure~\ref{pictureproof}. That is, it is a pre-Lagrangian annulus and thus cannot be convex. But we build a standard model for a neighborhood of $L_1\times [-1,1]$. Specifically, in $\R^3/\sim,$ where $(x,y,z)\sim (x+1, y,z)$, with the contact structure $dz-y\, dx$ we find an open set around $S=\{(x,y,z) : y=0, |z|\leq 1\}$ that is contactomorphic to a neighborhood of $L_1\times[-1,1]$ by a contactomorphism taking $S$ to $L_1\times [-1,1]$ and a neighborhood of $\partial A'$ to constant $z$ annuli with, say, positive $y$ coordinate. In this local model, we can deform $L_1\times[-1,1]$ by slightly pushing its interior to have negative $y$ coordinate. The characteristic foliation on this new annulus $A''$ now has Legendrian boundary, and on the interior flow lines that spiral to one boundary component in positive time and the other in negative time. We can finally slightly modify $A'$ in this local model so that the orbits near $\partial A'=\partial A''$ spiral towards the boundary components in the same way that those on $A''$ do. In particular, $A'\cup A''$ is now a convex torus with two dividing curves. One is in the center of $A'$ and the other is in $A''$. In addition, we see that $L_1$ is isotopic to the Legendrian divides on $A'\cup A''$, and again, $L_0$ sits on this torus intersection one of the dividing curves twice. Let $N$ be the solid torus bounded by $A'\cup A''$. This is a standard neighborhood of $L_1$, and we see that $L_0$ has contact twisting one less than $L_1$, and so when it is contact isotoped into the interior of $N$ we see that it has a neighborhood as claimed above. 

  The proof for the analogous case with the Legendrian knots $L'_0$ and $L'_1$ is the same. 

  For the second statement, notice that the annulus $A$ of slope $s$ from $T^2\times \{0\}$ to $T^2\times \{1\}$ with boundary ruling curves cannot have a boundary parallel dividing curve on $T^2\times \{1\}$, since if there were we could attach a bypass to $T^2\times \{1\}$ and get a convex torus of dividing slope outside of $[s_0,s_1]$ contradicting the minimal twisting of a basic slice. Thus the dividing curves on $A$ have some boundary parallel dividing curves on $A\cap (T^2\times \{0\})$ and the rest run across from one boundary component to the other. We can use the bypasses to destabilize $L''_0$ and then isotope it to $L''_1$. The signs of the destabilization are determined by the sign of the bypass and then number has to be as indicated, otherwise there would be a bypass on $T^2\times \{1\}$. 
\end{proof}

\subsection{Pairs of decorated paths}\label{pairsodecorated}
Suppose $P_1$ and $P'_1$ are minimal decorated paths in the Farey graph from $r$ clockwise to $s$, and $P_2$ and $P'_2$ are minimal decorated paths from $s$ to $t$. Below we will be interested when the contact structures $\xi_{P_1}\cup \xi_{P_2}$ and $\xi_{P'_1}\cup \xi_{P'_2}$ are contactomorphic. When $P_1\cup P_2$ cannot be shortened, we can completely understand this using Theorem~\ref{T2XIclass}, so we will focus on when the concatenated path can be shortened. In particular, here we will consider the following situation. Suppose that $q/p\in(n-1,n)$ and $P_1$ is a minimal decorated anti-clockwise path from $q/p$ to $n-1$ and $P_2$ is a minimal decorated clockwise path from $q/p$ to $n$. It is clear that the concatenation $\overline{P}_1\cup P_2$ can be shortened, and the specific shortening process was discussed in Sections~\ref{shorteningpaths} and~\ref{sortentoone}. We now study when $\xi_{P_1}\cup \xi_{P_2}$ is contactomorphic to a contact structure described by different decorations on the two paths. 

We consider the breakdown of $P_1$ and $P_2$ as in Section~\ref{sortentoone} (we will only discuss the case here, with the case of $(A_1,\ldots, A_{2n-1})$ and $(B_2,\ldots, B_{2m})$, and the case when the maximal odd index is smaller, being analogous). 

\begin{definition}\label{defnconsistent}
We will call a pair of decorated paths $(P_1,P_2)$ \dfn{$i$-consistent} if the signs of the decorations on the edges in $A_j$ and $B_j$ with $j\leq i$ are all the same and we call the paths \dfn{$i$-inconsistent} if $(P_1,P_2)$ is $(i-1)$-consistent but not $i$-consistent.
\end{definition}

Let $D_i$ denote $A_i$ if $i$ is even and $B_i$ if $i$ is odd. Suppose $P_1$ and $P_2$ are $i$-inconsistent for some $i>2$, then of course the paths are $(i-1)$-consistent. See the top row in Figure~\ref{compatible} for a $6$-inconsistent pair of decorated paths.
  \begin{figure}[htbp]
  \begin{overpic}{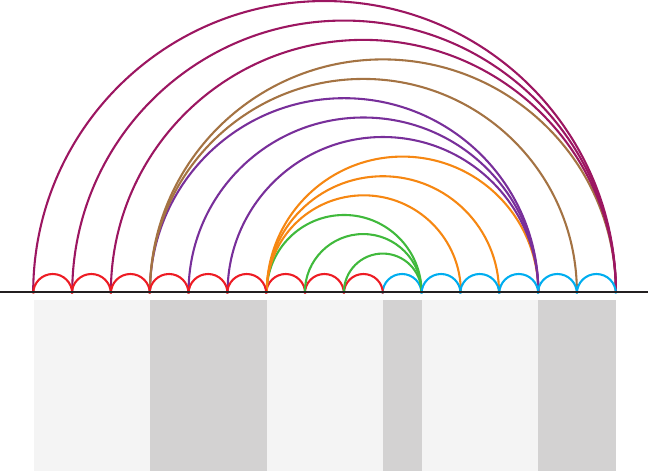}
    \put(188, 70){$B_1$}
    \put(150, 70){$A_2$}
    \put(227, 70){$B_3$}
    \put(93, 70){$A_4$}
    \put(276, 70){$B_5$}
    \put(38, 70){$A_6$}
    \put(11, 30){
      \setlength{\tabcolsep}{6.35pt}
      \begin{tabular}{ccccccccccccccc}
        -& - & + & - &-  &-  &-  &-  &-  &-  &-  &-  &-  & - & -\\
        -& - & - & + &+  &+  &+  &+  &+  &+  &+  &+  &+  & + & -\\
        -& - & - & + &-  &-  &-  &-  &-  &-  &-  &-  &-  & + & +\\
        -& - & - & - &-  &-  &+  &+  &+  &+  &+  &+  &-  & + & +\\
        -& - & - & - &-  &-  &+  &-  &-  &-  &+  &+  &+  & + & +\\
      \end{tabular}}
  \end{overpic}
  \caption{The signs in the top row give a $6$-inconsistent pair of paths. In the next row, we have shuffled signs in a continued fraction block to get a $5$-inconsistent pair of paths. Each of the subsequent rows is obtained from the previous row by shuffling a basic slice in a continued fraction block to get a $4$, then $3$, and finally $2$-inconsistent pair of paths.}
  \label{compatible}
\end{figure}
From the discussion in Sections~\ref{shorteningpaths} and~\ref{sortentoone}, we know that there is an edge in the Farey graph between the first vertex $v_i$ of $D_i$ and the last vertex $v'_i$ in $D_{i-1}$. Moreover, if $v_i''$ is the second-to-last vertex in $D_{i-1}$ then it is the Farey sum of $v_i$ and $v_i'$; in particular, it extends the continued fraction block $D_i$ by one extra jump. Since all the edges between $v_i$ and $v''_i$ have the same sign, Theorem~\ref{shortening} tells us that the contact structure described by the path between $v_i$ and $v''_i$ is a basic slice with sign the common sign of the edges in the path. Now in $D_i$ we know there is an edge with opposite sign, and since $D_i$ is a continued fraction block, Theorem~\ref{T2XIclass} says that one may assume it is the edge adjacent to $v_i$. So we can exchange the sign on the edge between $v''_i$ and $v_i$ and the first edge in $D_i$. By Corollary~\ref{lengthen}, this is equivalent to changing all the signs on the edges in $P_1$ and $P_2$ between $v_i$ and $v''_i$ as well as the sign of the first basic slice in $D_i$. After we have done this, we have a new pair of decorated paths $P_1'$ and $P_2'$. 

\begin{definition}
With the notation above, we say that $(P_1,P_2)$ is \dfn{$i$-compatible} with $(P'_1,P_2')$. 
\end{definition}

Notice, since one edge in $D_{i-1}$ kept its same sign, that $(P_1',P_2')$ is $(i-1)$-inconsistent. 

\begin{lemma}
  The contact structures described by $i$-compatible pairs of paths are isotopic. 
\end{lemma}

\begin{proof}
  The discussion before the definition makes it clear that the contact structures described by $i$-compatible pairs of paths are built by concatenating the same basic slices.
\end{proof}

Of course we can iterate and find decorated paths $(P_1'',P_2'')$ that are $(i-1)$-compatible with $(P_1',P_2')$ and continue until we have $(P_1^{(i-2)}, P_2^{(i-2)})$ which is $2$-inconsistent. 
\begin{definition}
We say two pairs of paths $(P_1,P_2)$ and $(P'_1,P'_2)$ are \dfn{compatible} if they are related by a sequence of pairs of paths where each pair is $i$-compatible with the previous pair for some $i$. 
\end{definition}

Iteratively applying the previous lemma allows us to conclude the following result. 
\begin{lemma}\label{comppathgivesame}
  Compatible pairs of decorated paths define the same contact structure. 
\end{lemma}

We end with a definition we will need later.
\begin{definition}\label{totallyinconsistentdef}
If every decoration of $(A_1, A_3, ..., A_{i-1})$ has the same sign and every decoration of $(B_2, ... ,B_i)$ has the opposite sign, then the paths are called \dfn{totally $i$-inconsistent}. 
\end{definition}
We also notice that if a pair of decorated paths $(P_1,P_2)$ is totally $2$-inconsistent, then it cannot be compatible with any decorated pairs that are $i$-inconsistent for any $i\geq 3$. This is because, in the discussion above, we see that if a $2$-inconsistent pair of paths is compatible with a $3$-inconsistent pair of paths, then $A_2$ (or $B_2$) will have a mixture of signs.

\subsection{Paths in the Farey graph and contact structures on \texorpdfstring{$S^3$}{S3}}\label{subsec:pathsinFG}
When studying non-loose torus knots in $S^3$, we will need to consider $S^3$ as $L_\infty^0$ (that is a lens space with lower meridian $\infty$ and upper meridian $0$, see Section~\ref{sec:lens}). We will describe contact structures on $L_\infty^0$ using paths in the Farey graph. More precisely, given a rational number $q/p$ we will write $L_\infty^0$ as the union of two solid tori: $V_1$ with lower meridian $\infty$ and convex boundary having slope $q/p$ and $V_2$ with upper meridian $0$ and convex boundary having slope $q/p$, where  $p,q$ are coprime integers and $|q| > p > 1$. According to Theorem~\ref{torusclass}, we will need two paths in the Farey graph to specify contact structures on these tori. Let $P_1$ be a path that describes a contact structure in $\Tight(S_\infty; q/p)$ and $P_2$ be a path describing a contact structure in $\Tight(S^0; q/p)$. Given these paths we get a contact structure $\xi_{P_1,P_2}$ on $S^3$. In this section we will see when the contact structures associated to two different decorated pairs of paths correspond to the same contact structure. In \cite{Matkovic18},  Matkovi\v{c} has done the same things for some small Seifert fibered spaces in terms of her ``characteristic vectors", and then in \cite{Matkovic20Pre} used this to understand when negative $(p,q)$-torus knots with $\tb<pq$ are in the same overtwisted contact structure. 

As in Theorem~\ref{torusclass}, we consider a minimal path $P_1$ from $q/p$ anti-clockwise to $\infty$ with all edges decorated by a sign except the last edge from $\lfloor q/p \rfloor$ to $\infty$ (this edge describes the unique tight contact structure on a solid torus with convex boundary having two longitudinal dividing curves). We have a similar discussion for $P_2$. If $q/p<-1$, then we need to consider $P_2$ as a decorated path from $q/p$ clockwise to $-1$ (the jump from $-1$ to $0$ describes the unique tight solid torus with given dividing curves and meridian). If $q/p>1$, then $P_2$ will be a decorated path from $q/p$ clockwise to $\infty$ (the jump from $\infty$ to $0$ describes the unique tight solid torus with given dividing curves and meridian). 

We say the pair $(P_1,P_2)$ is a \dfn{pair of paths representing $q/p$}. We describe the paths $P_1$ and $P_2$ in details below.

Below, we will see that when $pq<0$, the part of $P_2$ from $\lceil q/p \rceil$ to $-1$ plays a very different role in our analysis, and in Section~\ref{thealgorithm}, we mainly consider the part of $P_2$ from $q/p$ to $\lceil q/p \rceil$. Thus we denote by $P_2^\intercal$ the truncated path from $q/p$ to $\lceil q/p \rceil$.

\subsubsection{Case 1: $q/p<-1$}
Notice that $P_1$ is a minimal path in the Farey graph from $q/p$ anti-clockwise to $\lfloor q/p\rfloor$ and $P^\intercal_2$ is a minimal path from $q/p$ clockwise to $\lceil q/p\rceil$. Thus, if we concatenate the path $P_1$ reversed, which we denote $\overline{P_1}$, and $P_2^\intercal$ we will obtain a path going from $\lfloor q/p\rfloor$ clockwise to $\lceil q/p\rceil$. However, this path will not be minimal as there is a single edge going from $\lfloor q/p\rfloor$ to $\lceil q/p\rceil$. Thus, the path may be shortened. 

Before proceeding, we first state a couple of results that indicate why are are considering $P^\intercal_2$ instead of $P_2$.

\begin{lemma}\label{extrablockfornegative}
Let $(P_1,P_2)$ be a pair of decorated paths representing $q/p<-1$. Suppose all the signs of $P_1$ and $P^\intercal_2$ are the same, but some signs in $P_2-P^\intercal_2$ are different. There is no way to ``shuffle" signs to make $(P_1,P_2)$ inconsistent at an earlier stage. 
\end{lemma}

\begin{remark}
  Note that this lemma is in contrast to what we say for the pair $(P_1, P^\intercal_2)$. There, if the pair were $i$-inconsistent, then they would be compatible with other pairs of paths that were $j$-inconsistent for $j<i$. 
\end{remark}

\begin{proof}
  When all the edges in $P_1\cup P^\intercal_2$ have the same sign but some edge in between $\lceil q/p \rceil$ and $-1$ has a different sign, we know that $\overline P_1\cup P^\intercal_2$ describes a basic slice. One can see that ``shuffling" the signs of basic slices, see Theorem~\ref{T2XIclass}, the signs of all of the basic slices in $\overline P_1\cup P^\intercal_2$ will change (just as in our discussion in Section~\ref{pairsodecorated}, and so you will not get a pair of paths that is inconsistent at an earlier stage.  
\end{proof}

\begin{lemma}
  Suppose $P_1$ and $P_2^\intercal$ are totally consistent (that is all their signs are the same) but some of the signs in the path from $\lceil q/p\rceil$ can be different. The contact structure on $S^3$ corresponding to $(P_1,P_2)$ is tight. 
\end{lemma}

\begin{remark}
  Note that when there is a mixture of signs in $(P_1, P^\intercal_2)$, then the contact structure is automatically overtwisted, but this lemma shows that this is not the case for $(P_1, P_2)$. 
\end{remark}

\begin{proof}
  Notice that by Theorem~\ref{shortening} the path $\overline P_1\cup P_2^\intercal$ describes a basic slice and so the path from $\infty$ clockwise to $\lfloor q/p \rfloor$ followed by $\overline P_1\cup P_2$ describes the unique tight contact structure on a solid torus and using Lemma~\ref{gluingtoriandthickened} when one extends the path all the way to $-1$ we will have the unique tight contact structure on a solid torus with dividing slope $-1$. Notice that the complementary solid torus in $S^3$ also has a unique tight contact structure and the union of these tori is the tight contact structure on $S^3$. In other words, such a path does not describe an overtwisted contact structure and is unrelated to non-loose Legendrian knots.  
\end{proof}

Since $\overline P_1\cup P^\intercal_2$ is a path from $\lfloor q/p\rfloor$ to $\lceil q/p\rceil$ we can use the terminology in Section~\ref{pairsodecorated} about $ i$-consistency and compatibility. Specifically we will say $(P_1,P_2)$ are \dfn{$i$-compatible} or \dfn{consistent} if $P_1$ and $P^\intercal_2$ are. Our main observation is the following, which follows  directly from Lemma~\ref{comppathgivesame}.

\begin{lemma}
  Compatible pairs of decorated paths define the same contact structure on $S^3$. 
\end{lemma}

\subsubsection{Case 2: $q/p>1$}
%

%
%
 
Our discussion of the paths $P_1$ and $P_2$ follows as in Case~1, except when $pq>1$ we have one extra type of pair of decorated paths to consider. Suppose all the signs of $P_1$ and $P_2$ are the same, say negative. Let $i$ be an integer such that $i < q/p< i+1$. As all the signs in all the paths are the same, we can shorten $\overline P_1\cup P_2$ to a path going from $i$ to $\infty$ and describing a basic slice. Now split this path into $P$ going from $i$ to $i+1$ and $P'$ going from $i+1$ to $\infty$ and decorate the paths with the common sign. Now the path $P$ describes a contact structure on the solid torus $S_\infty$ that is the unique solid torus with longitudinal divides and so may be split into a path going from $\infty$ to $i$ and then to $i+1$ where the first jump corresponds to the unique contact structure on the solid torus with given slope and the second is a basic slice of either sign (see Lemma~\ref{gluingtoriandthickened}). We can choose the sign of the basic slice to be positive and then subdivide the path to $(\overline {P_1}\cup P_2)\cap[i,i+1]$ so that all the basic slices are positive. However, $P_2$ has one more edge going from $i+1$ to $\infty$ that is still negative. The paths with the new signs will be denoted $(P_1^{2m-1}, P_2^{2m-1})$ (here, we assume $2n>2m-1$ without loss of generality). Clearly $\xi_{P_1,P_2}$ and $\xi_{P_1^{2m-1}, P_2^{2m-1}}$ are the same as they are described by gluing together the same contact structures on solid tori. 

Notice that if the paths are broken into their continued fraction blocks 
\[
  (A_2,\ldots, A_{2n}) \text{ and } (B_1,B_3, \ldots, B_{2m-1}),
\] 
as above, this new pair of paths $(P_1^{2m-1}, P_2^{2m-1})$ is $2m-1$-inconsistent. As we saw above we will now get $k$-inconsistent pairs of paths $(P_1^k,P_2^k)$ for $k=2,3,\ldots, 2m-1$ that are all compatible. Notice that all the signs of the basic slices in $P^2_1$ are of a fixed sign, say positive, except the first one which is negative and all the basic slices of $P^2_2$ are negative, except the first one which is positive. 

\subsection{From decorated Farey graphs to contact surgery diagrams} \label{sec:FareytoSurgery}
Let $(P_1,P_2)$ be a pair of paths in the Farey graph representing $q/p$. As discussed in the previous section, once the paths are decorated they give a contact structure $\xi_{P_1,P_2}$; moreover, there is a convex torus $T$ with two dividing curves of slope $q/p$ that separates $S^3$ into two solid tori with contact structures described by $P_1$ and $P_2$.  
We will show in Lemma~\ref{lem:=pq} that a Legendrian divide on this torus will be a non-loose $(p,q)$-torus knot $L_{P_1,P_2}$ and all such torus knots with $\tb=pq$ and $\tor = 0$ will occur in this way, as we will show in Section~\ref{justifyalgorithm}. Here we would like to turn the Farey graph description of $L_{P_1,P_2}$ into a contact surgery diagram in $(S^3,\xi_{std})$. The relation between the surgery construction and paths in the Farey graph was originally observed by Matkovi\v{c} \cite{Matkovic18} in the case of small Seifert fibered spaces and then used to study negative torus knots in her paper \cite{Matkovic20Pre}.

Our main result is the following.
\begin{lemma}
Given a pair of paths $(P_1,P_2)$ in the Farey graph representing $q/p$ then the contact structure $\xi_{P_1,P_2}$ is given by the contact surgery diagram on the left in Figure~\ref{fig:surgery} and in Figure~\ref{fig:torus-knots}, and $L_{P_1,P_2}$ is given by the green curve in either figure. Recall that a contact surgery with a surgery coefficient different from $1/n$ is not unique and is determined by a signed path in the Farey graph. The paths $P_1$ and $P_2$ determine a specific contact surgery. 
\end{lemma}

\begin{figure}[htbp]
\vspace{0.1cm}
\begin{overpic}{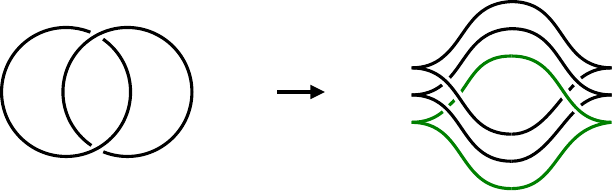}
  \put(-10, 80){$\frac{p'}{p-p'}$}
  \put(81, 80){$\frac{q-q'}{q'}$}
  \put(300, 41){$(\frac{p}{p-p'})$}
  \put(305, 59){$(\frac{q}{q'})$}
\end{overpic}
\caption{Left: smooth surgery diagram of $S^3$. Right: contact surgery diagram of $(S^3,\xi)$}
\label{fig:surgery}
\end{figure}
\begin{proof}
As discussed in Section~\ref{subsec:pathsinFG} the pair $(P_1,P_2)$ determines a contact structure $\xi_{P_1,P_2}$ on $S^3$. We can convert this decorated Farey graph into a contact surgery diagram in $(S^3,\xi_{std})$ for the contact structure $\xi_{P_1,P_2}$. To this end, we first consider a smooth surgery diagram for $S^3$ shown in the left drawing of Figure~\ref{fig:surgery}. 

Here, we denote $(q/p)^c$ by ${q'/p'}$. To see this manifold is $S^3$, think of this manifold as a result of Dehn filling $T^2 \times I$ along a curve on $-T^2\times\{0\}$ of slope ${p'/(p-p')}$ and a curve on $T^2\times\{1\}$ of slope ${q'/(q-q')}$. Let $\phi$ be a diffeomorphism of a torus whose matrix representation is 
\[
  \phi = \begin{pmatrix}
  p' & p'-p\\
  q'& q'-q
  \end{pmatrix}.
\]

After change the coordinates of $T^2 \times I$ using $\phi$, the meridional slopes of the two solid tori glued on $T^2 \times \{0\}$ and $T^2 \times \{1\}$ are $\infty$ and $0$, respectively. Thus, the surgered manifold is diffeomorphic to $S^3$. Next, convert this diagram into a contact surgery diagram as shown in the right drawing of Figure~\ref{fig:surgery}. Observe that the region between the two Legendrian unknots is a thickened torus $T^2\times I$ with an $I$-invariant contact structure having dividing slope $-1$. The green curve in Figure~\ref{fig:surgery} is a Legendrian divide on one of these convex tori. After changing the coordinates using $\phi$, we have
\[
  \begin{pmatrix}
    p' & p'-p\\
    q' & q'-q
  \end{pmatrix}
  \begin{pmatrix}
    1\\-1
  \end{pmatrix}
  =
  \begin{pmatrix}
    p\\q
  \end{pmatrix}.
\]
Thus the dividing slope of the torus is $q/p$ after change of the coordinates. Thus the two solid tori glued on $T^2 \times \{0\}$ and $T \times \{1\}$ are elements of $\Tight(S_\infty;q/p)$ and $\Tight(S^0;q/p)$ respectively. Recall that in \cite{DingGeigesStipsicz04}, Ding, Geiges, and Stipsicz provided an algorithm to convert a general contact surgery diagram to a contact $(\pm1)$-surgery diagram:
\begin{itemize}
	\item contact $(p/q)$-surgery on a Legendrian knot $L$ with $p/q<0$:
	\begin{enumerate}
		\item Stabilize $L,$ $|r_1+2|$ times, where
		\[
      \frac{p}{q} = r_1 + 1 - \frac{1}{r_2 - \frac{1}{r_3 \cdots -\frac{1}{r_n}}}
    \]
    for $r_i \leq -2$. Let the resulting Legendrian knot be $L_1$.
		\item For $i=2,\ldots,n$, let $L_i$ be the Legendrian push-off of $L_{i-1}$ and stabilize it $|r_i+2|$ times.
		\item Then a contact $({p/q})$-surgery on $L$ corresponds to a contact $(-1)$-surgery on a link $(L_1,\ldots,L_n)$.
	\end{enumerate}
	\item contact $({p/q})$-surgery on a Legendrian knot $L$ with ${p/q}>0$:
	\begin{enumerate}
		\item Choose a positive integer $k$ such that $q-kp<0$. Let $r'={p/(q-kp)}$.
		\item Let $L_1,\ldots,L_k$ be k successive Legendrian push-offs of $L$.
		\item Then a contact $({p/q})$-surgery on $L$ corresponds to $(+1)$-surgeries on $L_1,\ldots,L_{k}$ and a contact $(r')$-surgery on $L$.
	\end{enumerate}
\end{itemize}

Applying the second algorithm to the contact surgery diagram in Figure~\ref{fig:surgery}, we obtain the contact surgery diagram shown in Figure~\ref{fig:torus-knots}. To see this notice that since $q/p= (q/p)^a\oplus (q/p)^c$, see Lemma~\ref{adjacentvertices}, we know that $p-p'>0$ and $q-q'>0$ and $q$ and $q'$ will both either be positive or both negative. Choosing $k=1$ in the above algorithm will result in the surgery coefficients shown in Figure~\ref{fig:torus-knots} which can easily be checked to be less than $-1$. 

\begin{remark}
  Notice that the above algorithm actually produces Figure~\ref{fig:torus-knots} with the second and third surgery coefficients interchanged; however, these surgery diagrams produce the same contact structure, as we now explain. Since the two negative surgery coefficients are less than $-1$, the corresponding Legendrian knots must be stabilized in order to perform Legendrian surgery. It was shown in \cite{DaltonEtnyreTraynor21pre} that stabilized components of a $(4,-4)$-torus link (that is, the surgery link in Figure~\ref{fig:torus-knots}) can be arbitrarily permuted among the other components.
\end{remark}

In \cite{DingGeiges04}, Ding and Geiges showed that the choice of stabilizations on $L_1$ with contact surgery coefficient $(p/(p-p'))$ and $L_2$ with contact surgery coefficient $(q/q')$ corresponds to the choice of signs on the basic slices of $V_1$ and $V_2$, respectively. To be more precise, suppose 
\[
  -\frac{p}{p'} = r_1 + 1 - \frac{1}{r_2 - \frac{1}{r_3 \cdots -\frac{1}{r_u}}} \quad\text{and}\quad -\frac{q}{q-q'} = s_1 + 1 - \frac{1}{s_2 - \frac{1}{s_3 \cdots -\frac{1}{s_v}}}
\]
for $r_i \leq -2$ and $s_i \leq -2$. Let $[r_{p_1},...,r_{p_n}]$ be the subsequence of $[r_1, ..., r_u]$ such that $r_{p_i} < -2$. Now the choice of stabilizations on $L_{p_i}$ corresponds to the choice of signs on each basic slice in the continued fraction block $A_{2i-1}$. Similarly, let $[s_{q_1},...,s_{q_m}]$ be the subsequence of $[s_1, ..., s_v]$ such that $s_{q_i} < -2$. Now the choice of stabilizations on $L_{q_i}$ corresponds to the choice of signs on each basic slice in the continued fraction blocks $B_{2i}$. Observe that since $P_2$ represents a contact structure in $\Tight(S^0,{q/p})$, the positive stabilization corresponds to the negative basic slice (respectively, the negative stabilization corresponds to the positive basic slice). See Figure~\ref{fig:LHTsurgery} for examples.
\end{proof}

\subsection{Homotopy class of plane fields and rotation numbers}\label{htpclasses} 
Given a pair $(P_1,P_2)$ of decorated paths in the Farey graph for the $(p, q)$-torus knot, we saw in Section~\ref{sec:FareytoSurgery}  that we define a contact structure $\xi_{P_1,P_2}$ on $S^3$ and a non-loose Legendrian $(p,q)$-torus knot $L_{P_1,P_2} \in (S^3,\xi_{P_1,P_2})$ with $\tor=0$. Moreover, Lemma~\ref{lem:=pq} says that all such Legendrian knots come from this construction. 

To compute the $d_3$-invariant of $\xi_{P_1,P_2}$, we first convert the decorated Farey graph $(P_1,P_2)$ for the $(p,q)$-torus knot into the corresponding contact $(\pm1)$-surgery diagram as described in Section~\ref{sec:FareytoSurgery}. See Figure~\ref{fig:LHTsurgery} for examples. From the surgery diagram we can compute $d_3$-invariant of the contact structure on $S^3$ using \cite[Corollary~3.6]{DingGeigesStipsicz04}. 
\begin{lemma}
Let $L$ be a link that can be written as links $L'$ and $L''$ and $\xi$ be the contact structure obtained from $(S^3,\xi_{std})$ by performing contact $(-1)$-surgery on $L'$ and contact $(+1)$-surgery on $L''$. Then
\begin{align}
  d_3(\xi) = \frac{1}{4}(c^2 - 3\sigma(X) - 2(\chi(X)-1)) + q.\label{formula:d3}
\end{align}
Here $X$ is the  $4$-manifold obtained by attaching $2$-handles to the $4$-ball as indicated in the surgery diagram, and $q$ is the number of components in $L''$. Let $M$ be the intersection matrix of $X$, which is also the linking matrix of the surgery diagram, so the signature $\sigma(X)$ is the signature of this matrix and the Euler characteristic $\chi(X)$ is the rank of this matrix plus $1$. The quantity $c^2$ is $(\mathbf{rot})^{\intercal} M^{-1} \mathbf{rot}$, where $\mathbf{rot}$ is a vector of rotation numbers of each surgery component.  
\end{lemma}

Since $\mathbf{rot}^{\intercal}  M^{-1}  \mathbf{rot} = (-\mathbf{rot}^{\intercal})  M^{-1}  (-\mathbf{rot})$, we see that $\xi_{P_1,P_2}$ is the same as $\xi_{-P_1,-P_2}$.


We give two methods to compute the rotation number of $L_{P_1,P_2}$. The first method for computing $\rot(L_{P_1,P_2})$ involves the surgery diagram used above. In particular, we have the formula from \cite[Theorem~2.2]{DurstKegel16}:

\begin{align}
  \rot(L) = \rot_0 - \mathbf{rot}^\intercal \cdot M^{-1}\cdot \mathbf{lk},\label{formula:rot}
\end{align}
where $\rot_0$ is the rotation number of $L$ in the surgery diagram before surgery and $\mathbf{lk}$ is the vector of linking numbers between each surgery component and $L$. In our surgery diagram, it is clear that
\[
  \mathbf{lk} = \begin{bmatrix}
    -1\\ \vdots \\ -1
  \end{bmatrix}.
\]
In our examples $\rot_0=0$, so we see that $\rot(L_{P_1,P_2})=-\rot(L_{-P_1,-P_2})$.

The second method for computing the rotation number makes the computation directly from the Farey graph. Given a pair $(P_1,P_2)$ of decorated paths in the Farey graph for the $(p, q)$-torus knot, let $p_1=q/p, p_2, \ldots, p_k=\lfloor  q/p\rfloor$ be the vertices in $P_1$ and $q_1=q/p,q_2,\ldots, q_l$ be the vertices in $P_2$. Recall, when $pq<0, q_l=-1$ and when $pq>0, q_l=\infty$. Define 
\[
  r_m=\sum_{i=1}^{k-1} \epsilon_i \left((p_{i+1}\ominus p_{i}) \bigcdot \frac 10\right) 
\]
and
\[
  r_n=\sum_{i=1}^{l-1} \epsilon'_i \left((q_{i+1}\ominus q_{i})\bigcdot \frac 01\right) 
\]
where $\epsilon_i$ is the sign of the edge from $p_i$ to $p_{i+1}$ and $\epsilon'_i$ is the sign of the edge from $q_i$ to $q_{i+1}$. Then define
\[
  R(P_1,P_2)= p\,r_n + q\,r_m 
\]

\begin{lemma}\label{computer}
  The Legendrian knot $L_{P_1,P_2}$ has rotation number 
  \[
    \rot(L_{P_1,P_2})=R(P_1,P_2).
  \]
\end{lemma}

Notice that since $R(P_1,P_2)$ is simply the rotation number of $L_{P_1,P_2}$ it can also be computed from the surgery formula above. 

\begin{proof}
Suppose $T$ is a convex torus. Let $\lambda$ and $\mu$ be curves on the torus that form a basis for $H_1(T)$. If $\gamma$ is a curve on $T$ whose  homology class is $p\lambda + q\mu$, then when $T$ is isotoped through convex tori so that $\gamma$ is a Legendrian curve, it was shown in \cite[discussion before Lemma~4.11]{EtnyreHonda01}  that
\[
  \rot(\gamma)=p\, \rot(\lambda) + q\, \rot(\mu),
\]
where $\rot(\lambda)$, respectively $\rot(\mu)$, is computed by isotoping $T$ through convex tori so that $\lambda$, respectively $\mu$, is a Legendrian curve. 

In our setting, $T$ is a Heegaard torus for $S^3$ thought of as a neighborhood of an unknot. Then the standard longitude and meridian for the unknot are exactly $\lambda$ and $\mu$ and one bounds a compressing disk in $V_1$ and the other bounds one in $V_2$. From this, we see that the relative Euler classes of these two Heegaard tori are the rotation numbers of $\lambda$ and $\mu$. From \cite[proof of Proposition~4.22]{Honda00a} we can compute these relative Euler classes and see that $\rot(\lambda)=r_n$ and $\rot(\mu)=r_m$. The result follows. 
\end{proof}

Using this lemma, we can show that $\rot(L_{P_1,P_2})$ differs by the choice of the decorated paths $(P_1,P_2)$.

\begin{lemma} \label{lem:rotdiff}
  If $(P_1,P_2) \neq (P_1',P_2')$, then $\rot(L_{P_1,P_2}) \neq \rot(L_{P'_1,P'_2})$.
\end{lemma}

\begin{proof}
  We compute the rotation numbers using Lemma~\ref{computer}. That is $\rot(L_{P_1,P_2})= pr_n + qr_m$ and $\rot(L_{P'_1,P'_2})= pr'_n+qr'_m$ where $r_m$, $r_n$, $r'_m$ and $r'_n$ are computed in terms of the decorated paths as described in Section~\ref{htpclasses}. The numbers $r_m$ and $r_n$ are the relative Euler numbers for the contact structures on $V_1$ and $V_2$, and similarly for $r'_m$ and $r'_n$. In \cite{Honda00a}, Honda showed that tight contact structures on solid tori are determined by their relative Euler class. Since $(P_1,P_2)\not= (P'_1,P'_2)$, we know that either $r_m\not= r'_m$ or $r_n\not=r'_n$. 

  Arguing by contradiction we assume that $\rot(L_{P_1, P_2})=\rot(L_{P'_1,P'_2})$, so we have that 
  \[
    p(r_n-r'_n)+q(r_m-r'_m)=0.
  \]
  We first notice that $r_m-r'_m$ and $r_n-r'_n$ are both even, since from the formula in Section~\ref{htpclasses} we see that
  \[
    r_m-r'_m = \sum_{i=1}^{k-1} (\epsilon_i-\epsilon'_i)\left((p_{i+1}\ominus p_{i})\bigcdot \frac 10 \right),
  \]
  where the $\epsilon_i$ are the signs in $P_1$ and the $\epsilon'_i$ are the signs in $P'_1$. Thus $(\epsilon_i-\epsilon'_i)$ is even for all $i$ and we have a similar argument for $r_n-r'_n$. Moreover 
  \[
    |r_m|\leq \left|\left(\frac qp \ominus \left\lfloor \frac qp \right\rfloor\right)\bigcdot \frac 10\right| = p-1
  \]
  and similarly for $|r'_m|$. We also have 
  \[
    |r_n|\leq |q|-1
  \]
  since $|r_n|\leq |(\frac qp \ominus \frac 10)\bigcdot \frac 01|$ when $pq>0$ and $|r_n|\leq |(\frac qp \ominus \frac{-1}{1})\bigcdot \frac 01|$ when $pq<0$, and similarly for $r'_n$. 

  Since $\gcd(p,q)=1$, the only integer solutions to $pa+qb=0$ are $a=nq$ and $b=-np$ for $n \in \mathbb{Z}$. But given the above, we see that 
  \begin{align*}
    |r_m - r'_m| &< 2p,\\
    |r_n - r'_n| &< 2|q|.
  \end{align*}
  Thus, $p=|r_m-r'_m|$ and $|q|=|r_n-r'_n|$. However this implies that $p$ and $q$ are both even, contradicting the fact that $\gcd(p,q)=1$. Thus we have $\rot(L_{P_1, P_2})\not =\rot(L_{P'_1,P'_2})$ as claimed.
\end{proof}

\subsection{Contact structures on \texorpdfstring{$S^1 \times P$}{S1 X P}}
Consider $S^1\times P$ where $P$ is a pair of pants (a disk with two disjoint open sub-disks removed). Label the boundary components $T_1, T_2,$ and $T_3$ and consider the basis $S^1\times \{pt\}$ and $\mu_i=T_i\cap (\{\theta\}\times P)$ for $H_1(T_i)$. Let $S^1\times \{pt\}$ have slope $0$ and $\mu_i$ have slope $\infty$. Let $\Tight_0^{\operatorname{free}}(S^1\times P; r_1, r_2, r_3)$ be the set of tight contact structures, up to isotopy (not fixing the boundary point-wise, but preserving it set-wise), on $S^1\times P$ with convex boundary such that $T_i$ has two dividing curves of slope $r_i$ and having no convex torsion. 

The next lemma follows from \cite[Lemmas~10 and~11]{EtnyreHonda01a}, though we give a simple proof for completeness. 

\begin{lemma}\label{all000}
  $|\Tight_0^{\operatorname{free}}(S^1\times P; 0,0,0)|=1$
\end{lemma}

\begin{proof}
  Any contact structure $\xi \in \Tight_0^{\operatorname{free}}(S^1\times P; 0,0,0)$ has dividing curves parallel to the $S^1$-fibers on all boundary components. We can make the ruling curves have slope $\infty$ and then arrange for $P=\{\theta\}\times P$ to have its boundary be ruling curves and then make it convex. We need to consider two cases for the dividing curves on $P$.

  Case 1. There is a boundary-parallel dividing curve on $P$ for one of the tori $T_i$ for $1 \leq i \leq 3$. See the right drawing of Figure~\ref{fig:sdisk} for example. Without loss of generality, we can assume that there is a boundary-parallel dividing curve on $P$ for $T_1$. Then we can attach a bypass to $T_1$ to obtain a thickened torus $N_1$ with convex boundary where its back face has slope $\infty$ and its front face ($T_1$) has slope $0$. Now, as $T_3$ is convex, we can take a contact isotopic copy of it $T'_3$, inside of $S^1\times P$. Let $L$ be a Legendrian divide on $T'_3$. We now take a torus $T'_1$ parallel to $\partial N_1$ (but outside of $N_1$) that contains $L$. We make $T'_1$ convex relative to $L$ and since its contact twisting relative to $T'_1$ (which is the same as that relative to $T_3$) is zero, we see that the dividing curves on $T'_1$ must have slope $0$. Let $N'_1$ be the region between $T'_1$ and $T_1$. Since both boundaries of $\partial N'_1$ have dividing curves of slope $0$ and there are convex tori inside of $N'_1$ of with other dividing slopes, we see that $N'_1$ must contain half convex torsion. This contradiction shows there is no boundary-parallel dividing curves on $P$ for any $T_i$. (One small subtlety here is that $T'_1$ might have more than two dividing curves, but since $N_1'$ is rotative, we know we can find a convex torus parallel to $T'_1$ with slope $0$ that has two dividing curves, then we can use this new $T'_1$ to define $N'_1$.)
  

  \begin{figure}[htbp]
  \begin{overpic}{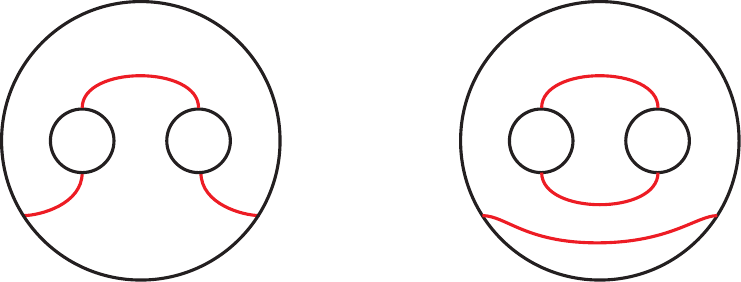}
    \put(10,65){$T_1$}
    \put(67,65){$T_2$}
    \put(20,1){$T_3$}
  \end{overpic}
  \caption{Some possible dividing sets on the pair of pants $P$.}
  \label{fig:sdisk}
  \end{figure}
  
  Case 2. There are no boundary parallel dividing curves on $P$ for any $T_i$. See the left drawing of Figure~\ref{fig:sdisk} for example. Honda showed in \cite[Lemma~4.1]{Honda00b}, that the tight contact structures on $S^{1} \times P$ with $0$ slope dividing curves on every $T_1$, $T_2$ and $T_3$ are in one-to-one correspondence with the choice of dividing set on $P$. There is a unique such configuration of dividing curves on $P$ up to some number of half twists near each boundary component. As we are allowing the boundary components to rotate among themselves, these twists can be undone and we can assume the dividing curves are as given on the left-hand side of Figure~\ref{fig:sdisk}.
\end{proof}

Given any contact structure in $\Tight_0^{\operatorname{free}}(S^1\times P; -1, \infty, 0)$ let $T'_3$ be a copy of $T_3$ that is contact isotopic to $T_3$.  We can take an annulus $A_i$ from a slope $0$ ruling curve on $T_i$ to a Legendrian divide on $T'_3$ and make it convex, for $i=1,2$. As argued in the previous proof one may take a neighborhood $N_i$ of $T_i\cup A_i$ to be a basic slice in $\Tight^{min}(T^2\times [0,1]; r_i,0)$ and the contact structure on the complement of the $N_i$ is the unique contact structure in $\Tight_0^{\operatorname{free}}(S^1\times P; 0,0,0)$. Let $s_i$ be the sign of the basic slice $N_i$ and denote the contact structure by $\xi_{s_1\, s_2}$.

\begin{lemma}\label{basicpants}
  We have 
  \[
    \Tight_0^{\operatorname{free}}(S^1\times P;-1, \infty, 0) = \{\xi_{++}, \xi_{+-}, \xi_{-+}, \xi_{--}\}.
  \]
  In $\xi_{\pm\pm}$ there is a convex annulus $A$ with boundary slope $0$ ruling curves on $T_1$ and $T_2$ that has two dividing curves each running from one boundary component to the other. In $\xi_{\pm\mp}$ the analogous annulus will always have boundary parallel dividing curves. 

  Let $\eta_\pm$ be the $\pm$ basic slice in $\Tight^{min}(T^2\times[0,1];-1,0)$ and $\zeta_\pm$ be the $\pm$ basic slice in $\Tight^{min}(T^2\times[0,1];0,\infty)$. Then $\xi_{\pm\mp}$ is obtained by gluing the front face of $\eta_\pm$ to the back face of $\zeta_\pm$ together and removing a standard neighborhood of a Legendrian divide on the convex torus of slope $0$. Similarly $\xi_{\pm\pm}$ is obtained by gluing the front face of $\eta_\pm$ to the back face of $\zeta_\mp$ together and removing a standard neighborhood of a Legendrian divide on the convex torus of slope $0$.
\end{lemma}

The construction of the contact structures $\xi_{\pm\pm}$ and the annulus in the lemma follows closely the construction in \cite[Lemma~4.13]{GhigginiSchonenberger03}, whereas the classification itself follows from \cite[Lemma~5.4]{GhigginiLiscaStipsicz07}. We give the proof of the lemma here, as we will need all the properties of the contact structures described in the lemma. 

\begin{remark}
  Notice that given a contact structure in $\Tight_0^{\operatorname{free}}(S^1\times P;-1, \infty, 0)$ the $0$ sloped ruling curves on $T_1$ and $T_2$ will be isotopic if and only if there is a convex annulus that they bound that has dividing curves running from one boundary component to the other. 
\end{remark}

\begin{proof}
  From the discussion before the lemma, it is clear that there are at most 4 contact structures, so we are left to see that the given contact structures are indeed tight and satisfy the required properties. 

  First consider the $\pm$ basic slice $\xi_\pm$ in $\Tight(T^2\times [0,1]; -1, \infty)$. We know there is a convex torus $T$ inside of $T^2\times [0,1]$ with dividing slope $0$. Notice that $T$ breaks $\xi_\pm$ into $\eta_\pm$ and $\zeta_\pm$. Remove a neighborhood of a dividing curve on $T$ to get a contact structure on $S^1\times P$. Clearly this contact structure contains no convex torsion and is in $\Tight_0(S^1\times P;-1, \infty, 0)$. Recall when considering $T^2\times[0,1]$ we orient $T^2\times \{0\}$ using the inward pointing normal vector and $T^2\times \{1\}$ with the outward pointing vector. However, when factoring a contact structure in $\Tight_0(S^1\times P;-1, \infty, 0)$ as above both the basic slices with back face $T_1$ and $T_2$ are oriented with the inward pointing vector. Thus the sign of the bypass on $T_2$ is opposite to what one sees when concatenating $\eta_\pm$ and $\zeta_\pm$. Thus the contact structure on $S^1\times P$ coming from $\xi_\pm$ is $\xi_{\pm\mp}$. Also, consider a convex annulus $A$ from a slope $0$ ruling curve on $T_1$ to a slope $0$ ruling curve on $T_2$. This will also be an annulus in $(T^2\times [0,1], \xi_\pm)$. We know the Poincar\'e dual of the relative Euler class of this contact structure is $\pm (1,-2)$ and so evaluates to $\pm 2$ on $A$. Since we know the relative Euler class evaluated on a convex surface is $\chi(A_+)-\chi(A_-)$, where $A_\pm$ is the $\pm$ regions of the convex surface $A$, we see that the dividing curves cannot run across $A$. 

  Now for the other two contact structures consider the $\pm$ basic slice $\xi'_\pm$  in $\Tight(T^2\times[0,1]; 0, \infty)$. Now let $T$ be a convex torus contact isotopic to $T^2\times \{1\}$ on the interior of $T^2\times[0,1]$. Let $L$ be a slope $0$ ruling curve on $T$. Removing a standard neighborhood of $L$ will result in a contact structure on $S^1\times P$. It will again clearly have no convex torsion and be an element in $\Tight_0(S^1\times P; -1, \infty, 0)$. By construction, there is a convex annulus with boundary components being ruling curves of slope $0$ on the boundary components of $S^1\times P$ with dividing slope $-1$ and $\infty$; moreover, this annulus has dividing curves going from one boundary component to the other. Thus the two contact structures on $S^1\times P$ coming from $\xi'_\pm$ are different from the ones coming from $\xi_\pm$ by their relative Euler classes. Thus they must be $\xi_{\pm\pm}$. Notice that by construction $\xi_{\pm\pm}$ is a union of some contact structure on the thickened tori $N_1$ and $N_2$ and the unique contact structure on $S^1\times P'$ (where $P'\subset P$) in $\Tight_0^{\operatorname{free}}(S^1\times P; 0,0,0)$. Notice that if one glues a solid torus $S$ to $T_3$ and extends the contact structure so that it is tight on the solid torus, then $(S^1\times P')\cup S$ will be an $I$ invariant contact structure on $T^2\times I$. Thus $(S^1\times P)\cup S$ will be the result of concatenating a basic slice in $\Tight(T^2\times[0,1];-1,0)$ and one in $\Tight(T^2\times[0,1];0,\infty)$. Since we have already identified $\eta_\pm\cup \zeta_\pm$ above, we see that the current contact structures must come from $\eta_\pm \cup \zeta_\mp$ by removing a $0$ sloped Legendrian divide from a convex torus. This establishes all the claimed properties. 
\end{proof}

\begin{remark}\label{rightbs}
  Notice that in the local model for $\xi_{\pm\pm}$ we see that if we attach a $\pm$ basic slice from $\Tight(T^2\times [0,1]; 0, \infty)$ to $\xi_{\pm\pm}$ we will still have a tight contact manifold, but if we attach the $\mp$ basic slice then it will become overtwisted. 
\end{remark}

\subsection{Non-rotative layers and properties of bypasses}\label{sec:non-rotative}
In this subsection, we will review basic definitions and properties of non-rotative layers. After that, we will review some useful properties of bypasses which will be used to prove Lemma~\ref{lem:staytight}. 

A \dfn{non-rotative $T^2 \times I$ layer}, or a \dfn{non-rotative layer} in short, is a tight $T^2 \times I$ with convex boundary such that any convex torus parallel to the boundary has the same dividing slope. We will denote $T_i = T^2 \times \{i\}$ for $i=0, 1$ and let $n_0$ and $n_1$ be the number of dividing curves on $T_0$ and $T_1$, respectively. Note that $n_0$ and $n_1$ are always even. 

Honda showed that some non-rotative layers are embedded in $I$-invariant neighborhoods. The following result follows from the second paragraph of the proof of \cite[Theorem~1.3]{Honda01}.

\begin{theorem}\label{thm:non-rotative-trivial}
  Let $T^2 \times [0,1]$ be a non-rotative layer with $n_0=2$ and $n_1\geq 2$. Then there is a non-rotative layer $T^2\times [1,2]$ so that $T^2\times \{1\}$ has the same dividing curves and characteristic foliation in both non-rotative layers, such that the result of gluing the two non-rotative layers results in an $I$-invariant neighborhood $T^2\times [0,2]$ where $n_0 = n_2 = 2$.  
\end{theorem}

We say a convex annulus $A$ in a non-rotative layer is \dfn{horizontal} if it has Legendrian boundary and intersects each dividing curve on $T_1$ and $T_2$ exactly once. Let $(M,\xi)$ be a tight contact $3$-manifold with a torus boundary $T$. Then a \dfn{non-rotative outer layer for $T$} is a non-rotative layer $N=T^2 \times I$ in $(M,\xi)$ such that $T_1 = T$, $n_0 = 2$ and $n_1 \geq 2$. Given two such non-rotative outer layers $N_1$ and $N_2$ for the boundary component $T$ of  $(M,\xi)$, let $A_i$ be a horizontal annulus in $N_i$ such that $A_1\cap T=A_2\cap T$ and denote this curve $c$. 
We say that $N_1$ and $N_2$ are \dfn{disk-equivalent} if there exist a disk $D$ and embeddings $\phi_i:A_i \to D$ such that $\phi_i(c) = \phi_i(c) = \partial D$, $(\phi_i)|_{c} = (\phi|_i)|_{c}$, and $\Gamma_1$ is isotopic to $\Gamma_2$ on $D$ where $\Gamma_i$ is obtained from $\phi_i(\Gamma_{A_i})$ by extending over $D- \phi_i(A_i)$ by a single arc (here $\Gamma_S$ denotes the dividing curves on a convex surface $S$). See Figure~\ref{fig:diskequiv} for example.

\begin{figure}[htbp]{\small
  \vspace{0.1cm}
  \begin{overpic}{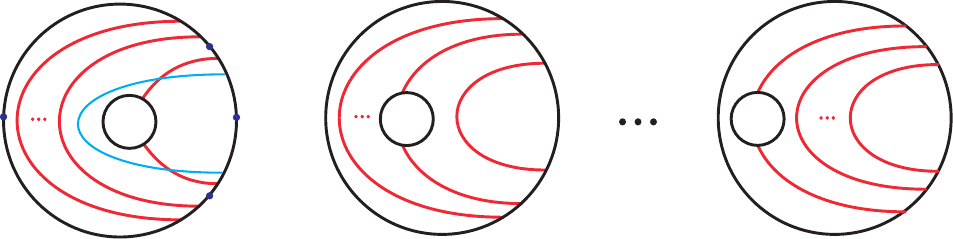}
    \put(-13,57){$p_1$}
    \put(108,17){$p_{k}$}
    \put(120,57){$p_{k+1}$}
    \put(108,92){$p_{k+2}$}
  \end{overpic}}
  \caption{Disk-equivalent annuli. The blue arc is a Legendrian arc.}
  \label{fig:diskequiv}
\end{figure}

The following theorem, which can be found in \cite[Theorem~3.1] {Honda01}, shows that any two non-rotative outer layers for a fixed torus are disk-equivalent.

\begin{theorem} \label{thm:disk-equivalence}
  Let $(M,\xi)$ be a tight contact $3$-manifold with convex boundary and $T$ be a torus boundary component. Then any two non-rotative outer layers for $T$ are disk-equivalent. 
\end{theorem}

Even though $(M,\xi)$ can contain two different (but disk-equivalent) non-rotative outer layers for a fixed torus, the complements of these layers are contactomorphic. This is the content of \cite[Corollary~2.1]{Honda01} which we now quote. 

\begin{theorem} \label{thm:non-rotative-contactomorphic}
  Let $(M,\xi)$ be a tight contact $3$-manifold with convex boundary and $T$ be a torus boundary component. Suppose $(M,\xi) = (M_i,\xi_i) \cup N_i$ for $i = 0,1$ where the $N_i$ are non-rotative outer layers for $T$. Then there is a co-orientation preserving contactomorphism between $(M_0,\xi_0)$ and $(M_1,\xi_1)$. 
\end{theorem}

The next theorem easily follows from Theorem~\ref{thm:non-rotative-trivial}.

\begin{theorem} \label{thm:solid-torus}
  Let $V$ be a tight solid torus having convex boundary $T$ with two dividing curves of slope $s$. Let $N$ be a non-rotative layer of slope $s$ with $n_0 =2$ and $n_1>2$. Then $V \cup N$ is tight. 
\end{theorem}

Consider a tight solid torus $V$ with convex boundary $T$. Let $n$ be the number of dividing curves on the boundary and $q/p$ be the slope of the dividing curves. We can find (after perturbation) a Legendrian meridian $c$ on $T$, which intersects the dividing curves in $2k$ points.  Label these intersection points as $p_0, \ldots, p_{2k-1}$ consecutively. Let $D$ be a convex meridian disk in $V$ bounded by $c$. Clearly, there exist $2k$ dividing curves on $D$. We say a bypass on $D$ is a \dfn{bypass for $p_i$} if there exists a bypass with attaching arc containing $p_{i-1}$, $p_i$, $p_{i+1}$ (indices are considered as an element in $\mathbb{Z}_{2k}$). We also say a bypass is \dfn{effective} if the attaching arc of the bypass passes three dividing curves and the center dividing curve is different from the others. Recall that Honda \cite{Honda00a} showed that attaching an effective bypass to a torus will decrease the number of dividing curves or change the dividing slope. He also showed that if a bypass in a tight solid torus is not effective, then it is contained in an $I$-invariant neighborhood of $T$. The following theorem follows from \cite[Proposition~5.8]{Honda00a}. Note the result only needs the existence of a non-rotative outer layer, but not its uniqueness, which is not true, see \cite{Honda01}. We recall that a bypass that, when attached to a torus, increases the number of dividing curves is called a \dfn{folding bypass}. And such a bypass can only increase the number of dividing curves by $2$. 

\begin{theorem}\label{thm:effective-bypasses}
  Let $V$ be a tight solid torus with convex boundary $T$ having two dividing curves. If a bypass in $V$ whose attaching arc is on $T$ is not effective, then the bypass is contained in an $I$-invariant neighborhood of $T$. In fact, it is either a trivial bypass or a folding bypass. 
\end{theorem}

Now we review Colin's isotopy discretization, which is the key idea of ``state transition'' (see \cite[Section~2]{Honda02}).

\begin{theorem}[Isotopy Discretization, Colin \cite{Colin97}, see also Honda \cite{Honda02}]\label{thm:discretization}
  Let $\Sigma$ and $\Sigma'$ be two convex surfaces with the same Legendrian boundary. If there is a smooth isotopy between them rel boundary, then there exists a sequence of surfaces $\Sigma_0 = \Sigma, \ldots \Sigma_n = \Sigma'$ with the same boundary and $\Sigma_{i+1}$ is obtained from a single bypass attachment to $\Sigma_i$. 
\end{theorem}

We say a bypass on a disk is \dfn{non-nested} if it is associated to a dividing curve that separates the disk into two components, one of which has no dividing curves. We say there are \dfn{nested bypasses} for a point $p$ if there are consecutive dividing curves parallel to a non-nested bypass for $p$. The number of dividing curves for nested bypasses is called the \dfn{length of the nested bypasses}. See Figure~\ref{fig:dot-balancing} for example.

\begin{figure}[htbp] {\small
  \vspace{0.2cm}
  \begin{overpic}{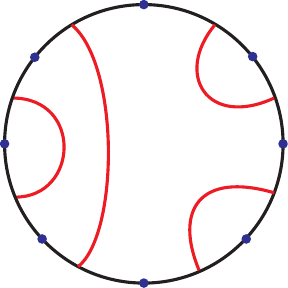} 
    \put(-15, 68){$p_1$}
    \put(5, 15){$p_2$}
    \put(66, -10){$p_3$}
    \put(125, 15){$p_4$}
    \put(145, 68){$p_5$}
    \put(127, 117){$p_6$}
    \put(66, 145){$p_7$}
    \put(4, 117){$p_8$}
  \end{overpic}}
  \vspace{0.1cm}
  \caption{Nested bypasses for $p_1$ with length $2$. There are three non-nested bypasses for $p_1$, $p_4$ and $p_6$.}
  \label{fig:dot-balancing}
\end{figure}

\section{An algorithm to classify non-loose torus knots}\label{thealgorithm}

In this section we give a user's guide to the complete classification of non-loose Legendrian $(p,q)$-torus knots. We prove that this algorithm really gives the complete classification in Section~\ref{justifyalgorithm}. A brief outline of the algorithm is as follows: 

\smallskip
\noindent
{\bf Input:}  A pair of relatively prime integers $(p,q)$ with $|q|>p>1$ (recall that a $(p,q)$-torus knot is isotopic to a $(q,p)$-torus knot).

\smallskip
\noindent
{\bf Output:} A list of overtwisted contact structures supporting a non-loose $(p,q)$-torus knot and a list of all non-loose $(p,q)$-torus knots in these contact structures.

The steps in the algorithm are as follows:
\begin{enumerate}
\setcounter{enumi}{-1}
\item Compute the pair of paths $(P_1,P_2)$ representing $q/p$ as discussed in Section~\ref{subsec:pathsinFG}.
\item Consider all $2$-inconsistent decorations on these paths, construct the associated surgery diagrams (as described in Section~\ref{sec:FareytoSurgery}), and then compute the corresponding $d_3$ invariants associated to these $d_3$ invariants. These, together with $\xi_1$ if $pq>0$, will be the only overtwisted contact structures with non-loose $(p,q)$-torus knots. (We recall that $k$-inconsistent pairs of paths and compatible pairs of paths were defined in Section~\ref{pairsodecorated}.)
\item For each pair $(p,q)$, there will be a unique contact structure supporting non-loose knots with a special mountain range. These mountain ranges are shown in Figure~\ref{fig-exceptionMR}. Compute the invariants of these non-loose knots. (Note for $pq<0$, this is completely done in Figure~\ref{fig-exceptionMR}.
\item For each $2$-inconsistent decorated pair of paths, find all the other decorated pairs of paths that are compatible with this $2$-inconsistent pair. Compute the classical invariants of the non-loose $(p,q)$-torus knots associated to these decorated pairs of paths. These give the non-loose knots with $\tor=0$ in the overtwisted contact structure given by the $2$-inconsistent decorated pair of paths. 
\item Finally, determine which of the above non-loose knots one can add convex torsion to and obtain other non-loose knots. This will give a complete list of non-loose Legendrian knots with $\tor>0$.  
\end{enumerate}

Item (0) was discussed in Section~\ref{subsec:pathsinFG}. Details of Items (1), through (3) are given in the next subsection, while details of Item~(4) are given in the following subsection. 

The proof that the algorithm below classifies all non-loose $(p,q)$-torus knots is given in Section~\ref{justifyalgorithm}, though below we will cite other specific sections where details of the algorithm are established. 

\subsection{The classification of non-loose torus knots without convex torsion}\label{classwogt}
Below is an algorithm to classify non-loose $(p,q)$-torus knots with $\tor = 0$. 

\noindent
\textbf{Step 1.\ Determine the overtwisted contact structures that support non-loose $(p,q)$-torus knots.}\label{step1} 
Given $q/p$ proceed as follows:
\begin{enumerate}
\item  Find all $2$-inconsistent pairs of decorated paths representing $q/p$:
\[
  \{ (\pm P_1^1,\pm P_2^1), \ldots, (\pm P_1^n, \pm P_2^n)\}.
\]
Recall $k$-inconsistent decorated paths were defined in Section~\ref{pairsodecorated}.
\item For each $i\in \{1,\ldots, n\}$ draw a contact surgery diagram for $L_{P_1^i,P_2^i}$ as described in Section~\ref{sec:FareytoSurgery}. 
\item Use the formula~(\ref{formula:d3}) in Section~\ref{htpclasses} to compute the $d_3$-invariant of the contact structure. (We note that there may be some cases where one can identify the $d_3$-invariant without finding the surgery diagram. See, for example, Section~\ref{classofpos22np1}.)
\end{enumerate}
These, and $\xi_1$ when $pq>0$, are the only contact structures supporting non-loose $(p,q)$-torus knots without convex torsion. 

\noindent
\textbf{Step 2.\ Compute the non-loose Legendrian knots with ``exceptional" mountain ranges.}\label{step2} For any $(p,q)$-torus knot, there is an exceptional overtwisted contact structure where the classification is qualitatively different from all the others. These are shown in Figure~\ref{fig-exceptionMR} and defined as follows. 
\begin{definition}
The \dfn{exceptional} contact structure for the $(p,q)$-torus knot is $\xi_1$ if $pq>0$ and $\xi_{|pq|-|p|-|q| + 1}$ if $pq<0$. 
\end{definition}

We will consider each case of exceptional contact structure separately.
\subsubsection{The case of $pq>1$}

In Section~\ref{exceptionalctstr} it will be shown that for $pq>0$ the contact structure $\xi_1$ is described by the pair of paths $(P_1,P_2)$ with all signs the same and also by the $2$-inconsistent pair of paths $(P_1',P_2')$ where the signs of all the basic slices in $P'_1$ are the same, say $\pm$, except for one edge in the first block that is $\mp$, while all blocks $P'_2$ are labeled with $\mp$ except for one edge in the first block that is $\pm$. 

Let $(P_1,P_2)$ be a pair of paths representing $q/p$. 
Recall that we can arrange the continued fraction blocks in $P_1$ and $P_2$ as follows.
\[
  (A_1,A_3,\ldots,A_{m-1}) \text{ and } (B_2,B_4,\ldots,B_m)
\] 
(notice that there are several other cases, but they can be dealt in the same way). Let $s_k$ be the slope in the $k^{th}$ continued fraction block that is farthest from $q/p$. Set $n_k = |s_k \bigcdot \frac qp|$ for $k=2,\ldots,m$. 

In Section~\ref{diamondsforxi1} and~\ref{classgeqpq} we will classify non-loose Legendrian knots in $\xi_1$. Specifically, there is an infinite $\textsf{V}$ with bottom vertex having $r=0$ and $\tb=pq-p-q+2$. That is, there are Legendrian knots $L_\pm^i$ for $i>pq-p-q+2$ and $L^{pq-p-q+2}$ with 
\[
  \tb(L_\pm^i)=i,\,\, \rot(L_\pm^i)=\mp(i-pq + p+q-2), \text{ and } \tor(L_\pm^i) = 0,
\]
\[
  \tb(L^{pq-p-q+2})=pq-p-q+2,\,\, \rot(L^{pq-p-q+2})=0, \text{ and } \tor(L^{pq-p-q+2}) = 0,
\]
such that  
\[
  S_\pm(L_\pm^i)=L_\pm^{i-1}, \text{ for i $\geq pq-p-q+4$, and } S_\pm(L_\pm^{pq-p-q+3})=L^{pq-p-q+2},
\]
and 
\[
  S_\mp(L_\pm^i) \text{ and } S_\pm(L^{pq-p-q+2}) \text{ are loose}.
\]

Moreover, there are Legendrian knots $L_{k,\pm}^{pq}$ with $k=2,\ldots, m$  such that 
\[
  \tb(L_{k,\pm}^{pq})=pq,\,\, \rot(L_{k,\pm}^{pq})=\mp\left(p + q  - 2n_{k}\right), \text{ and } \tor(L_{k,\pm}^{pq})=0
\]
and a stabilization of $L_{k,\pm}^{pq}$ is non-loose if and only if it stays on or above the $\textsf{V}$ described above. Lastly, the Legendrian knots described above are coarsely Legendrian simple, so if stabilizations of any two of them have the same Thurston-Bennequin invariant and rotation number, then they are equivalent. 

\subsubsection{The case of $pq<-1$}

In Section~\ref{exceptionalctstr} it will be shown that when $pq<0$ the contact structure $\xi_{|pq|-|p|-|q|+1}$ is described by the $2$-inconsistent pair of paths $(P_1,P_2)$ where all the signs in $P_1$ are the same and opposite to all the signs in $P_2$. 

In Section~\ref{sec:wings}, \ref{classgeqpq} ,and~\ref{classextra} we will classify non-loose Legendrian knots in $\xi_{|pq|-|p|-|q| + 1}$. Specifically, we show that there are non-loose Legendrian knots $L_\pm^i, i\in \Z$ and $L_e$ with
\[
  \tb(L_\pm^i)=i,\,\,  \rot(L_\pm^i)=\mp(i - |pq| + |p|+|q|), \text{ and } \tor(L_\pm^i) = 0,
\]
\[
  \tb(L_e)=|pq|-|p|-|q|,\,\, \rot(L_e)=0, \text{ and } \tor(L_e) = 0,
\]
such that 
\[
  S_\pm(L_\pm^i)=L_\pm^{i-1} \text{ and } S_\pm(L_e)=L_\pm^{|pq|-|p|-|q|-1}
\]
and
\[
  S_\mp(L_\pm^i) \text{ is loose.}
\]
Notice that all these Legendrian knots are determined by the Thurston-Bennequin invariants and the rotation numbers, except when $\tb=|pq|-|p|-|q|$, there are $3$ distinct Legendrian knots all with the rotation number $0$. 

\noindent
\textbf{Step 3.\ Compute the non-loose Legendrian knots with ``generic" mountain ranges.}\label{stpe3}
All other mountain ranges are as shown in Figure~\ref{fig-genericXwing}. In Section~\ref{sec:wings} and~\ref{classgeqpq} we will establish that this is the classification and the classification is described as follows. 

\begin{enumerate}
\item For each $2$-inconsistent decorated pair of paths $(P_1,P_2)$ representing $q/p$ that is not compatible with the ones discussed in Step~2 do the following.
\item $(P_1,P_2)$ may be compatible with other decorated pairs of paths as discussed in Section~\ref{subsec:pathsinFG}. Let $\{(P_1^i,P_2^i)\}_{i=2}^n$ be the collection of all pairs of paths compatible with $(P_1,P_2)$ where $(P_1^i,P_2^i)$ is $i$-inconsistent and $(P_1^2,P_2^2)=(P_1,P_2)$.
\item Recall the truncated path $P_2^\intercal=P_2\cap \left[q/p,\lceil q/p\rceil\right]$. Let $s_k$ be the slope in the $k^{th}$ continued fraction block of $P_1\cup P_2^\intercal$ that is farthest from $q/p$. Set $n_k=|s_k \bigcdot \frac qp|$ for $k=2,\ldots, n$. 
\item In the contact structure $\xi_{P_1,P_2}$ there are non-loose Legendrian knots $L_+^i$ and $L_-^i$ for $i\in \Z$ with invariants 
\[
  \tb(L_\pm^i)=i, \text{ and } \rot(L_\pm^i)= \begin{cases} \mp(i - pq + |R(P_1,P_2)|)\quad pq>0,\\ \mp(i - pq - |R(P_1,P_2)|)\quad pq<0, \end{cases}
\]
where $R(P_1,P_2)$ is defined in Section~\ref{htpclasses}, such that 
\[
  S_\pm(L_\pm^i)=L_\pm^{i-1} \text{ and } S_\mp(L_\pm^i) \text{ is loose.}
\]
Moreover, if $n \geq 3$, there are Legendrian knots $L_{k,\pm}^{pq}$ with $k=2,\ldots, n-1$  such that 
\[
  \tb(L_{k,\pm}^{pq})=pq,\, \text{ and } \tor(L_{k,\pm}^{pq})=0
\]
and
\[
  \rot(L_{k,\pm}^{pq})= \begin{cases} \mp (|R(P_1,P_2)| - 2(n_{k}-1)) &\,pq>0,\\ \mp (-|R(P_1,P_2)| - 2(n_{k}-1)) &\,pq < 0. \end{cases}
\]
and $S_\pm^iS_\mp^j(L_{k,\pm}^{pq})$ is non-loose if and only if $j\leq n_k-1$.
\item  Lastly, if stabilizations of the $L_{k,\pm}^{pq}$ and $L_{l,\pm}^{pq}$ have the same invariants, then they are equivalent, while non-loose stabilizations of $L_{k,\pm}^{pq}$ and $L_{l,\mp}^{pq}$ are never equivalent. Notice that when $pq<0$ then non-loose stabilizations of $L_{k,\pm}^{pq}$ and $L_{l,\mp}^{pq}$ will never share the same invariants but when $pq>0$ they will. See Figure~\ref{fig-genericXwing}.
\item One subtlety arises when $pq>0$ and all blocks in $P_1$ have the same sign and all blocks $P_2$ have the opposite sign. In this case, we will see in Lemma~\ref{standardstructures} that $\xi_{P_1,P_2}$ is simply $\xi_{-pq+p+q}$ which is obtained from $\xi_{std}$ by a half Lutz twist on the unique maximal self-linking number transverse representative of the $(p,q)$-torus knot. In this case, the knots $L_\pm^{i}$ will have $\tor(L_\pm^i) = 1/2$ when $i\leq pq-p-q$ and otherwise have $\tor(L_\pm^i) = 0$.
\end{enumerate}

\subsection{The classification of non-loose torus knots with convex torsion}\label{classwithtorsion}
We now consider non-loose Legendrian knots with $\tor = n$ for $n \in \mathbb{N}\cup \{0\}$. These results are established in Section~\ref{classmoretorsion}. 

\subsubsection{The case of $pq>1$}
For any pair of paths $(P_1,P_2)$ representing $q/p$ that is totally $2$-inconsistent, the classification of non-loose Legendrian with torsion is as follows. 

\begin{enumerate}
\item When $(P_1,P_2)$ is not the pair of paths such that $P_1$ has only one sign while $P_2$ has only the other sign we have the following.
\begin{enumerate}
\item In the contact structure $\xi_{P_1,P_2}$ there are Legendrian knots $L_\pm^{i,n}$ for $i\in \Z$ with invariants 
\[
  \tb(L_\pm^{i,n})=i \text{ and } \tor(L_\pm^{i,n})=n,
\]
and 
\[
  \rot(L_\pm^{i,n}) = \begin{cases} \mp(i -pq + |R(P_1,P_2)|)\quad pq>0, \\ \mp(i - pq - |R(P_1,P_2)|)\quad pq<0, \end{cases} 
\]
where $R(P_1,P_2)$ is defined in Section~\ref{htpclasses}. We also have
\[
  S_\pm(L_\pm^{i,n})=L_\pm^{i-1,n} \text{ and } S_\mp(L_\pm^{i,n}) \text{ is loose.}
\]
Moreover, $L_\pm^{i,0}$ corresponds to $L_\pm^i$ in the previous section. Notice that the mountain range for non-loose Legendrian knots in $\xi_{P_1,P_2}$ does not contain any extra ``wings" as seen for some contact structures on the previous section.
\item In the contact structure $\xi_{P_1,P_2}'$ with 
\[
d_3(\xi_{P_1,P_2}')=d_3(\xi_{P_1,P_2}) + |R(P_1,P_2)| - pq
\]
 there are non-loose Legendrian knots $L_\pm^{i,n+\scriptscriptstyle\frac 12}$ in $\xi'_{P_1,P_2}$ for $i\in \Z$ with invariants 
\[
  \tb(L_\pm^{i,n+\scriptscriptstyle\frac 12})=i, \text{ and } \tor(L_\pm^{i,n+\scriptscriptstyle\frac 12}) = n + \frac12
\]
and
\[
  \rot(L_\pm^{i,n+\scriptscriptstyle\frac 12}) = \begin{cases} \mp(i + pq-|R(P_1,P_2)|) \,& pq>0, \\ \mp(i + pq+|R(P_1,P_2)|) &pq<0. \end{cases} 
\]
We also have 
\[
  S_\pm^j(L_\pm^{i,n+\scriptscriptstyle\frac 12})=L_\pm^{i-j,n+\scriptscriptstyle\frac 12} \text{ and } S_\mp(L_\pm^{i,n+\scriptscriptstyle\frac 12}) \text{ is loose.}
\]
Notice that the mountain range for non-loose Legendrian knots in $\xi_{P_1,P_2}$ does not contain any extra ``wings" as seen for some contact structures on the previous section. 
\end{enumerate}
\item When $P_1$ has all one sign and $P_2$ has only the other sign, as noted above, $\xi_{P_1,P_2}$ is $\xi_{-pq+p+q}$. In this case the classification of contact structures on $\xi_{-pq+p+q}$ is as stated above except 
\[
  \tor(L_\pm^{i,n}) = \begin{cases} n \quad &i>pq-p-q,\\ n+ 1/2 \quad &i\leq pq-p-q.\end{cases}
\] 
Similarly, in $\xi_0=\xi'_{P_1,P_2}$ the classification is as stated above except
\[
  \tor(L_\pm^{i,n+\scriptscriptstyle\frac 12}) = \begin{cases} n+ 1/2 \quad &i>pq-p-q,\\ n+1 \quad &i\leq pq-p-q.\end{cases}
\] 
\end{enumerate}

\subsubsection{The case of $pq>1$}
When $pq<0$ the classification in all cases is as in the case for $pq>0$ and $(P_1,P_2)$ is not the pair of paths such that $P_1$ has only one sign while $P_2$ has only the other sign, except the $d_3$-invariant of $\xi_{P_1,P_2}'$ is $d_3(\xi_{P_1,P_2}) - |R(P_1,P_2)| - pq$.

\section{Classification of non-loose \texorpdfstring{$(2, \pm(2n+1))$}{(2,(2n+1))}-torus knots}\label{fullclass22np1}

In this section, we apply the algorithm in Section~\ref{thealgorithm} to prove Theorem~\ref{thm:(2,2n+1)},  Theorem~\ref{thm:(2,-(2n+1))}, Theorem~\ref{thm:(2,2n+1)-transverse}, and Theorem~\ref{thm:(2,-(2n+1))-transverse}. We note that the classification of the $(2,2n+1)$-torus knots is quite straightforward and many steps in the algorithm are not necessary. To see the algorithm carried out in its full generality, please see Sections~\ref{secneg22n},~\ref{sec58}, and~\ref{sec58neg}.

\subsection{Non-loose \texorpdfstring{$(2,2n+1)$}{(2,2n+1)}-torus knots}\label{classofpos22np1}
We begin by classifying non-loose Legendrian $(2,2n+1)$-torus knots. 

\begin{proof}[Proof of Theorem~\ref{thm:(2,2n+1)}]
First, we apply Step $1$ of the algorithm. The pair of paths representing $(2n+1)/2$ is  
\[
  P_1=\left\{\frac{2n+1}{2}, n\right\}, \text{ and } P_2=\left\{\frac{2n+1}{2},n+1, \infty\right\},
\]
and the continued fraction blocks in $P_1$ and $P_2$ are
\[
  A_1=\left\{\frac{2n+1}{2}, n\right\}, \text{ and } B_2=\left\{\frac{2n+1}{2},n+1, \infty\right\}.
\]  
Now we can list all non-loose decorations of $(\overline{P_1}\,;\,P_2)$ as follows:
\[
  \pm(+\,;\,+,+),\,\pm(-\,;\,+,-),\,\pm(-\,;\,+,+)
\]
(Since $P_2$ is a continued fraction block, $\pm(-,-,+)$ and $\pm(-,+,-)$ are the same). As stated in the algorithm, we know that when all the signs are the same, the contact structure will be $\xi_1$. We also know that the sign choices $\pm(-)$ on $P_1$ and $\pm(+,-)$ on $P_2$ will also give $\xi_1$ since they are compatible with $\pm(+,+,+)$, see the last two paragraphs of Section~\ref{subsec:pathsinFG}. Thus we only need to compute the $d_3$-invariant of $\pm(-,+,+)$, which is $-pq+p+q=1-2n$ by Lemma~\ref{standardstructures}.

Now applying Step $2$ of the algorithm, we have the Legendrian knots in $\xi_1$ as follows:
\begin{gather*}
  L_\pm^i \,\text{ for $i>2n+1$},\\ L^{2n+1},\\
  L_{2,\pm}^i \,\text{ for } 2n+4 \leq i \leq 4n+2,\\ L_2^{2n+3}
\end{gather*}
with
\begin{gather*}
  \tb(L_\pm^i)=i, \text{ and } \rot(L_\pm^i)=\mp(i-2n-1),\\
  \tb(L^{2n+1})=2n+1, \text{ and } \rot(L^{pq-p-q+2})=0,\\
  \tb(L_{2,\pm}^i)=i, \text{ and } \rot(L_{2,\pm}^i)= \mp (i - 2n-3),\\
  \tb(L_2^{2n+3})=2n+3, \text{ and } \rot(L_2^{2n+3})=0
\end{gather*}
such that 
\begin{gather*}
  S_\pm(L_\pm^i)=L_\pm^{i-1}, \text{ for $i \geq 2n+3$, and } S_\pm(L_\pm^{2n+2})=L^{2n+1},\\
  S_\pm(L_{2,\pm}^i)=L_{2,\pm}^{i-1}, \text{ for $i\geq 2n+5$, and } S_\pm(L_{2,\pm}^{2n+4})= L_2^{2n+3},\\
  S_\mp(L_{2,\pm}^i)=L_\pm^{i-1}, \text{ for } i\geq 2n+4, S_\mp(L_2^{2n+3})=L_\pm^{2n+2},
\end{gather*}
and $S_\mp(L_\pm^i)$ and $S_\pm(L^{2n+1})$ are loose. All these Legendrian knots have $\tor = 0$. See Figure~\ref{fig:RHT-mountain}.

Applying Step $3$ of the algorithm, we obtain the classification of non-loose Legendrian knots in $\xi_{1-2n}$ as follows. Note that the decoration $\pm(-,+,+)$ are totally $2$-inconsistent, so we do not have to consider the ``wings''. We must first compute $R(P_1,P_2)$. One may easily check that $r_m=1$ and $r_n=2n$ and hence 
\[
  R(P_1,P_2)=(2n+1)\cdot 1 + 2\cdot (2n)=6n+1.
\]
Step 3 of the algorithm now gives non-loose Legendrian knots $L_\pm^{i,k}$ for $i \in \Z$ and $k\in \N\cup \{0\}$ such that 
\[
  \tb(L_\pm^{i,k})=i,\, \text{ and } \rot(L_\pm^{i,k})=\mp(i+2n-1)
\]
and
\[
  \tor(L_\pm^{i,k}) = \begin{cases} k \,& i>2n-1,\\ k+\frac12 &i\leq 2n-1. \end{cases} 
\]
We also have
\[
  S_\pm(L_\pm^{i,k})=L_\pm^{i-1,k}  
\]
and $S_\mp(L_\pm^{i,k})$ is loose. See Figure~\ref{fig:RHT-mountain}.

Finally we consider the non-loose Legendrian knots coming from adding half convex torsion to the complements of $L_-^{i,0}$. This will give the contact structure obtained from $\xi_{1-2n}$ by a half Lutz twist on the transverse push-off of $L_-^{i,0}$, which is $\xi_0$. So in $\xi_0$, we have non-loose Legendrian knots $L_\pm^{i,k+\scriptscriptstyle\frac 12}$ for $i \in \Z$ and $k\in \N\cup \{0\}$ such that 
\[
  \tb(L_\pm^{i,k+\scriptscriptstyle\frac 12})=i \text{ and } \rot(L_\pm^{i,k+\scriptscriptstyle\frac 12})=\mp(i-2n+1)
\]
and
\[
  \tor(L_\pm^{i,k+\scriptscriptstyle\frac 12}) = \begin{cases} k+\frac12 \,&i>2n-1,\\ k+1 &i\leq2n-1. \end{cases}
\]
We also have
\[
  S_\pm(L_\pm^{i,k+\scriptscriptstyle\frac 12})=L_\pm^{i-1,k+\scriptscriptstyle\frac 12} 
\]
and $S_\mp(L_\pm^{i,k+\scriptscriptstyle\frac 12})$ is loose. See Figure~\ref{fig:RHT-mountain}.
\end{proof}
We now turn to the classification of non-loose transverse $(2,2n+1)$-torus knots.
\begin{proof}[Proof of Theorem~\ref{thm:(2,2n+1)-transverse}]
Since the classification of transverse knots is equivalent to the classification of Legendrian knots up to negative stabilization \cite[Proof of Theorem~2.10]{EtnyreHonda01}, the theorem follows immediately from Theorem~\ref{thm:(2,2n+1)}. In particular, then $T^k$ and the $T^{k+\scriptscriptstyle\frac 12}$ are transverse push-offs of $L_-^{0,k}$ and $L_-^{0,k+\scriptscriptstyle\frac 12}$, respectively. 
\end{proof}

\subsection{Non-loose \texorpdfstring{$(2,-(2n+1))$}{(2,-(2n+1))}-torus knots}\label{secneg22n}
We begin with the classification of non-loose Legendrian $(2,-(2n+1))$-torus knots. 

\begin{proof}[Proof of Theorem~\ref{thm:(2,-(2n+1))}] 
According to Step $1$ of the algorithm in the previous section, we first find all decorated pair paths $(P_1,P_2)$ representing $-(2n+1)/2$. We have 
\[
  P_1=\{-\frac{2n+1}{2}, -n-1\}, \text{ and } P_2=\{-\frac{2n+1}{2}, -n, -n+1, \ldots, -1\}.
\] 
Moreover, the breakdown into continued fraction blocks is 
\[
  A_2=\{-\frac{2n+1}{2}, -n-1\}, \text{ and } B_1=\{-\frac{2n+1}{2}, -n\},\, B_3=\{-n,\ldots -1\}.
\] 
Now we will list all non-loose decorations for $(\overline{P_1}\,;\,P_2)$. Since $(\pm,\pm,\ldots,\pm)$ describes $\xi_{std}$ by Lemma~\ref{extrablockfornegative}, we have $2n$ non-loose decorations 
\[
  \pm(-,+\,;\,\overbrace{+,\ldots,+}^k,\overbrace{-,\ldots,-}^{n-1-k}) 
\]
for $0 \leq k \leq n-1$. See the top drawing of Figure~\ref{fig:LHTsurgery} (notice that we can shuffle the signs in $B_3$). We now convert these pairs of decorated paths into contact surgery diagrams as discussed in Section~\ref{sec:FareytoSurgery}. To this end notice that
\[
  -\frac{p}{p'} = -2 = [-2], \text{ and } -\frac{q}{q-q'} = -\frac{2n+1}{n+1} = [-2,-n-1].
\] 
We thus obtain the diagrams in Figure~\ref{fig:LHTsurgery}.

\begin{figure}[htbp]{\tiny
  \vspace{0.1cm}
  \begin{overpic}{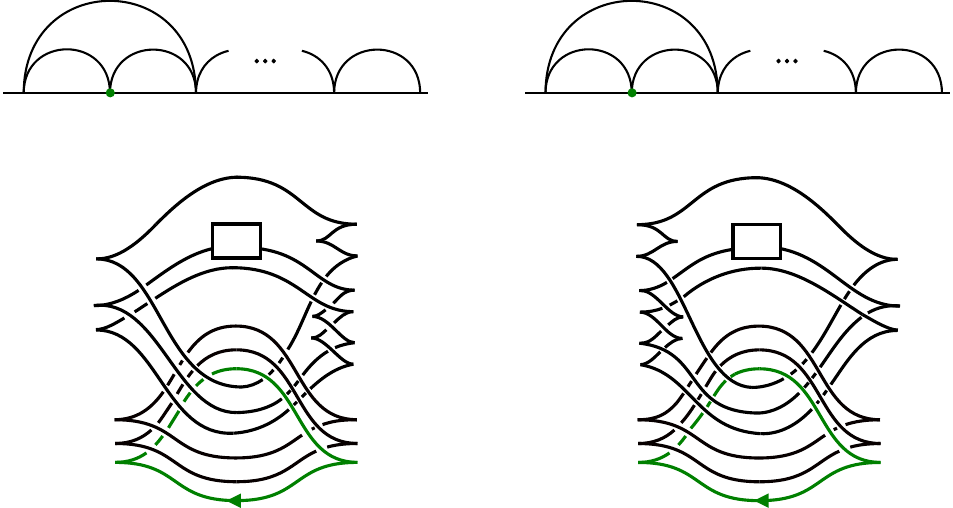}
    \put(174, 135){$(-1)$}
    \put(174, 105){$(-1)$}
    \put(174, 65){$(-1)$}
    \put(174, 42){$(+1)$}
    \put(174, 30){$(+1)$}
    \put(154, 8){$L_-$}
    \put(105, 127){$n-1$}

    \put(435, 118){$(-1)$}
    \put(435, 96){$(-1)$}
    \put(435, 83){$(-1)$}
    \put(429, 44){$(+1)$}
    \put(429, 28){$(+1)$}
    \put(406, 8){$L_+$}
    \put(355, 127){$n-1$} 
    
    \put(-4, 188){$-n-1$}
    \put(37, 188){$-\frac{2n+1}2$}
    \put(93, 188){$-n$}
    \put(154, 188){$-2$}
    \put(195, 188){$-1$}
    \put(28, 208){$+$}
    \put(71, 208){$-$}
    \put(179, 208){$\pm$}
    
    \put(248, 188){$-n-1$}
    \put(289, 188){$-\frac{2n+1}2$}
    \put(339, 188){$-n$}
    \put(405, 188){$-2$}
    \put(446, 188){$-1$}
    \put(280, 208){$-$}
    \put(322, 208){$+$}
    \put(430, 208){$\pm$}
  \end{overpic}}
  \caption{Surgery diagrams for the non-loose $(2,-(2n+1))$ torus knots with $\tb = -4n-2$. In the upper box there are $n-1$ stabilizations. The signs of the stabilizations depend on the signs in the continued fraction block $A_2$.}
  \label{fig:LHTsurgery}
\end{figure}

We now use the formula~(\ref{formula:d3}) in Section~\ref{htpclasses} to compute $d_3$-invariants. In particular, we have the linking matrix $M$ for the diagram
\[
  M =
  \begin{bmatrix}
    -3 & -1 & -1 & -1 & -1 \\
    -1 & -n-2 & -2 & -1 & -1 \\
    -1 & -2 & -3 & -1 & -1 \\
    -1 & -1 & -1 & 0 & -1 \\
    -1 & -1 & -1 & -1 & 0\\
  \end{bmatrix},
\] 
from which we can compute that $\sigma(X)=-1$ and $\chi(X)=6$.  There exist $2n$ rotation vectors:
  \[
    \pm\mathbf{rot} =
    \begin{bmatrix}
      -1\\-(l+1)\\-1\\0\\0
    \end{bmatrix},
  \]
  for $l \in \{-n+1, -n+3, ..., n-3, n-1\}$. From these, we can compute 
    \[
    d_3(\xi) = n + l + 1.
  \]
One may also use the surgery diagram to compute 
\[
  |R(P_1,P_2)|=4n + 2l + 3.
\]
In Section~\ref{htpclasses} we also gave a formula for $R(P_1,P_2)$ in terms of the decorated pair of paths, 
since relating this computation to the one done in terms of the surgery diagram is a little involved, we give details. First, notice that
\begin{gather*}
  p_1=-\frac{2n+1}{2},\; p_2=-n-1,\\
  q_1=-\frac{2n+1}{2},\\ 
  q_i = -n+i-2\; \text{ for }\; 2 \leq i \leq n+1.
\end{gather*}
In $P_1$, the signs of blocks are all negative. In $P_2$, there exist $k+1$ positive blocks, which are exactly $[q_{1},q_{2}],\ldots,[q_{k+1},q_{k+2}]$. Thus we can calculate $r_m$ and $r_n$ as follows:
\begin{gather*}
  r_m = (-1)\left((p_2 \ominus p_1) \bigcdot \frac10\right) = -1,\\
  r_n = (q_{k+2} \ominus q_1) \bigcdot \frac01 + (-1)\left((q_{n+1} \ominus q_{k+2})\bigcdot \frac01\right) = 2k+2.
\end{gather*}
We could now easily compute $R(P_1,P_2)$ in terms of $k$, but we would like to make the computation in terms of the rotation numbers in the surgery diagram. The rotation number of the second link component of the surgery diagram in Figure~\ref{fig:LHTsurgery} is equal to the difference between the number of negative blocks and positive blocks (notice that $P_2$ corresponds to the solid torus with the upper meridian).  Thus we have 
\begin{align*}
  -(l+1) &= (n-k-1) - (k+1)\\ 
  &= n -2k -2,
\end{align*}
and this implies that
\[
  k = (n+l-1)/2.
\]
Finally, we have
\begin{align*}
  R(P_1,P_2) &= 2r_n + (-2n-1)r_m\\ 
  &= 4n + 2l +3.
\end{align*}

Now we apply Step 2 of the algorithm and consider the exceptional contact structure corresponding to the decoration $(\pm,\mp,\ldots,\mp)$, where $P_1$ has edges all of one sign and $P_2$ has edges of the opposite sign. This is the contact structure $\xi_{2n}$ (since $P_1$ and $P_2$ are totally $2$-inconsistent we also obtain non-loose knots with convex torsion and we list them now). According to the algorithm in Section~\ref{thealgorithm}, in this contact structure, we have the following non-loose Legendrian knots $L_{n-1,\pm}^{i, k}$ for  $i\in \Z, k\in\N\cup\{0\}$ and $L_e$ such that 
\begin{gather*}
  \tb(L_{n-1,\pm}^{i,k})=i,\, \rot(L_{n-1,\pm}^{i,k})=\mp(i - 2n + 1) \text{ and } \tor(L_{n-1,\pm}^{i,k}) = k \\
  \tb(L_e)=2n-1,\, \rot(L_e)=0 \text{ and } \tor(L_e)=0.
\end{gather*}
We also have
\[
  L_{n-1,\pm}^{i,k}=S_\pm(L_{n-1,\pm}^{i-1,k}) \text{ and } S_\pm(L_e)=L_{n-1,\pm}^{2n-2,0},
\]
\[
  S_\mp(L_{n-1,\pm}^{i,k}) \text{ is loose.}
\]
See Figure~\ref{fig:LHT-mountain}. 

Now we apply Step $3$ of the algorithm and consider the other decorations. Since all these decorated paths are totally $2$-inconsistent, they will also contribute to non-loose Legendrian knots with $\tor > 0$. 
 
Now in $\xi_{n + l + 1}$, except for $l\not=n-1$ which was handled above, we have non-loose Legendrian knots $L_{l,\pm}^{i,k}$ with 
\[
  \tb(L_{l,\pm}^{i,k})=i,\, \rot(L_{l,\pm}^{i,k})=\mp(i - 2l-1) \text{ and } \tor(L_{l,\pm}^{i,k}) = k,
\]
\[
  S_\pm(L_{l,\pm}^{i,k})=L_{l,\pm}^{i-1,k} \text{ and } S_\mp(L_{l,\pm}^{i,k}) \text{ is loose.}
\]
See Figure~\ref{fig:LHT-mountain}.

In addition, when attaching half convex torsion to the complements of standard neighborhoods of the Legendrian knots above, we get non-loose Legendrian knots in the contact structures $\xi_{n-l}$ for $l\in \{-n+1, -n+3, ..., n-3, n-1\}$. In $\xi_{n-l}$, we have the non-loose Legendrian knots $L_{l,\pm}^{i,k+\scriptscriptstyle\frac 12}$ for $i\in \mathbb{Z}, k \in \mathbb{N} \cup \{0\}$ with 
\[
  \tb(L_{l,\pm}^{i,k+\scriptscriptstyle\frac 12})=i,\, \rot(L_{l,\pm}^{i,k+\scriptscriptstyle\frac 12})=\mp(i+2l+1) \text{ and } \tor(L_{l,\pm}^{i,k+\scriptscriptstyle\frac 12})= k + \frac12,
\]
\[
  S_\pm(L_{l,\pm}^{i,k+\scriptscriptstyle\frac 12})=L_{l,\pm}^{i-1,k+\scriptscriptstyle\frac 12} \text{ and } S_\mp(L_{l,\pm}^{i,k+\scriptscriptstyle\frac 12}) \text{ is loose}.
\]
See Figure~\ref{fig:LHT-mountain}.
\end{proof}

The classification of transverse $(2,-(2n+1))$-torus knots now follows.

\begin{proof}[Proof of Theorem~\ref{thm:(2,-(2n+1))-transverse}]
  Since the classification of transverse knots is equivalent to the classification of Legendrian knots up to negative stabilization \cite[Proof of Theorem~2.10]{EtnyreHonda01}, the theorem follows immediately from Theorem~\ref{thm:(2,-(2n+1))}. In particular, $T_l^k$ is the transverse push-off of $L_{l,-}^{0,k}$ in $\xi_{n+l+1}$ and  $T_l^{k+\scriptscriptstyle\frac 12}$ are transverse push-offs of  $L_{l,-}^{0,k+\scriptscriptstyle\frac 12}$ in $\xi_{n-l}$. 
\end{proof}

\section{Classification of non-loose \texorpdfstring{$(5,\pm8)$}{(5,8)}-torus knots}\label{fullclass58}

In this section, we apply the algorithm in Section~\ref{thealgorithm} to prove Theorem~\ref{leg58}, Theorem~\ref{leg5-8}, Theorem~\ref{trans58} and Theorem~\ref{trans5-8}. This will demonstrate the use of our algorithm in a more complicated setting than considered in the previous section. 

\subsection{Non-loose \texorpdfstring{$(5,8)$}{(5,8)}-torus knots}\label{sec58}
In this section we will classify non-loose Legendrian and transverse $(5,8)$-torus knots. We begin with the Legendrian representatives. 

\begin{proof}[Proof of Theorem~\ref{leg58}]
  The pair of paths representing $8/5$ are 
  \[
    P_1=\left\{\frac85, \frac32, 1\right\}, \text{ and } P_2=\left\{\frac85, \frac53, 2, \infty\right\},
  \] 
  and the continued fraction blocks in $P_1$ and $P_2$ are 
  \[
    A_1=\left\{\frac85, \frac32\right\},\, A_3=\left\{\frac32, 1\right\}, \text{ and } B_2=\left\{\frac85, \frac53, 2\right\},\, B_4=\left\{2, \infty\right\}. 
  \]
  Since we have
  \[
    -\frac{p}{p'} = -\frac53 = [-2, -3] \text{ and } -\frac{q}{q-q'} = -\frac83 = [-3, -3], 
  \]
  the linking matrix $M$ of the surgery diagram is
  \[
    M = \begin{bmatrix} -4 & -2 & -1 & -1 & -1 & -1 \\ -2 & -3 & -1 & -1 & -1 & -1 \\ -1 & -1 & -5 & -3 & -1 & -1 \\ -1 & -1 & -3 & -4 & -1 & -1 \\ -1 & -1 & -1 & -1 & 0 & -1 \\ -1 & -1 & -1 & -1 & -1 & 0 \end{bmatrix} 
  \]
  and from this we can compute $\sigma(X) = -4$ and $\chi(X) = 7$. The rotation vectors $\pm\mathbf{rot}$ gotten by different choices of stabilizations are  
  \[
    \begin{bmatrix} 0 \\ -1 \\ 1 \\ 0 \\ 0 \\ 0 \end{bmatrix},\, \begin{bmatrix} 0 \\ 1 \\ -1 \\ -2 \\ 0 \\ 0 \end{bmatrix},\, \begin{bmatrix} -2 \\ -1 \\ 1 \\ 2 \\ 0 \\ 0 \end{bmatrix},\, \begin{bmatrix} 2 \\ 1 \\ -3 \\ -2 \\ 0 \\ 0 \end{bmatrix},\, \begin{bmatrix} 0 \\ -1 \\ -1 \\ 0 \\ 0 \\ 0 \end{bmatrix},\, \begin{bmatrix} 0 \\ 1 \\ -3 \\ -2 \\ 0 \\ 0 \end{bmatrix},\, \begin{bmatrix} -2 \\ -1 \\ 1 \\ 0 \\ 0 \\ 0 \end{bmatrix},\, \begin{bmatrix} -2 \\ -1 \\ -1 \\ 0 \\ 0 \\ 0 \end{bmatrix},\, \begin{bmatrix} 0 \\ -1 \\ -1 \\ -2 \\ 0 \\ 0 \end{bmatrix},\, \begin{bmatrix} 0 \\ -1 \\ -3 \\ -2 \\ 0 \\ 0 \end{bmatrix},\, \begin{bmatrix} -2 \\ -1 \\ -1 \\ -2 \\ 0 \\ 0 \end{bmatrix},\,  \begin{bmatrix} -2 \\ -1 \\ -3 \\ -2 \\ 0 \\ 0 \end{bmatrix}.
  \]
  We can compute the $d_3$-invariant of each decoration of $(\overline{P_1}\,;\,P_2)$. The first $4$ rotation vectors give $\xi_1$ and correspond to the decorations on the path $(\overline{P_1}\,;\,P_2)$
  \[
    \pm(+,-\,;\,+,-,-),\, \pm(-,+\,;\,+,+,-),\, \pm(-,-\,;\,-,-,+),\, \pm(+,+\,;\,+,+,+)
  \]
  (note that $\pm(+,-\,;\,-,+,-)$ and $\pm(+,-\,;\,+,-,-)$ are the same). These decorations are, respectively, $2$, $3$, $4$-inconsistent and the last is totally consistent. 
  The next two rotation vectors give $\xi_{-1}$ and correspond to the decorations 
  \[
    \pm(+,-\,;\,+,-,+) \text{ and } \pm(-,+\,;\,+,+,+)
  \]
  (note that $\pm(+,-\,;\,+,-,+)$ and $\pm(+,-\,;\,-,+,+)$ are the same). 
  The first being $2$-inconsistent while the second is $3$-inconsistent. The remaining $6$ rotation vectors give distinct $d_3$-invariants and correspond to $2$-inconsistent pairs of decorated paths. In particular, in the order of the rotation vectors above we have 
  the decorations for $\xi_{-3}$ are 
  \[
    \pm(-,-\,;\,+,-,-)
  \]
  (note that $\pm(-,-\,;\,-,+,-)$ and $\pm(-,-\,;\,+,-,-)$ are the same). The decorations for $\xi_{-7}$ are
  \[
    \pm(-,-\,;\,+,-,+)
  \]
  (note that $\pm(-,-\,;\,-,+,+)$ and $\pm(-,-\,;\,+,-,+)$ are the same). The decorations for $\xi_{-9}$ are
  \[
    \pm(+,-\,;\,+,+,-).
  \]
  The decorations for $\xi_{-15}$ are
  \[
    \pm(+,-\,;\,+,+,+).
  \]
  The decorations for $\xi_{-19}$ are
  \[
    \pm(-,-\,;\,+,+,-).
  \]
  The decorations for $\xi_{-27}$ are
  \[
    \pm(-,-\,;\,+,+,+).
  \]

  We begin with the exceptional contact structure $\xi_1$. According to Step $2$ of the algorithm in Section~\ref{thealgorithm}, there are non-loose Legendrian knots $L_\pm^i$ for $i>29$ and $L^{29}$ such that 
  \[
    \tb(L_{\pm}^i)=i, \rot(L_\pm^i)=\mp(i-29), \tb(L^{29})=29, \rot(L^{29})=0,
  \]
  \[
    S_\pm(L_\pm^i)=L_\pm^{i-1}\,\text{ for $i>30,$ } \text{ and } S_\pm(L_\pm^{30})=L^{29}, 
  \]
  \[
    S_\mp(L_\pm^i) \text{ and } S_\pm(L^{29}) \text{ are loose.}
  \]

  To determine the other non-loose Legendrian $(5,8)$-torus knots in $\xi_1$, we note that $s_2=2$, $s_3=1$, $s_4=\infty$ and that $n_2=2$, $n_3=3$, $n_4=5$.
  Thus we have Legendrian knots $L_{k,\pm}^{40}$ for $k=2,3,4$ with $\tb=40$ and 
  \[
    \rot(L_{2,\pm}^{40})=\mp 9, \rot(L_{3,\pm}^{40})=\mp 7, \text{ and } \rot(L_{4,\pm}^{40})=\mp 3.
  \]
  Stabilizations of these Legendrian knots and the $L_\pm^i$ and $L^{29}$ with the same invariants will be equivalent and they will remain non-loose until they are stabilized outside the $\textsf{V}$ defined by the $L_\pm^i$ and $L^{29}$. All these knots have no convex torsion. 

  Now we consider the contact structure $\xi_{-1}$. First, notice that $\pm(+,-,+,-,+)$ are $2$ inconsistent and $\pm(-,+,+,+,+)$ are $3$-inconsistent, and they are compatible. Moreover, one easily computes that for the first decorations that $|R(P_1,P_2)|=19$ and for the second decorations that $|R(P_1,P_2)|=21$. Thus Step $3$ of the algorithm gives us non-loose Legendrian knots $L_\pm^i$ for $i\in \Z$ and $L_{2,\pm}^i$ for $i\leq 40$ such that
  \[
    \tb(L_\pm^i)=\tb(L_{2,\pm}^i)=i,\, \rot(L_\pm^i)=\mp (i-21),\, \rot(L_{2,\pm}^i)=\mp(i-19),
  \]
  \[
    S_\pm(L_\pm^i)=L_\pm^{i-1},\, \text{ and } S_\pm(L_{2,\pm}^i)=L_{2,\pm}^{i-1}, 
  \]
  \[
    S_\mp(L_{2,\pm}^i)=L_\pm^{i-1}, \text{ and } S_\mp(L_\pm^i) \text{ is loose.}
  \] 
  None of $L_+^i$ or $L_{2,+}^i$ is equivalent to $L_-^j$ or $L_{2,-}^j$ for any $i,j\in\mathbb{Z}$. All these Legendrian knots have no convex torsion. 

  In $\xi_{-3}$ and $\xi_{-7}$, we have following classification. First, notice that $\pm(-,-,+,-,-)$ and $\pm(-,-,+,-,+)$ are $2$-inconsistent, so in each contact structure the mountain range is an infinite $\textsf{X}$. Since they are not totally $2$-inconsistent, there are no non-loose Legendrian knots with $\tor>0$ in these contact structures. One can also check that $|R(P_1,P_2)| = 27$ and $|R(P_1,P_2)| = 37$, respectively. Thus there are Legendrian knots $L_\pm^i$ with $\tb(L_\pm^i)=i$ and $\tor(L_\pm^i)=0$ in each contact structure and
  \[
    \rot(L_\pm^i)=\begin{cases}
    \mp(i-13) & \text{ in } \xi_{-3},\\
    \mp(i-3)& \text{ in } \xi_{-7},
    \end{cases}
  \]
  \[
    S_\pm(L_\pm^i)=L_\pm^{i-1} \text{ and } S_\mp(L_\pm^i) \text{ is loose.}
  \]

  All the decorations in $\xi_{-9}$, $\xi_{-15}$, $\xi_{-19}$ and $\xi_{-27}$ are totally $2$-inconsistent, so these contact structures can have non-loose Legendrian representatives with convex torsion. One may compute that $|R(P_1,P_2)|$ for these four pairs of paths is $41, 51, 57,$ and $67$, respectively. So in each contact structure we have $L_\pm^{i,k}$ where $\tb(L_\pm^{i,k})=i$ and  
  \[
    \rot(L_\pm^{i,k})=\begin{cases}
    \mp(i+1) & \text{ in } \xi_{-9},\\
    \mp(i+11)& \text{ in } \xi_{-15},\\
    \mp(i+17)& \text{ in } \xi_{-19},\\
    \mp(i+27)& \text{ in } \xi_{-27},
    \end{cases}
  \]
  \[
    S_\pm(L_\pm^{i,k})=L_\pm^{i-1,k} \text{ and } S_\mp(L_\pm^{i,k}) \text{ is loose.}
  \]
  Moreover, $\tor(L_\pm^{i,k}) = k$ if it is not in $\xi_{-27}$, in that case we have 
  \[
    \tor(L_\pm^{i,k}) = \begin{cases}k \,& i>27,\\ k+\frac12 & i\leq 27.  \end{cases}
  \] 

  Finally we can add half convex torsion to these latter four contact structures. This results in the contact structures $\xi_{-8}$, $\xi_{-4}$, $\xi_{-2}$, and $\xi_0$, and we have non-loose Legendrian $(5,8)$-torus knots $L_{\pm}^{i,k+\scriptscriptstyle\frac 12}$ where $\tb(L_\pm^{i,k+\scriptscriptstyle\frac 12})=i$ and 
  \[
    \rot(L_\pm^{i,k+\scriptscriptstyle\frac 12})=\begin{cases}
    \mp(i-1) & \text{ in } \xi_{-8},\\
    \mp(i-11)& \text{ in } \xi_{-4},\\
    \mp(i-17)& \text{ in } \xi_{-2},\\
    \mp(i-27)& \text{ in } \xi_{0},
    \end{cases}
  \]
  \[
    S_\pm(L_\pm^{i,k+\scriptscriptstyle\frac 12})=L_\pm^{i-1,k+\scriptscriptstyle\frac 12} \text{ and } S_\mp(L_\pm^{i,k+\scriptscriptstyle\frac 12}) \text{ is loose.}
  \]
  Moreover, $\tor(L_\pm^{i,k+\scriptscriptstyle\frac 12}) = k + 1/2$ if it is not in $\xi_{0}$, and in that case we have
  \[
    \tor(L_\pm^{i,k+\scriptscriptstyle\frac 12}) = \begin{cases} k + \frac12 \,& i>27,\\ k+1 & i \leq 27. \end{cases}
  \]  
\end{proof}
We now turn to the transverse $(5,8)$-torus knots.

\begin{proof}[Proof of Theorem~\ref{trans58}]
  This theorem follows directly from Theorem~\ref{leg58} given that the classification of transverse knots is the same as the classification of Legendrian knots up to negative stabilization \cite[Proof of Theorem~2.10]{EtnyreHonda01}.  
\end{proof}

\subsection{Non-loose \texorpdfstring{$(5,-8)$}{(5,-8)}-torus knots}\label{sec58neg}
In this section we will classify non-loose Legendrian and transverse $(5,-8)$-torus knots. We begin with the Legendrian representatives. 

\begin{proof}[Proof of Theorem~\ref{leg5-8}]
  The pair of paths representing $-8/5$ is 
  \[
    P_1=\left\{-\frac85, -\frac53, -2\right\}, \text{ and } P_2=\left\{-\frac85, -\frac32, -1\right\}
  \] 
  and the continued fraction blocks in $P_1$ and $P_2$ are
  \[
    A_2 = \left\{-\frac85, -\frac53, -2\right\}, \text{ and } B_1 = \left\{-\frac85, -\frac32\right\},\,B_3 = \left\{-\frac32, -1\right\}.
  \] 
  Since we have
  \[
    -\frac{p}{p'} = -\frac52 = [-3, -2], \text{ and } -\frac{q}{q-q'} = -\frac85 = [-2, -3, -2].
  \]
  the linking matrix $M$ of the surgery diagram is 
  \[
    M = \begin{bmatrix} -4 & -3 & -1 & -1 & -1 & -1 & -1 \\ -3 & -4 & -1 & -1 & -1 & -1 & -1 \\ -1 & -1 & -4 & -3 & -2 & -1 & -1 \\ -1 & -1 & -3 & -4 & -2 & -1 & -1 \\ -1 & -1 & -2 & -2 & -3 & -1 & -1 \\ -1 & -1 & -1 & -1 & -1 & 0 & -1 \\ -1 & -1 & -1 & -1 & -1 & -1 & 0 \end{bmatrix}
  \]
  and from this we can compute $\sigma(X) = -3$ and $\chi(X) = 8$. Here, we list all rotation vectors depending on the choice of stabilizations:
  \[
    \pm\mathbf{rot} = \begin{bmatrix} 0 \\ 0 \\ 0 \\ 0 \\ -1 \\ 0 \\ 0 \end{bmatrix},\, \begin{bmatrix} -2 \\ -2 \\ 0 \\ 0 \\ 1 \\ 0 \\ 0 \end{bmatrix},\, \begin{bmatrix} 0 \\ 0 \\ -2 \\ -2 \\ -1 \\ 0 \\ 0\end{bmatrix},\,  \begin{bmatrix} -2 \\ -2 \\ 0 \\ 0 \\ -1 \\ 0 \\ 0 \end{bmatrix},\, \begin{bmatrix} -2 \\ -2 \\ -2 \\ -2 \\ -1 \\ 0 \\ 0 \end{bmatrix}.
  \]
  We can compute the $d_3$-invariant of each decoration of $(\overline{P_1}\,;\,P_2)$. The first $2$ rotation vectors above give $\xi_2$ and correspond to the decorations on the paths $(\overline{P_1}\,;\,P_2)$ given by 
  \[
    \pm(+,-\,;\,+,-) \text{ and } \pm(-,-\,;\,-,+)
  \]
  (note that $\pm(-,+\,;\,+,-)$ and $\pm(+,-\,;\,+,-)$ are the same). The decorations are, respectively, $2$ and $3$-inconsistent. The remaining rotation vectors give distinct $d_3$-invariants and correspond to $2$-inconsistent pairs of paths. The third rotation vector gives  $\xi_8$ and corresponds to the decorations 
  \[
    \pm(+,-\,;\,+,+)
  \]
  (note that $\pm(-,+\,;\,+,+)$ and $\pm(+,-\,;\,+,+)$ are the same). The fourth rotation vector gives  $\xi_{14}$ and corresponds to the decorations
  \[
    \pm(-,-\,;\,+,-).
  \]
  The last rotation vector gives  $\xi_{28}$ and corresponds to the decorations
  \[
    \pm(-,-\,;\,+,+).
  \]

  We begin with the exceptional contact structure corresponding to the path $P_1$ having all one sign and $P_2$ having the other. This is the contact structure $\xi_{28}$. In this contact structure, we have the following non-loose Legendrian knots $L_\pm^{i, k}, i\in \Z, k\in\N\cup\{0\}$ and $L_e$ such that 
  \[
    \tb(L_\pm^{i,k})=i,\, \rot(L_\pm^{i,k})=\mp(i - 27) \text{ and } \tor(L_\pm^{i,k}) = k,
  \]
  \[
    \tb(L_e)=27 \text{ and } \rot(L_e)=0,
  \]
  \[
    L_\pm^{i,k}=S_\pm(L_\pm^{i-1,k}),\, S_\pm(L_e)=L_\pm^{26,0} \text{ and } \tor(L_\pm^{i-1,k}) = k,
  \]
  \[
    S_\mp(L_\pm^{i,k}) \text{ is loose.}
  \]

  Now we consider $\xi_2$. First, notice that the decorations $\pm(+,-\,;\,+,-)$ are $2$-inconsistent and $\pm(-,-\,;\,-,+)$ are $3$-inconsistent, and they are compatible. Also, $|R(P_1,P_2)| = 15$ and $17$ for each decoration, respectively. In addition, the $2$-inconsistent decorations are not totally $2$-inconsistent, so none of the non-loose knots in $\xi_2$ will have convex torsion. Thus the algorithm yields the following non-loose Legendrian knots $L_\pm^i$ for $i\in \Z$ and $L_{2,\pm}^i$ for $i\leq -40$ such that 
  \[
    \tb(L_\pm^i)=\tb(L_{2,\pm}^i)=i, \text{ and } \tor(L_\pm^i)=\tor(L_{2,\pm}^i)=0,
  \]
  \[
    \rot(L_\pm^i)=\mp(i+25), \text{ and } \rot(L_{2,\pm}^i)=\mp(i+23),
  \]
  \[
    S_\pm(L_\pm^i)=L_\pm^{i-1}, \text{ and } S_\pm(L_{2,\pm}^i)=L_{2,\pm}^{i-1}, 
  \]
  \[
    S_\mp(L_{2,\pm}^i)=L_\pm^{i-1}, \text{ and } S_\mp(L_\pm^i) \text { is loose.}
  \]

  Next, we consider $\xi_8$. Because the decorations for $\xi_{8}$ are not totally $2$-inconsistent and $|R(P_1,P_2)| = 35$, we have non-loose Legendrian knots $L_\pm^i$ with 
  \[
    \tb(L_\pm^i)=i,\, \rot(L_\pm^i)=\mp(i+5) \text{ and } \tor(L_\pm^i) = 0,
  \]
  \[
    S_\pm(L_\pm^i)=L_\pm^{i-1} \text{ and } S_\mp(L_\pm^i) \text{ is loose.}
  \]

  The decorations for $\xi_{14}$ are totally $2$-inconsistent and $|R(P_1,P_2)| = 47$, so $\xi_{14}$ contains the non-loose Legendrian knots $L_\pm^{i,k}$ for $i\in \Z$ and $k\in \N\cup \{0\}$ satisfying 
  \[
    \tb(L_\pm^{i,k})=i,\, \rot(L_\pm^{i,k})=\mp(i-7) \text{ and } \tor(L_\pm^{i,k})=k,
  \]
  \[
    S_\pm(L_\pm^{i,k})=L_\pm^{i-1,k} \text{ and } S_\mp(L_\pm^{i,k}) \text{ is loose.}
  \]

  Finally, to the totally $2$-inconsistent pairs of paths (that are the ones for $\xi_{28}$ and $\xi_{14}$), we can also add half convex torsion. As described in Section~\ref{classwithtorsion}, this yields the contact structures $\xi_1$ and $\xi_{7}$. In each of these contact structures we have non-loose Legendrian knots $L_\pm^{i,k+\scriptscriptstyle\frac 12}$  satisfying
  \[
    \tb(L_\pm^{i,k+\scriptscriptstyle\frac 12})=i \text{ and } \tor(L_\pm^{i,k+\scriptscriptstyle\frac 12})=k + \frac12,
  \]
  \[
    \rot(L_\pm^{i,k+\scriptscriptstyle\frac 12})=\begin{cases} 
    \mp(i + 27) & \text{ in } \xi_1,\\
    \mp(i + 7)& \text{ in } \xi_7,
    \end{cases}
  \]
  \[
    S_\pm(L_\pm^{i,k+\scriptscriptstyle\frac 12})=L_\pm^{i-1,k+\scriptscriptstyle\frac 12} \text{ and } S_\mp(L_\pm^{i,k+\scriptscriptstyle\frac 12}) \text{ is loose.}
  \]
\end{proof}

We end with the classification of non-loose transverse $(5,-8)$-torus knots. 

\begin{proof}[Proof of Theorem~\ref{trans5-8}]
  This theorem follows directly from Theorem~\ref{leg5-8} given that the classification of transverse knots is the same as the classification of Legendrian knots up to negative stabilization \cite[Proof of Theorem~2.10]{EtnyreHonda01}.  
\end{proof}

\section{Tight contact structures on torus knot complements}\label{classificationoncomplement}

In this section, we investigate the tight contact structures on the complements of torus knots, which are Seifert fibered spaces over the disk with two singular fibers. The first classification results on such spaces were obtained in \cite{DingLiZhang13} and used in \cite{GeigesOnaran20a} to give their classification results for non-loose torus knots and expanded upon in \cite{Matkovic20Pre}. Thus, several of the classification results below were already known, but as observed in \cite[Section~4.2]{GeigesOnaran20a}, most non-loose torus knot complements are not covered by the results of \cite{DingLiZhang13}. We also note that in \cite{Ghiggini06b} some tight contact structures on these spaces were also constructed. 

This section is a long section that is the heart of our classification results. So we start the section by giving an overview of the whole strategy for the classification of non-loose torus knots. Hopefully, this will motivate the details carried out in the rest of this section. 

\subsection{Legendrian knots and their complements}\label{overview}
Before discussing torus knots, we consider a general approach to understanding Legendrian knots. 

Let $K$ be a null-homologous knot in a closed oriented connected $3$-manifold $M$. We can consider its complement $C_K=\overline{M-N_K}$ where $N_K$ is a tubular neighborhood of $K$. If $L$ is a Legendrian representative of $K$ in a contact manifold $(M,\xi)$, we can consider its standard neighborhood $N_L$. This is a neighborhood with convex boundary having two dividing curves of slope $\tb(L)$. Let $C_L=\overline{M-N_L}$ be the complement of the standard neighborhood of $L$. If $L$ is non-loose or $\xi$ is tight, then $\xi$ restricted to the complement $C_L$ is tight. So we have a tight contact structure on $C_K$ (which can be identified with $C_L$) with convex boundary having two dividing curves of slope $\tb(L)$. 

Reversing this construction, if we have any tight contact structure $\xi'$ on the complement  $C_K$ with convex boundary and dividing slope $n\in\Z$, then gluing in a solid torus to recover $M$ and extending the contact structure so it is tight on the solid torus will result in a Legendrian knot $L_{\xi'}$  that has this solid torus as its regular neighborhood and $\tb(L_{\xi'})=n$. The knot $L_{\xi'}$ will either be in a tight contact structure on $M$ or be a non-loose knot in an overtwisted contact structure on $M$.

Thus, we see that understanding non-loose Legendrian knots in the knot type $K$ in overtwisted contact structures on $M$ or Legendrian knots in the knot type $K$ in tight contact structures on $M$ is equivalent to understanding tight contact structures on $C_K$ with convex boundary having dividing slope $n\in\Z$. 

Moreover, if $\xi_L$ on $C_K$ is the contact structure on the complement of a standard neighborhood of $L$, then the contact structure on the complement of the stabilization $S_\pm(L)$ is obtained from $\xi_L$ by attaching a $\pm$-bypass to $\xi_L$ with dividing slopes $\tb(L)-1$ and $\tb(L)$. Similarly, $L$ will destabilize if inside of $(C_K,\xi_L)$ we can find a convex torus $T$ that is isotopic to the boundary of $C_K$ with two dividing curves of slope $\tb(L)+1$. Notice that $T$ separates $C_K$ into two pieces, one of which is a basic slice. The sign of this basic slice corresponds to the sign of the destabilization.  

With the above understood, we now turn to our strategy for understanding non-loose torus knots (and torus knots in the tight contact structure on $S^3$). For any $|q|>p>0$, the classification is built around two classes of knots:
\begin{enumerate}
\item non-loose knots with $\tb=pq$ and
\item an ``extra" non-destabilizable knots with $\tb=|pq|-|p|-|q|$.
\end{enumerate}

We first discuss the knots of Type~(1). Any such knot is constructed as follows. In Section~\ref{subsec:pathsinFG} we saw that given any decorated pair of paths $(P_1,P_2)$ that represent $q/p$, there is a contact structure $\xi_{P_1,P_2}$ on $S^3$ and inside of this contact structure there is a distinguished convex torus with dividing slope $q/p$. We let $L_{P_1,P_2}$ be a Legendrian divide on this convex torus. One may easily see that $\tb(L_{P_1,P_2})=pq$. All such knots come from this process. Below, we will show:
\begin{enumerate}[label=(\roman*)]
  \item If $(P_1,P_2)$ is not $2$-inconsistent, any convex torus in the complement of a standard neighborhood of $L_{P_1,P_2}$ that is isotopic to the boundary will have dividing slope $pq$. Thus $L_{P_1,P_2}$ does not destabilize.

  \item If $(P_1,P_2)$ is $2$-inconsistent then there is a basic slice embedding in the complement of the standard neighborhood of $L_{P_1,P_2}$ with one boundary component having dividing slope $pq$ and agreeing with the boundary of the standard neighborhood and the other having dividing slope $\infty$. This will prove that $L_{P_1,P_2}$ will destabilize infinitely often with one sign. 

  \item For any $(P_1,P_2)$, there exists a sign such that $L_{P_1,P_2}$ admits arbitrarily many stabilizations of that sign and remains non-loose, except in $\xi_1$ when $pq>0$. If we stabilize $L_{P_1,P_2}$ with the opposite sign, it immediately becomes loose.

  \item If $(P_1,P_2)$ is totally $2$-inconsistent, then the contact structure on the complement of a standard neighborhood of $L_{P_1,P_2}$ remains tight after adding arbitrary amounts of convex torsion (of one fixed sign). This will give us non-loose Legendrian knots with convex torsion in their complements.

  \item If $(P_1,P_2)$ is not totally $2$-inconsistent, then the contact structure on the complement of a standard neighborhood of $L_{P_1,P_2}$ will become overtwisted when adding a basic slice with dividing slopes $\infty$ and $pq$. Thus, these will not contribute to Legendrian knots with convex torsion in their complements. 

  \item We will not do this in this section, but in Section~\ref{nonloosetorusknots} we will see that if $(P_1,P_2)$ is $i$-inconsistent for $i>2$, then one can add basic slices of the opposite sign to the ones mentioned in Item~(iii). This will show that one can stabilize $L_{P_1,P_2}$ until it agrees with a stabilization of a Legendrian knot coming from a $2$-inconsistent decorated pair of paths. These will account for ``wings" and ``diamonds" discussed in the introduction. 
\end{enumerate}

We now turn to the knots of Type~(2). For each $(p,q)$, there will be one such knot. If $pq>0$, then this knot will be in the tight contact structure on $S^3$ and its stabilizations will give all Legendrian $(p,q)$-torus knots in $(S^3,\xi_{std})$. In addition, one may add convex torsion to these knots to obtain non-loose knots. It is this exceptional knot that is responsible for the anomalous behavior concerning Giroux torsion seen in Theorem~\ref{parity}. When $pq<0$, this extra knot will be in an overtwisted contact structure that also has Legendrian representatives coming from decorated pairs of paths. In addition, after a single stabilization, the extra Legendrian will become isotopic to a stabilization of the Legendrian knot coming from the decorated pairs of paths. In this section, we will establish the existence of the ``extra" knot complements, while in the next section the rest of the fact just mentioned will be established. 

\subsection{Torus knot complements}\label{knotcomp}
We start by building a topological model of the complement of a torus knot by following \cite[Section~3.1]{EtnyreLaFountainTosun12}. 
Let $F_1 \sqcup F_2$ be the Hopf link in $S^3$ and $V_1$ and $V_2$ neighborhoods of $F_1$ and $F_2$, respectively. Then there is a torus $T$ in the complement of the Hopf link that separates $S^3$ such that $S^3 = V_1 \cup (T \times [0,1]) \cup V_2$. We may take the $(p,q)$-torus knot $T_{p,q}$ to sit on $T \times \{1/2\}$. Let $N$ be a neighborhood of $T_{p,q}$ in $T\times [0,1]$ and $C=S^3\setminus N$. If we set $A'=(T\times \{1/2\})\cap C$, then we can consider $T \times [0,1]$ as the union of $N$ and $N(A')$, a neighborhood of the annulus $A'$ in $C$.  See the left drawing of Figure~\ref{fig:knot-complement}. In $N(A')$, we can find an annulus $A$ for which each of the boundary components is a $(p,q)$-curve, one on $\partial{V_1}$ and the other on $\partial{V_2}$. See the right drawing of Figure~\ref{fig:knot-complement}. Here we use the coordinate system on any torus parallel to $T$ coming from the Seifert framing of $V_1$ (so the meridian of $V_1$ has slope $\infty$ and the meridian of $V_2$ has slope $0$). We denote this coordinate system $\mathcal{F}_1$. Since we can also think of $N(A')$ as a neighborhood of $A$, we have the following model for $C$,
\[
  C = V_1 \cup N(A) \cup V_2.
\]
We notice that $C$ is a Seifert fibered space over a disk with two singular fibers. The regular fiber is a $(p,q)$-torus knot in $S^3 = C \cup N$ and this will be called a \dfn{vertical curve}, as will any curve isotopic to it. 

\begin{figure}[htbp]{\footnotesize
  \vspace{0.1cm}
  \begin{overpic}
  {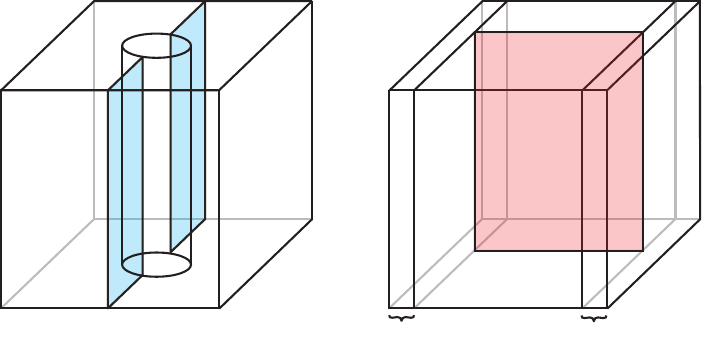}
  \put(35, 88){$A' \rightarrow$}
  \put(71, 32){$N$}
  \put(260, 88){$A$} 
  \put(189,-3){$V_1$}
  \put(282,-3){$V_2$}
  \end{overpic}}
  \caption{The complement of a neighborhood of a $(p,q)$-torus knot. Each cube is $T^2\times [0,1]$ (the top and bottom are identified, as are the front and back) with curves on the right and left collapsed to give $S^3$. On the left, we see that annulus $A'$ in $C$ that separates $C$ into two solid tori. On the right is the annulus $A$ going from $\partial V_1$ to $\partial V_2$.}
  \label{fig:knot-complement}
\end{figure}

We will use two different framing conventions for the torus $-\partial{C} (= \partial{N})$. One, which we denote $\mathcal{C}_1$, is the Seifert framing of $T_{p,q}$, and the other one, which we denote $\mathcal{C}_2$, comes from the torus $T$ on which $T_{p,q}$ sits. We can convert from the first framing to the second framing by using the coordinate change map
\[
  \psi = 
  \begin{pmatrix}
    1 && 0\\
    -pq && 1 
  \end{pmatrix},
\]
where $(a,b)^\intercal = a\lambda + b\mu$ and $\lambda$ is a longitude from the given framing and $\mu$ is a meridian of $N$.  

There is another convenient topological model for $C$. Notice that if we take the neighborhood $N$ of $T_{p,q}$  to be contained in the interior of $T^2\times [0,1]$, then $(T^2\times [0,1])\setminus N$ is $S^1\times P$ where $P$ is $D^2$ with two disjoint open disks removed. Now $C=V_1\cup (S^1\times P)\cup V_2$. We have $\partial (S^1\times P)=T_1\cup T_2\cup T_3$ where $T_i$ is identified with $\partial V_i$ for $i=1,2$ and $T_3=\partial C$. For each $T_i$, we can take coordinates so that $S^1\times\{pt\}$ has slope $0$ and $(\{\theta\}\times P)\cap T_i$ has slope $\infty$. Notice that on $T_3$ this framing agrees with the framing $\mathcal{C}_2$ that comes from the torus $T$. On the other $T_i$, this gives a coordinate system $\mathcal{F}_2$, and one can convert from the Seifert coordinates, $\mathcal{F}_1$, to this one by the map 
\[
  \phi =
  \begin{pmatrix}
    q' && -p'\\
    -q && p
  \end{pmatrix}
\]
where $q'/p'= (q/p)^c$, the largest rational number satisfying $pq'-p'q = 1$. See Section~\ref{fareygraph} for this notation. 

\subsection{Contact structures on \texorpdfstring{$C$}{C}} \label{contactonCsec}
To classify non-loose torus knots, we first classify tight contact structures on $C$ having convex boundary with dividing slope $n \in \mathbb{N}$ (here we use the Seifert coordinates $\mathcal{C}_1$ on $-\partial C=\partial N$) and without convex torsion. This will be done in Lemma~\ref{lem:>pq}, but we begin by considering tight contact structures on $C$ having convex boundary with any dividing slope $s$ and without convex torsion.

\subsubsection{Minimally twisting contact structures on \texorpdfstring{$C$}{C}}
We note that as we are thinking of $C$ as the complement of a knot, we will orient its boundary as $-\partial C$, in particular, the slopes of dividing curves of tori parallel to $-\partial C$ will change in a clockwise direction as they move into $C$.  We will always assume that $|q|>p>0$. Also, recall that we use the notation $\frac ab \bigcdot \frac cd$ to denote the quantity $ad-bc$. 
\begin{definition}
When considering a vertical curve $\gamma$ in $C$ it has a natural framing coming from the Seifert fibered structure (nearby regular fibers frame $\gamma$). If $\gamma$ is Legendrian, then we call the contact framing relative to this Seifert framing the \dfn{twisting of $\gamma$} or the \dfn{twisting along $\gamma$}. 
\end{definition}

Given the complement $C$ of a $(p,q)$-torus knot $T_{p,q}$ and a slope $s$ on the boundary of $C$, we denote by 
\begin{align*}
\Tight_i(C; s) = \left\{ \begin{aligned}
&\text{Tight contact structures on $C$ up to isotopy, with convex boundary having} \\
&\text{two dividing curves of slope $s$ and convex $i$-torsion for } i \in \tfrac{1}{2}\mathbb{N} \cup \{0\}.
\end{aligned} \right\}
\end{align*}
We begin with positive torus knots. 

\begin{lemma}\label{lem:thickening-positive}
  Consider the complement $C$ of a positive $(p,q)$-torus knot. For any rational number $s$ with $s > pq$ and any contact structure $\xi\in \Tight_0(C;s)$, there is a contact structure $\xi'\in \Tight_0(C;\infty)$ and a contact structure $\xi''\in \Tight^{min}(T^2\times [0,1]; s, \infty)$ such that $\xi$ is isotopic to $\xi'\cup \xi''$ under the natural identification $C\cong C\cup (T^2\times [0,1])$.  
\end{lemma}

\begin{proof}
 In this proof, slopes of curves on $\partial C$ are measured with respect to the Seifert framing. A `$0$-twisting curve' is one with $0$-twisting with respect to the framing induced by the Seifert fibered structure of $C$. Suppose  $(C,\xi)$ contains a $0$-twisting vertical Legendrian curve $\gamma$. Let $T$ be a torus parallel to the boundary of $C$ that contains this $0$-twisting vertical Legendrian curve. Recall that a vertical curve is one isotopic to a regular fiber of the Seifert fibered structure on $C$; hence, on $T$ it is a $pq$-curve. Since the twisting of $\gamma$ is $0$, we know that the torus framing and the contact framing agree. Thus, we may make the torus convex relative to $\gamma$, see \cite[Proposition~6.5]{Kanda98}. Since the contact framing and the torus framing agree, the dividing curves must be parallel to $\gamma$. Hence, $T$ can be perturbed into a convex torus with two dividing curves of slope $pq$. Thus $T$ cuts $C$ into two pieces, one being $T^2\times [0,1]$ with dividing slopes $s$ and $pq$. Note that this $T^2\times[0,1]$ is minimally twisting, since otherwise $\xi \notin \Tight_0(C;\infty)$.  Inside of $T^2\times[0,1]$ there will be a convex torus $T'$ with dividing slope $\infty$. This torus gives the claimed splitting of the contact structure $\xi$. 
  
  Thus, we will consider the case when $(C,\xi)$ does not contain a $0$-twisting vertical Legendrian curve. This implies that the dividing slope on any convex realization of $\partial V_1$ is less than $q/p$ and the dividing slope on any convex realization of $\partial V_2$ is bigger than $q/p$ or negative. Here, we are using $\mathcal{F}_1$, the Seifert coordinates of $V_1$.  

  Now we perturb $\partial{V_1}$ and $\partial{V_2}$ so that the ruling curves become $(p,q)$-curves, and perturb $A$ so that it becomes convex and one boundary component is a ruling curve of $\partial{V_1}$ and the other a ruling curve of $\partial{V_2}$. We change the coordinate system of $\partial V_1$ from $\mathcal{F}_1$ to $\mathcal{F}_2$ using $\phi$ from the previous section. Suppose that the dividing slopes of $V_1$ and $V_2$ are ${n_1/m_1}$ and ${n_2/m_2}$ respectively. Then $\phi$ maps $(m_i,n_i) \mapsto (q'm_i-p'n_i, -qm_i+pn_i)$ for $i=1,2$. If $|qm_1-pn_1| \neq |qm_2 -pn_2|$, then the twisting numbers of the boundary components of $A$ are different. Thus, by the Imbalance Principle \cite[Proposition~ 3.17]{Honda00a}, there is a bypass on $A$ along one of the boundary components, and we can thicken $V_1$ or $V_2$ using this bypass. Thus we can keep thickening the $V_i$ until $|qm_1-pn_1| = |qm_2 -pn_2|$.

  Recall if the dividing slope of $V_1$ is $r$, then when a bypass is attached to $\partial V_1$, the dividing curves of the resulting torus will have slope $r'$, which is clockwise of $r$ and the closest point to $q/p$ with an edge to $r$. Since we can start with $V_1$ as a standard neighborhood of a Legendrian knot with very negative Thurston-Bennequin invariant, then the possible dividing slopes on $\partial V_1$ will be
  
  \begin{itemize}
    \item $(1,n)$ for $n \leq \lfloor q/p\rfloor$,
    \item $(m_1,n_1)$ in the shortest path from $\lfloor q/p\rfloor$ clockwise to $q/p$ in the Farey graph.
  \end{itemize}

  Similarly, if $\partial V_2$ has dividing slope $s$, then when a bypass is attached to $\partial V_2$ the dividing curves will have slope $s'$ which is anti-clockwise of $s$ and the closest point to $q/p$ with an edge to $s$. So if we start with $V_2$ being a  standard neighborhood of a Legendrian knot with very negative $\tb$, then the possible dividing slopes on $\partial V_2$ will be
  
  \begin{itemize}
    \item $(m,1)$ for $m \leq 0$,
    \item $(m_2,n_2)$ in the shortest path from $q/p$ clockwise to $\infty$ in the Farey graph.
  \end{itemize}

  We denote $(q/p)^c$ by $q'/p'$ and $(q/p)^a$ by $q''/p''$ and recall that there is an edge in the Farey graph between all three points $(p,q), (p',q'),$ and $(p'',q'')$ (See Section~\ref{fareygraph}). 
  
  As explained above, we can keep thickening $V_1$ and $V_2$ until $|qm_1-pn_1| = |qm_2 -pn_2|$. In this situation, we can still thicken $V_1$ and $V_2$ if $A$ contains a boundary parallel dividing curve. Hence we may assume $A$ does not contain boundary parallel dividing curves. That is, the dividing curves on $A$ run from one boundary component to the other. There are three cases to consider:  

  \noindent
  {\bf Case 1: $\mathbf {|q - pn| = |p - qm|}$.} We begin by assuming that $q-pn=p-qm$. In \cite[Section~3.1]{EtnyreLaFountainTosun12}, the first author, LaFountain, and Tosun considered this case when $m, n \leq -1$ and $ {q - pn = p - qm}$, and showed that the solutions are $m = pk-1$ and $n = qk-1$ for $k \leq 0$. They determined in these cases that the dividing curves on the boundary of $C'=V_1 \cup N(A) \cup V_2$ had slope $s_k={(pq-p-q)/(1-k)}$ (please note that the slope convention in \cite{EtnyreLaFountainTosun12} is the reciprocal of the one used here). Thus $C\setminus C'$ is a thickened torus $T^2\times [0,1]$ with a contact structure that rotates from $s$ clockwise to $s_k$. In particular, there will be a torus in $T^2\times [0,1]$ with dividing slope $\infty$, and this torus provides the desired decomposition of $C$. Notice that when $k>0$ then $m$ is positive, the formula for $s_k$ still holds and we can still find a torus of slope $\infty$ in $C\setminus C'$.

  Now in the case that $q-pn=qm-p$, one can show that $n$ and $m$ will not satisfy the conditions in the bullet point above. 

  \noindent
  {\bf  Case 2:  $\mathbf{|qm_1 - pn_1| = |p - qm|}$ or $\mathbf{|q - pn| = |pn_2 - qm_2|}$.} We will show that there is no solution to these equations. Consider the first equation and the case when $qm_1 - pn_1 = p - qm$. Let $a_0 = \lfloor q/p\rfloor, a_1, \ldots, a_k=q/p$ be a shortest path in the Farey graph from $\lfloor q/p\rfloor,$ clockwise to $q/p$. As discussed in \cite[Remark~2.13]{ChakrabortyEtnyreMin20Pre} we see that $|\frac qp \bigcdot a_i|<|\frac qp\bigcdot a_{i-1}|$, and hence the maximum value for $qm_1 - pn_1$ is $q-a_0p$. If $m \leq -1$, we have $q-a_0p\geq qm_1 - pn_1 = p - qm\geq p+q$, but we are taking $p>0$ and so there is no solution to this equation. We will deal with the case $m = 0$ in Case 3. Now consider the case that $pn_1-qm_1=p-qm$. One may check that the left-hand side is negative, while the right-hand side is positive when $m\leq -1$ and again the case of $m=0$ will be handled in Case 3. 

  For the second equation in the case that $q - pn = pn_2 - qm_2$, we can similarly argue that the maximal value of $pn_2 - qm_2$ is $p$ and so for $n< \lfloor q/p\rfloor$ we have $p \geq  pn_2 - qm_2 = q - pn > p$; and this contradiction shows that there is no solution to the second equation. We will deal with the case of $n=\lfloor q/p\rfloor$ in Case 3. We may dispense with the case  $q - pn = qm_2-pn_2$ as above. 

  \noindent
  {\bf  Case 3. $\mathbf{|qm_1 - pn_1| = |pn_2 - qm_2|}$.}  We must have that $qm_1 - pn_1 = pn_2 - qm_2$, since both the right and left-hand sides are positive. We will show that the only solution is $(m_1,n_1) = (p'',q'')$ and $(m_2,n_2) = (p',q')$. Observe that if  $qm_1 - pn_1 = pn_2 - qm_2$, then ${(n_1+n_2)/(m_1+m_2)} = q/p$. We know from our choice of $(p',q')$ and $(p'',q'')$ that $q=q'+q''$ and $p = p'+p''$, so $qp''-pq''=pq'-qp'$. 

  As above let $a_1=\lfloor q/p\rfloor,a_1,\ldots, a_k=q/p$ be a shortest path in the Farey graph from $\lfloor q/p\rfloor$ clockwise to $q/p$ and similarly $b_0=q/p, \ldots b_l=\infty$ be the shortest path from $q/p$ clockwise to $\infty$. Notice that $q''/p''=a_{k-1}$ and $q'/p'=b_1$. The path $a_0,\ldots,a_{k-1}, a_k=b_0,b_1,\ldots , b_l$ can be shortened to a single jump from $\lfloor q/p\rfloor$ to $\infty$. In the process of shortening the path we first remove $q/p=a_k=b_0$ and the edges adjacent to it, and add the edge from $a_{k-1}$ to $b_1$, we will then remove one of $a_{k-1}$ or $b_1$, the edges adjacent to the removed vertex and add another edge in the Farey graph. Notice each vertex removed is the Farey sum of the two adjacent vertices. Thus the size of the numerator and denominator of the vertices $a_i$ and $b_j$ get smaller as we move out from $q/p$. Thus we see that if $(m_1,n_1) \neq (p'',q'')$ or $(m_2, n_2) \neq (p',q')$, then $m_1 < p''$ and $n_1 < q''$ or $m_2 < p'$ and $n_2 < q'$. Thus $m_1 + m_2 < p$ and $n_1 + n_2 < q$. But we observed above that ${(n_1 + n_2)/(m_1 + m_2)} = q/p$ which contradicts the fact that $\gcd(p,q) = 1$, so we must have that $(m_1,n_1) = (p'',q'')$ and $(m_2,n_2) = (p',q')$.

  When the dividing slopes of $V_1$ and $V_2$ are ${q''/p''}$ and ${q'/p'}$, there are two dividing curves on $A$, running from one boundary component to the other since we assume that there is no $0$-twisting vertical Legendrian curve in $C$. 

  Since $\phi$ maps $(p',q') \mapsto (0,1)$ and $(p'',q'') \mapsto (1,-1)$, we can compute the boundary slope of $C$ under the coordinates $\mathcal{C}_2$ after rounding the edges of $\partial{V_1} \cup \partial{N(A)} \cup \partial{V_2}$ to be 
  \[
    \frac{1}{0 - 1 + 1} = \frac{1}{0}.
  \]  
  Now we use the coordinate change map $\psi$ to compute the slope using Seifert coordinates $\mathcal{C}_1$ of $T_{p,q}$. Since $\psi^{-1}$ maps $(0,1) \mapsto (0,1)$, the slope of $\partial{V_1} \cup \partial{N(A)} \cup \partial{V_2}$ is $\infty$ and we have our desired splitting of the contact structure on $C$. 
\end{proof}

Now we will consider negative $(p,q)$-torus knot with $-q > p > 1$, but first recall that $r^a$ is defined in Lemma~\ref{adjacentvertices} as the farthest anti-clockwise point from $r$ in the Farey graph that is less than $r$ with an edge to $r$

\begin{lemma}\label{lem:thickening-negative}
  Consider the complement $C$ of a negative $(p,q)$-torus knot. Consider the slopes $e_k= {(|pq|-|p|-|q|)/k}$ with $k\geq1$ and $\gcd(|pq|-|p|-|q|,k) = 1$. For any $s>pq$ and $\xi\in \Tight_0(C;s)$ with $s\not\in(e_k^a,e_k]$, there is a contact structure $\xi'\in \Tight_0(C;\infty)$ and a contact structure $\xi''\in \Tight^{min}(T^2\times [0,1]; s, \infty)$ such that $\xi$ is isotopic to $\xi'\cup \xi''$ under the natural identification $C\cong C\cup (T^2\times [0,1])$. 

  If $s\in (e^a_k,e_k]$, then either $\xi$ is as above or there is a contact structure $\xi'_k\in \Tight_0(C; e_k)$ and a contact structure $\xi''_k\in \Tight^{min}(T^2\times [0,1]; s, e_k)$ such that $\xi$ is isotopic to $\xi'_k\cup \xi''_k$ under the natural identification $C\cong C\cup (T^2\times [0,1])$. Also, the contact structures $\xi'_k$ have the property that any convex torus in $C$ parallel to the boundary has dividing slope $e_k$. 
\end{lemma}

\begin{remark}\label{mostZthicken}
  Notice that if $s$ is an integer larger than $pq$ then $s\not\in (e_k^a,e_k]$ for any $k$ unless $s=|pq|-|p|-|q|$. Thus all such integer values of $s\not=|pq|-|p|-|q|$ are in the first case of the lemma. 
\end{remark}

\begin{proof}
  As in the positive case, if there is a $0$-twisting vertical Legendrian in $(C,\xi)$, then we have the desired splitting. So we assume there is no $0$-twisting vertical Legendrian in $(C,\xi)$. As in the proof of Lemma~\ref{lem:thickening-positive} we can take $V_1$ and $V_2$ to be neighborhoods of Legendrian knots with very negative Thurston-Bennequin invariant and then consider the annulus $A$ in $C$ between $V_1$ and $V_2$. We can make $A$ convex and if the dividing curves on $A$ do not all run from one boundary component to the other, we may attach a bypass to $\partial V_1$ or $\partial V_2$. As argued in Lemma~\ref{lem:thickening-positive}, we know that the dividing slope of $\partial V_1$ is less than $q/p$ and of the form

  \begin{itemize}
    \item $(1,n)$ for $n \leq \lfloor q/p \rfloor$,
    \item $(m_1,n_1)$ where $(m_1,n_1)$ is in the shortest path from $\lfloor q/p \rfloor$ clockwise to $q/p$ in the Farey graph.
  \end{itemize}

  The dividing curves of $\partial V_2$ are greater than $q/p$ but less than zero and of the form 

  \begin{itemize}
    \item $(m,1)$ for $m \leq -1$,
    \item $(m_2,n_2)$ where $(m_2,n_2)$ is in the shortest path from $q/p$ clockwise to $-1$ in the Farey graph.
  \end{itemize}

  When the dividing curves on $A$ all run from one boundary component to the other, we have three cases to consider. 

  \noindent
  {\bf  Case 1. $\mathbf{|pn - q| = |p - qm|}$.} We must have that $pn-q = p - qm$, since both the right and left-hand sides are negative. The solutions are $m = -pl + 1$ and $n = ql + 1$ for $l \geq 1$. Now change the coordinates using $\phi$ and we have $(m,1)\mapsto(-pq'l-p'+q', pql+p-q)$ and $(1,n) \mapsto (-p'ql-p'+q', pql+p-q)$. 
  Now we round the edges of $C'=V_1\cup N(A)\cup V_2$, and its dividing slope will be 
  \[
    \frac{pql+p-q}{(-pq'l - p' + q') - (-p'ql - p' + q') + 1} = \frac{pql+p-q}{-l+1}
  \] 
  and there will be $2\gcd(pql+p-q,l-1)$ dividing curves. 
  Now using $\psi^{-1}$ we can express the slope under the Seifert coordinates $\mathcal{C}_1$; $(-l+1, pql+p-q)\mapsto (-l+1, pq+p-q)$. Relabel $l-1$ as $k$, we have
  \[
    \frac{-pq-p+q}{k} = \frac{|pq|-|p|-|q|}{k}
  \]
  for $k \geq 0$. If $\gcd(|pq|-|p|-|q|,k)=1$, we can identify $\xi|_{C'}$ with an element of $\Tight_0(C;e_k)$. The argument in \cite[Lemma~3.3]{EtnyreLaFountainTosun12} shows that any convex torus in $C'$ that is parallel to $\partial C'$ is contact isotopic to $\partial C'$ when $\gcd(|pq|-|p|-|q|,k) = 1$. Now $C\setminus C'$ is $T^2\times [0,1]$ and $\xi$ on this thickened torus is an element of $\Tight^{min}(T^2\times [0,1]; s, e_k)$. If $s\not\in (e_k^a, e_k]$, then as shown in \cite[Proposition~3.10]{EtnyreLaFountainTosun12}, we may find bypasses for $\partial V_1$ or $\partial V_2$. We can continue to thicken these tori until we again have the annulus $A$ having no bypasses. In which case we will be in Case 1, 2, or 3 again, but if in Case 1, the new dividing slope $s_{k'}$ on $\partial C'$ will be larger than $e_k$. We may continue as above, until either $s\in (e_k^a, e_k]$, and we have a splitting as in the theorem, or we are in Case 2 or 3.
  
  If $\gcd(|pq|-|p|-|q|, k) > 1$, $\partial C'$ will have more than two dividing curves. Then as indicated in \cite[Remark~3.8]{EtnyreLaFountainTosun12} and since $\partial C$ has fewer dividing curves, we can thicken $V_1$ and $V_2$ and, as above, come back to Case~1, 2, or~3 and eventually find the desired splitting. 

  We note that the contact structure on $C$ is obtained by taking standard neighborhoods $V_1$ and $V_2$ of knots with Thurston-Bennequin invariants $-pk-1$ and $-qk-1$ and attaching an $I$ invariant neighborhood $N(A)$ of $A$. The tight contact structure will be denoted by $\xi'_k$ and has the property that any convex torus in $C$ parallel to the boundary has dividing slope $e_k$. This follows an argument identical to that of the proof of \cite[Lemma~3.3]{EtnyreLaFountainTosun12}, which we briefly recall for the convenience of the reader. This lemma is about positive torus knots, but the same argument works for negative torus knots. First, $C$ has a $|pq|$-fold cover $\widetilde C$ unwrapping the meridian $|pq|$-times which is diffeomorphic to $S^1\times \Sigma$ where $\Sigma$ is a Seifert surface for the $(p,q)$-torus knot. Using the similar argument in \cite[Lemma~3.3]{EtnyreLaFountainTosun12}, one can show that the pullback contact structure $\widetilde{\xi}_k$ on $\widetilde{C}$ can be isotoped so that the $S^1$-fibers are all Legendrian with twisting number $(pq(k + 1) + p - q)$ (relative to the framing on the fibers coming from the product structure $S^1\times \Sigma$). One can show that any Legendrian knot in $(\widetilde{C},\widetilde{\xi}_k)$ which is smoothly isotopic to a $S^1$-fiber of $S^1\times \Sigma$ must have twisting number less than or equal to $(pq(k + 1) + p - q)$, see \cite[Lemma~3.6]{Giroux01}. Now, suppose that there is a convex torus $T'$ in $C$ that has dividing slope larger than $e_k$. From our discussion above that means we can find another torus in $C$ that has dividing slope $e_{l}$, for some $l<k$, and hence $(C,\xi_{l})$ embeds in $(C,\xi_k)$ but then in the $pq$-fold cover of $(C,\xi_k)$ we have the $pq$-fold cover of $(C,\xi_l)$ and hence an $S^1$-fiber with twisting number $(pq(l+1)+p-q)$ which contradicts the claim above.

  \noindent{\bf  Case 2. $\mathbf{|q - pn| = |pn_2 - qm_2|}$ or $\mathbf{|qm_1 - pn_1| = |p - qm|}$.} As in the proof of Lemma~\ref{lem:thickening-positive}, we can see there are no solutions to these equations.    

  \noindent{\bf  Case 3. $\mathbf{|qm_1 - pn_1| = |-qm_2 + pn_2|}$.} As in the proof of Lemma~\ref{lem:thickening-positive}, we can see that the only solution is $(m_1,n_1) = (p'',q'')$ and $(m_2,n_2) = (p',q')$ and the boundary of $V_1\cup N(A) \cup  V_2$ has $2$ dividing curves with slope $\infty$ after edge rounding.
\end{proof}

The next lemma tells more about the exceptional contact structures $\xi'_k\in \Tight_0(C,e_k)$.

\begin{lemma}\label{extraisUT}
  The contact structures $\xi_k' \in \Tight_0(C,e_k)$ from Lemma~\ref{lem:thickening-negative} are universally tight and remain so after gluing any amount of convex torsion on $T^2\times [0,1]$ to $(C,\xi_k')$. 
\end{lemma}

\begin{proof}
  To prove $\xi_k'$ is universally tight, we recall the proof of \cite[Lemma~3.3]{EtnyreLaFountainTosun12}.  We continue our discussion of the $pq$-cover of $(C,\xi_k')$ from just before the start of Case~2. The bound on the twisting of any fiber of $(\widetilde{C}, \widetilde{\xi'_k})$ implies that the contact structure on $\widetilde C$ is tight (recall if a contact structure is overtwisted, there is no upper bound on the twisting number of any smooth knot type). Any further finite cover of $\widetilde C$ will be diffeomorphic to $S^1\times \Sigma'$ for some surface $\Sigma'$ and the $S^1$-fibers can all be made to be Legendrian with some fixed twisting number. Thus we see that they will also have to be tight. Since any finite cover of $(\widetilde{C},\widetilde{\xi}_k)$ is tight, we see that $(C,\xi_k')$ must be universally tight. 

  We now show that $(C,\xi_k')$ remains tight after one adds convex torsion. To this end, let $(C_+,\xi)$ be the contact structure on the complement of the unique Legendrian representative of the $(p,|q|)$-torus knot in $(S^3,\xi_{std})$ with maximal Thurston-Bennequin invariant. Then $(C_+,\xi)$ is universally tight and it remains tight after convex torsion is added, see \cite{EtnyreVelaVick10}. Moreover, its $|pq|$-fold cover $\widetilde C_+\cong S^1\times \Sigma$ is also foliated by Legendrian curves with twisting $-(p+|q|)$. There is a further cover of $\widetilde C_+$ and $\widetilde C$ such that they are both $S^1\times \Sigma$ and both foliated by Legendrian curves with the same twisting number. This implies that the contact structures on this common cover are isotopic (notice that one may cut $S^1\times \Sigma$ open along some annuli to get a contact structure on a solid torus with a unique contact structure, see the proof of \cite[Lemma~3.3]{EtnyreLaFountainTosun12}). Now if one adds a basic slice to $(C,\xi_k')$ and an appropriate contact structure on $T^2\times [0,1]$ to $(C_+,\xi)$, then when pulled back to the common cover, the contact structures will again be the same and hence $(C,\xi_k')$ with a basic slice attached will be universally tight. Since a basic slice is rotative, we can apply \cite[Theorem 4.7]{HondaKazezMatic02}) and say that adding convex torsion to $(C,\xi_k')$ results in a universally tight contact structure.
\end{proof}

\subsubsection{Minimally twisting contact structures on \texorpdfstring{$C$}{C} with integral dividing slope}
Here we consider the complements of non-loose Legendrian knots with a focus on those with no convex torsion in their complement, though our first result does apply to those as well. 

\begin{lemma}\label{lem:<pq}
  If $L$ is a non-loose Legendrian $(p,q)$-torus knot in an overtwisted contact structure on $S^3$ with $\tb(L)<pq$, then $L$ destabilizes. 
\end{lemma} 

\begin{proof}
By Lemma~\ref{seestab}, it is sufficient to find a convex torus parallel to the boundary of the knot complement with dividing slope one larger than the $\tb$ of the knot. The contact structure on the complement $C = S^3 \setminus N(L)$ is in $\Tight_i(C;\tb(L))$ for $i \geq 0$. If $i>0$, we can clearly destabilize $L$. If $i=0$, a destabilization of $L$ would correspond to finding $C' \subset C$, which is diffeomorphic to $C$ such that $\partial C'$ is convex with dividing slope $\tb(L)+1$. If there is a $0$-twisting vertical Legendrian in $C$, we can clearly find such $C'$. Assume there is no $0$-twisting vertical Legendrian in $C$. For negative torus knots, the argument in the proof of Lemma~\ref{lem:thickening-negative} shows that all contact structures in $\Tight_0(C;\tb(L))$ contain $(C',\xi')$ or $(C',\xi'_k)$, and the dividing slope of $\partial C'$ is $\infty$ or $e_k > pq$. For positive torus knots, the same argument in the proof of Lemma~\ref{lem:thickening-negative} works, and all contact structures in $\Tight_0(C;\tb(L))$ contain some $C'$, whose boundary slope is $\infty$ or $e_k < pq$. It is known that all of these contact structures with boundary slope $e_k$ exist in $(S^3,\xi_{std})$, see \cite[Section~3.1]{EtnyreLaFountainTosun12}, and hence are not of concern when studying non-loose Legendrian knots. 
\end{proof}

\begin{lemma}\label{inftyclass}
  Let $n(p,q)$ be the number of tight contact structures on $L(p,-q) \# L(q, -p)$ and $C$ the complement of the $(p,q)$-torus knot. Then we have 
  \[
    \left| \Tight_0(C;\infty)\right|=n(p,q).
  \]
\end{lemma}

\begin{proof}
We begin by proving $ \left| \Tight_0(C;\infty)\right|\leq n(p,q)$.  

Recall that our model for $C$ from Section~\ref{knotcomp} is $V_1\cup N(A)\cup V_2$, that is the union of two solid tori $V_1$ and $V_2$ and a neighborhood $N(A)=A\times [-1,1]$ of an annulus. In the proofs of Lemmas~\ref{lem:thickening-positive} and~\ref{lem:thickening-negative}, we saw that any contact structure $\xi\in \Tight_0(C,\infty)$ either 

  \begin{enumerate}
    \item contains a $0$-twisting vertical Legendrian curve, or
    \item $pq>0$ and there is a subset $C'$ of $C$ diffeomorphic to $C$ such that $C\setminus C'$ is $T^2\times [0,1]$, $\xi|_{C'}$ is in $\Tight_0(C, e_k)$ and $\xi|_{T^2\times [0,1]}$ is in $\Tight^{min}(T^2\times [0,1]; \infty, e_k)$ where $e_k={(pq-p-q)/k}$ and $\gcd(k, pq-p-q)=1$, and similarly for $pq<0$, or 
    \item $\xi|_{V_1}\in \Tight(V_1, (q/p)^a)$, $\xi|_{V_2} \in \Tight(V_2, (q/p)^c)$, and on $N(A)$ is an $I$-invariant neighborhood of a convex annulus $A$ with two dividing curves. 
  \end{enumerate} 

  Given $\xi\in \Tight_0(C;\infty)$, suppose that there exists a $0$-twisting vertical Legendrian curve in $C$. Then there is a copy $C'$ of $C$ sitting inside $C$ such that $C\setminus C'=T^2\times [0,1]$, $\xi|_{C'}\in \Tight_0(C;pq)$, and $\xi_{T^2\times [0,1]}\in \Tight^{min}(T^2\times[0,1]; \infty, pq)$. If all convex tori in $C'$ parallel to the boundary have dividing slope $pq$, then $\xi$ is overtwisted by Remark~\ref{nothickot} below. If $\xi|_{C'}$ contains tori parallel to the boundary with slopes different from $pq$, then by Lemmas~\ref{lem:thickening-positive} and~\ref{lem:thickening-negative}, it contains $C''$ with $\partial C''$ having dividing slope $\infty.$ Thus $C\setminus C''$ will be a convex torsion layer. That is, $\xi$ is not in $\Tight_0(C;\infty)$. So we know that $(C,\xi)$ does not contain a $0$-twisting vertical Legendrian curve.

  Now suppose that $\xi$ is of the form given in Item~(2), we will show that it is also of the form given in Item~(3) so that we will know that all $\xi \in \Tight_0(C;\infty)$ are of the form given in Item~(3). To this end we recall that there is an annulus $A'$ in $C$ with boundary slope $pq$ curves on $\partial C$, such that if we remove a neighborhood of $A'$ from $C$ we are left with $V_1$ and $V_2$, see Section~\ref{knotcomp}.  We can assume the ruling curves on $\partial C$ are $pq$ curves and $\partial A'$ consists of two ruling curves. Since the dividing slope of $\partial C$ is $\infty$ we know that each component of $\partial A'$ intersects the dividing set twice. Note that when we make $A'$ convex it must have two dividing curves that each run from one boundary component to the other, since if not we could Legendrian realize the core curve of $A'$ giving a $0$-twisting vertical Legendrian curve in $C$ but we have already ruled this out above. Now consider cutting $C$ along $A'$ and rounding corners. The result will be the solid tori $V_1$ and $V_2$. Notice that the $pq$ curve on $\partial C$ with respect to the coordinate system $\mathcal{C}_1$ will become the $q/p$ curve on $V_1$ and $V_2$ with respect to the coordinate system $\mathcal{F}_1$; moreover, we know the dividing curves will intersect the $q/p$ curves at most two times. In fact, it must be two times, since if not we again could Legendrian realize a $0$-twisting vertical Legendrian curve in $C$. Thus we know that the slope of the dividing curves has an edge in the Farey graph to the $(p,q)$-curve. We claim that we can assume that the slope on $\partial V_1$ is $(q/p)^a$ and on $\partial V_2$ is $(q/p)^c$.  To see this suppose that the slope is $r$ on $\partial V_1$, then inside of $V_1$ we can realize tori with dividing slopes larger than $-\infty$ and less than $r$. Thus there is a torus in $V_1$ parallel to the boundary with dividing slope $(q/p)^a$. Similarly we can take the dividing slope on $\partial V_2$ to be $(q/p)^c$. Now consider the convex annulus $A$ used in the proof of Lemma~\ref{lem:thickening-positive} that connects $\partial V_1$ to $\partial V_2$. We know that the dividing curves on $A$ must run from one boundary component to the other or else we could Legendrian realized the core curve of $A$ giving a $0$-twisting vertical Legendrian curve in $C$ which we have ruled out above. Thus as in the proof of Case~3 in the proof of Lemma~\ref{lem:thickening-positive} we see that the boundary of $C'=V_1\cup N(A)\cup V_2$ is convex with dividing slope $\infty$ and $C\setminus C'$ must be an $I$ invariant contact structure on $T^2\times I$. 

  So all the contact structures $\xi\in \Tight_0(C;\infty)$ are as described in Item~(3). Thus the number of contact structures in $\Tight_0(C;\infty)$ is bounded above by the number of tight contact structures on $(V_1; (q/p)^a)$ times the number of tight contact structures on $(V_2; (q/p)^c)$. Recall that when discussing slopes on $\partial V_1$ or $\partial V_2$ we are using longitude-meridian coordinates coming from $V_1$ and as such $V_1$ has lower meridian $-\infty$ and $V_2$ has upper meridian $0$ (see Section~\ref{oldclassification} for terminology). Thus as noted in Lemma~\ref{lensandtorus} we know 
  \[
    |\Tight(L_\infty^{q/p})|=|\Tight(S_\infty; (q/p)^a)| \text{ and } |\Tight(L_{q/p}^0)|=|\Tight(S^0;(q/p)^c)|. 
  \]
  From Section~\ref{oldclassification} we also know that $L_{q/p}^0=L(q,-p)$ and $L_\infty^{q/p}=L(p,-q)$. Thus we see that 
  \[
    n(p,q)=|\Tight(S_\infty; (q/p)^a)|\cdot|\Tight(S^0;(q/p)^c)|
  \] 
  is an upper bound on the number of contact structures in $\Tight_0(C;\infty)$. 

We now prove that $ \left| \Tight_0(C;\infty)\right|\geq n(p,q)$. To this end we first consider a topological Dehn filling. That is, if we glue a solid torus $S$ to $C$ so that the meridian goes to a $pq$ curve on $\partial C$ then the result is $L(p,-q)\#L(q,-p)$. The reason for this is that one can use two meridian disks in $S$ to cap off the boundary components of $A'$ to get a $2$-sphere that will divide the resulting manifold into two pieces. Once piece $V_1'$ is $V_1$ with a $2$-handle $h_1$ attached along the $q/p$ curve and the other $V_2'$ is $V_2$ with a $2$-handle attached $h_2$ to $V_2$ along the $q/p$ curve (notice $S$ minus the two meridian disks can be thought of as two $2$-handles, $h_1\cup h_2$). Now, if one glues a $3$-ball to each of these pieces, one gets the claimed lens spaces. 

  Now we return to contact geometry. Given a contact structure  $\eta\in \Tight(V_1, (q/p)^a)$ and $\eta' \in \Tight(V_2, (q/p)^c)$ we can connect them with an $I$ invariant contact structure on $N(A)=A\times I$ to get a contact structure on $C$. We first show that this is indeed tight. To this end, we Dehn fill $C$ with the solid torus $S$ as above. Notice that the solid torus glued in has meridional slope $pq$, and we are gluing it to a convex torus with dividing slope $\infty$. Thus, there is a unique tight contact structure on $S$ with these boundary conditions. We claim this contact structure on $L(p,-q)\#L(q,-p)$ is tight, and hence the one on $C$ must have been too. To see this, notice that each of the $2$-handles $h_1$ and $h_2$ can be thought of an $I$-invariant neighborhood of a meridian disk. Note that gluing a tight $3$-ball to $V_1'=V_1\cup h_1$ is the same as gluing a tight solid torus to $V_1$.  The meridian of this torus goes to the $q/p$ curve. In other words, one gets the tight contact structure on $L(p,-q)$ determined by the path in the Farey graph corresponding to the path determined by the contact structure $\eta$ on $V_1$. Similarly, we get a tight structure on $L(q,-p)$ determined by $\eta'$. Thus, the contact structure on the Dehn filling of $C$ is indeed a tight contact structure on $L(p,-q)\#L(q,-p)$. Colin \cite{Colin97} showed there is a one-to-one correspondence between tight contact structures on $M_1\# M_2$ and pairs of contact structures, one on $M_1$ and one on $M_2$. Thus, the lower bound is $n(p,q)$. 
\end{proof}

\begin{lemma}\label{lem:>pq}
  Let $n(p,q)$ be the number of tight contact structures on $L(p,-q) \# L(q, -p)$ then for the complement $C$ of the $(p,q)$-torus knot and any integer $n>pq$, we have 
  \[
    |\Tight_0(C; n)|=2n(p,q),
  \]
  unless $pq<0$ and $n = |pq| - |p| - |q|$, in which case we have 
  \[
    |\Tight_0(C;n)|=2n(p,q) + 1.
  \]
\end{lemma}

\begin{proof}
We begin with positive $(p,q)$-torus knots. 

\noindent{\bf An upper bound.}   
By Lemma~\ref{lem:thickening-positive} and the proof of Lemma~\ref{inftyclass} we know that for any $n > pq$ and $\xi\in \Tight_0(C; n)$ there is a subset $C'$ of $C$ diffeomorphic to $C$ such that $\xi|_{C'}\in \Tight_0(C;\infty)$ and $C\setminus C'$ is $T^2\times [0,1]$ and $\xi|_{T^2\times [0,1]}\in \Tight^{min}(T^2\times [0,1]; n,\infty)$. Since $\Tight^{min}(T^2\times [0,1]; n,\infty)$ contains exactly two elements, we see that an upper bound on $|\Tight_0(C; n)|$ is $2n(p,q)$. 

\noindent{\bf A lower bound.}   
  We now show that the obvious map 
  \[
    \Tight_0(C;\infty)\times  \Tight^{min}(T^2\times [1,2]; n,\infty) \to \Tight_0(C; n)
  \] 
  is injective and thus $|\Tight_0(C; n)|=2n(p,q)$. We first show that it is well defined, that is given $\eta \in \Tight_0(C;\infty)$ and $\zeta\in  \Tight^{min}(T^2\times [1,2]; n,\infty)$ the contact structure $\eta\cup \zeta$ on $C\cong C\cup (T^2\times[0,1])$ is in $\Tight_0(C;n)$. Suppose that $\zeta$ is a positive basic slice. Recall in the proof of Lemma~\ref{inftyclass} we showed that the result of gluing a solid torus $S$ to $(C,\eta)$ with meridional slope $pq$ and extending $\eta$ by the unique tight contact structure on $S$, was a tight contact structure on $L(p,-q)\# L(q,-p)$. Notice that $S$ is a standard neighborhood of a Legendrian knot $L$ in $L(p,-q)\# L(q,-p)$. If we stabilize $L$ positively and let $S'$ be a standard neighborhood of that knot, then $S\setminus S'$ is a positive basic slice $\zeta'\in \Tight^{min}(T^2\times [0,2]; pq+1,\infty)$. We may find a convex torus $T=T^2\times \{1\}$ in $S\setminus S'$ with two dividing curves of slope $n$ and this torus splits $\zeta'$ into a $\zeta$ on $T^2\times [1,2]$ and some other contact structure on $T^2\times[0,1]$. Thus we see that $\eta\cup \zeta$ is sitting in the contact Dehn filling of $(C,\eta)$ which is tight. Thus $\eta\cup \zeta$ is tight. To see that it is actually in $\Tight_0(C;n)$ we note that it cannot contain a convex torus parallel to the boundary with dividing slope $pq$ or else the contact structure on $L(p,-q)\# L(q,-p)$ would not be tight. If $\zeta$ were a negative basic slice then we could similarly show $\eta\cup \zeta$ is in $\Tight_0(C;n)$ by negatively stabilizing $L$. 

  To see that the above map is injective, we take $\eta$ and $\eta'$ in $\Tight_0(C;\infty)$ and distinct $\zeta$ and $\zeta'$ in $\Tight^{min}(T^2\times [1,2]; n,\infty)$. Let $\zeta$ be the positive basic slice and $\zeta'$ the negative one. We first note that $\eta\cup \zeta$ and $\eta'\cup \zeta'$ must be distinct for any two $\eta$ and $\eta'$ (even if they are the same). To do this, we consider Dehn filling $C\cup (T^2\times [1,2])$ by a solid torus $S''$ with meridional slope $pq$. When $n>pq+1$, there are two ways we can extend the contact structures over the added torus as a universally tight contact structure (there is a unique contact structure on $S''$ when $n=pq+1$, we will discuss this case next). One extension will have only positive basic slices $\xi$, and the other extension will have all negative basic slices $\xi'$. Suppose we choose the one with positive basic slices, then on $(T^2\times [0,1])\cup S''$ the contact structure $\zeta\cup \xi$ is tight and $\zeta'\cup \xi$ is overtwisted. But $\zeta\cup \xi$ is simply the unique tight contact structure on the solid torus $S$ from the previous paragraph. That is $\eta\cup \zeta\cup \xi$ is a tight contact structure on $C\cup (T^2\times [0,1])\cup S''=L(p,-q)\# L(q,-p)$ while $\eta'\cup \zeta'\cup \xi$ is overtwisted. Thus $\eta\cup\zeta$ and $\eta'\cup \zeta'$ are distinct. 

  We are left to see that if $\eta\cup \zeta$ is isotopic to $\eta'\cup \zeta$ then $\eta$ is isotopic to $\eta'$ (and similarly for $\zeta'$).  This is clear since gluing $S''$ with the contact structure $\xi$ will result in the same contact structures on $L(p,-q)\# L(q,-p)$ and these will also be the result of Dehn filling $(C,\eta)$ and $(C,\eta')$. From the proof of Lemma~\ref{inftyclass} we know that this implies $\eta$ is isotopic to $\eta'$. (Notice this argument also works when $n=pq+1$.)

  We must now see that $\eta\cup \zeta$ and $\eta\cup \zeta'$ are not isotopic when $n=pq+1$. This follows as the relative Euler classes of these two contact structures evaluate differently on a $\Sigma$ where $\Sigma$ is a minimal genus Seifert surface for the $(p,q)$-torus knot in $C$. 

  We now turn to the case of negative torus knots $-q>p>0$. 
  
  If $n\not=|pq|-|p|-|q|$ then the argument is identical to the above argument using Lemma~\ref{lem:thickening-negative}, the proof of Lemma~\ref{inftyclass}, and Remark~\ref{mostZthicken}. When $n=|pq|-|p|-|q|$ then we get the upper bound of $2n(p,q)+1$ since Lemma~\ref{lem:thickening-negative} gave an extra contact structure on this case. From Lemma~\ref{lem:thickening-negative} we know that this extra contact structure $\xi$ is tight and any convex torus parallel to the boundary of $C$ has dividing slope $n$. Thus $\xi \in \Tight_0(C;n)$ and it is distinct from the other contact structures since they all contain convex tori parallel to the boundary with dividing slope $\infty$. 
\end{proof}

Before our next result we establish some notation:
\begin{align*}
  \mathcal{L}(p,q)=\{&\text{Legendrian realizations of $(p,q)$-torus knot with $\tb=pq$ in some contact}\\ 
  & \text{structure on $S^3$ whose complement is tight without convex torsion}\}.
\end{align*}

\begin{lemma}\label{lem:=pq}
  Let $m(p,q) =  |\Tight(S^1\times D^2; p/q)| \cdot |\Tight(S^1\times D^2; q/p)|$. Then
  \[
    |\mathcal{L}(p,q)|=m(p,q). 
  \]
\end{lemma}

\begin{proof}
  We will try to understand contact structures in $\Tight_0(C;pq)$. While we will not quite get a classification of these, we will come close enough to identify all Legendrian realizations of $(p,q)$-torus knots. 

  Recall from Section~\ref{knotcomp} that we can take $V_1\cup (S^1\times P)\cup V_2$ as a model for $C$ and we will use the notation established there for $\partial (S^1\times P)=T_1\cup T_2\cup T_3$ and the coordinates on the $T_i$ discussed there as well. 

  Since any $\xi\in \Tight_0(C;pq)$ has dividing curves of slope $pq$ on $\partial C$, we know they are isotopic to $S^1\times \{pt\} \subset S^1\times P$. We can then use convex annuli between $\partial C$ and $\partial V_i$ to thicken the solid tori $V_i$ until they have dividing slope $q/p$. 

  The contact structure $\xi$ on $S^1\times P$ has dividing curves on all boundary components isotopic to the $S^1$-fibers. We can make the ruling curves have $\infty$ slope and arrange them for $\partial P=\{\theta\}\times \partial P$ to be ruling curves and then make $P$ convex. According to Lemma~\ref{all000}, there is a unique contact structure on $S^1\times P$ up to isotopy (not fixing the boundary point-wise). But notice that when $V_1$ and $V_2$ are glued back into $S^1\times P$, the fact that the isotopy did not fix $T_1$ or $T_2$ is irrelevant because the rotation of these $T_i$ can be extended to the interior by rotating the $V_i$'s. So the contact structure on $C$ is uniquely determined, up to isotopy (not fixing the boundary of $C$ point-wise) by the contact structures on $V_1$ and $V_2$. Now notice that when the neighborhood $N$ of the $(p,q)$-torus knot is glued to $\partial C$, we again do not need to be concerned about the fact that the classification above allowed $\partial C$ to move. 

  Thus when studying contact structures on $C$ that come from the complement of non-loose Legendrian $(p,q)$-torus knots, there is a unique tight contact structure on $S^1\times P$ that we need to consider and a model for it comes from taking a neighborhood $T^2\times [0,1]$ of a convex torus with two dividing curves of slope $q/p$ and then removing a neighborhood $N$ of a dividing curve on $T^2\times \{1/2\}$. Now gluing $V_1$ and $V_2$ into this model, we see that an upper bound on $|\mathcal{L}(p,q)|$ comes from the number of tight contact structures on $V_1$ and $V_2$. Since the dividing slope on $\partial V_1=\partial V_2$ is $q/p$ we see that $|\Tight(V_1)|=|\Tight(S_\infty;q/p)|=\Tight(S^1\times D^2; q/p)|$ and $|\Tight(V_2)|=|\Tight(S^0;q/p)|=|\Tight(S^1\times D^2; p/q)|$ (see Section~\ref{oldclassification} for notation about upper and lower meridians). This shows that $m(p,q)$ is the upper bound of $|\mathcal{L}(p,q)|$. Now we will describe all the possible elements in $\mathcal{L}(p,q)$ and find the lower bound. 

  \noindent{\bf Construction of contact structures in $\Tight_0(C;pq)$.}
  Consider a decorated pair of paths $(P_1,P_2)$ in the Farey graph representing $q/p$ (see Section~\ref{subsec:pathsinFG}). Recall that the union $P_1\cup P_2$ gives a contact structure $\xi_{P_1,P_2}$ on $S^3$ (it might or might not be the tight) and inside of it there is a convex torus $T$ with dividing slope $q/p$. Let $L_{P_1,P_2}$ be a Legendrian divide on $T$. 

  We now claim that all of these Legendrian knots are indeed in $\mathcal{L}(p,q)$, that is, the complement of $L_{P_1,P_2}$ is tight and has no convex torsion. This will follow by showing that Legendrian surgery on $L_{P_1,P_2}$ results in a tight contact structure. To see this, recall that any $L_{P_1,P_2}$ can be realized as a Legendrian knot shown in Figure~\ref{fig:torus-knots}, see Section~\ref{sec:FareytoSurgery}. The Legendrian surgery on $L_{P_1,P_2}$ cancels one of the $(+1)$-surgery components, and we obtain a tight contact structure.
  From Lemma~\ref{lem:rotdiff}, we know that $L_{P_1,P_2}$ and $L_{P'_1,P'_2}$ have different rotation numbers and thus are not Legendrian isotopic if $(P_1,P_2) \neq (P_1',P_2')$. This gives the lower bound of $|\mathcal{L}(p,q)|$ and completes the proof. 
\end{proof}
\begin{remark}
The non-looseness of the knots in Lemma~\ref{lem:=pq} is seen by showing that Legendrian surgery on them is tight, but we note that the proof of Proposition~\ref{wings} can also be applied here to show non-looseness using convex surfaces and state transition. 
\end{remark}

Referring to the construction of contact structures in $\Tight_0(C;pq)$ above, each one will be of the form $\xi_{P_1,P_2}$ where $P_1$ is a minimal path in the Farey graph that goes from $\infty$ clockwise to $q/p$ and decorated so that it represents an element in $\Tight(S_\infty;q/p)$ and $P_2$ is a minimal path from $q/p$ clockwise to $0$ and decorated so that it represents an element in $\Tight(S^0;q/p)$. Recall that we defined $i$-consistent and $i$-inconsistent decorated pairs of paths in Definition~\ref{defnconsistent}. We also recall that in the previous proof we defined the Legendrian $(p,q)$-torus knot $L_{P_1,P_2}$ as a Legendrian divide on the convex torus in $\xi_{P_1,P_2}$ that separates the solid torus described by $P_1$ from the solid torus described by $P_2$. 

\begin{lemma}\label{lem:thickenornot}
  Let $\xi'_{P_1,P_2} \in \Tight_0(C;pq)$ be the complement of $L_{P_1,P_2}$. In $\xi'_{P_1,P_2}$, all convex tori parallel to $\partial C$ have slope $pq$ if and only if $P_1$ and $P_2$ are $2$-consistent. Moreover, if $P_1$ and $P_2$ are $2$-inconsistent, then there is a subset $C'$ of $C$ that is isotopic to $C$ and $(\xi'_{P_1,P_2})|_{C'} \in \Tight_0(C;\infty)$.  
\end{lemma}

\begin{remark}\label{2inconsistentandinfty}
  Notice that by the proof of Lemma~\ref{inftyclass} any element in $\Tight_0(C;\infty)$ is described by a pair of paths $P_1', P_2'$ where $P_1'$ is a path from $(q/p)^a$  anti-clockwise to $\infty$ and $P_2'$ is a path from $(q/p)^c$ clockwise to $0$. One can extend $P_i'$ to start at $q/p$ and then add a $\pm$ sign to one and a $\mp$ to the other edges added. This will result in two different $2$-inconsistent paths corresponding to the element in $\Tight_0(C;\infty)$ that give two elements in $\Tight_0(C;pq)$ described by $2$-inconsistent paths. This observation, coupled with the above lemma implies that the number of $2$-inconsistent elements in $\Tight_0(C;pq)$ is $2|\Tight_0(C,\infty)|$. 
\end{remark}

\begin{proof}
  We first show that if $P_1$ and $P_2$ are $2$-inconsistent, then we can find $C'$ such that $(\xi'_{P_1,P_2})|_{C'} \in \Tight_0(C;\infty)$. As discussed in Section~\ref{knotcomp} we write $C$ as $V_1\cup (S^1\times P)\cup V_2$ and use the notation there for $\partial (S^1\times P)=T_1\cup T_2\cup T_3$ and the coordinates on the $T_i$ given there as well. 

  We can arrange that the $\partial V_i$ are convex and the slope of the dividing curves on $\partial V_1$ is $(q/p)^a$ and the slope on $\partial V_2$ is $(q/p)^c$. Using the coordinates on $T_i=\partial V_i$ coming from $S^1\times P$, we see that the slope of the dividing curves on $T_1, T_2,$ and $T_3$ is $-1, \infty,$ and $0$, respectively. Moreover, there are convex tori $T'_i$ in $S^1\times P$ parallel to $T_i$ with dividing slope $0$, for $i=1,2$. The thickened tori $N_i$ cobounded by $T_i$ and $T'_i$ are basic slices. Because the paths $P_1$ and $P_2$ are $2$-inconsistent, we know we can arrange these basic slices to have the same sign. To see this, notice that $P_i$ describes a contact structure on $V_i$ and we took the outermost basic slice and added it to $S^1\times P$. By assumption we can arrange that the outermost basic slices of $V_1$ and $V_2$ have opposite signs, but recall the contact structure on $V_2$ is described as an element of $\Tight(S^0;q/p)$. Thus this outermost basic slice, when oriented as a basic slice in $S^1\times P$, has the same sign as the one coming from $V_1$. Now Lemma~\ref{basicpants} says there is a convex annulus $A$ going from a $0$ slope ruling curve on $T_1$ to a ruling curve on $T_2$ that has dividing curves running from one boundary component to the other. As in the proof of Lemma~\ref{lem:thickening-positive}, if we round the corners of $T_1\cup T_2\cup N(A)$, the convex torus $T = \partial(T_1\cup T_2\cup N(A))$ will have dividing curves with $\infty$ slope. The torus $T$ is parallel to $T_3=\partial C$ and again as in the proof of Lemma~\ref{lem:thickening-positive}, the dividing slope of $T$ in Seifert coordinates is still $\infty$. Thus $T$ and $T_3$ cobound a thickened torus $N$ that is a basic slice. So $\xi_{P_1,P_2}$ can be written as the union of a contact structure in $\Tight_0(C;\infty)$ and a basic slice in $\Tight_0(T^2\times[0,1];pq,\infty)$. 

  By the proof of Lemma~\ref{inftyclass} we know that if there is a subset $C'$ of $C$ such that $(\xi'_{P_1,P_2})|_{C'} \in \Tight_0(C;\infty)$, then $(\xi'_{P_1,P_2})|_{C'}$ will be as described in the previous paragraph and hence $P_1$ and $P_2$ will be $2$-inconsistent. Now the proof of Lemma~\ref{lem:thickening-positive} shows that for $pq>0$, if $\xi_{P_1,P_2}$ is a contact structure in $\Tight_0(C;pq)$ and there is a torus $T$ parallel to $\partial C$ with slope different from $pq$ then there will be a subset $C'$ of $C$ such that $(\xi'_{P_1,P_2})|_{C'} \in \Tight_0(C;\infty)$ (since if $T$ has slope different from $pq$, then we can assume it has slope slightly larger than $pq$). Thus $P_1$ and $P_2$ are $2$-inconsistent. If $pq<0$, then Lemma~\ref{lem:thickening-negative} says  if $\xi'_{P_1,P_2}$ is a contact structure in $\Tight_0(C;pq)$ and there is a torus $T$ parallel to $\partial C$ with slope different from $pq$ then either there will be a subset $C'$ of $C$ such that $(\xi'_{P_1,P_2})|_{C'} \in \Tight_0(C;\infty)$, or there will be a subset $C''$ that is isotopic to $C$ and $\partial C''$ is convex with dividing slope $|pq|-|p|-|q|$. In the latter case, in $C\setminus C''$ we can find a convex torus $T$ with dividing slope $|pq|-|p|-|q|-1$ and Proposition~\ref{prop:extra} below says that inside the component of $C\setminus T$ that is not a thickened torus, we can find a subset $C'$ such that $(\xi'_{P_1,P_2})|_{C'} \in \Tight_0(C;\infty)$. So in either case, we see again that $P_1$ and $P_2$ are $2$-inconsistent. Thus we have that in $\xi_{P_1,P_2}$ there is a convex torus parallel to $\partial C$ have slope different from $pq$ if and only if $P_1$ and $P_2$ are $2$-inconsistent; moreover, in this case there is a subset $C'$ of $C$ such that $(\xi'_{P_1,P_2})|_{C'} \in \Tight_0(C;\infty)$.  
\end{proof}

We can also see that for almost all contact structures in $\Tight_0(C;\infty)$ any convex tori parallel to $\partial C$ will have dividing slope $\infty$. 

\begin{lemma}\label{infinitythicken}
  If $pq<0$, then for any contact structure in $\Tight_0(C;\infty)$, any convex torus parallel to $\partial C$ will have dividing slope $\infty$. If $pq>0$, the same is true for all contact structures but two. These contact structures are obtained from the complement of a standard neighborhood of Legendrian $(p,q)$-torus knots with maximal Thurston-Bennequin invariant in $(S^3,\xi_{std})$ by adding a $\pm$ basic slice in $\Tight_0(T^2\times[0,1];\infty,pq-p-q)$. 
\end{lemma}

For context we recall that there is a unique maximal Thurston-Bennequin invariant Legendrian $(p,q)$-torus knot when $pq>0$, see \cite{EtnyreHonda01}. 

\begin{proof}
  Suppose $pq<0$ and $\xi\in \Tight_0(C,\infty)$. If there is a convex torus in $(C,\xi)$ parallel to the boundary that has dividing slope $s$ different from $\infty$ then it separates off a $T^2\times [0,1]$ where the contact structure rotates from $\infty$ to $s$, so there will be a convex torus parallel to $\partial C$ with dividing slope $n$ for some negative integer $n$. Now according to Lemmas~\ref{lem:thickening-negative} and~\ref{lem:<pq} if $n\not = |pq|-|p|-|q|$, then there is another torus $T$ that is parallel to $\partial C$, has dividing slope $\infty$ or $pq$. If the slope of $T$ is $\infty$, then $T$ separates off $T^2\times [0,1]$ that is a half convex torsion layer, contradicting the fact that $\xi\in \Tight_0(C,\infty)$. 

  If the slope of $T$ is $pq$, then by Lemma~\ref{lem:thickenornot} it may happen that all convex tori further from $\partial C$ than $T$ will have dividing slope $pq$. This happens when $T$ separates $C$ into $T^2\times [0,1]$ and $C'$ where $\xi|_{C'}$ is in $\Tight_0(C;pq)$ and corresponds to a pair of $2$-consistent decorated paths $P_1, P_2$. In particular, $\xi|_C'$ is $\xi'_{P_1,P_2}$, which is the complement of $L_{P_1,P_2} \subset (S^3,\xi_{P_1,P_2})$. Thus $(C,\xi)$ is the union of a basic slice in $\Tight(T^2\times [0,1]; \infty, pq)$ and $(C',\xi'_{P_1,P_2})$ and by Lemma~\ref{lem:overtwisted}, $\xi$ is overtwisted, contradicting that $\xi \in \Tight_0(C,\infty)$.

  Now if $pq>0$ and $\xi\in \Tight_0(C,\infty)$, then we can argue as above, using Lemmas~\ref{lem:thickening-positive} and~\ref{lem:<pq}, and see that if there are convex tori in $C$ parallel to $\partial C$ with slope different from $\infty$ then there is one $T$ with slope $pq-p-q$. In this case $T$ separates $C$ into $T^2\times[0,1]$ and $C'$. The contact structure restricted to the former space is simply a bypass with dividing slopes $\infty$ and $pq-p-q$. According to Case 1 of the proof of Lemmas~\ref{lem:thickening-positive} and \cite[Section~3.1]{EtnyreLaFountainTosun12} we see that $\xi|_{C'}$ is simply the complement of the unique maximal Thurston-Bennequin Legendrian representative $L$ of the $(p,q)$-torus knot in $(S^3,\xi_{std})$. 

  We now know the only possibilities for $(C,\xi)$ to have a convex torus of slope different from $\infty$, but need to prove that the contact structure described above is indeed tight and so in $\Tight_0(C;\infty)$. But this was already proven in \cite[Lemma~3.1]{EtnyreVelaVick10}. 
\end{proof}

\subsubsection{Adding torsion to minimally twisting contact structures on \texorpdfstring{$C$}{C}}
The next two lemmas show that $2$-inconsistency also controls when certain basic slices can be attached to a contact structure in $\Tight_0(C;pq)$ while preserving tightness. Before proceeding, we suggest the reader consult the Definition~\ref{totallyinconsistentdef} concerning total $i$-consistency.

The first lemma demonstrates that if the pair of paths $(P_1, P_2)$ is not totally $2$-inconsistent, then $\Tight_n(C; \infty)$ is empty for all $n>0$. Moreover, if $(P_1, P_2)$ are $2$-consistent, then $\Tight_n(C; \infty)$ is empty for all $n\geq0$.

\begin{lemma}\label{lem:overtwisted}
  Let $\xi'_{P_1,P_2}$ be a contact structure in $\Tight_0(C;pq)$ such that $P_1$ and $P_2$ are not totally $2$-inconsistent. Gluing any basic slice in $\Tight_0(T^2\times [0,1]; \infty, pq)$ to $(C,\xi'_{P_1,P_2})$ will result in an overtwisted contact structure. 
\end{lemma}

\begin{remark}\label{nothickot}
  Note that Lemmas~\ref{lem:thickenornot} and~\ref{lem:overtwisted} show that if a contact structure in $\Tight_0(C;pq)$ has all convex tori parallel to $\partial C$ having dividing slope $pq$, then adding a basic slice in  $\Tight_0(T^2\times [0,1]; \infty, pq)$ to $C$ will yield an overtwisted contact structure. 
\end{remark}

\begin{proof}
  Let $\xi$ be the contact structure on $C$ resulting from gluing any basic slice in $\Tight_0(T^2\times [0,1]; \infty, pq)$ to $\xi_{P_1,P_2}$. We can decompose $C$ as two solid tori $V_1$ and $V_2$ and $S^1\times P$ as in Section~\ref{knotcomp} and arrange that the dividing slope of $\partial V_1=T_1, \partial V_2=T_2,$ and $\partial C= T_3$ is $(q/p)^a, (q/p)^c,$ and $\infty$, respectively. Using the coordinates on $T_i$ coming from $S^1\times P$, we see the dividing slope on $T_1, T_2,$ and $T_3$ is $-1, \infty,$ and $\infty$, respectively. Now there is a convex torus $T'_i$ in $S^1\times P$ parallel to $T_i$ that has dividing slope $0$. Let $N_i$ be the thickened tori that $T_i$ and $T'_i$ cobound and set $S^1\times P'=(S^1\times P)\setminus \cup_{i=1}^3 N_i$. Notice that $\xi|_{N_3}$ is the basic slice that was added to $\xi_{P_1,P_2}$. By the hypothesis that $P_1$ and $P_2$ are not totally $2$-inconsistent paths, we know that $N_1$ and $N_2$ can be taken to be basic slices with different signs, so one of them has the same sign as $N_3$ (recall, as in the proof of Lemma~\ref{lem:thickenornot}, a basic slice as seen in $V_2$ has the opposite sign when seen as a basic slice in $S^1\times P$). 

  Suppose that $N_1$ and $N_3$ have the same sign. By Lemma~\ref{basicpants} we know that there is a convex annulus $A$ with boundary $0$ sloped ruling curves on $T_1$ and $T_3$ and the dividing curves on $A$ run from one boundary component to the other. If we now cut $(S^1\times P')\cup N_1 \cup N_2 \cup N_3$ along this annulus we will obtain $T^2\times [0,1]$ with one boundary component $T_2$ and the other boundary component having dividing slope
  \[
    \frac{1}{1 + 0 - 1} = \infty.
  \]  
  Since $N_2$ contains a $0$-twisting vertical Legendrian curve, $T^2\times [0,1]$ is a half convex torsion layer and the union of this and $V_2$ (which is contained in $(C,\xi)$) is overtwisted. 

  Similarly if $N_2$ and $N_3$ have the same sign then we write $N_2$ as the union of two basic slices $N_2'$ and $N_2''$ where $N_2'$ had dividing slopes $-1$ and $0$. Now Lemma~\ref{basicpants} again implies the existence of a $0$ sloped convex annulus $A$ between $T_3$ and one boundary component of $N'_2$ with dividing slope $-1$. The dividing curves on A run from one boundary component to the other. As $N_2''$ has a similar annulus we can extend $A$ to an annulus in $(S^1\times P)\cup N_1 \cup N_2\cup N_3$ between $T_2$ to $T_3$. If we now cut $(S^1\times P')\cup N_1 \cup N_2 \cup N_3$ along this annulus we will obtain $T^2\times [0,1]$ with one boundary component $T_1$ and the other boundary component having dividing slope
  \[
    \frac{1}{0 + 0 - 1} = -1
  \]  
  Again, this $T^2\times [0,1]$ is a half convex torsion layer and the union of this and $V_1$ (which is contained in $(C,\xi)$) is overtwisted. 
\end{proof}

The next lemma considers the case when $(P_1,P_2)$ is a totally $2$-inconsistent pair of paths.

\begin{lemma}\label{lem:staytight}
    Let $(P_1,P_2)$ be a totally $2$-inconsistent pair of paths and $\xi'_{P_1,P_2} \in \Tight_0(C;pq)$. Gluing one basic slice in $(T^2\times [0,1]; \infty, pq)$ to $(C,\xi'_{P_1,P_2})$ will result in a tight contact structure $\xi$ on $C$ with $\infty$ dividing slope on $\partial C$, while gluing the other basic slice will result in an overtwisted contact structure. Moreover, adding any amount of convex torsion (with a correct co-orientation) to $(C,\xi)$ will result in a tight contact structure.  
\end{lemma}

\begin{proof}
  According to the proof of Lemma~\ref{lem:thickenornot}, $\xi'_{P_1,P_2}$ is the union of a contact structure $\xi' \in \Tight_0(C;\infty)$ and a $\pm$-basic slice $\eta_\pm \in \Tight_0(T^2\times[0,1];pq,\infty)$. Now let $\zeta_\pm$ be the $\pm$-basic slice in $\Tight_0(T^2\times[0,1];\infty,pq)$. Gluing $\zeta_\mp$ to $\eta_\pm$ is overtwisted, so we see that attaching one of the basic slices to $\xi'_{P_1,P_2}$ will result in an overtwisted contact structure. 

  Suppose $\xi'_{P_1,P_2}$ is the union of $\xi'$ and $\eta_-$. We will show that gluing $\zeta_-$ to $\xi'_{P_1,P_2}$ will result in a tight contact structure, and moreover, it will remain tight after adding any amount of convex torsion (with a correct co-orientation). Let $\xi$ be the contact structure resulting from this gluing. As usual, we consider $(C, \xi)$ as the union of two solid tori $V_1$ and $V_2$ and $S^1\times P$. We will use the notation from Section~\ref{knotcomp} except that we will use the coordinates on $T_1$ and $T_2$ coming from the longitude-meridian coordinates on $\partial V_1$ and the coordinates on $T_3$ coming from the longitude-meridian coordinates on $\partial C$. In particular, we can take the dividing curves on $T_1$ to have slope $(q/p)^a$, on $T_2$ to have slope $(q/p)^c$, and on $T_3$ to have slope $\infty$. Moreover, we can thicken $V_1$ and $V_2$ so that the slopes of $T_1$ and $T_2$ become $q/p$. Notice that in the coordinates on $\partial V_i, i=1,2,$ coming from $S^1\times P$ the slopes are $0$. Honda \cite[Lemma~5.1]{Honda00b} showed that this $P \times I$ is universally tight even after adding a universally tight rotative $T^2\times I$ layer to $T_i$ (with the correct sign). Thus, it is tight after we add convex torsion to $T_3$. After we make $P$ convex, we obtain the dividing curves on $P$ as shown in Figure~\ref{fig:pants} (see the proof of Lemma~\ref{lem:finite} to rule out other possibilities).

  Let $A'$ be an annulus in $S^1\times P$ separating $T_1$ from $T_2$, such that the boundary is two $pq$ slope ruling curves on $T_3$. See the first drawing of Figure~\ref{fig:knot-complement}. We can perturb $A'$ rel boundary so that it becomes convex and the intersection between $P$ and $A'$ is the Legendrian arc as shown in Figure~\ref{fig:pants}. It is not hard to see that $A'$ contains a $0$-twisting Legendrian curve (since $\xi$ has a half convex torsion layer in it), so we obtain two non-rotative layers $N_1$ and $N_2$ once we cut $S^1 \times P$ along $A'$.  By Theorem~\ref{thm:solid-torus}, $\widehat{V}_1 := V_1 \cup N_1$ and $\widehat{V}_2 := V_2\cup N_2$ are tight. By adding a sufficiently large amount of convex torsion to $T_3$, we can assume that $\widehat{V}_1$ and $\widehat{V}_2$ have a large number of dividing curves. 

  \begin{figure}[htbp] {\scriptsize
  \begin{overpic}{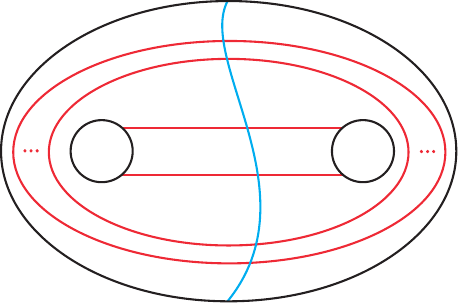}
    \put(55,50){$T_1$}
    \put(155,50){$T_2$}
    \put(20,15){$T_3$}
  \end{overpic}}
  \caption{Dividing curves on the pair of pants $P$. The blue arc is a Legendrian arc, which is the intersection of $A'$ and $P$.}
  \label{fig:pants}
  \end{figure}

  Suppose $(C,\xi)$ was not tight. Then we can smoothly isotope $A'$ so that it would be disjoint from an overtwisted disk. We can then use isotopy discretization (Theorem~\ref{thm:discretization}) to find a sequence of convex annuli $A_1=A',A_2,\ldots, A_n$ such that there is an overtwisted disk in the complement of $A_n$ and each $A_i$ is obtained from $A_{i-1}$ by a bypass attachment. We will inductively show that $A_i$ is contained in an $I$-invariant neighborhood $A \times [1,2]$, where $A$ is an annulus contact isotopic to $A'$ rel boundary. (The annulus $A$ used for each $A_i$ does depend on $i$, but they all have the claimed property, so we have dropped the dependence on $i$ from the notation for $A$ to keep the notation simple.) Then clearly $C \setminus A_i$ is tight for all $i$, and this contradiction will establish the tightness of $\xi$. 

  For each $i$, cutting $C$ along $A_{i}$ will result in two solid tori $\widehat{V}^{i}_1$ and $\widehat{V}^{i}_2$. Clearly, $A_1$ satisfies the inductive hypothesis. Now, assume that $A_{i-1}$ satisfies the inductive hypothesis.  We know $A_i$ is obtained by attaching a bypass to $A_{i-1}$. We assume the bypass was contained in $\widehat{V}^{i-1}_1$. The argument when the bypass is contained in $\widehat{V}^{i-1}_2$ is almost identical, except for one issue that is discussed in Remark~\ref{V2change}. By the inductive hypothesis, $A_{i-1}$ is contained in an $I$-invariant neighborhood $A \times [1,2]$. Since $A\times\{2\}$ is contact isotopic to $A'$ we know that $C\setminus (A\times \{2\})$ consists of two solid tori contact isotopic to $\widehat{V}_1$ and $\widehat{V}_2$, so we will think of $C\setminus (A\times \{2\})$ as $\widehat{V}_1 \cup \widehat{V}_2$. We know that $A_i$ is contained in $\widehat{V}_1$ and $A_i$ cuts $C$ into two solid tori, one of which is contained in $\widehat{V}_1$. Denote the boundary of this solid torus by $\widehat{T}_i$. We need to find an annulus $A_i'$ inside the solid torus bounded by $\widehat T_i$ and contact isotopic to $A\times \{2\}$. When we do this, $A_i'$ and $A\times \{2\}$ will co-bound a thickened annulus $A\times[1,2]$ containing $A_i$ on which the contact structure is $I$-invariant, thus completing the inductive step. 
  
  Since $A_{i-1}$ is contained in an $I$-invariant neighborhood $A \times [1,2]$, the number of dividing curves on $A_{i-1}$ is greater than or equal to the number of dividing curves on $A\times\{2\}$. To see this, we notice that if we glue the boundary components of $A\times [1,2]$ together using the identity diffeomorphisms, then we will obtain a tight contact structure on $T^3$. The convex torsion of this contact structure will agree with half the number of dividing curves on $A\times \{2\}$, say it is $n$. In the proof of the classification of tight contact structures on $T^3$, see \cite[Section~7]{Kanda97}, we know that if we take a curve $\gamma$ on $T^2\times \{2\}$ that together with a dividing curve forms a basis for the homology of the torus, then any Legendrian curve isotopic to $\gamma$ will have twisting bounded above by $-n$. However, if there was a torus $T$ in $A\times[1,2]$ with fewer than $2n$ dividing curves, then we could Legendrian realize a curve isotopic to $\gamma$ with twisting larger than $-n$.

Using \cite[Theorem~2.2 Part~(4)]{Honda00a}, one may easily check that if the number of dividing curves is the same, then $A_{i-1}$ and $A\times\{2\}$ are contact isotopic. We claim that the number of dividing curves on $A_i$ is also greater than or equal to the number of dividing curves on $A\times\{2\}$. To prove the claim, we only need to consider the case when the number of dividing curves on $A_{i-1}$ is the same as the number of dividing curves on $A\times \{2\}$ (since the number of dividing curves on $A_i$ can only differ from those on $A_{i-1}$ by $2$). In this case, $A_{i-1}$ is contact isotopic to $A\times\{2\}$ and the torus $\widehat T_{i-1}$ is contact isotopic to $\partial\widehat{V}_1$. Label the points of intersection of the dividing curves on $\widehat T_{i-1}$ (and $\partial\widehat{V}_1$) with $A_{i-1}$ by $p_i$ as shown in first drawing of Figure~\ref{fig:diskequiv} (notice that the annulus in the figure is equivalent to the horizontal annulus for $\widehat{V}_1 \setminus V_1$, which is the subsurface of $\{\theta\} \times P$ as shown in Figure~\ref{fig:pants}). To proceed, we note the following result, which will be established later. 

\begin{lemma}\label{lem:nobypasses}
  Suppose $(P_1,P_2)$ is totally $2$-inconsistent. Then there exist no effective bypasses in $\widehat{V}_1$, except for the ones for $p_1$.
\end{lemma} 
  
Notice that $A_{i-1}$ contains all dividing curves of $\widehat T_{i-1}$ except for the one containing $p_1$. By Lemma~\ref{lem:nobypasses}, there is no effective bypass on $A_{i-1}$, so any bypass cannot reduce the number of dividing curves on $A_{i-1}$ and this completes the claim. Thus the solid torus bounded by $\widehat T_i$ contains a solid torus $\overline{V}_1$ with two dividing curves of the same slope as the dividing curves on $\widehat T_i$ (which is the same slope as the dividing curves on $\partial\widehat{V}_1$). Now $\widehat{V}_1\setminus \overline{V}_1$ and $\widehat{V}_1\setminus V_1$ are both non-rotative outermost layers and so by Theorem~\ref{thm:non-rotative-contactomorphic} we know that $\overline{V}_1$ and $V_1$ are contactomorphic. 

Now consider a horizontal annulus $\widehat A$ for the thickened torus bounded by $\widehat T_i$ and $\partial \widehat{V}_1$ that is isotopic to a subsurface of $\{\theta\}\times P$ and then extend it to $\overline A$, a horizontal annulus for the thickened torus bounded by $\partial \overline{V}_1$ and $\partial \widehat{V}_1$. We can make these annuli convex with Legendrian boundary. By Theorem~\ref{thm:disk-equivalence} and Lemma~\ref{lem:nobypasses}, the horizontal annulus $\overline{A}$ must be disk equivalent to the horizontal annulus for $\widehat{V}_1 \setminus V_1$. Label the points where $\overline{A}$ intersects the dividing curves on $\partial \widehat{V}_1$ as shown in the first drawing of Figure~\ref{fig:diskequiv}.  

Notice that $\partial \widehat{V}_1$ and $\widehat T_i$ agree in a neighborhood of the dividing curve corresponding to $p_1$. From this we can see that the dividing set on $\widehat{A}$ does not have any bypasses for $\partial \widehat{V}_1$. This is because the only effective bypasses for $\partial \widehat{V}_1$ in $\widehat{V}_1$ are bypasses for $p_1$, but since $\partial {V}_1$ and $\widehat T_i$ agree near the dividing curve corresponding to $p_1$ there is no bypass here either. Notice this implies every dividing curve on $\widehat{A}$ that starts on $\widehat{A}\cap \partial \widehat{V}_1$ must end on $\widehat{A}\cap\widehat T_i$. Now inside of $\overline{A}\setminus \widehat{A}$, we can find a closed curve parallel to $\partial\overline{A}$ that intersects the dividing set the same number of times that $\overline{A}\cap \partial \widehat{V}_1$ does. Legendrian realize this curve and take a convex torus $T$ parallel to $\widehat T_i$ intersecting $\overline{A}$ in that curve. The thickened torus between $T$ and $\partial\widehat{V}_1$ is an $I$-invariant neighborhood of $\partial\widehat{V}_1$ containing $\widehat T_i$ (and hence $A_i$); note that this follows since both boundary components have the same number of dividing curves and the horizontal annulus connecting them has all the dividing curves running from one boundary component to the other. Again, since $\partial{V}_1$ and $T$ agree near the dividing curve corresponding to $p_1$, we can fix the surface $S = \partial{V}_1 \setminus A \times \{2\}$ in the $I$-invariant neighborhood. This implies that $A\times\{2\}$ and $T \setminus S$ are contact isotopic. This completes the inductive step.  
\end{proof} 

To prove Lemma~\ref{lem:nobypasses}, we need a preliminary observation. Lemma~\ref{lem:nobypasses} and the following lemma both apply to $\widehat{V}_2$ once we use the mirror image of the first drawing in Figure~\ref{fig:diskequiv}.

\begin{lemma}\label{lem:dijointbypasses}
  In $\widehat{V}_1$, the only points which can have an effective bypass are $p_1$ and $p_{k+1}$. If there exists an effective bypass for $p_{k+1}$, then there exist length $k-1$ nested bypasses for $p_1$ which are disjoint from a, possibly different, effective bypass for $p_{k+1}$.
\end{lemma}

\begin{proof}
  According to Theorem~\ref{thm:disk-equivalence}, any two horizontal annuli for $\widehat{V}_1$ are disk-equivalent. The first horizontal annulus for $\widehat{V}_1$ is shown in Figure~\ref{fig:pants} and is obtained by cutting the pair-of-pants along the blue Legendrian arc in the figure. So any other horizontal annulus for a non-rotative outer layer in $\widehat{V}_1$ will be disk equivalent to this one and all such possibilities are shown in Figure~\ref{fig:diskequiv}. If there is an effective bypass for $p_i$ when  $i \neq 1,k+1$, then we can attach it to reduce the number of dividing curves on $\partial \widehat{V}_1$. We can then find further bypasses to get a torus $T$ with just two dividing curves of slope $0$ (the same slope as the dividing cures on $\partial \widehat{V}_1$). Now we can find a horizontal annulus for the region between $\partial \widehat{V}_1$ and $T$ on which the original bypass sits. 
  This annulus cannot be disk-equivalent to the one for $\widehat{V}_1$. Thus there is no effective bypass for $p_i$ when $i\neq 1, k+1$. 

  Now, suppose there is an effective bypass for $p_{k+1}$. To find the claimed bypasses for $p_1$ we will construct a meridian disk for $\widehat{V}_1$ from the horizontal annulus $H$ shown in Figure~\ref{fig:pants}. To this end, recall that the horizontal annulus is for the non-rotative layer $\widehat{V}_1\setminus V_1$, and the meridian for $V_1$ is a $-p/q'$ sloped curve in the coordinates on $\partial P$ coming from $S^1\times P$ (we are using the change of coordinates $\phi$ from Section~\ref{knotcomp}). We want to extend this meridian to a meridian for $\widehat{V}_1$ by using copies of the horizontal annulus $H$ (which has slope $\infty$). Smoothly we can do this by taking $p$ copies of $H$, labeled $H_0, \ldots, H_{p-1}$, cutting each of them $p'$ times by an arc running from one boundary component to the other and then gluing one side of the cut on $H_i$ to the other side of the cut on $H_{i+1}$ (with indices taken modulo $p$). See the blue curves in Figure~\ref{fig:ma} (there $p=5$ and $p'=2$). This will give an annulus in $\widehat{V}_1\setminus V_1$ that can be extended to a meridian disk for $\widehat{V}_1$. 

  We now perform the construction, paying attention to the contact geometry. Consider the torus $T$ formed by taking the product of $S^1$ with the blue arc on the left of Figure~\ref{fig:diskequiv} and the black arc containing $p_{k+1}$ and connecting the end points of the blue arc. Notice that $\partial V_1$ intersected with $H$ is the inner circle in Figure~\ref{fig:diskequiv}, and we see that the region $R$ bounded by $T$ and $\partial V_1$ is a thickened torus with an $I$-invariant contact structure. We denote the outer boundary component by $\partial_oR$ and the inner boundary component by $\partial_i R$. We now take $p$ copies of the convex horizontal annulus $H$, shown on the left of Figure~\ref{fig:diskequiv}, and perform the construction above. The cutting and re-gluing take place in a small neighborhood of a $0$-twisting Legendrian arc in $H$ connecting $p_{k+1}$ to a point in the inner boundary component (here, by `$0$-twisting' we mean that the arc does not intersect any dividing curve on $H$). In particular, we will modify our annuli in the region $R$ where the contact structure is $I$-invariant. Let $H'=H\cap R$, and we can assume this is a horizontal annulus for $R$. In a neighborhood of the dividing curve corresponding to $p_{k+1}$, the characteristic foliation can be assumed to be as shown on the left of Figure~\ref{fig:ma} 
  \begin{figure}[htbp]
    {\scriptsize
    \begin{overpic}{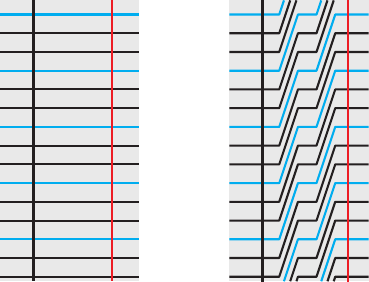}
    \end{overpic}}
    \caption{The gray is an annular neighborhood of the dividing curve corresponding to $p_{k+1}$, shown in red (the top and bottom of the rectangle are identified). On the left, we see the original annulus with a Legendrian divide shown vertically, and the horizontal curves are ruling curves. The blue curves are the intersection of the copies of $H'$ with the region (they are also ruling curves). The annulus can be isotoped relative to its boundary, so the foliation is as shown on the right-hand side. Notice that the blue will now be a single curve on $\partial_o R$.}
    \label{fig:ma}
  \end{figure}
and the copies of $H'$ intersecting this region are also shown. By a small perturbation of $\partial_i R$ we can arrange the characteristic foliation to be as shown on the right of Figure~\ref{fig:ma}. Notice that the green curve (when extended to the rest of $\partial_oR$ by arcs in copies of $H'\cap \partial_o R$) is a $p/p'$ curve $\gamma$ and can be assumed to agree with copies of $H'\cap \partial_o R$ outside of the region shown in the figure. We can make this same perturbation to the torus making up the inner boundary component of $R$. When we have done this to both boundary components of $R$, we can again assume that the contact structure is $I$-invariant on $R$ and $\gamma\times I$ will be an annulus. This annulus must have $2p$ dividing curves running from one boundary component to the other and agrees with copies of $H'$ outside a small neighborhood of the dividing curve corresponding to $p_{k+1}$ times $I$. The copies of $H'$ already had $2p$ dividing curves in the region where they agree with $\gamma\times I$, and so we can take these to be the dividing curves on $\gamma\times I$. Notice that on $\partial_o R$ we can glue copies of $H\setminus H'$ to $\gamma\times I$ to get an annulus $A_m$ for $\widehat{V}_1\setminus V_1$ that can be extended by a meridian disk $D'$ for $V_1$ to a meridian disk of $\widehat{V}_1$. By construction, the dividing set on $A_m$ will be obtained from $H$ by taking a $p$-fold cover, the dividing set on $D'$ will consist of $p$ arcs (of which we have no control). 

  A potential dividing set for the meridian disk $D=A_m\cup D'$ is shown in Figure~\ref{fig:largedisk}. 
  \begin{figure}[htbp]
  \vspace{0.2cm}
  {\scriptsize
  \begin{overpic}{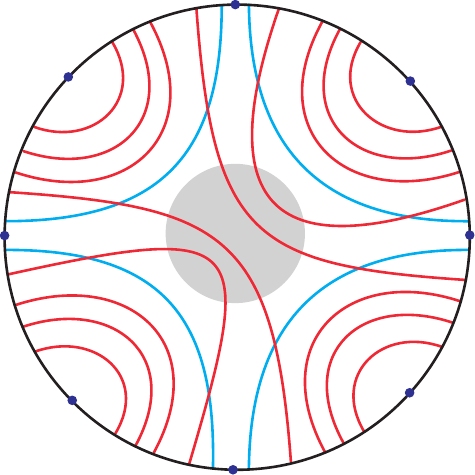}
    \put(20,200){$p_1$}
    \put(20,25){$p_1$}
    \put(200,28){$p_1$}
    \put(200,195){$p_1$}
    \put(-25, 115){$p_{k-1}$}
    \put(235, 115){$p_{k-1}$}
    \put(105, 234){$p_{k-1}$}
    \put(105, -6){$p_{k-1}$}
  \end{overpic}}
  \caption{Some possible meridian disk for $\widehat{V}_1$ constructed from horizontal annuli. The shaded center region is $D'$. Here, $p=k=4$.}
  \label{fig:largedisk}
  \end{figure}
  For some configurations of dividing curves on $D'$ we will immediately see a bypass for $p_{k+1}$. If this happens then we can take the bypass for $p_{k+1}$ to be on this meridian disk and then the length $(k-1)$ nested bypasses for $p_1$ can be found on a parallel copy of the meridian disk. 

  We now suppose that there is not a bypass for $p_{k+1}$ on $D$.  By hypothesis there is an effective bypass for $p_{k+1}$ and as we did in the first paragraph of this proof, we can assume that it is on some meridian disk for $\widehat{V}_1$ and by sliding the boundary of this disk along $\partial \widehat{V}_1$ we can assume that it has the same boundary on $D$. By the isotopy discretization (Theorem~\ref{thm:discretization}), we obtain a sequence of disks $D_0 = D, D_1, \ldots, D_{n-1}, D_n = D'$ such that $D_{i+1}$ can be obtained by a single bypass attachment from $D_i$. Then by the construction of $D$, there is an integer $j$ such that $D_j$ does not contain a non-nested bypass for $p_{k+1}$ and $D_{j+1}$ contains a non-nested bypass for $p_{k+1}$. 

  Consider the maximal nested bypasses for $p_1$. If one of these has length $k-1$, then the proof is done. This is because $D_{j+1}$ is obtained from a single bypass attachment from $D_j$, so we can make $D_j$ and $D_{j+1}$ disjoint after perturbation. Then the bypass for $p_{k+1}$ on $D_{j+1}$ is disjoint from the bypasses on $D_j$. 

  Recall that from the first part of the proof and our assumption that there is no bypass for $p_{k+1}$ on $D_j$, we know that the only boundary parallel dividing curves (by which we mean that it co-bounds with an arc in $\partial D_j$ a disk containing no other dividing curves) are bypasses for $p_1$. We also know from the construction of our meridian disk that there are $pk$ arcs in the dividing set and their end points are interlaced with the $p$ copies of each of the points $p_i$. We claim that the combinatorics of arcs on a disk as described above implies that one of the $p_1$ must have nested bypass of length $k-1$ and so the proof is complete. To see this suppose that all of the nested bypasses for $p_1$ on $D_j$ have length less than $k-1$. Consider a sub-disk $D'_j$ of $D_j$ such that the annulus $D_j\setminus D'_j$ contains all of the nested bypasses for the copies of $p_1$ and the rest of the dividing curves just run from one boundary component to the other. Since we are assuming all the nested bypasses for $p_1$ have length less than $k-1$, there will still be at least $2p$ arcs in $D'_j$. Thus there will be a boundary parallel arc $\gamma$ on $D'_j$. We can extend $\gamma$ across the annulus $D_j\setminus D'_j$ and it will either be boundary parallel on $D_j$ or not. If it is boundary parallel then it gives a bypass for $\partial \widehat{V}_1$ along some $p_i$ with $i\not=1$, which is a contradiction. But if the arc is not boundary parallel, then it will surround one of the nested bypasses for $p_1$, but this is also a contradiction since we said all the nested bypasses for $p_1$ were outside of $D'_j$. Thus $p_1$ must have a nested bypass of length at last $k-1$, as claimed. 
\end{proof}

We are now ready to prove Lemma~\ref{lem:nobypasses}. Since the proof is somewhat long, we give a brief outline. We first show that if there is an effective bypass for $p_{k+1}$ in $\widehat{V}_1$, then its sign must be opposite from the sign of the first continued fraction block of $V_1$. Then by Lemma~\ref{lem:dijointbypasses}, if such a bypass exists, then we can also find nested bypasses for $p_1$ in $\widehat{V}_1$, disjoint from the effective bypass for $p_{k+1}$. Attaching these nested bypasses produces a solid torus contactomorphic to $V_1$, while the effective bypass for $p_{k+1}$ remains effective for this solid torus. However, the sign of this bypass would be opposite to that of the first continued fraction block of $V_1$, which is impossible. Thus there is no effective bypass for $p_{k+1}$ in $\widehat{V}_1$. Combining this with Lemma~\ref{lem:dijointbypasses}, the only effective bypasses in $\widehat{V}_1$ are a bypass for $p_1$.

\begin{proof}[Proof of Lemma~\ref{lem:nobypasses}]
Suppose there is an effective bypass for $p_{k+1}$. We claim that the sign of this bypass is different from the sign of the first continued fraction block of $V_1$ (it has the single sign since $(P_1, P_2)$ is totally $2$-inconsistent). To see this, we first notice that the sign of the bypass will agree with the sign of the region on the convex surface $P$ between the two horizontal dividing curves in Figure~\ref{fig:pants}.  (To see this, we will always orient the attaching arc of such a bypass so that it passes the dividing curves corresponding to $p_{k+1}$ in the same direction that $\partial H$ does. Now the co-orientation on a contact structure orients the dividing curve corresponding to $p_{k+1}$. The sign of a bypass is determined by whether or not the orientation of the dividing curve agrees with the co-orientation on the bypass disk when oriented as above, and so is fixed by the contact structure and orientation on its attaching arc.) 
  
We claim this sign must be opposite to the sign in the first continued fraction block of $V_1$. We assume this is not the case and derive a contradiction. We can do this by replacing $V_1$ and $V_2$ in $C$ with the universally tight solid tori whose sign of the first continued fraction block is the same as $V_1$ and $V_2$, respectively. By Lemma~\ref{standardstructures}, this is contactomorphic (possibly co-orientation reversing) to the complement of the binding of an open book supported by $(p,q)$-torus knot with a convex torsion layer added. By Colin's gluing theorem (originally proven in \cite{Colin99}, but as used here see \cite[Theorem~4.7]{HondaKazezMatic02}), the contact structure is universally tight even after adding any amount of convex torsion. Inside $V_1$ we have another solid torus $V'_1$ with dividing slope $(q/p)^a$, and if we extend the convex pair of pants $P$ across an annulus running between $\partial V_1$ and $\partial V_1'$, then we will see a bypass on this annulus with sign given by the sign of the bypasses in the first continued fraction block of $V_1$. Since we are assuming that it is the same sign as the region between the horizontal lines in Figure~\ref{fig:pants}, we see that it will give a bypass for $V_2$. Attaching this bypass will result in a torus about $V_2$ with dividing slope $(q/p)^a$. Since there are vertical ruling curves on $A'$ disjoint from the attached bypass, we can thicken this torus further to have slope $q/p$. This gives a Giroux torsion layer in a solid torus, so the contact structure would be overtwisted, but we know that is not the case. So the sign of the region bounded by the horizontal dividing curves in Figure~\ref{fig:pants} is opposite from the signs of the first continued fraction block of $V_1$. 

  We now return to the setting where $V_1$ and $V_2$ are given by any decorated pair of paths that is totally $2$-inconsistent. We still know that the sign of the region bounded by the horizontal dividing curves in Figure~\ref{fig:pants} is opposite that of the signs in the first continued fraction block of $V_1$. Thus any effective bypass for $p_{k+1}$ will have sign opposite as well. 
  
  From Lemma~\ref{lem:dijointbypasses}, we can find an effective bypass for $p_{k+1}$ disjoint from the length $k-1$ nested bypasses for $p_1$. Attach these nested bypasses to obtain a solid torus with two dividing curves inside of $\widehat{V}_1$, which is contactomorphic to $V_1$ by Theorem~\ref{thm:non-rotative-contactomorphic}. Thus we just call this solid torus $V_1$. While attaching the nested bypasses, we never modify the dividing curve on $\partial\widehat{V}_1$ passing through $p_{k+1}$, so the bypass for $p_{k+1}$ is also effective for $V_1$. However, the sign of the effective bypass for $p_{k+1}$ is different from the sign of the first continued fraction block of $V_1$. This is impossible since the sign of the first continued fraction block is determined by the sign of effective bypasses in a solid torus. This contradiction implies that there is no effective bypass for $p_{k+1}$ in $\widehat{V}_1$. Combining with Lemma~\ref{lem:dijointbypasses}, the only effective bypasses in $\widehat{V}_1$ are a bypass for $p_1$.
\end{proof}

\begin{remark}\label{V2change}
  Lemmas~\ref{lem:nobypasses} and~\ref{lem:dijointbypasses} also hold for $\widehat{V}_2$ and the proofs are the same except for a part of the proof of Lemma~\ref{lem:dijointbypasses}. Specifically, in that proof we used the fact that we could find the bypass for $V_1\setminus V_1'$ on $P$. This is because in the coordinates on $\partial V_1$ coming from $S^1\times P$ we see that the dividing slope on $\partial V_1$ is $0$ and on $V_1'$ is $\infty$. Thus the annulus in $V_1\setminus V_1'$ that extends $P$, which has slope $\infty$, will contain a bypass for $\partial V_1$. When we consider $\widehat{V}_2$ the relevant slopes on $\partial V_2$ and $\partial V_2'$ are $0$ and $-1$, respectively. Thus we cannot find a bypass for $\partial V_2$ on an annulus of slope $\infty$. To proceed in this case we need to change the section of $S^1\times P$ that we are using. Specifically, if we take $\{\theta\}\times P$ and cut it along an arc connecting $\partial V_1 \cap P$ to $\partial V_2 \cap P$, we can then push one side of the cut along the $S^1$ fibers until it returns and is glued to the other side of the cut. Notice that if one pushed in the correct direction, then the slope of this new section, call it $P'$, on $\partial V_2$ is $-1$. Now running the whole argument with $S^1\times P'$ instead of $S^1\times P$ will prove the lemma for $\widehat{V}_2$.
\end{remark}

We now turn to contact structures on $C$ with convex torsion. 

\begin{lemma}\label{postorsion}
  For any $(p,q)$-torus knot and integer $k$ we have that 
  \[
    |\Tight_l(C;k)|=2|(a_1+1)\cdots(a_{m-1}+1)||(b_1+1)\cdots (b_{n-1}+1)|
  \]
  for any $l\in \frac 12 \N$, where the $a_i$ and $b_i$ are defined in Section~\ref{introgeneral}. This is the same as the number of totally $2$-inconsistent pairs of paths representing $q/p$. 
\end{lemma}

\begin{proof}
  We begin by considering $\Tight_l(C;pq)$.
  Given a contact structure $\xi\in \Tight_l(C;pq)$, there is an embedding of a convex $l$-torsion layer, i.e. there is an embedding of $T^2\times [0,1]$ such that $T^2\times \{0\}=\partial C$ and $\xi$ restricted to $T^2\times [0,1]$ is a convex $l$-torsion layer. Let $C'=C\setminus T^2\times [0,1]$ and $\xi'$ be $\xi$ restricted to $C'$. Notice that $(C',\xi')$ has no convex torsion, otherwise $(C,\xi)$ would have torsion larger than $l$. Thus $\xi'\in \Tight_0(C;pq)$ and $C\setminus C'$ is a convex $l$-torsion layer. By Lemma~\ref{lem:overtwisted}, the pair of paths describing $\xi'$ must be totally $2$-inconsistent.  Theorem~\ref{twolayers} says there are two possibilities for the contact structure on $T^2\times[0,1]$, but by Lemma~\ref{lem:staytight} we see that only one can possibly result in $\xi$ being tight. Thus for every element in $\xi\in \Tight_l(C;pq)$ there is a unique element $\xi'\in \Tight_0(C;pq)$ corresponding to a totally $2$-inconsistent pair of paths and a unique convex $l$-torsion layer on $T^2\times [0,1]$. Moreover, given an element in $\Tight_0(C;pq)$ described by a totally $2$-inconsistent pair of paths there is a unique convex $l$-torsion layer on $T^2\times [0,1]$ that can be glued to it to give a tight contact structure.

  We will prove the following lemmas later, but finish the proof assuming they are true. 

\begin{lemma}\label{ClaimA}
  If $\xi \in \Tight_l(C;pq)$, then $\tor(\xi) = l$, i.e., there is no embedding of a convex $l'$-torsion layer for any $l'>l$.  
\end{lemma}

  Thus from above we see that $|\Tight_l(C;pq)|$ is bounded above by the number of totally $2$-inconsistent paths describing $q/p$. 

\begin{lemma}\label{ClaimB}
 Adding convex $l$-torsion to two distinct elements of $\Tight_0(C;pq)$ results in distinct contact structures.
 \end{lemma}

  From this lemma, we now know that $|\Tight_l(C;pq)|$ is the number of totally $2$-inconsistent paths describing $q/p$. 

  To see that the number of such paths is given by the formula in the lemma we note that the first continued fraction block in $P_1$ must be all one sign and the first in $P_2$ must be of the opposite sign. Thus there are $2$ possibilities for these two continued fraction blocks. The number of possible decorations on the remainder of $P_1$ is $|(a_1+1)\cdots(a_{m-1}+1)|$ and on $P_2$ is $|(b_1+1)\cdots (b_{n-1}+1)|$.

  Now given $k>pq$ we know by Lemma~\ref{lem:>pq} and Remark~\ref{2inconsistentandinfty} that for each $\xi\in \Tight_0(C;k)$ there is a unique $2$-inconsistent path describing a contact structure $\xi'\in \Tight_0(C;pq)$ such that $\xi$ embeds in $(C,\xi')$. Thus we can clearly add convex torsion to $\xi$ to get an element in $\Tight_l(C;k)$. Moreover, given any element in $\Tight_0(C;pq)$ there will be a unique element in $\Tight_0(C;k)$ embedded in it. Thus $|\Tight_l(C;k)|$ is the number of totally $2$-inconsistent paths describing $q/p$. 

  Suppose that $k<pq$ and if $pq>0$ then assume $k>pq-p-q$. If $\xi\in \Tight_0(C;k)$ then we can glue in a solid torus to get a non-loose Legendrian knot in $S^3$ and by Lemma~\ref{lem:<pq} that knot must destabilize to a knot with $\tb=pq$. In other words, there is a subset $C'$ of $C$ such that $\xi|_{C'}=\xi'$ is in $\Tight_0(C;pq)$. Now if $\xi'$ is described by a $2$-inconsistent pair of paths, then as above we see that we can add convex torsion to $\xi$ to get an element in $\Tight_l(C;k)$. If the Legendrian corresponding to $\xi$ cannot be destabilized to a contact structure in $\Tight_0(C;pq)$ corresponding to a totally $2$-inconsistent pair of paths, then after adding twisting to $\xi$ so that the boundary slope is $\infty$ the contact structure will be overtwisted by Lemma~\ref{lem:overtwisted}. Thus once again we see that $|\Tight_l(C;k)|$ is the number of totally $2$-inconsistent paths representing $q/p$. 

  Finally, if $pq>0$ and $k\leq pq-p-q$, then we can make the same argument as above, except notice that according to Lemma~\ref{infinitythicken} two of the elements in $\Tight_0(C;pq)$ corresponding to totally $2$-inconsistent paths contain a convex torus $T$ with two dividing curves of slope $pq-p-q$ and the rest contain a convex torus $T$ with two dividing curves of slope $\infty$ (and in both cases any torus further from the boundary have the same dividing slope). Thus adding a contact structure on $T^2\times [0,1]$ with dividing slopes $k$ and $pq$ on the boundary to $C$ with the first two contact structures will have convex $1/2$-torsion, while adding the same contact structure to the other contact structures will still have no convex torsion. However, we can add a convex $1/2$-torsion layer to these to get contact structures in $\Tight_{\scriptscriptstyle{1/2}}(C;k)$ and thus the count of such structures is still the same. One can similarly argue for $\Tight_l(C;k)$. 
\end{proof}

\begin{proof}[Proof of Lemma~\ref{ClaimA}]
  We begin with a specific example. Consider the pair of paths $(P_1,P_2)$ representing $q/p$ with all the signs of $P_1$ positive and all the signs of $P_2$ negative. From Lemma~\ref{standardstructures} below, we know that when $pq<0$, the contact structure $\xi_{P_1,P_2}$ is the one supported by the open book with binding the $(p,q)$-torus knot and when $pq>0$, the contact structure is obtained from the tight contact structure on $S^3$ by a half Lutz twist. 

  Let $\xi_l$ be the contact structure obtained from $\xi_{P_1,P_2}$ by performing an $l$-fold Lutz twist on a transverse push-off of $L_{P_1,P_2}$ for $l \in \frac 12 \mathbb{N} \cup \{0\}$. Similarly $\xi_l$ can also be obtained from the complement of $L_{P_1,P_2}$ by attaching a $T^2\times [0,1]$ with convex $l$-torsion and then a tight solid torus that is a neighborhood of a Legendrian knot $L_l$ in $\xi_l$ with $\tb=pq$. In \cite[Section~5]{Etnyre13}, it was shown that $\tor(L_l) = l$. Let $(C,\xi'_l)$ be the contact structure on the complement of $L_l$. Clearly $(C,\xi'_l)$ has convex torsion $l$, but not $l+1$. Thus we have established the claim for the contact structure $\xi_{P_1,P_2}$. 

  The key to proving $\xi_l$ has Giroux torsion $l$ from \cite[Section~5]{Etnyre13} is to consider a specific arc $\gamma$ on a Seifert surface for the torus knot and showing that in $\xi_l$ the maximal twisting $\overline{tw}(\gamma)=-2l$ where $\overline{tw}(\gamma)$ is defined to be the maximum of the twisting of the contact planes along any Legendrian approximation of $\gamma$ (with endpoints fixed) with respect to the framing on $\gamma$ coming from the Seifert surface. It is easy to see that $\overline{tw}(\gamma)\leq -2l$, but to show it is exactly $-2l$ one needs to use that the contact structure on the complement of the torus knot is universally tight. Thus we can only apply this argument to the contact structure considered above (for the universally tightness, see the proof of Lemma~\ref{standardstructures}). 

  To prove Lemma~\ref{ClaimA} for the other contact structures we proceed as follows. Recall that $C$ can be thought of as the union of two solid tori $V_1$ and $V_2$ and $S^1\times P$ (see Section~\ref{knotcomp}). A Seifert surface for the $(p,q)$-torus knot can be constructed by taking $q$ meridian disks in $V_1$, $p$ meridian disks in $V_2$, and $pq$ $1$-handles in $S^1\times P$ that connect the disks. The arc $\gamma$ above can be taken to be a co-core to one of these $1$-handles and hence lives in $S^1\times P$. Also recall there is an annulus $A'$ that has both boundary components on $\partial C$ and when $C$ is cut along $A'$ one obtains two solid tori, one containing $V_1$ and the other $V_2$. The curve $\gamma$ can also be taken to be a curve on $A'$ and the framing given to $\gamma$ by the Seifert surface is the same as the one given by $A'$. Thus we can measure $\overline{\tw}(\gamma)$ using the $A'$ framing. 

  Now since we are considering pairs of decorated paths $P_1$ and $P_2$ that are totally $2$-inconsistent, we know the contact structure on $S^1\times P$ is contactomorphic to $\xi_{\pm\pm}$ and hence (up to switching co-orientations on the contact planes) independent of the choice of totally $2$-inconsistent pairs of paths. Now let $\widetilde\xi_l$ be the result of attaching $l$ convex torsion to $S^1\times P$ along $\partial C$. We have that $\overline{tw}(\gamma)$ in $\widetilde\xi_l$ is $-2l$. This is because it must be greater than or equal to $-2l$ by construction, but it cannot be larger than $-2l$ since if it were, that would contradict the fact that in $\xi'_l$ considered above we have that $\overline{tw}(\gamma)=-2l$ (recall that $\widetilde\xi_l$ is a subset of $\xi'_l$). 

  Finally consider any totally $2$-inconsistent pair of paths $(P_1,P_2)$. We can construct contact structures $\xi'_l$ as above associated to $\xi_{P_1,P_2}$ and inside $\xi'_l$ we have the contact structures $\widetilde\xi_l$ on $S^1\times P$. We can again consider $\overline{tw}(\gamma)$ in $\xi'_l$ and we again clearly have $\overline{tw}(\gamma)\geq -2l$. Suppose $\overline{tw}(\gamma)>-2l$. Then we can smoothly isotope the annulus $A'$ (relative to its boundary), so that it contains a Legendrian realization of $\gamma$ with twisting larger than $-2l$. As in the proof of Lemma~\ref{lem:staytight} we can use the isotopy discretization (Theorem~\ref{thm:discretization}) to get annuli $A_1=A',\ldots A_k$ such that $A'$ is our original annulus (that by construction contains a Legendrian realization of $\gamma$ with twisting $-2l$) and $A_k$ contains a Legendrian realization of $\gamma$ with twisting larger than $-2l$, and $A_i$ is obtained from $A_{i-1}$ by a bypass attachment. But recall, in the proof of Lemma~\ref{lem:staytight} we showed by induction that $A_i$ is contained on $S^1\times P$ with a contact structure contactomorphic to $\widetilde\xi_l$. This contradicts that $\overline{tw}=-2l$ in $\widetilde{\xi}_l$ and completes the proof of the claim. 
\end{proof}

\begin{proof}[Proof of Lemma~\ref{ClaimB}] 
  Suppose that $\xi$ and $\xi'$ are two contact structures in $\Tight_0(C;pq)$ associated to totally $2$ inconsistent pairs of decorated paths $(P_1,P_2)$ and $(P'_1, P'_2)$ representing $q/p$, respectively, such that $(P_1,P_2)\not= (P'_1,P'_2)$. Now let $\xi_l$ and $\xi'_l$ be the result of adding convex $l$-torsion to $\xi$ and $\xi'$, respectively. We first note that by Lemma~\ref{ClaimA} we know that $\xi_l$ is not contactomorphic to $\xi'_k$ if $l\not=k$, so we are left to see that $\xi_l$ and $\xi'_l$ are distinct. 

  From Lemma~\ref{lem:rotdiff}, we know that $\rot(L_{P_1,P_2}) \neq \rot(L_{P'_1,P'_2})$. Recall that the rotation number is the relative Euler class of the contact structure evaluated on a Seifert surface of the knot. Since adding full torsion does not change the relative Euler class and adding half torsion changes the sign of the relative Euler class, we see that the relative Euler class of $\xi_l$ and $\xi'_l$ are distinct for all $l\in \frac12\mathbb{N}$, thus proving Lemma~\ref{ClaimB}. 
\end{proof}

\section{Classification of non-loose torus knots}\label{nonloosetorusknots}

In this section, we begin by identifying some special contact structures and their associated pairs of decorated paths, and end by proving that our algorithm from Section~\ref{thealgorithm} really does give a complete classification of non-loose Legendrian $(p,q)$-torus knots. 

\subsection{Contact structures described by special pairs of decorated paths}\label{exceptionalctstr}
It will be useful to understand explicitly some of the contact structures associated with pairs $(P_1,P_2)$ of decorated paths in the Farey graph for the $(p, q)$-torus knot. The first statement of the following lemma was observed in \cite{Matkovic20Pre} and previously for some negative torus knots in \cite{GeigesOnaran20a}, in terms of contact surgery diagrams.

\begin{lemma}\label{standardstructures}
  Let $(P_1,P_2)$ be a pair of paths representing $q/p$ and decorated such that $P_1$ has only positive signs and $P_2$ has only negative signs. We have the following:

  \begin{enumerate}
  \item If $pq<0$ then $\xi_{\pm P_1,\pm P_2}$ is the contact structure $\xi_{|pq|-|p|-|q|+1}$ and is supported by the open book with binding the $(p,q)$-torus knot $T_{p,q}$.
  \item If $pq>0$ then $\xi_{\pm P_1,\pm P_2}$ is the contact structure $\xi_{-pq+p+q}$ which is obtained from $\xi_{std}$ by a half Lutz twist on the unique maximal self-linking number transverse representative of $T_{p,q}$ in $(S^3,\xi_{std})$.
  \end{enumerate}
\end{lemma}

\begin{proof}
 In \cite{EtnyreVelaVick10}, the first author and Vela-Vick showed that for a closed contact $3$-manifold $(M,\xi)$, the complement of a neighborhood of a transverse knot supporting $(M,\xi)$ is universally tight, and after adding convex torsion to its boundary, it remains tight. We will use this result to identify $\xi_{\pm P_1,\pm P_2}$ as the contact structure supported by $T_{p,q}$ when $pq<0$. 

Consider a contact structure $\xi'_{P'_1, P'_2} \in \Tight_0(C;pq)$, which is constructed by gluing two solid tori together with contact structures determined by the decorated paths $(P'_1, P'_2)$ and then removing a neighborhood of a Legendrian divide from the torus. Notice that when pulled back to the universal cover of $C$, the solid tori will completely unwrap (that is, their pre-image under the covering map will consist of copies of the universal cover of the solid tori). Thus, for $\xi'_{P'_1, P'_2}$ to be universally tight, each path can have only one sign. So there are at most four universally tight contact structures. Moreover, if $(P'_1, P'_2)$ are not totally 2-inconsistent then Lemma~\ref{lem:overtwisted} says adding Giroux torsion to $(C,\xi'_{P'_1, P'_2})$ will result in an overtwisted contact structure; while by Lemma~\ref{lem:staytight} if $(P'_1,P'_2)$ are totally 2-inconsistent, then $(C,\xi'_{P'_1, P'_2})$ remains tight even after one adds Giroux torsion to it. Thus $\xi_{\pm P_1, \pm P_2}$ are the only two possible contact structures that are universally tight and remain tight when Giroux torsion is added; in addition, Section~\ref{htpclasses} says $\xi_{P_1,P_2}$ and $\xi_{-P_1,-P_2}$ are isotopic (since they are homotopic as plane fields). 
  
In Proposition~\ref{wings}, we will see that all the negative stabilizations of $L_{P_1,P_2}$ are non-loose and hence the transverse push-off of $L_{P_1,P_2}$ has both these properties, and it is the only Legendrian with these properties (the contact structure $\xi_{-P_1,-P_2}$ also has these properties, but the transverse push-off of $L_{-P_1,-P_2}$ will be loose since a single negative stabilization of it is loose, see Proposition~\ref{wings} below). Thus $\xi_{P_1,P_2}$ is the contact structure supported by $T_{p,q}$ when $pq<0$. The $d_3$-invariant can be computed from \cite[Theorem~1.2]{Hedden08} or \cite[Corollary~1.2]{BakerEtnyreVanHornMorris12} (note that those papers consider the Hopf invariant, which in our context is $-d_3$).

  When $pq>0$, $T_{p,q}$ supports $\xi_{std}$, but adding a half Lutz twist to $\xi_{std}$ along the maximal self-linking number representative of $T_{p,q}$ will also have these properties and contain a non-loose Legendrian realization of $T_{p,q}$ with $\tb=pq$ and $\tor=0$. Thus the contact structure must be $\xi_{P_1,P_2}$. The $d_3$-invariant changes by subtracting the self-linking number of the transverse knot \cite[Proof of Theorem~4.3.1]{Geiges08} thus we see the contact structure is $\xi_{-pq+p+q}$. 
\end{proof}

We explicitly identify another contact structure in terms of decorated paths. This was also observed in \cite{Matkovic20Pre} in terms of contact surgery diagrams. 

\begin{lemma}\label{sdt+xi1}
  Let $(P_1,P_2)$ be a pair of decorated paths in the Farey graph for the $(p,q)$-torus knot. Suppose that the signs of all basic slices in $P_1$ and $P_2$ are the same, then $\xi_{P_1,P_2} = \xi_{std}$ for $pq<0$, and $\xi_{P_1,P_2} = \xi_1$ for $pq>0$.  
\end{lemma}

\begin{remark}
  If $q/p < -2$, then Lemma~\ref{sdt+xi1} remains true even if the basic slices in the last continued fraction block in $P_2$ have any signs. 
\end{remark}

\begin{proof}
  When $pq<0$, the last continued fraction block in $P_2$ is $n+1, n+2, \ldots, -1$ for $n = \lfloor q/p \rfloor$. Given the hypothesis on $(P_1,P_2)$, the concatenated path $\overline{P_1}\cup P_2$ can be shortened to $ n, n+1, \ldots, -1$ where each edge in the path has some sign (notice the edge from $n$ to $n+1$ represents a basic slice since all of the edges that were shortened had the same sign). Now extending this path by adding the edge from $\infty$ to $n$, we may think of this path as describing a contact structure on the solid torus with lower meridian $\infty$ and dividing slope $-1$. According to Lemma~\ref{gluingtoriandthickened}, we know this contact structure is tight and by Theorem~\ref{torusclass} it is unique. Now the path from $0$ to $-1$ also represents the unique tight structure on this solid torus. The union of these two tight contact structures on the solid tori now gives the tight contact structure on $S^3$. 

  When $pq>0$, the same argument shows that $\xi_{P_1,P_2}$ is obtained by gluing a solid torus with lower meridian $\infty$ and dividing slope $1$ to a solid torus with upper meridian $0$ and dividing slope $1$. This is clearly the result of performing a half Lutz twist on the maximal self-linking number unknot in the standard tight contact structure on $S^3$ and hence $\xi_1$. 
\end{proof}

It will be useful to have an explicit description of $\xi_1$ in terms of torus knots. This is given in the following lemma, but we first need another description of $\xi_1$ in terms of pairs of paths. At the end of Section~\ref{subsec:pathsinFG} we saw that $\xi_1$ is also described by paths $P_1'$ and $P_2'$ such that all the signs of the basic slices in $P_1'$ are $\mp$, except the first one which is $\pm$ and all the basic slices of $P'_2$ are $\pm$, except the first one which is $\mp$. 

\begin{lemma}\label{xi1}
  For any positive $(p,q)$-torus knot $T_{p,q}$, the contact structure $\xi_1$ can be described as follows. Let $(C,\xi)$ be the complement of a standard neighborhood of the Legendrian representative of $T_{p,q}$ with $\tb=pq-p-q$ in $(S^3,\xi_{std})$. Attach a basic slice in $\Tight^{min}(T^2\times [0,1]; \infty, pq-p-q)$ to $(C,\xi)$ and then a basic slice in $\Tight^{min}(T^2\times [0,1];pq,\infty)$ with the opposite sign to the result. Finally glue the unique tight contact structure on a solid torus with meridional slope $\infty$ and dividing slope $pq$.

  The final torus is a standard neighborhood of a non-loose $(p,q)$-torus knot with $\tb=pq$ that is described by the pair of decorated paths $(P'_1,P'_2)$ where the first edge in $P'_1$ is $\pm$ and all the others are $\mp$ while the first edge of $P'_2$ is $\mp$ and all the others are $\pm$. 
\end{lemma}

\begin{proof}
  Let $\xi_\pm$ be the contact structure which is the result of attaching a $\pm$-basic slice to $C$ with dividing slopes $\infty$ and $pq-p-q$. Since we discussed that $\xi_1 = \xi_{P'_1,P'_2}$ above, it is sufficient to show that $\xi_{P'_1,P_2'}$ is the result of attaching a $\mp$-basic slice to $\xi_\pm$ with dividing slopes $pq$ and $\infty$.

  Since the basic slices in $(P'_1,P'_2)$ adjacent to $q/p$ are of opposite signs, we may argue as in the proof of Lemma~\ref{lem:thickenornot} that $\xi_{P_1',P_2'}$ may be factored into two solid tori $V_1'$ and $V_2'$ and $S^1\times P$, where $P$ is a pair of pants. Also the contact structure on $S^1\times P$ admits a convex annulus $A$ running from $\partial V_1'$ to $\partial V_2'$ with two dividing curves that run from one boundary component to the other. Moreover, the contact structure on $V_1'$ is described by a path in the Farey graph whose signs are all $\mp$ and the path for $V_2'$ has all signs $\pm$. Notice that adding an $I$-invariant neighborhood of $A$ to $V_1'\cup V_2'$ yields a manifold $C'$ that is isotopic to $C$ and $\partial C'$ is convex with dividing slope $\infty$. Thus $C\setminus C'$ is $T^2\times [0,1]$ and the contact structure $\xi_{P'_1,P'_2}$ restricted to $T^2\times [0,1]$ will be a basic slice with dividing slopes $pq$ and $\infty$. Below we will see that the contact structure on $C'$ is $\xi_\pm$ discussed above and the basic slice has sign $\mp$. 
 
  Now consider the pair of paths $(P_1,P_2)$ representing $q/p$ with $P_1$ having all signs $\pm$ and $P_2$ having all signs $\mp$. From Lemma~\ref{standardstructures} we know that $\xi_{P_1,P_2}$ is obtained from performing a half-Lutz twist on the maximal self-linking transverse representative of $T_{p,q}$ in $\xi_{std}$ and then removing a solid torus with convex boundary having dividing slope $pq$. Notice that this contact structure can be described by taking $\xi_\pm$ on $C'$ and then adding another basic slice with dividing slopes $pq$ and $\infty$ with sign $\pm$. 

  We note that $\xi_{P_1,P_2}$ can be decomposed into pieces as we did for $\xi_{P_1',P_2'}$. In particular inside $(C,\xi_{P_1,P_2})$ we see a submanifold $C'$ isotopic to $C$ such that $C\setminus C'$ is a basic slice with dividing slopes $pq$ and $\infty$, and sign $\pm$. Moreover, we see that $\xi_{P_1,P_2}$ and $\xi_{P_1', P_2'}$ restricted to $C'$ are the same since they are given by attaching a thickened annulus to $V_1'$ and $V_2'$ with the same contact structures on them. Thus the only difference between $\xi_{P_1,P_2}$ and $\xi_{P_1', P_2'}$ is the sign of the bypass added, thus giving the desired result. (Notice that $\xi_{P_1,P_2}$ and $\xi_{P'_1,P'_2}$ cannot be the same contact structure since they are contact structures on the complement of Legendrian knots with different rotation number, see Lemma~\ref{lem:rotdiff}.)
\end{proof}

\subsection{Non-loose torus knots with \texorpdfstring{$\tb\leq pq$}{tb <= pq}}\label{classlesspq}
In this section, we will classify non-loose $(p, q)$-torus knots with $\tb \leq pq$ that are stabilizations of the Legendrian knots $L_{P_1,P_2}$ constructed in the proof of Lemma~\ref{lem:=pq}. 
We will see how each of these $L_{P_1,P_2}$ generates a ``wing'' or a ``diamond'', see Definition~\ref{quantfeatures}, and see how these wings and diamonds for different $L_{P_1,P_2}$ interact. 

\subsubsection{Wings for \texorpdfstring{$i$}{i}-inconsistent paths}\label{sec:wings}
We now consider a pair of decorated paths $(P_1, P_2)$ that is $i$-inconsistent for some $i\geq 2$ that describe a $(p,q)$-torus knot. We assume here that $(P_1,P_2)$ does not describe a positive torus knot in $(S^3,\xi_1)$ or a negative torus knot in the tight contact structure $\xi_{std}$. As above, see the beginning of this section or Section~\ref{subsec:pathsinFG},  we break the truncated paths $(P_1,P_2)$ into the continued fraction blocks  
\[
  (A_2,\ldots, A_{2n}) \text{ and } (B_1,B_3, \ldots, B_{2m-1}).
\] 
(we will only discuss this case here, with the case of $(A_1,\ldots, A_{2n-1})$ and $(B_2,\ldots, B_{2m})$ being analogous). As defined in Steps~2 and~3 of Section~\ref{classwogt}, let $s_k$ be the slope in $A_k$ or $B_k$ which is farthest from $q/p$, $T_k$ the convex torus in $V_1$ or $V_2$ with two dividing curves of slope $s_k$, and $L_k$ a Legendrian ruling curve on $T_k$ of slope $q/p$. Finally set $n_k= |s_k\bigcdot \frac qp|$.

\begin{proposition}\label{wings}
  If $pq>0$, then we assume that the ambient contact structure is not $\xi_1$, and if $pq<0$, we assume that the ambient contact structure is not $\xi_{std}$. 
  
Given an $i$-inconsistent pair of paths $(P_1, P_2)$ for $q/p$ as above. Assume that $i$ is even and all the basic slices in the continued fractions blocks $A_2,\ldots, A_{i-2}, B_1, \ldots, B_{i-1}$ are negative while some in $A_i$ are positive. Then 
 \[
 S_+^kS_-^l(L_{P_1,P_2})
\]
is non-loose for $k < n_{i-1}$ and $l \geq 0$, but $S_+^{n_{i-1}}(L_{P_1,P_2})$ is loose.

Similarly, $S_+^kS_-^l(L_{-P_1,-P_2})$ is non-loose for $k \geq 0$ and $l < n_{i-1}$, but $S_-^{n_{i-1}}(L_{-P_1,-P_2})$ is loose.  See Figure~\ref{wingsfig}.

  When $i$ is odd, the same result holds if all basic slices in the continued fraction blocks $A_2, \ldots, A_{i-1}$, $B_1, \ldots, B_{i-2}$ are positive and some in $B_i$ are negative. 
\end{proposition}

\begin{remark}
With the notation from the proposition, notice that Lemma~\ref{seestab} implies that the $q/p$-sloped ruling curve $L_k$ on $T_k$ is Legendrian isotopic to $S_+^{n_k}(L_{P_1,P_2})$ if $k$ is odd and $S_-^{n_k}(L_{P_1,P_2})$ if $k$ is even (since the sign of bypass will change if we consider the basic slice is from $T^2 \times \{1\}$ to $T^2 \times \{0\}$). 
\end{remark}

\begin{definition}\label{wingdef}
We will call the set 
\[
  W_{P_1,P_2}=\{S_+^kS_-^l(L_{P_1,P_2}) \text{ for }  k < n_{i-1} \text{ and } l \geq 0\},
\] 
the \dfn{wing of $L_{P_1,P_2}$}, and similarly for $L_{-P_1, -P_2}$. We think of these as the non-loose Legendrian knots generated by $L_{\pm P_1, \pm P_2}$. See Figure~\ref{wingsfig}.
\end{definition}

When $pq<0$ and $\tb<pq$ this proposition also follows from \cite[Corollary~4.3]{Matkovic20Pre}, though the structure of the wings was not made explicit. The result for the negative trefoil and $\tb<-6$ or equal to $5$ was also established in \cite{GeigesOnaran20a}. 

\begin{figure}[htbp]
\begin{overpic}{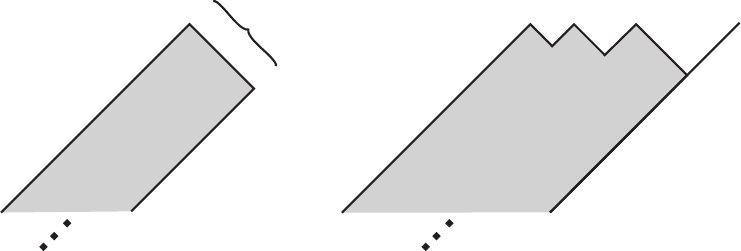}
  \put(125, 110){$n_{i-1}$}
\end{overpic}
\caption{On the left is the wing $W(L_{P_1,P_2})$ from Proposition~\ref{wings}. On the right is the wing $W$ from Proposition~\ref{merge} generated from all the pairs of paths compatible with $(P_1,P_2)$. Each integral point in the shaded region, whose coordinates sum to be odd, is realized by a unique non-loose Legendrian knot with $\tor = 0$. The peaks are at $\tb=pq$ and there are $i-1$ peaks (for an $i$-inconsistent path) each corresponding to a $k$-inconsistent pair of paths $(P_1^k, P_2^k)$ compatible with $(P_1,P_2)$ for $2 \leq k \leq i$ and the distance between the peaks corresponding to $(P_1^k, P_2^k)$ and $(P_1^{k-1}, P_2^{k-1})$ is $2n_{k-1}'$ (see the proof of Proposition~\ref{merge}). Once one computes a rotation number of one of the peaks using Lemma~\ref{computer} the others are determined by the distance between the peaks.  The wings for $(-P_1,-P_2)$ are obtained by reflecting these wings about a vertical line.}
\label{wingsfig}
\end{figure}

\begin{remark}
  We will see in the proof below that stabilizations of the $L_{P_1,P_2}$ become loose because they can be put on a convex torus that allows the path $\overline{P_1}\cup P_2$ to be shortened merging two basic slices with opposite sign. This does not happen if $\xi_{P_1,P_2}$ is $\xi_{std}$ since the ambient contact structure is tight. We will address the case when $pq>0$ and $\xi_{P_1,P_2}$ is $\xi_{-pq+p+q}$ at the end of this section and in Section~\ref{diamondsforxi1} below we will see what is different when $pq>0$ and $\xi_{P_1,P_2}$ is $\xi_1$. 
\end{remark}

We first observe the following results. 

\begin{lemma}\label{increasing} 
  With the notation above, the integers $n_i=|s_i\bigcdot \frac{q}{p}|$ start at $1$ and are strictly increasing (unless $q/p={(2n+1)/2}<0$ in which case $n_1=n_2$).
\end{lemma}

\begin{proof}
  The claim for $q/p={(2n+1)/2}<0$ can easily be checked, so we assume that we have some other $q/p$. By Lemma~\ref{morethanone-mustbeone}, we know that there is an edge from $q/p$ to $s_1$ and there is not an edge between $q/p$ and $s_2$. Thus $n_1=|\frac qp\bigcdot s_1|=1<|\frac qp \bigcdot s_2|=n_2$. We now inductively prove the $n_i$ are strictly increasing. To this end we assume this has been proven for $i<j$ and establish that $n_i<n_j$. Recall, from Lemma~\ref{obs} we know that there is an edge from $s_{j-1}$ to $s_j$ and the minimal path from $s_j$ to $s_{j-2}$ is given by $v_0=s_j, v_1, \ldots, v_k=s_{j-2}$, where $v_l = ls_{j-1} \oplus s_j $ for $0 \leq l \leq k$. Then we have 
  \begin{align*}
    s_{j-2}\bigcdot \frac qp &= (ks_{j-1} \oplus s_j) \bigcdot \frac qp \\
    &= (ks_{j-1}   \bigcdot \frac qp) + (s_j \bigcdot \frac qp).
  \end{align*}
  Since $(s_{j-1} \bigcdot \frac qp)$ and $(s_{j} \bigcdot \frac qp)$ have opposite sign and $(s_{j-2} \bigcdot \frac qp)$ and $(s_{j} \bigcdot \frac qp)$ have the same sign, we have 
  \begin{equation} \label{eq:5}
    \left|s_{j-2}\bigcdot \frac qp\right| + \left|(ks_{j-1}) \bigcdot \frac qp\right| = \left|s_j \bigcdot \frac qp\right|, 
  \end{equation}
  confirming that the $n_i$ are increasing. 
\end{proof}

\begin{proof}[Proof of Proposition~\ref{wings}]
The basic idea of the proof is to show that if a stabilization of $L_{P_1,P_2}$ can be put on a convex torus in one of the solid tori described by $P_1$ or $P_2$ so that it splits $S^3$ into a pair of solid tori, one of which is obviously overtwisted, then the Legendrian knot is loose. But if this is not the case, then one may use discretization of isotopy arguments to show the Legendrian knot is still non-loose. We detail this argument below.

  From Lemma~\ref{seestab} we know that $L_{i-1}$ is the same as $S_+^{n_{i-1}}(L_{P_1,P_2})$. Notice, by Lemma~\ref{obs}, there is an edge in the Farey graph from $s_{i-1}$ to $s_i$. Thus the path $\overline{A}_i,\ldots, \overline{A}_2, B_1,\ldots B_{i-1}$ can be shortened. Since not all the signs of the basic slices in this path are the same, the resulting contact structure on $T^2\times [0,1]$ is overtwisted. That is we have found an overtwisted disk in the complement of $S_+^{n_{i-1}}(L_{P_1,P_2})$.
  
  Now we will show that $L := S_+^{n_{i-1}-1}S_-^{l}(L_{P_1,P_2})$ for $l \geq 0$ is non-loose. We can put $L$ on a convex torus $T$, which is contained in $V_2$ with slope $s_{i-3}$ (not as a standard ruling curve). To see this, notice that a ruling curve on $T$ would be $S_+^{n_{i-3}}(L_{P_1,P_2})$ by Lemma~\ref{seestab} and $n_{i-3}<n_{i-1}$ by Lemma~\ref{increasing}; thus any further stabilization can be put on $T$ but not as a ruling curve. Let $\overline{S^3 \setminus T} = V_1' \cup V_2'$. Clearly, $V_1'$ and $V_2'$ are tight as the paths in the Farey graph describing them are either minimal or can be shortened to be minimal at vertices whose adjacent edges have the same sign. 

  Suppose $L$ is loose. Then there is an overtwisted disk in the complement of a standard neighborhood $N$ of $L$. Notice that $T \cap (\overline{S^3 \setminus N})$ is an annulus $A$ and there is a smooth isotopy of $A$, rel boundary, to an annulus disjoint from the overtwisted disk. We can assume that the boundary of $A$ is Legendrian and perturb $A$ to be convex. By isotopy discretization (Theorem~\ref{thm:discretization}), there is a sequence of annuli $A_1=A, \ldots, A_k$ such that $A_k$ is disjoint from the overtwisted disk and each $A_j$ is obtained from $A_{j-1}$ by attaching a bypass. Notice that the $A_j$ can be extended to tori $T_j$ containing $L$, which is just $A_j \cup (T \cap N)$ (after perturbation). Each $T_j$ is obtained from $T_{j-1}$ by a bypass attachment in the complement of $L$, and $T_k$ is disjoint from the overtwisted disk.  

  Each $T_j$ breaks $S^3$ into two solid tori $V_1^j$ and $V_2^j$. By construction, we know that $V_1^1$ and $V_2^1$ are both tight. We will inductively prove that each $V_1^j$ and $V_2^j$ is tight, and this will contradict the fact that there is an overtwisted disk in the complement of $T_k$, thus showing that there could not have been an overtwisted disk in the complement of $N$ and that $L$ is non-loose.
 
  We recall that the part of the path $\overline{P_1}\cup P_2$ between $s_{i-3}$ and $s_{i-2}$ consists of all negative basic slices while the part between $s_{i-2}$ and $s_i$ contains some positive basic slices (and possibly some negative ones too). 

  We inductively assume that 
  \begin{itemize}
    \item the slope $t_{j-1}$ of the dividing curves on $T_{j-1}$ is in $(s_{i}, s_{i-1})$, 
    \item if $t_{j-1} \in (s_{i}, s_{i-2})$, then the contact structures on the tori are given by consistently dividing the path $\overline{P_1}\cup P_2$, by which we mean shortening the path at vertices whose adjacent edges have the same sign or dividing an edge into two edges of the same sign.
  \end{itemize}
  Notice that this condition guarantees that the contact structures on $V_1^{j-1}$ and $V_2^{j-1}$ are tight since the contact structures correspond to subdividing the path $\overline{P_1}\cup P_2$ at $t_{j-1}$ and when doing this only the continued fraction blocks with the same sign can be shortened.

  First, suppose $t_j \geq t_{j-1}$. By the inductive hypothesis, $t_j \in (s_i,s_{i-1}]$ and we will show that $t_j \neq s_{i-1}$. Assume $t_j = s_{i-1}$. Then the ruling curves on $T_{j}$ are Legendrian isotopic to $S_+^{n_{i-1}}(L_{P_1,P_2})$ by Lemma~\ref{seestab} and any Legendrian curve on $T_{j}$ is a stabilization of the ruling curve. However, since $L = S_+^{n_{i-1}-1}S_-^{l}(L_{P_1, P_2})$, it cannot be a stabilization of the ruling curve and $t_j \neq s_{i-1}$.

  Next, suppose $t_j < t_{j-1}$. By the inductive hypothesis, $t_j \in [s_{i},s_{i-1})$. Assume $t_j < s_{i-2}$ and the sign of the basic slice between $T_{j-1}$ and $T_{j}$ is positive. Recall the proof of Lemma~\ref{increasing}. We labeled the vertices in $A_i$ as $v_0=s_i, v_2,\ldots, v_k=s_{i-2}$ and $v_l= ls_{i-1} \oplus s_i$ for $1 \leq l \leq k$. Also, from Equation~(\ref{eq:5}), we have 
  \[
    \left|v_{l} \bigcdot \frac qp\right| = \left|s_i \bigcdot \frac qp\right| - l\left|s_{i-1} \bigcdot \frac qp\right|.
  \]  
  Clearly this implies
  \begin{equation} \label{eq:6}
    \left|v_{l} \bigcdot \frac qp\right| - \left|v_{l+1} \bigcdot \frac qp\right| = \left|s_{i-1} \bigcdot \frac qp\right|.
  \end{equation}
  Returning to our problem, the sign of the basic slice implies that $t_{j-1} = v_{l+1}$ for some $0 \leq l < k$. This is because if $v_{j-1}$ were between two $v_i$ then $v_j$ would also be between them and the basic slice would have to be negative (since the basic slice between the two $v_i$ are negative by our hypothesis on $t_{j-1}$). Clearly $t_j \in [v_l,v_{l+1}]$ and by \cite[Remark 2.13]{ChakrabortyEtnyreMin20Pre}, we have 
  \[
    \left|v_{l} \bigcdot \frac qp\right| \leq \left|t_j \bigcdot \frac{q}{p}\right|.
  \]
  Thus combining it with Equation~(\ref{eq:6}), we can conclude
  \begin{align*}
    \left|s_{i-1} \bigcdot \frac qp\right| &= \left|v_{l} \bigcdot \frac qp\right| - \left|v_{l+1} \bigcdot \frac qp\right| \\
    &\leq \left|t_j \bigcdot \frac qp\right| - \left|t_{j-1} \bigcdot \frac qp\right| \\
    &\leq \left|(t_j \ominus t_{j-1}) \bigcdot \frac qp\right|.
  \end{align*}
  Therefore, the ruling curves on $T_{j}$ is Legendrian isotopic to $S_+^kS_-^l(L_{P_1,P_2})$ for $k \geq n_i$ and $l \geq 0$ by Lemma~\ref{seestab} and any Legendrian curve on $T_{j}$ is a stabilization of the ruling curve. However, since $L = S_+^{n_{i-1}-1}S_-^{l}(L_{P_1,P_2})$, it cannot be a stabilization of the ruling curve, and the basic slice cannot be positive.
\end{proof} 

Let $(P^i_1, P^i_2)$ be an $i$-inconsistent pair of paths that describes a $(p,q)$-torus knot and assume it is not compatible with an $(i+1)$-inconsistent pair of paths (see Section~\ref{htpclasses} for terminology); moreover, if $pq>0$ assume that the contact structure given by the paths is not $\xi_1$.  As discussed in Section~\ref{htpclasses}, we know that these paths are compatible with a unique $k$-inconsistent pair of paths $(P_1^k,P_2^k)$ for all $k=2, 3, \ldots, i$. Let $L_{P_1^k,P_2^k}$ be the Legendrian $(p,q)$-torus knots corresponding to the paths $(P_1^k,P_2^k)$. We know they are all in the same contact structure and each generates a wing by Proposition~\ref{wings}. We will see that all of these wings merge in the sense that when two Legendrian knots in different wings have the same classical invariants, then they are isotopic. 

As discussed at the beginning of this section and in Section~\ref{subsec:pathsinFG}, we break the decorated paths $(P_1^k, P_2^k)$ into their continued fraction blocks  
\[
  (A^k_2,A^k_4,\ldots, A^k_{2n}) \text{ and } (B^k_1,B^k_3, \ldots, B^k_{2m-1}).
\] 
Notice that the paths $A^k_l$ and $B^k_l$ in the Farey graph are independent of $k$, only the signs on the edges vary with $k$. 

\begin{proposition}\label{merge}
  With the notation above, there is a fixed line of slope $\pm 1$ that contains the lower edge of each wing $W_{P_1^k,P_2^k}$ and the union of the wings 
  \[
    W=\bigcup_{k=2}^i W_{P_1^k,P_2^k}
  \]
  is coarsely Legendrian simple, i.e. any two Legendrian knots in $W$ with the same $\tb$ and $\rot$ are equivalent. See Figure~\ref{wingsfig}.
\end{proposition}
When $pq<0$ and $\tb<pq$ this proposition also follows from \cite[Corollary~4.3]{Matkovic20Pre}, though the coarse Legendrian simplicity was not made explicit. 

\begin{proof}
We will assume that $i$ is even but the discussion for $i$ odd is entirely analogous. Since the pair $(P_1^i, P_2^i)$ is $i$-inconsistent, we can assume that all the basic slices in $A^i_2,\ldots, A^i_{i-2}$, $B^i_1, \ldots, B^i_{i-1}$ have the same sign, say negative (the positive case being entirely analogous), and $A_i$ has some positive basic slices. 
In Section~\ref{subsec:pathsinFG} we saw that one gets $(P_1^{i-1}, P_2^{i-1})$ from $(P_1^i,P_2^i)$ as follows: the union of $A^i_2,\ldots, A^i_{i-2}$, $B^i_1,\ldots, B^i_{i-3}$ and all but the last basic slice of $B^i_{i-1}$ can be shortened to a single edge in the Farey graph,  which will have a negative sign, and that edge extends $A^i_i$ to a longer continued fraction block. Thus, we exchange the positive basic slice in $A^i_i$ with this new edge and break the new edge back into its previous edges, but now all having positive signs. Specifically, this means $A^{i-1}_l=A^i_l$ and $B^{i-1}_l=B^i_l$ for all $l>i$, $A_i^{i-1}$ agrees with $A_i^i$ except one of the positive basic slices has turned into a negative one, $B_{i-1}^{i-1}$ consists of one negative basic slice and all the others are positive, and finally $A_l^{i-1}$ and $B_l^{i-1}$ all have only positive basic slices for $l<i-1$.  Continuing this shuffling, one sees that the $A_l^k$ and $B_l^k$ for $l<k$ will all have the same sign and the signs are negative if $k$ is even and positive is $k$ is odd. See Figure~\ref{compatible}. 

We also recall that $s_k$ is the slope in $A^i_k$ or $B^i_k$ which is farthest from $q/p$, $T_k$ is the convex torus in $V_1$ or $V_2$ with two dividing curves of slope $s_k$, and $L_k$ is a Legendrian ruling curve on $T_k$ of slope $q/p$. We also set $n_k=|s_k \bigcdot \frac qp|$. 

 With the notation established above, we begin with the Legendrian simplicity of $W$.
  Consider the contact structure $\xi_{P^i_1,P^i_2}$. 
  Let $s_{i-1}'$ be the slope in $B_{i-1}$ closest to $q/p$ with an edge to $s_{i-1}$ (that is, it is the slope of the second to the last vertex in $B_{i-1}$).  Let $T_{i-1}'$ be the convex torus in $V_2$ with two dividing curves of slope $s'_{i-1}$ and let $L_{i-1}'$ be a ruling curve on $T_{i-1}'$ with slope $q/p$. By Lemma~\ref{seestab} we know that 
  \[
    L'_{i-1} = S_+^{n'_{i-1}}(L_{P_1^i,P_2^i}) 
  \]
  for $n_{i-1}'=|s_{i-1}' \bigcdot \frac qp|$ (since the sign of all basic slices in $B_1, \ldots, B_{i-1}$ are negative). Now consider two solid tori $V'_1$ and $V'_2$ that $T_{i-1}'$ breaks $S^3$ into, and the contact structure on $V'_1$ is given by the path $\overline{P_1}$ followed by $B^i_1\cup \cdots \cup B^i_{i-3}$ followed by all but the last edge in $B_{i-1}$. We can thus exchange the basic slices in the continued fraction block as discussed above. Now it is clear that $T_{i-1}'$ is also a torus in the contact structure $\xi_{P_1^{i-1}, P_2^{i-1}}$ and hence its ruling curve is 
  \[
    L_{i-1}'=S_-^{n'_{i-1}}(L_{P_1^{i-1},P_2^{i-1}}). 
  \] 
  In other words, $S_+^{n_{i-1}'}(L_{P_1^i,P_2^i})$ is Legendrian isotopic to $S_-^{n_{i-1}'}(L_{P_1^{i-1},P_2^{i-1}})$. It is also the first time a stabilization of $L_{P_1^i,P_2^i}$ could be isotopic to a stabilization of $L_{P_1^{i-1},P_2^{i-1}}$. Moreover, it is clear that any Legendrian knot in 
  \[
    W(L_{P_1^{i},P_2^{i}})\cap W(L_{P_1^{i-1},P_2^{i-1}}) 
  \]
  is a stabilization of $L_{i-1}'$ and hence 
  \[
    W(L_{P_1^{i},P_2^{i}})\cup W(L_{P_1^{i-1},P_2^{i-1}})
  \] 
  is coarsely Legendrian simple. One may now similarly show that  $W(L_{P_1^{k},P_2^{k}})\cup W(L_{P_1^{k-1},P_2^{k-1}})$ is coarsely Legendrian simple for all $k$, thus yielding the second part of the proposition. 

  We now consider the first statement that there is a fixed line, which is the lower edge of all the wings. Notice that the lower boundary of all wings is contained in a line of slope $1$ (or $-1$ for $(-P_1,-P_1)$), and this line is determined by how many positive stabilizations make one of the $L_{P_1^l, P_2^l}$ loose. Now recall that $L_{P_1^i,P_2^i}$ becomes loose after exactly $n_{i-1}$ positive stabilizations and $L_{P_1^{i-1}, P_2^{i-1}}$ will become loose after exactly $n_{i-2}$ positive stabilizations. Moreover, we just saw that $S_+^{n_{i-1}'}(L_{P_1^i,P_2^i})$ is isotopic to $S_-^{n_{i-1}'}(L_{P_1^{i-1},P_2^{i-1}})$. We claim that $n_{i-2} = n_{i-1}-n_{i-1}'$.
  If this is true, then it is clear that the line defining the lower edge of the wing of $L_{P_1^i,P_2^i}$ and $L_{P_1^{i-1}, P_2^{i-1}}$ will be the same, and the same argument works for all adjacent wings. It is not hard to see from the Farey graph that $s_{i-1}'=s_{i-1}\oplus s_{i-2}$. Since $s_{i-1}$ and $s_{i-2}$ are on opposite sides of $q/p$, their intersection number with $q/p$ will have opposite sign. Thus we have 
  \[
    \left|s_{i-1}'\bigcdot \frac qp\right|=\left|(s_{i-1}\oplus s_{i-2})\bigcdot \frac qp\right| = \left|s_{i-1} \bigcdot \frac qp\right| - \left|s_{i-2} \bigcdot \frac qp\right|.
  \] 
\end{proof}

\begin{proposition}
  With the notation above and in Proposition~\ref{merge}, let 
  \[
    \overline{W}=\bigcup_{k=2}^i W_{-P_1^k,-P_2^k}.
  \]
  No Legendrian element in $W$ is equivalent to an element of $\overline{W}$. 
\end{proposition}

\begin{remark}
  In Section~\ref{justifyalgorithm} below we will see that when $pq<0$, $W$ is disjoint from $\overline{W}$ and hence $W\cup \overline{W}$ is coarsely Legendrian simple. However when $pq>0$ $W\cap \overline{W}\not=\emptyset$ and hence $W\cup \overline{W}$ is not Legendrian simple. 
\end{remark}

\begin{proof}
  Notice any element in $W$ will become loose after a finite number of negative stabilizations, while elements of $\overline{W}$ will stay non-loose after any number of negative stabilizations so no element in $W$ can be equivalent to an element of $\overline{W}$. 
\end{proof}

\subsubsection{Diamonds in \texorpdfstring{$\xi_1$}{xi_1} when \texorpdfstring{$pq>0$}{pq > 0}}\label{diamondsforxi1}
Let $(P_1,P_2)$ be a pair of paths that represent $q/p$ with $pq>0$ and assume all the signs in the paths are the same, say negative. From Lemma~\ref{xi1} we know that $\xi_{P_1,P_2}$ is $\xi_1$. 

\begin{proposition}\label{diamonds}
  Given $P_1$ and $P_2$ as above, 
  the Legendrian knots $S^k_\pm S^l_\mp (L_{\pm P_1,\pm P_2})$ are non-loose if and only if $k<p$ and $l<q$.
\end{proposition}
\begin{definition}\label{rigorousdiamond}
We define the \dfn{diamond of $L_{\pm P_1, \pm P_1}$} to be the set
\[
  D_{\pm P_1, \pm P_1}= \{S^k_\pm S^l_\mp (L_{\pm P_1,\pm P_2}) : k<p, l<q\}
\] 
and think of these as the non-loose Legendrian knots generated from $L_{\pm P_1, \pm P_1}$. 
\end{definition}
\begin{proof}
  Notice that the path $\overline{P_1}$ starts at $\infty$ and goes clockwise to $q/p$. However, the first edge in the path goes from $\infty$ to $\lfloor q/p \rfloor$ and corresponds to the unique tight contact structure on the solid torus with longitudinal dividing curves. So the part of $\overline{P_1}$ that determines the contact structure is the path that starts at $\lfloor q/p \rfloor$ and goes clockwise in some number of jumps to $q/p$. By Lemma~\ref{gluingtoriandthickened} we can represent this contact structure by the unique contact structure on the solid torus with convex boundary $0$ and then a contact structure on $T^2\times [0,1]$ given by the path $0,1, \ldots,  \lfloor q/p \rfloor$ followed by $\overline{P_1}$ and the signs on the edges between $0$ and $\lfloor q/p \rfloor$ can be chosen arbitrarily.  In particular, we can choose them to be negative (that is, the same sign as the signs in $P_1$ and $P_2$). Thus, inside $V_1$ we have a convex torus $T_0$ with two dividing curves of slope $0$ such that the path from $0$ to $q/p$ consists of all negative signs. Similarly, we have a convex torus $T_\infty$ of slope $\infty$ inside $V_2$, and again the path from $q/p$ to $\infty$ consists of all negative signs. Let $L_0$ and $L_\infty$ be ruling curves of slope $q/p$ on the tori $T_0$ and $T_\infty$, respectively. By Lemma~\ref{seestab} we know that $L_0$ is $S_-^q(L_{P_1,P_2})$ and that $L_\infty$ is $S_+^p(L_{P_1,P_2})$. Notice that $T_0$ separates $S^3$ into two solid tori, one of which has meridional slope $0$, and hence we see that a dividing curve on $T_0$ bounds an overtwisted disk in this solid torus, and hence $L_0$ is loose. Similarly, $L_\infty$ is also loose. Thus we see that $S^k_+ S^l_- (L_{P_1,P_2})$ is loose if either $k\geq p$ or $l\geq q$. 

  Now if $k<p$ and $l <q$, then the Legendrian knot $S^k_\pm S^l_\mp (L_{\pm P_1,\pm P_2})$ cannot be put on either $T_0$ or $T_\infty$ and the same isotopy discretization argument as in the proof of Proposition~\ref{wings} shows that $S^k_\pm S^l_\mp (L_{\pm P_1,\pm P_2})$ is non-loose.

  A similar argument establishes the result for $L_{-P_1,-P_2}$. 
\end{proof}

\begin{proposition}\label{dmerge}
  Given $P_1$ and $P_2$ as above, the union $D'=D_{P_1,P_2}\cup D_{-P_1,-P_2}$ is coarsely Legendrian simple, that is, any two Legendrian knots in $D'$ with the same $\tb$ and $\rot$ are equivalent. See Figure~\ref{diamondsfig}. 
\end{proposition}

\begin{figure}[htbp]
\begin{overpic}{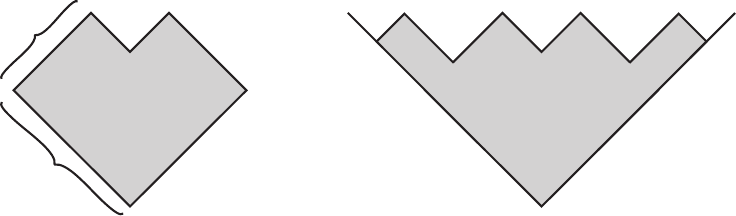}
  \put(7, 89){$p$}
  \put(17, 17){$q$}
\end{overpic}
\caption{The union $D'$ of the diamonds associated with two completely consistent paths describing $q/p$ for $pq>0$ is shown on the left. The peaks occur at $pq$, and the central valley occurs after stabilizing a peak $q-p$ times. The union of the diamonds of all pairs of paths compatible with the original paths is shown on the right. Each integral point in the shaded region, whose coordinates sum to be odd, is realized by a unique non-loose Legendrian knot with $\tor = 0$.}
\label{diamondsfig}
\end{figure}

\begin{proof}
  As argued in the proof of the previous proposition, inside $V_1$ we find a convex torus $T_1$ with two dividing curves of slope $1$. Let $L_1$ be a ruling curve on $T_1$ with slope $q/p$. Lemma~\ref{seestab} tells us that $L_1$ is $S_-^{q-p}(L_{P_1,P_2})$. 

  Notice that $T_1$ breaks $S^3$ into two solid tori $V_1'$ and $V_2'$ each having two longitudinal dividing curves, so there is a unique tight contact structure on each, the first described by a path that goes from $\infty$ clockwise to $1$ and the other going from $0$ anti-clockwise to $1$. As argued above we can break the first path into a path from $\infty$ to $0$ and then $0$ to $1$. The first edge describes the unique tight contact structure on a solid torus with longitudinal divides and the second edge can have any signs and here, we choose a positive sign. Similarly for the second path we may subdivide the edges from $1$ to $\lfloor q/p\rfloor$ and then the edges in $\overline{P_1}\cup P_2$ that are from $\lfloor q/p \rfloor$ to $\infty$ and we may assume that all the edges have a positive sign. This shows that $L_1$ also sits in the contact structure $\xi_{-P_1,-P_2}$ and in particular is $S_+^{q-p}(L_{-P_1,-P-2})$ and we see that $S_-^{q-p}(L_{P_1,P_2})$ is Legendrian isotopic to $S_+^{q-p}(L_{-P_1,-P_2})$. Since all other stabilizations of $L_{P_1,P_2}$ and $L_{-P_1,-P_2}$ with the same classical invariants are stabilizations of $S_-^{q-p}(L_{P_1,P_2}) = S_+^{q-p}(L_{-P_1,-P_2})$, the result follows.  
\end{proof}

As all the signs in all paths $P_1$ and $P_2$ are the same, we can shorten $\overline P_1\cup P_2$ to a path going from $\lfloor q/p \rfloor$ clockwise to $\infty$. Inside the contact structure on $T^2\times [0,1]$ described by this path we can find a torus $T$ with dividing slope $\lceil q/p \rceil$. Now $T$ divides $(S^3, \xi_{P_1,P_2})$ into two tight solid tori. One with lower meridian $\infty$ and convex boundary of slope $\lceil q/p \rceil$ and the other with upper meridian $0$ and convex boundary of slope $\lceil q/p \rceil$. So the first solid torus has longitudinal dividing curves and hence there is only one possible contact structure on it. Moreover, by Lemma~\ref{gluingtoriandthickened} we may split this torus into one with boundary slope $\lfloor q/p \rfloor$ and basic slice with boundary slopes  $\lfloor q/p \rfloor$ and  $\lceil q/p \rceil$ and the sign of the basic slice can be chosen arbitrarily. We choose the sign to be positive and then subdivide the path to $\overline {P_1}\cup P_2$ so that all the basic slices are positive except the last one in $P_2$ going from $\lceil q/p \rceil$ to $\infty$ which is still negative. Denote the paths with the new signs by $P_1^{2m-1}, P_2^{2m-1}$. Break the paths into their continued fraction blocks 
\[
  (A^{2m-1}_2,\ldots, A^{2m-1}_{2n}) \text{ and } (B^{2m-1}_1,\ldots, B^{2m-1}_{2m-1})
\] 
as in Section~\ref{subsec:pathsinFG}. Then this new pair of paths is $(2m-1)$-inconsistent (that is, maximally inconsistent). We leave the almost identical case when the continued blocks in $P_2$ have even subscripts to the reader. As we saw in Section~\ref{subsec:pathsinFG}, we will now get $k$-inconsistent pairs of paths $P_1^k,P_2^k$ for $k=2,3,\ldots, 2m-1$ that are all compatible. Notice that all the signs of the basic slices in $P^2_1$ are negative, except the first one, which is positive, and all the basic slices of $P^2_2$ are positive, except the first one, which is negative. 

\begin{proposition}\label{alldiamonds}
 With the notation above, consider the pairs of paths $(\pm P_1,\pm P_2)$ where both $P_1$ and $P_2$ have all positive signs. From Proposition~\ref{dmerge} we have the union of diamonds. 
 \[
 D'=D_{P_1,P_2}\cup D_{-P_1,-P_2}
 \]
 Consider the $\textsf{V}$ formed by the two rays starting at the bottom vertex of $D'$, tangent to the lower boundary of $D'$, and with the top of the $\textsf{V}$ at $\tb=pq$. Each pair of paths $(\pm P_1^k, \pm P_2^k)$ constructed above gives a Legendrian knot $L_{\pm P_1^k,\pm P^k_2}$ with $\tor = 0$ and $\tb = pq$, and stabilizations of it will remain non-loose exactly when the resulting Legendrian has its classical invariants on or above the $\textsf{V}$ described above. The set of non-loose stabilizations of $L_{\pm P_1^k,\pm P_2^k}$ gives the diamond $D_{\pm P_1^k,\pm P_2^k}$ of $L_{\pm P_1^k,\pm P_2}$. The union 
\[
  D=D'\cup \bigcup_{k=2}^{2m-2} D_{\pm P_1^k,\pm P_2^k}
\]
is coarsely Legendrian simple, i.e. any two Legendrian knots in $D$ with the same $\tb$ and $\rot$ are equivalent. See the right-hand side of Figure~\ref{diamondsfig}.
\end{proposition}

\begin{proof}
  We first relate a stabilization of $L_{P_1^{2m-1}, P_2^{2m-1}}$ and a stabilization of $L_{P_1,P_2}$.  We use the notation from the paragraph preceding the statement of the proposition. Notice that in the contact structure $\xi_{P_1,P_2}$ we see that inside of $V_2$ there is a convex torus $T_{\lceil \scriptscriptstyle{q/p} \rceil}$ with two dividing curves of slope $\lceil q/p \rceil$. Let $L_{\lceil \scriptscriptstyle{q/p} \rceil}$ be a ruling curve on $T_{\lceil \scriptscriptstyle{q/p} \rceil}$ with slope $q/p$. By Lemma~\ref{seestab} we know that $L_{\lceil \scriptscriptstyle{q/p} \rceil}$ is isotopic to the result of positively stabilizing $L_{P_1,P_2}$ exactly $|\lceil \frac qp \rceil\bigcdot \frac qp|$ times. As noted in the paragraph above, we also know that $T_{\lceil \scriptscriptstyle{q/p} \rceil}$ is a convex torus inside $\xi_{P_1^{2m-1}, P_2^{2m-1}}$ and from this we see that $L_{\lceil \scriptscriptstyle{q/p} \rceil}$ is also the result of negatively stabilizing $L_{P_1^{2m-1},P_2^{2m-1}}$ exactly $|\lceil \frac qp \rceil\bigcdot \frac qp|$ times, thus all further stabilizations of these knots will remain isotopic. 

  Now just as in the proof of Proposition~\ref{wings} we see that $L_{P_1^{2m-1},P_2^{2m-1}}$ positively stabilized $n_{2n}=|s_{2n}\bigcdot \frac qp|$ times is loose but stabilizing any fewer times remains non-loose (here, $2n = 2m-2$). Also, as in the proof of Proposition~\ref{merge} we see that $S_-^{n_{2n}-1}(L_{P_1^{2m-1},P_2^{2m-1}})$ can be negatively stabilized some number of times to agree with the left corner of $D'$. From this we see that we get the diamond of $L_{P_1^{2m-1},P_2^{2m-1}}$ and when an element shares classical invariants with $D'$ it is isotopic to the corresponding element of $D'$. 

  The diamonds for the other paths $P_1^k, P_2^k$ follow from the same arguments as in Propositions~\ref{wings} and~\ref{merge} and the arguments above.  
\end{proof}

\subsection{Non-loose torus knot with \texorpdfstring{$\tb\geq pq$}{tb >= pq}}\label{classgeqpq}

In this section, we will classify non-loose $(p, q)$-torus knots with $\tb \geq pq$ and $\tor = 0$, which stabilize to $L_{P_1,P_2}$ for some $2$-inconsistent $(P_1,P_2)$. 

\begin{proposition}\label{propxwing}
  Let $(P_1,P_2)$ be a $2$-inconsistent pair of paths representing $q/p$. If $pq>0$, then we assume that $(P_1,P_2)$ are not $\pm(P'_1,P'_2)$ in Lemma~\ref{xi1}. Then $L_{P_1,P_2}$ and $L_{-P_1,-P_2}$ contribute an infinite $\textsf{X}$, that is there are non-loose Legendrian knots $L_-^k$ and $L_+^k$, for $k\in \Z$ with invariants 
  \[
    \tb(L_\pm^k)=k \text{ and } (L_\pm^k)=\pm r_0\mp k
  \]
  for some $r_0$, and such that 
  \[
    S_\pm^i(L_\pm^k)=L_\pm^{k-i} \text{ and } S_\mp(L_\pm^k) \text{ is loose.}
  \]
  See the left-hand side of Figure~\ref{xwingfig}.
\end{proposition}

The classification of $(p,np + 1)$-torus knots with $\tb= np2 + p + 1$  and $(p,-(np-1))$-torus knots with $\tb= -np2+p+1$ was also established in \cite{GeigesOnaran20a}.

\begin{remark}
  We will see in the proof that $L_{P_1,P_2}$ is either $L_+^{pq}$ or $L_-^{pq}$ and the other one is $L_{-P_1,-P_2}$, so the two pairs produce the same knots $L_\pm^k$. 
\end{remark}

\begin{remark}
  Notice that $L_+^{r_0}$ and $L_-^{r_0}$ both have rotation number zero, but are not equivalent since they behave differently under stabilization. 
\end{remark}
  
\begin{figure}[htbp]
\begin{overpic}{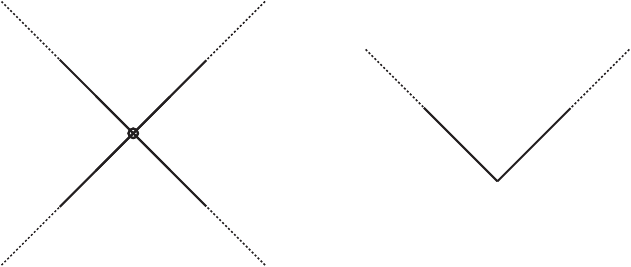}
\end{overpic}
\caption{On the left are the Legendrian realizations of a $(p,q)$-torus knot related to the $2$-inconsistent pairs $\pm(P_1,P_2)$ when $(P_1,P_2)\not=\pm(P_1',P_2')$ (see Lemma~\ref{xi1}). On the right we see the same when $(P_1,P_2)=\pm(P'_1,P'_2)$. For each integral point, 
there is a unique non-loose Legendrian representative except for the crossing point of the $\textsf{X}$. There are exactly two non-loose Legendrian representatives on that point.}
\label{xwingfig}
\end{figure}

\begin{remark}
  From the paragraph before Proposition~\ref{alldiamonds} we see that the excluded $2$-inconsistent pair of paths in the proposition is compatible with the pair of paths whose signs are all the same and hence by Lemma~\ref{sdt+xi1} we know the contact structure is $\xi_1$.
\end{remark}

\begin{proof}
  From Proposition~\ref{wings} we get the Legendrian knots $L_\pm^k$ for $k\leq pq$ with the desired properties. We now recall that Lemma~\ref{lem:thickenornot} says that there is a tight contact structure $\xi$ on the complement $C$ of the $(p,q)$-torus knot that has convex boundary with two dividing curves of slope $\infty$ and no convex torsion such that adding a basic slice $(T^2\times I,pq,\infty)$ to $(C,\xi)$ will result in the complement of a standard neighborhood of $L_{P_1,P_2}$. Suppose this basic slice was negative. Then given any integer $n>pq$ we can factor $T^2\times [0,1]$ into negative basic slices given by the path $pq, \ldots, n-1, n, \infty$ in the Farey graph. That is there is a convex torus $T_n$ in $T^2\times[0,1]$ with two dividing curves of slope $n$. This torus separates $C\cup T^2\times [0,1]$ into two pieces, one, denoted $C_n$, is diffeomorphic to $C$ and clearly the complement of a Legendrian $(p,q)$-torus knot $L_-^n$ with $\tb = n$. Moreover, since the complement of $L_{P_1,P_2}=L_-^{pq}$ is obtained by attaching $n-pq$ negative basic slices to $C_n$, we see that $L_{P_1,P_2}=S_-^{n-pq}(L_-^n)$. Similarly, all the tori $T_k$ for $k\in [pq, n]\cap \Z$ give rise to Legendrian knots $L_-^k$ with the desired properties. Since $n$ was arbitrary we see that we have constructed $L_-^k$ for any $k \in \mathbb{Z}$. Notice that one attaches a positive basic slice to $C_n$ with dividing slopes $n$ and $n-1$, then in the result we have a contact structure on $T^2\times [0,1]$ given by the path $\infty, n, n-1$ in the Farey graph and the signs on each edge are different. Since the path can be shortened the contact structure is overtwisted. Thus a positive stabilization of $L_-^n$ is loose. So we have constructed $L_-^k$ with the desired properties for all $k$. Similarly we can get the $L_+^k$ from $L_{-P_1,-P_2}$.
\end{proof}

\begin{proposition}\label{inftyV}
  Suppose that $pq>0$ and let $(P_1,P_2)$ be the pair of decorated paths $(P'_1,P'_2)$ in Lemma~\ref{xi1}. The Legendrian knots $L_{P_1,P_2}$ and $L_{-P_1,-P_2}$ contribute a $\textsf{V}$ of non-loose Legendrian $(p,q)$-torus knots in $\xi_{P_1,P_2}$. That is, there are non-loose Legendrian knots $L_\pm^k$ for $k> pq-r_0$, where $r_0=|R(P_q,P_2)|$ (see Lemma~\ref{computer}), and $L^{pq-r_0}$ in $\xi_{P_1,P_2}$, with invariants
  \[
    \tb(L_\pm^k)=k, \text{ and } \rot(L_\pm^k)=(\pm pq \mp r_0) \mp k
  \]
  \[
    \tb(L^{pq-r_0})=pq-r_0, \text{ and } \rot(L^{pq-r_0})=0,
  \]
  such that  
  \[
    S_\pm(L_\pm^i)=L_\pm^{i-1} \text{ and } S_\pm(L_\pm^{pq-r_0+1})=L^{pq-r_0},
  \]
  and 
  \[
    S_\mp(L_\pm^i) \text{ and } S_\pm(L^{pq-r_0}) \text{ are loose}.
  \]
  See the right-hand side of Figure~\ref{xwingfig}.
\end{proposition}

The classification for the right-handed trefoil with $\tb=7$ was also established in \cite{GeigesOnaran20a}.

\begin{remark}
  Notice that $L_{P_1,P_2}$ is either $L_+^{pq}$ or $L_-^{pq}$ and the other one is $L_{-P_1,-P_2}$. So $(P_1,P_2)$ and $(-P_1,-P_2)$ determine the same $\textsf{V}$. Moreover, it is also clear that the Legendrian knots in the $\textsf{V}$ are determined by their $\tb$ and $\rot$.
\end{remark}

\begin{proof}
  The Legendrian knots $L_\pm^k$ for $k\leq pq$ are given by Proposition~\ref{dmerge} and those for $k>pq$ are founds exactly as in the proof of Proposition~\ref{propxwing}. 
\end{proof}

\subsection{The extra torus knot when \texorpdfstring{$pq<0$}{pq < 0}}\label{classextra}
By Lemma~\ref{lem:>pq}, when $pq < 0$ there exists one extra contact structure $\xi_e$ in $\Tight_0(C;|pq| - |p| - |q|)$. If we glue a tight solid torus to $C$ to obtain $S^3$, then the added solid torus is a standard neighborhood of a non-loose Legendrian knot $L_e$ with $\tor = 0$. We now study the properties of this extra Legendrian knot. 

\begin{proposition} \label{prop:extra}
  Suppose $pq < 0$. Let $(P_1,P_2)$ be the paths describing a Legendrian $L_{P_1,P_2}$ such that the edges in $P_1$ have only positive signs and the edges in $P_2$ have only negative signs. Let $\xi_{p,q}$ be the contact structure supported by the open book with binding the $(p,q)$-torus knot. We have the following:

  \begin{enumerate}
    \item The transverse push-off of $L_{P_1,P_2}$ is the binding of an open book supporting $\xi_{p,q}$.
    \item $d_3(\xi_{p,q}) = |pq|-|p|-|q|+1$. 
    \item In $\xi_{p,q}$, there are non-loose Legendrian knots $L_\pm^i, i\in \Z$ such that 

    \begin{enumerate}
      \item $\tb(L_\pm^i)=i, \rot(L_\pm^i)=\pm(|pq|-|p|-|q|)\mp i$, 
      \item $L_\pm^i=S_\pm(L_\pm^{i-1})$, $S_\mp(L_\pm^i)$ is loose,  and 
      \item $L_\pm^{pq}=L_{\mp P_1,\mp P_2}$. 
    \end{enumerate}

    \item The extra Legendrian $L_e$ is in the contact manifold $(S^3,\xi_{p,q})$ such that 

    \begin{enumerate}
      \item $\tb(L_e)=|pq|-|p|-|q|$, $\rot(L_e)=0$, and 
      \item $S_\pm(L_e)=L_\pm^{|pq|-|p|-|q|-1}$.
    \end{enumerate}

  \end{enumerate}
\end{proposition}

\begin{proof}
  Item~(1) and~(2) are the content of Lemma~\ref{standardstructures} and its proof, while Item~(3) is Proposition~\ref{propxwing} except for the computation of the rotation numbers which will be done below. So we are left to check Item~(4).

  In Lemma~\ref{lem:thickening-negative}, we saw that when $pq<0$, there is an extra contact structure $\xi_e\in \Tight_0(C;|pq|-|p|-|q|)$ such that all convex tori parallel to $\partial C$ have dividing slope $|pq|-|p|-|q|$. Thus, if we glue a solid torus to $C$ and extend $\xi_e$ (there is a unique way to do this), then we get a Legendrian knot $L_e$ with standard neighborhood the glued in solid torus. Clearly $\tb(L_e)=|pq|-|p|-|q|$. By Lemma~\ref{extraisUT}, we also know that $\xi_e$ is universally tight and remains so after gluing any amount of convex torsion. Thus if $\xi_\pm$ is the result of adding a $\pm$-basic slice in $\Tight^{min}(T^2\times[0,1]; \infty, |pq|-|p|-|q|)$ to $\xi_e$, we know it is tight. Moreover, we may factor $\xi_\pm$ into a contact structure $\xi_\pm^i\in \Tight_0(C;i)$ and a $\pm$-basic slice with slopes $\infty$ and $i$ for $i< |pq|-|p|-|q|$. Clearly $\xi_\pm^i$ is the complement of a non-loose Legendrian knot $\widetilde L_\pm^i$ and $\widetilde L_\pm^i$ is a $(|pq|-|p|-|q|-i)$-fold $\pm$-stabilization of $L_e$. Notice that $\widetilde L_\pm^{pq}$ are non-loose Legendrian knots whose complements are universally tight and remain so after adding any amount of convex torsion. Thus by the proof of Lemma~\ref{standardstructures}, we know that $\widetilde L_\pm^{pq}$ is equivalent to $L_\pm^{pq}$ and thus all the $\widetilde L_\pm^i$ for $i<|pq|-|p|-|q|$ are equivalent to $L_\pm^i$ by Lemma~\ref{propxwing} (indeed we know there are only $2n(p,q)$ non-loose knots with $\tor=0$ having these invariants and only one $L_\pm^i$ can stabilize to $L_\pm^{pq}$ so $\widetilde L_\pm^i$ must agree with this Legendrian knot). Since we know that $\rot(L_-^{pq})=-\rot(L_+^{pq})$ we see that $\rot(L_e)$ must be zero. This establishes Item~(4). 

  The computation of the rotation numbers for the $L_\pm^i$ now follows since we know the rotation number of $L_e$ and how it relates to the $L_\pm^i$. 
\end{proof}

\subsection{The Giroux torsion of the examples above}\label{nogt}
In this section, we will see that all the examples constructed in Section~\ref{classlesspq} have no convex torsion in their complement unless $pq>0$ and we are in $\xi_{pq-p-q}$, in which case some of the Legendrian knots have half convex torsion.

\begin{remark}
  For the Legendrian knots discussed in Section~\ref{classgeqpq} and~\ref{classextra} that have $\tb\geq pq$ we already know they have no convex torsion in their complement because their complements are in $\Tight_0(C;n)$ for some $n\geq pq$ which by definition have no convex torsion.
\end{remark}

\begin{proposition}\label{shownotor}
  Given a pair of decorated paths $(P_1,P_2)$, any non-loose stabilization of $L_{P_1,P_2}$ has $\tor=0$, unless $pq>0$ and $(P_1,P_2)$ is the one from Lemma~\ref{standardstructures}. In the latter case, the non-loose stabilizations of $L_{P_1,P_2}$ will have $\tor = 0$ if $\tb>pq-p-q$, and $\tor = 1/2$ if $\tb\leq pq-p-q$. 
\end{proposition}

\begin{proof}
  We consider three cases: first when $(P_1,P_2)$ is $2$-inconsistent but not totally $2$-inconsistent, then when $(P_1,P_2)$ is $2$-consistent, and finally when $(P_1,P_2)$ is totally $2$-inconsistent. 

  We deal with the first case. By possibly replacing $(P_1,P_2)$ by $(-P_1,-P_2)$ if necessary, we can assume that $S^k_-(L_{P_1,P_2})$ is non-loose and $S_+(L_{P_1,P_2})$ is loose. (When we consider the paths with opposite signs the role of $\pm$ stabilizations is reversed). Let $(C,\xi)$ be the complement of $S^k_-(L_{P_1,P_2})$ and assume $\xi$ contains half convex torsion. By Lemma~\ref{lem:<pq} and Lemma~\ref{lem:finite}, we can split $C$ into $C'$ and $T^2\times [0,1]$ where $\xi|_{C'} \in \Tight_0(C;pq)$ and $\xi|_{T^2\times [0,1]}$ has a convex torsion layer in it. Since $\xi$ is tight and is obtained from $\xi|_{C'}$ by attaching a convex torsion layer, we know that $\xi|_{C'}$ must be associated to a totally $2$-inconsistent pair of paths $(P'_1,P_2')$ by Lemma~\ref{lem:overtwisted}. Thus we can add an arbitrarily amount of convex torsion to $\xi|_{C'}$ and the result is still tight by Lemma~\ref{lem:staytight}. But this, of course, implies that we can add an arbitrary amount of convex torsion to $\xi$ and the result is still tight, which contradicts Lemma~\ref{lem:overtwisted}. Thus $S^k_-(L_{P_1,P_2})$ has $\tor = 0$. 

  Now consider a pair of decorated paths $(P_1,P_2)$ that is $2$-consistent. We note that $(P_1,P_2)$ is compatible with a $2$-inconsistent pair of paths $(P'_1,P'_2)$ that is not totally $2$-inconsistent (see Section~\ref{subsec:pathsinFG}). Moreover, Propositions~\ref{merge} and~\ref{alldiamonds} say that any stabilization of $L_{P_1,P_2}$ can be further stabilized to be a stabilization of $L_{P'_1,P'_2}$ and since the latter does not have any convex torsion in its complement, neither does the former. 

  We are left to consider totally $2$-inconsistent pairs of paths, and the proof of Lemma~\ref{postorsion} gives the result in this case. 
\end{proof}

\subsection{Non-loose torus knots with convex torsion}\label{classmoretorsion}
We begin by noticing that all non-loose torus knots have finite torsion.

\begin{lemma}\label{lem:finite}
  If $L$ is a non-loose Legendrian torus knot, then $\tor(L) < \infty$. 
\end{lemma}

\begin{proof}
  The idea of the proof is to realize the complement of a standard neighborhood of $L$ as an element of $\Tight_l(C;s)$ for some $l \in \tfrac12\N$ and then use Lemma~\ref{ClaimA} to conclude that $\tor(L)$ is finite.

  Suppose $\tor(L) \neq 0$. Then we can stabilize or destabilize $L$ to make $\tb(L) = pq$. Let $C$ be the complement of a standard neighborhood of $L$. As in Section~\ref{knotcomp}, we can decompose $C$ into $V_1 \cup (S^1 \times P) \cup V_2$ where $P$ is a pair of pants and $V_1$, $V_2$ are solid tori. We use the coordinate system $\mathcal{F}_2$ from Section~\ref{knotcomp} so that (a push-off of) $L$ is considered as a $0$-twisting vertical Legendrian curve in $S^1 \times P$. Use this $0$-twisting vertical Legendrian curve to thicken $V_1$ and $V_2$ so that their dividing slopes become $0$. Let $T_1$, $T_2$, and $T_3$ be the boundary of $S^1 \times P$ and perturb them so that $\partial P$ is a collection of ruling curves. After that, perturb $P$ to be convex and there exist two possible dividing sets on $P$ as shown in Figure~\ref{fig:pants2}.

  We will first show that the dividing set shown in the first drawing of Figure~\ref{fig:pants2} results in an overtwisted contact structure. In the first drawing of Figure~\ref{fig:pants2}, we can find a bypass for $T_2$ and thicken $V_2$ so that the dividing slope becomes $\infty$. However, $V_2$ contains a convex torus with slope $(q/p)^c$ measured in the coordinate system $\mathcal{F}_1$, which is $\infty$ measured in the coordinates system $\mathcal{F}_2$. Thus $V_2$ contains a half convex torsion after thickening, and we can find an overtwisted disk in $C$.

  Thus, the dividing set on $P$ should be the one shown in the second drawing of Figure~\ref{fig:pants2}. Choose a $0$-twisting vertical Legendrian curve in $S^1 \times P$, which is in an $I$-invariant neighborhood of $T_1$. Then we can find a convex torus $T$ which contains this curve and is smoothly isotopic to $T_3$, and its ruling curve sits on $P$ and intersects the dividing curves at two points as shown in Figure~\ref{fig:pants2}. Now cut $S^1 \times P$ along the torus $T$ and we obtain a contact structure on $S^1 \times P$ with boundary slope $0$ and the dividing set on $P$ being as shown in the first drawing of Figure~\ref{fig:sdisk}. Since $T$, $T_1$ and $T_2$ co-bound an $S^1\times P$ and from Lemma~\ref{all000}, there exists a unique tight contact structure on this $S^1 \times P$ up to boundary twisting. Let $C'$ be the union of this $S^1 \times P$ and $V_1$, $V_2$. Then $C$ is decomposed into $C'$ and a finite convex $l$-torsion layer for some $l \in \frac12\N$. Clearly, the contact structure on $C'$ is in $\Tight_0(C;pq)$. Thus the contact structure on $C$ is in $\Tight_l(C;pq)$, and by Lemma~\ref{ClaimA}, the amount of convex torsion is exactly $l$, so $\tor(L)$ is finite. 
\end{proof}

\begin{figure}[htbp]
\begin{overpic}{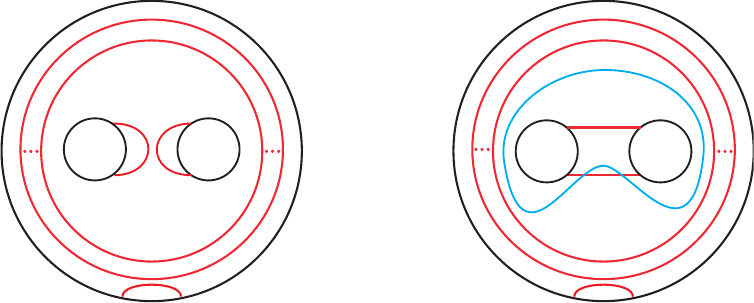}{\small
  \put(32,50){$T_1$}
  \put(82,50){$T_2$}
  \put(20,5){$T_3$}}
\end{overpic}
\caption{Some possible dividing curves on the pair of pants $P$. The blue curve is a Legendrian curve.}
\label{fig:pants2}
\end{figure}

Next, we show that all non-loose torus knots with convex torsion in their complement are obtained from the knots $L_{P_1,P_2}$ for some pair of totally 2-inconsistent paths $(P_1,P_2)$. To do so, we recall that by Proposition~\ref{wings} either $L_{P_1, P_2}$ or $L_{-P_1,-P_2}$ remains non-loose after arbitrarily many negative stabilizations, and the other becomes loose after a single negative stabilization. From our discussion in Section~\ref{lutzsection} we know that the transverse push-off of the one that remains non-loose under negative stabilizations will be a non-loose transverse knot. We call this transverse knot the \emph{transverse knot associated to $(P_1,P_2)$}.

\begin{proposition}\label{moretorsion}
  Let $L$ be any non-loose $(p,q)$-torus knot in $(S^3,\xi)$ with $\tor(L) = n$. Then there is some pair of totally 2-inconsistent paths $(P_1,P_2)$ representing $q/p$ such that the complement of a standard neighborhood of $L$ is obtained from the complement of a standard neighborhood of $L_{P_1,P_2}$ by attaching a basic slice $(T^2\times [0,1]; \infty, pq)$ and then attaching a convex $(n - 1/2)$-torsion layer, and finally a basic slice $(T^2\times[0,1]; \tb(L),\infty)$. In particular, the complement of a standard neighborhood of $L$ is a contact structure in $\Tight_n(C; \tb(L))$ and any element in $\Tight_n(C;\tb(L))$ gives a non-loose Legendrian knot with $\tor(L) = n$. 

  In addition, $\xi=\xi_{P_1,P_2}$ if $n$ is an integer, or $\xi$ is obtained from $\xi_{P_1,P_2}$ by a half Lutz twist on the non-loose transverse knot associated to $(P_1,P_2)$. 
\end{proposition}

\begin{proof}
Let $C$ denote the complement of the $(p,q)$-torus knot. Towards the end of Section~\ref{contactonCsec} we defined $\Tight_n(C;k)$ to be the space of contact structures on $C$ with boundary being convex with two dividing curves of slope $k$ and exactly $n$ convex torsion along a torus parallel to the boundary. 
We recall that Lemma~\ref{postorsion} says that for any $(p,q)$-torus knot and integer $k$ 
  \[
    |\Tight_n(C;k)|=2|(a_1+1)\cdots(a_{m-1}+1)||(b_1+1)\cdots (b_{n-1}+1)|
  \]
for any $n\in \frac 12 \N$, where the $a_i$ and $b_i$ are defined in Section~\ref{introgeneral}. This is the same as the number of totally $2$-inconsistent pairs of paths representing $q/p$. 

The contact structures in $\Tight_n(C;k)$ were constructed as described in the statement of the proposition. Moreover, any Legendrian representative of the $(p,q)$-torus knot with $\tb=k$ and convex $l$-torsion will have complement in $\Tight_n(C;k)$ and any element of $\Tight_n(C;k)$ will give such a Legendrian knot by gluing a solid torus with meridional slope $\infty$ to the contact structure and extending the contact structure over the solid torus to be the unique tight contact structure on the solid torus with the given boundary conditions. 

The discussion in Section~\ref{lutzsection} shows that $\xi$ is the claimed contact structure. 
\end{proof}

\subsection{Proof that the algorithm gives a complete classification}\label{justifyalgorithm}
We will now show that the algorithm from Section~\ref{thealgorithm} does indeed give all non-loose torus knots. 

We first consider non-loose Legendrian $(p,q)$-torus knots with $\tor = 0$. We first note that any such knot with $\tb = pq$ will be of the form $L_{P_1,P_2}$ for some pair of decorated paths $(P_1,P_2)$ representing $q/p$ by Lemma~\ref{lem:=pq}. Moreover, if a non-loose Legendrian knot with $\tor = 0$ has $\tb<pq$, then it will stabilize to one with $\tb=pq$ by Lemma~\ref{lem:<pq}. Thus we know that it will be in a Wing or a Diamond of $L_{P_1,P_2}$ for some decorated pair $(P_1,P_2)$ by Propositions~\ref{wings} and~\ref{alldiamonds}. 

Now if a non-loose Legendrian $(p,q)$-torus knot with $\tor = 0$ has $\tb=n>pq$, then its complement is in $\Tight_0(C;n)$ and hence is a destabilization of some $L_{P_1,P_2}$ for some decorated pair of paths by Lemma~\ref{lem:>pq}, and moreover they must be $2$-inconsistent by Lemma~\ref{lem:thickenornot}. Thus we see that such a knot must be in an infinite $\textsf{X}$ or $\textsf{V}$ from Propositions~\ref{propxwing} and~\ref{inftyV}. 

These observations show that the classification algorithm in the generic case (Steps 1 and 3 of the algorithm) give the desired result except when $pq>0$ and we are in the situation where $P_1$ has all one sign and $P_2$ has all the other sign. The only things that might not be immediately clear is the rotation numbers of $L^{pq}_{k,\pm}$. However, those easily follow from the computation of $R(P_1,P_2)$ for the $2$-inconsistent paths according to Lemma~\ref{computer}, and the proofs of Proposition~\ref{wings} and Proposition~\ref{alldiamonds} that indicates when compatible pairs of decorated paths stabilize to become the same. In the excluded case, we will not have an infinite $\textsf{X}$ associated to $L_{P_1,P_2}$ and $L_{-P_1,-P_2}$ with $\tor = 0$. Only the knots in the $\textsf{X}$ with $\tb>pq-p-q$ will have no convex torsion. Those with $\tb\leq pq-p-q$ will have half convex torsion by Proposition~\ref{shownotor}.

\begin{remark}
  To see that the generic $\textsf{X}$-wings are as depicted in Figure~\ref{fig-genericXwing}, we need to see that the crossing of the $\textsf{X}$ is above $pq$ when $pq<0$ and otherwise is below $pq$. This is actually clear by considering the inequalities in Theorem~\ref{looseLegbound} (shown graphically in Figure~\ref{geogfig}). Indeed, suppose the crossing of the $\textsf{X}$ was below $pq$ when $pq<0$ then the top part of the $\textsf{X}$ would not fit through the allowable range when $\tb=0$ (we see that when $\tb=0$ we must have $\rot$ between $-|pq|+|p|+|q|$ and $|pq|-|p|-|q|$). We can similarly argue for $pq>0$.
\end{remark}

In the exceptional cases (Step 2), we first consider $pq>0$. In this case, the above discussion shows that in $\xi_1$, we have an infinite $\textsf{V}$ together with some other diamonds. The only thing to consider is the claimed values for the rotation numbers. To see this we first consider the pair of paths $(P_1,P_2)$ with all signs the same. We saw in Proposition~\ref{dmerge} that the diamonds associated to $L_{P_1,P_2}$ and $L_{-P_1,-P_2}$ have a common lowest vertex that has $\tb=pq-p-q+2$. Now for the $2$-inconsistent pairs of decorated paths that are compatible with these paths, we see that they must be stabilized either strictly positively or strictly negatively to get to this lowest vertex (see Proposition~\ref{alldiamonds}). Thus we get the desired rotation numbers for these two Legendrian knots and the rotation numbers for the others follow from the proof of Proposition~\ref{alldiamonds}. 

We now consider the exceptional case when $pq<0$. Here the classification follows directly from the above discussion and Proposition~\ref{prop:extra}. 

Finally, the classification of non-loose torus knots with convex torsion in their complement follows directly from Proposition~\ref{moretorsion} and the fact that the non-loose representatives with half Giroux torsion come from applying a half Lutz twist along a transverse push-off $T$ of $L_-^i$ (see Section~\ref{classwogt} for the definition of this knot) in the contact structure $\xi_{P_1,P_2}$. (Notice $L_-^i$ has the same transverse push-off for any $i \in \mathbb{Z}$.) So the contact structure $\xi'_{P_1,P_2}$ discussed in the algorithm is the result of a half Lutz twist along $T$.


In Section~\ref{lutzsection}, we recalled how a half Lutz twist along a transverse knot affects the $d_3$-invariant of the contact manifold. 
Using this, we see that
\[
  d_3(\xi'_{P_1,P_2}) = \begin{cases} d_3(\xi_{P_1,P_2}) + |R(P_1,P_2)| - pq \, &pq>0,\\ d_3(\xi_{P_1,P_2}) - |R(P_1,P_2)| - pq \, &pq<0.\end{cases}
\]

\section{General results of non-loose torus knots}

Theorem~\ref{gen1} claims that any Legendrian $(p,q)$-torus knot destabilizes if $\tb\not =pq$ except for one with $\tb=|pq|-|p|-|q|$ when $pq<0$ and some with $\tb=pq$ do but others do not. 

\begin{proof}[Proof of Theorem~\ref{gen1}]
  This follows directly from Lemma~\ref{lem:<pq} and Lemmas~\ref{lem:thickenornot} and \ref{lem:thickening-negative}.
\end{proof}

The parity of the $d_3$-invariants of contact structures supporting non-loose torus knots is given in Theorem~\ref{parity} which we now prove.

\begin{proof}[Proof of Theorem~\ref{parity}]
We begin by noticing that when $pq>0$ the contact structure $\xi_0$ is obtained from the standard tight contact structure on $S^3$ by a full-Lutz twist on the maximal self-linking number transverse $(p,q)$-torus knot and $\xi_{-pq+p+q}$ is obtained on the same transverse knot by a half Lutz twist. The claimed result then follows from Proposition~\ref{shownotor} and the discussion of Lutz twist in Section~\ref{lutzsection}.

If $(P_1,P_2)$ are a pair of decorated paths such that $\xi_{P_1,P_2}$ supporting a non-loose $(p,q)$-torus knot with $\tb=pq$ (and all contact structures supporting non-loose torus knots have such a non-loose Legendrian knot), then we can draw a surgery diagram for $\xi_{P_1,P_2}$ as described in Section~\ref{sec:FareytoSurgery} and then use Equation~\eqref{formula:d3} to compute its $d_3$-invariant. Notice that in that equation the only term that depends on the decorations on $(P_1,P_2)$ is $c^2$. Recall that $c$ is the vector of rotation numbers of the link in the surgery presentation of $\xi_{P_1,P_2}$ and $c^2$ is computed with the intersection pairing given by the linking matrix $M$ of the surgery diagram. The class $c$ is a characteristic element of the pairing $M$ (see the proof of Corollary~3.6 in \cite{DingGeigesStipsicz04}, where they show that $c$ is related to $c_1$ of a complex structure, which is known to be characteristic by $9q$ where $q$ is the number of $(+1)$-contact surgeries in the diagram. Since $q=2$ in our case we see $c$ is characteristic). Now since the surgery diagram presents $S^3$ we know that $M$ is unimodular, we know that $c^2$ is congruent to the signature of $M$ modulo $8$. 
Hence all the decorated paths $(P_1,P_2)$ have the same $d_3$-invariant modulo $2$. 

Since for $pq>1$ we know there are always non-loose $L_{P_1,P_2}$ in $\xi_1$, see Section~\ref{diamondsforxi1}, we know that all $d_3$-invariants of contact structures supporting non-loose Legendrian knots with $\tor = 0$ must have odd $d_3$-invariants. Moreover, those with $\tor = n \in \mathbb{N}$ will have the same $d_3$-invariants since full Lutz twists do not change the $d_3$-invariant and those with $\tor = (2n-1)/2$ will have even $d_3$-invariants since half Lutz twists will change the $d_3$-invariant by the self-linking number of the transverse knot which is Lutz twisted about, see Section~\ref{classwithtorsion}, and we know these are all odd. We have a similar result for $pq<0$ since there is always some non-loose representative with $\tor=0$ in the contact structure $\xi_{|pq|-|p|-|q|+1}$, see Lemma~\ref{standardstructures}. 
\end{proof}

Theorem~\ref{thm:>pq} details all the possible Legendrian knots with $\tor=0$ and $\tb > pq$. 

\begin{proof}[Proof of Theorem~\ref{thm:>pq}]
  This follows directly from the classification given in Section~\ref{justifyalgorithm}, or more specifically Propositions~\ref{propxwing}, \ref{inftyV}, and~\ref{prop:extra}. 

  We are left to show that all $L^i_{\pm,k}$ can be realized as a Legendrian knot in Figure~\ref{fig:tb=pq+m}. We start with $i = pq+1$. First, a simple Kirby calculus shows that $L_-$ and $L_+$ are smooth $(p,q)$-torus knots. See Figure~\ref{fig:kirby}. 
\begin{figure}[htbp]{\tiny
  \vspace{0.1cm}
\begin{overpic}[scale=1,tics=10]{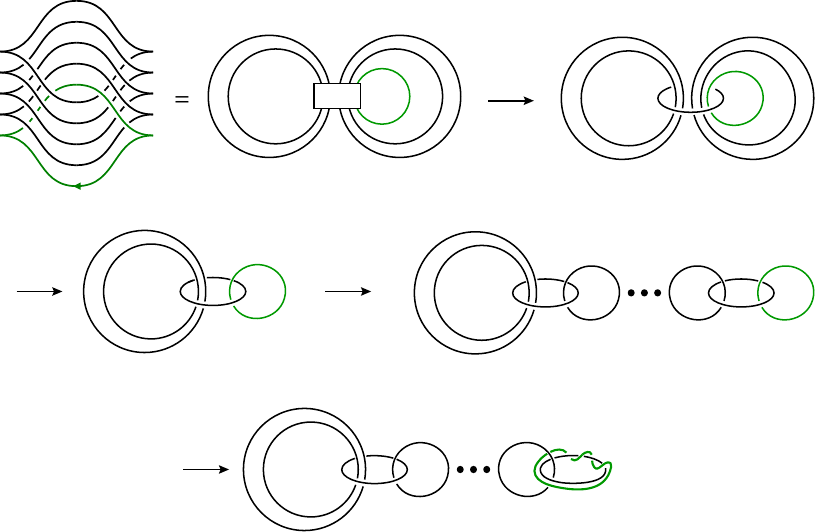}
  \put(-7,245){$(-\frac{p}{p'})$}
  \put(-33,223){$(-\frac{q}{q-q'})$}
  \put(-20,208){$(+1)$}
  \put(-20,198){$(+1)$}

  \put(87,235){$-\frac{p+p'}{p'}$}
  \put(115,215){$\frac{2q-q'}{q'-q}$}
  \put(201,220){$0$}
  \put(213,235){$0$}
  \put(157,208){$-1$}
    
  \put(265,235){$-\frac{p}{p'}$}
  \put(285,216){$\frac{q}{q'-q}$}
  \put(369,220){$1$}
  \put(383,235){$1$}
  \put(310,206){$1$}
  \put(233,215){$\text{blow up}$}
    
  \put(33,140){$-\frac{p}{p'}$}
  \put(55,125){$\frac{q}{q'-q}$}
  \put(75,113){$-1$}
  \put(0,120){$\text{blow down}$}

  \put(193,140){$-\frac{p}{p'}$}
  \put(215,125){$\frac{q}{q'-q}$}
  \put(235,113){$-2$}
  \put(257,125){$-2$}
  \put(278,133){$-2$}
  \put(330,133){$-2$}
  \put(352,125){$-1$}
  \put(155,120){$\text{blow up}$}

  \put(113,57){$-\frac{p}{p'}$}
  \put(133,40){$\frac{q}{q'-q}$}
  \put(152,28){$-2$}
  \put(174,40){$-2$}
  \put(196,48){$-2$}
  \put(246,48){$-2$}
  \put(270,42){$-1$}
  \put(75,35){$\text{handle slide}$}
\end{overpic}}
\caption{Various Kirby diagrams of a $(p,q)$-torus knot. On the upper left is a contact surgery presentation for the $(p,q)$-torus knot with $\tb=pq$.}
\label{fig:kirby}
\end{figure}
If we perform a Legendrian surgery on $L_\pm$ in the first row of Figure~\ref{fig:tb=pq+m}, then we obtain tight $L(p,-q) \# L(q,-p)$. Thus $\tor(L_\pm)=0$. Since $S^3_{pq}(T_{p,q}) \cong L(p,-q) \# L(q,-p)$ and $S^3_m(T_{p,q}) \not\cong S^3_n(T_{p,q})$ for $m \neq n$, the smooth surgery coefficient on $L_\pm$ 
must be $pq$, which is $\tb-1$. Thus $\tb(L_\pm)=pq+1$. Clearly there exist $n(p,q)$ different $L_-$ and the same for $L_+$. Now suppose some $L_-$ and $L_+$ are equivalent. This implies that the Legendrian surgery on them results in the same contact structure. We can calculate the rotation numbers of these $L_\pm$ using the Formula~(\ref{formula:rot}). Clearly, $\mathbf{rot}, M, \mathbf{lk}$ are the same for both $L_\pm$. The only difference is $r_0$. Thus $\rot(L_-) \neq -\rot(L_+)$, so they cannot be equivalent. Thus the first row of Figure~\ref{fig:tb=pq+m} represents $2n(p,q)$ non-loose torus knots with $\tb = pq + 1$ and $\tor=0$. 

Now consider $i = pq + m$ for $m>1$. If we perform a Legendrian surgery on $L_\pm$ in the second row of Figure~\ref{fig:tb=pq+m}, then we obtain Stein fillable contact manifolds and it is diffeomorphic to $S^3_{pq+m-1}(T_{p,q})$. Thus the smooth surgery coefficient on $L_\pm$ is $pq+m-1$ and $\tb(L_\pm) = pq + m$. The Legendrian surgery on each $L_\pm$ produces $n(p,q)$ different contact structures distinguished by the $\spinc$ structures on the Stein filling given by the surgery diagrams in Figure~\ref{fig:tb=pq+m} without the $(+1)$-surgery components, see~\cite[Theorem~1.2]{LiscaMatic97}. Finally, each $L_-$ and $L_+$ are not equivalent since we can calculate their rotation numbers as above and they are different. Thus the second row of Figure~\ref{fig:tb=pq+m} represents $2n(p,q)$ non-loose torus knots with $\tb = pq + m$ and $\tor = 0$. Since we know there are exactly $2n(p,q)$ of non-loose $(p,q)$-torus knots with $\tb>pq$, except with $pq<0$ and $\tb=|pq|-|p|-|q|$, we have established the theorem except in the exceptional case. 

In the case that $pq<0$ and $\tb=|pq|-|p|-|q|$, we know there are $2n(p,q)+1$ non-loose Legendrian knots, the extra one we are denoting $L_e$. We claim that none of the surgery diagrams in Figure~\ref{fig:tb=pq+m} gives $L_e$, given this we have completed the proof of Theorem~\ref{thm:>pq}. To see this, recall that $S_\pm(L_e)$ is non-loose for either choice of stabilization, but for all the other non-loose Legendrian knots with $\tb=|pq|-|p|-|q|$ one sign of stabilization will give a loose knot, while the other will remain non-loose. So we will establish our claim by showing that all the Legendrian knots in Figure~\ref{fig:tb=pq+m} become loose after one stabilization of the correct sign. 

To this end, consider the lower left diagram in Figure~\ref{fig:tb=pq+m}. Let $S$ be a standard neighborhood of a $\tb=-1$ unknot that contains the bottom Legendrian unknot with contact framing $(+1)$ which we will call $K$. We can assume that $L_+$ is contained in this neighborhood too. We will show that the complement of $S_-(L_+)$ in $S$ is overtwisted and thus the complement of $S_-(L_+)$ in $S^3$ is also overtwisted.  Let $S'$ be a standard neighborhood of $K$ in $S$, notice that $L_+$ sits on $\partial S'$ as a Legendrian divide. So $S\setminus S'$ is a positive basic slice with dividing slopes $-2$ and $-1$.  When we perform contact $(+1)$ surgery on $K$ we remove $S'$ from $S$ are replace it with a solid torus $S_{-1}$ with lower meridian ${-1}$ and dividing slope $-2$, call the result $S_K$. Clearly $S_K$ is overtwisted, but when we remove $S_{-1}$ from $S_K$ we get a tight basic slice. Notice that $S_{-1}$ is a standard neighborhood of $L_+$ in $S_K$ (since it sits on $\partial S_{-1}$ as a Legendrian divide). If we negatively stabilize $L_+$ in $S_{-1}$ the result will have a standard neighborhood with boundary slope $-\infty$ (and lower meridian $-1$) and the complement of the standard neighborhood in $S_{-1}$ will be a negative basic slice with dividing slopes $-\infty$ and $-2$. Thus the complement of this neighborhood in $S_K$ will be the union of a positive basic slice with dividing slopes $-2$ and $-1$ and a negative basic slice with dividing slopes $-\infty$ and $-2$. Since the path from $-\infty$ to $-2$ to $-1$ can be shortened and our basic slices have opposite sign, we see that the complement of $S_-(L_+)$ in $S_K$ is overtwisted as claimed. 
\end{proof}


Theorem~\ref{thm:<=pq} gives the number of non-loose Legendrian knots $\tor = 0$ and $\tb = pq$. 

\begin{proof}[Proof of Theorem~\ref{thm:<=pq}]
  From Lemma~\ref{lem:=pq}, we know that the number of Legendrian $(p,q)$-torus knots with tight complement and $\tor = 0$ and $\tb = pq$ is $m(p,q)$. For $pq>0$ there are no such Legendrian knots in $(S^3,\xi_{std})$ so in this case the number of non-loose such knots is $m(p,q)$. However, by \cite{EtnyreHonda01}, we know that there are $2\lceil q/p\rceil$ such knots in $(S^3,\xi_{std})$ and hence we have the claimed number of non-loose Legendrian knots. The fact that they come from the claimed surgery diagram was shown in Section~\ref{sec:FareytoSurgery}. 
\end{proof}

We discuss that there can be arbitrarily many peaks and deep valleys in Theorem~\ref{bigwings}.  
\begin{proof}[Proof of Theorem~\ref{bigwings}]
  It is clear that one can choose $q/p$ so that there are arbitrarily many continued fraction blocks in $P_1$ and $P_2$ and that these blocks are arbitrarily long except for the first one. According to Proposition~\ref{merge} and~\ref{alldiamonds}, the result about the number of peaks now follows by that we can choose $i$-inconsistent decorated pair of paths for arbitrarily large $i$. According to the proof of Proposition~\ref{merge}, the depth of the valleys is determined by the difference between the lengths of two continued fraction blocks $A_{i-1}$ and $B_i$, or $B_{i-1}$ and $A_i$, which we can make arbitrarily large. 
\end{proof}

In Theorem~\ref{numbersupportingnonloose}, we give an upper bound on the number of overtwisted contact structures supporting non-loose $(p,q)$-torus knots. 

\begin{proof}[Proof of Theorem~\ref{numbersupportingnonloose}]
  Any non-loose Legendrian torus knot with $\tor = 0$ is in a contact structure given by a $2$-inconsistent pair of paths by the discussion in Section~\ref{subsec:pathsinFG} and the classification given in Section~\ref{thealgorithm}. Moreover, there are $2n(p,q)$ of such pairs by Lemmas~\ref{lem:>pq} and~\ref{lem:thickenornot}. Since $(P_1,P_2)$ and $(-P_1,-P_2)$ give the same contact structures, we see an upper bound is $n(p,q)$ as claimed. 

  Now allow non-loose Legendrian knots with any convex torsion. According to Proposition~\ref{moretorsion}, the extra contact structures only come from totally $2$-inconsistent pairs of paths. According to Lemma~\ref{postorsion}, the number of totally $2$-inconsistent pairs of paths is twice what we want. Again, since $(P_1,P_2)$ and $(-P_1,-P_2)$ give the same contact structures, we see an upper bound in the formula is correct. 
\end{proof}

We now establish Theorem~\ref{thm:torsion} about the convex torsion in non-loose torus knot complements. 

\begin{proof}[Proof of Theorem~\ref{thm:torsion}]
  This directly follows from Lemmas~\ref{ClaimA}, \ref{ClaimB} and \ref{lem:finite}.
\end{proof}

We end by considering non-loose transverse knots by giving the proof of Theorem~\ref{gentransverse}.

\begin{proof}[Proof of Theorem~\ref{gentransverse}]
  As noted in \cite[Theorem~2.10]{EtnyreHonda01}, the classification of transverse knots is equivalent to the classification of Legendrian knots up to negative stabilization. So any non-loose transverse knot will be the transverse push-off of some non-loose Legendrian knot. Suppose $\xi$ supports non-loose Legendrian knots with a mountain range given in Figure~\ref{fig-genericXwing}. Since we only need to consider Legendrian knots up to negative stabilization, we only need to consider the lower left of the figure. If the ``wings" are non-trivial (that is, there is more than just an  $\textsf{X}$ in the mountain range), then none of the Legendrian knots can have Giroux torsion. This is because by Lemma~\ref{lem:thickenornot} the $\textsf{X}$ is associated to a $2$-inconsistent decorated pair of paths, but one cannot add torsion to any Legendrian knot in the $\textsf{X}$ unless the decorated pair of paths is totally $2$-inconsistent by Lemma~\ref{lem:overtwisted}. For there to be non-trivial wings in the mountain range, the $2$-inconsistent decorated pair of paths would need to be compatible with a $3$-inconsistent decorated pair of paths. Finally, we notice that a totally $2$-inconsistent decorated pair of paths cannot be 
  compatible with a $3$-inconsistent pair of paths, see Section~\ref{pairsodecorated} for the construction of compatible decorated pairs of paths. 
Thus, in the case there are non-trivial wings, see the transverse push-offs of these Legendrian knots give transverse knots as in Item~\eqref{1} of Theorem~\ref{gentransverse}. 

If the mountain range has just an  $\textsf{X}$, then it might support non-loose knots with convex torsion or not. If there is no convex torsion, then we are in the case above, if there is convex torsion then we know that for every point in the mountain range, there is an infinite number of Legendrian knots with different convex torsion in their complement. Their transverse push-offs will give transverse knots as in Item~\eqref{2} of the theorem. 
\end{proof}

\providecommand{\bysame}{\leavevmode\hbox to3em{\hrulefill}\thinspace}
\providecommand{\MR}{\relax\ifhmode\unskip\space\fi MR }
\providecommand{\MRhref}[2]{%
  \href{http://www.ams.org/mathscinet-getitem?mr=#1}{#2}
}
\providecommand{\href}[2]{#2}

\end{document}